\documentclass[11pt]{amsart}

\usepackage{amssymb}
\usepackage{amsmath}
\usepackage{amsthm}

\newcommand{\abs}[1]{\left| #1\right| }
\newcommand{\diag}{\mathop{\mathrm {diag}}\nolimits}

\newcommand{\Tr}{\mathop{\mathrm {Tr}}\nolimits}
\newcommand{\Ad}{\mathop{\mathrm {Ad}}\nolimits}
\newcommand{\Hom}{\mathop{\mathrm {Hom}}\nolimits}

\newcommand{\Ind}{\mathop{\mathrm {Ind}}\nolimits}

\newcommand{\id}{\mathop{\mathrm {id}}\nolimits}

\newcommand{\Wh}{\mathop{\mathrm {Wh}}\nolimits}

\newcommand{\ip}[2]{\langle #1,#2\rangle }
\newcommand{\efct}[2]{s(#1,#2)}
\newcommand{\evec}[2]{S_{#1}\left(#2\right)}
\newcommand{\eblock}[1]{\mS_\sigma (#1)}
\newcommand{\hatoplus}{\mathop{\widehat{\bigoplus}}}
\newcommand{\pmat}[3]{\mathcal{C}^{#1}_{#2 #3}}
\newcommand{\Isp}[1]{\mathcal{I}_{#1}}

\newcommand{\hs}{\hspace{0.5mm}}
\newcommand{\mhs}{\hspace{-2mm}}
\newcommand{\bkt}[1]{\left( #1 \right) }
\newcommand{\chid}[2]{\chi_{#1}^{\scriptscriptstyle (#2)} }
\newcommand{\wt}[1]{\gamma^{\scriptscriptstyle #1} }

\newcommand{\gpt}[6]{
\left( \begin{array}{c}
\\[-5.5mm]
\mhs 
\scriptscriptstyle #1 \hs #2 \hs #3 \mhs \\[-2.3mm]
\scriptscriptstyle #4 \hs #5 \\[-2.3mm]
\scriptscriptstyle #6 \\[-0.5mm]
\end{array} \right)}

\newcommand{\cgpt}[3]{
\left( \begin{array}{c}
\\[-6mm]
\mhs \scriptscriptstyle #1 \hs #2 \mhs \\[-2.3mm]
\scriptscriptstyle #3 \\[-1.0mm]
\end{array} \right)}

\newcommand{\lgpt}[6]{
\left( \begin{array}{c}
#1 \ #2 \ #3 \\
#4 \ #5 \\
#6 \\
\end{array} \right)}

\newcommand{\lcgpt}[3]{
\left( \begin{array}{c}
#1 \ #2 \\
#3 \\
\end{array} \right)}

\newcommand{\cp}[6]{
c^{\lambda }_{[#1 #2;#5 #4;#6]}\bkt{#3} 
}

\newcommand{\cpa}[4]{
#1^{#2}_{[#3]}\bkt{#4} 
}

\newcommand{\cpd}[6]{
c^{\widehat{\lambda }}_{[#1 #2;#5 #4;#6]}#3 
}

\newcommand{\cpr}[4]{
r_{[#1 #2;#3 #4]} 
}

\newcommand{\cpra}[2]{
r_{[#1;#2]} 
}

\newcommand{\epe}[2]{
[+#1 #2] 
}
\newcommand{\eme}[2]{
[-#1 #2]
}

\newcommand{\mC}{\mathbf{C}}

\newcommand{\mR}{\mathbf{R}}
\newcommand{\mS}{\mathbf{S}}
\newcommand{\mZ}{\mathbf{Z}}
\newcommand{\me}{\mathbf{e}}
\newcommand{\mm}{\mathbf{m}}
\newcommand{\mmu}{\mathbf{u}}

\newcommand{\ga}{\mathfrak{a}}
\newcommand{\g }{\mathfrak{g}}
\newcommand{\gh}{\mathfrak{h}}
\newcommand{\gk}{\mathfrak{k}}
\newcommand{\gl}{\mathfrak{l}}
\newcommand{\gn}{\mathfrak{n}}
\newcommand{\gp}{\mathfrak{p}}
\newcommand{\gu}{\mathfrak{u}}

\newtheorem{thm}{Theorem}[section]
\newtheorem{lem}[thm]{Lemma}
\newtheorem{prop}[thm]{Proposition}

\theoremstyle{definition}
\newtheorem{defn}[thm]{Definition}
\newtheorem{rem}[thm]{Remark}

\numberwithin{equation}{section}
\theoremstyle{remark}

\begin{document}

\title{The $(\g ,K)$-module structures of 
principal series representations of $Sp(3,\mR )$.}

\author[Tadashi Miyazaki]{Tadashi Miyazaki}
\address{Department of Mathematical Sciences, University of Tokyo}
\email{miyaza@ms.u-tokyo.ac.jp}

\allowdisplaybreaks
%%%%%%%%%%	TEXT START	%%%%%%%%%%

\maketitle
\begin{abstract}
We describe explicitly the whole structures of the $(\g ,K)$-modules 
of principal series representations of $Sp(3,\mR )$. 
We apply this result to determine the holonomic system 
characterizing those Whittaker functions.
\end{abstract}

\section{Introduction}
\label{sec:introduction}
In the investigation of a representation $\pi $ of a reductive 
Lie group $G$, it is a standard method to pass from the original 
$\pi $ to its associated $(\g ,K)$-module. 
Also in many applications, it is important to understand 
various `good' realizations of $\pi $ in some function spaces. 
These functions can be sometimes described as solutions of some differential 
equations, say, the Casimir equations derived from the $(\g ,K)$-module 
structure of the given representation $\pi $. 

For some `small' semisimple Lie groups $G$, the $(\g ,K)$-module 
structures of the standard representations are completely described. 
For example, the description of them for $SL(2,\mR)$ 
is found in standard textbooks, and there are rather complete results 
for some groups of real rank $1$, e.g.\ 
$SU(n,1)$ in \cite{MR0330355} and $Spin(1,2n)$ in \cite{MR0453925}. 
But, for Lie groups of higher rank there are few references 
as far as the author knows. 
It seems to be difficult to describe the whole 
$(\g ,K)$-module structures even for standard representations of 
classical groups of higher rank. 
However, in this paper we consider the case of 
the real symplectic group $Sp(3,\mR )$ of rank $3$, 
and solve this problem generalizing the result of 
the paper \cite{pre_standard_1_2006} 
for $Sp(2,\mR )$ of rank $2$. \nocite{pre_standard_2}

Before describing our situation for $Sp(3,\mR )$, let us explain 
the problem in a more precise form. 
Let $G$ be a real semisimple Lie group and $\g $ the Lie algebra of $G$. 
Fix a maximal compact subgroup $K$ of $G$. 
Since any standard $(\g ,K)$-modules are realized as subspaces of 
$L^2(K)$ as $K$-modules, we investigate the $K$-module structure of 
standard $(\g ,K)$-modules by the Peter-Weyl theorem. 
Because of the Cartan decomposition $\g =\gk \oplus \gp $, 
in order to describe the action of $\g $ or $\g_\mC =\g \otimes_\mR \mC $ 
it suffices to investigate the action of $\gp $ or $\gp_\mC $. 
Therefore, the investigation of the action of $\gp $ 
or $\gp_\mC $ is essential 
to give the explicit $(\g ,K)$-module structure of 
a standard representation. 
To investigate the action of $\gp_\mC $, we compute the linear map 
$\Gamma_{\tau ,i}$ defined as follows. 
Let $(\pi , H_\pi )$ be a standard representation of $G$ with the 
subspace $H_{\pi ,K}$ of $K$-finite vectors.
For  a $K$-type $(\tau ,V_\tau )$ of $\pi $, and 
a nonzero $K$-homomorphism $\eta \colon V_\lambda \to H_{\pi ,K}$, 
we define a linear map 
$\tilde{\eta }\colon \gp_\mC \otimes_\mC V_\lambda \to H_{\pi ,K}$ 
by $X\otimes v \mapsto X\cdot \eta (v) $. 
Then $\tilde{\eta }$ is a $K$-homomorphism with $\gp_\mC $ endowed with 
the adjoint action $\Ad $ of $K$. 
Let $V_\tau \otimes_\mC \gp_\mC \simeq \bigoplus_{i\in I}V_{\tau_i}$ 
be the decomposition into a direct sum of irreducible $K$-modules and 
$\iota_i$ an injective $K$-homomorphism from 
$V_{\tau_i}$ to $V_\tau \otimes_\mC \gp_\mC$ for each $i$. We define a 
linear map $\Gamma_{\tau ,i}\colon \Hom_K(V_\tau ,H_{\pi ,K})\to 
\Hom_K(V_{\tau_i},H_{\pi ,K})$ 
by $\eta \mapsto \tilde{\eta }\circ \iota_i$. 
These linear maps $\Gamma_{\tau ,i}\ (i\in I)$ characterize the action of 
$\gp_\mC $. Our purpose of this paper is to give the explicit expressions 
of $\Gamma_{\tau ,i}$ when $\pi $ is a principal series representation 
of $G=Sp(3,\mR )$. This is described in Theorem \ref{th:main}. 

As an application of the explicit expressions of $\Gamma_{\tau ,i}$, 
we get a system of differential equations satisfied 
by some types of spherical functions. Here we consider only the 
case of the Whittaker functions. 
In order to introduce this application, let us recall the general setting 
of the theory of the spherical functions. 
Fix a closed subgroup $R$ of $G$. 
Take a character $\xi $ of $R$ and consider its
$C^\infty$-induction $C^\infty \Ind_R^G(\xi )$. For an irreducible admissible 
representation $(\pi ,H_\pi )$ of $G$ with  the subspace $H_{\pi ,K}$ of 
$K$-finite vectors, we consider the subspace 
$\Hom_{(\g_\mC ,K)}(H_{\pi ,K},C^\infty \Ind_R^G(\xi ))$ of
 intertwining operators. 
Consider the restriction of elements in this space to specific $K$-type 
as follows. Let $(\tau,V_{\tau})$ be a multiplicity one $K$-type of 
$\pi$ and let $\iota \colon V_{\tau}\to H_\pi $ be a nonzero $K$-homomorphism. 
For $\Phi \in \Hom_{(\g_\mC ,K)}(H_{\pi ,K},C^\infty \Ind_R^G(\xi ))$, 
we can define the function $\phi_{\pi ,\tau^* }$ contained in the space 
$C_{\xi ,\tau^*}^\infty (R\backslash G/K)$ of $V_\tau^* $-valued smooth 
functions on $G$ satisfying $f(rgk)=\xi (r)\tau^* (k)^{-1}f(g)$ for 
all $(r,g,k)\in R\times G\times K$ by 
$\Phi (\iota (v))(g)=\ip{v}{\phi_{\pi ,\tau^* }}\ 
(g\in G,\ v^*\in V_{\tau})$. Here $\tau^*$ means the contragradient 
representation of $\tau $ and $\ip{}{}$ is the canonical pairing of 
$V_{\tau}\times V_{\tau^*} $. 
When $R$ is a maximal unipotent subgroup of $G$ and $\xi $ is a 
unitary character of $R$, the space 
$\Hom_{(\g_\mC ,K)}(H_{\pi ,K},C^\infty \Ind_R^G(\xi ))$
is called the space of Whittaker functionals and $\phi_{\pi ,\tau^* }$ 
is called a Whittaker function. 

In the case of $G=Sp(2,\mR )$, the explicit formulas of Whittaker functions 
for various standard representations are given in the papers \cite{MR2238640}, 
\cite{MR2149916}, \cite{MR1342321}, \cite{MR1362994}, 
\cite{MR1272882} as well as the generalized Whittaker functions 
in the papers \cite{MR2061467}, \cite{MR1768466}. 
In these papers, the explicit formulas of spherical functions are given 
as solutions of the system of differential equations which are obtained from 
shift operators and the Casimir element. Since the Casimir element 
is represented by a composite of the shift operators, 
the utilization of the shift operator is 
essential. Our operator $\Gamma_{\tau ,i}$ is compatible 
with the shift operator and we give the holonomic systems 
of differential equations characterizing Whittaker functions 
in Theorem \ref{th:submain}. 
Therefore, we can compute explicit formula of Whittaker functions of 
principal series representations of $G=Sp(3,\mR )$ by using 
the result of this paper. We hope that this interesting possibility 
will be considered in future work. 

We give the contents of this paper. In Section \ref{sec:preliminaries}, 
we recall the structure of $Sp(3,\mR )$ and define a principal 
series representation, that is, a standard representation obtained by 
a parabolic induction with respect to the minimal parabolic subgroup 
$P_{\min }$.  
In Section \ref{sec:K-modules}, we introduce 
the monomial basis of a finite dimensional 
irreducible representation of $\gk_\mC \simeq \g \gl (3,\mC )$ and 
investigate adjoint representation of $K$ on $\gp_\mC $. 
Here the monomial basis is an alias of Gelfand Zelevinsky basis, 
Which is twisted dual of the crystal basis of Kashiwara or 
the canonical basis of Lusztig. 
In Section \ref{sec:Clebsch-Gordan}, we give the explicit expressions 
of $\iota_i\colon  V_{\tau_i}\to V_\tau \otimes_\mC \gp_\mC $. 
Section \ref{sec:structure} 
is the main body of this paper and we give the matrix representation of 
$\Gamma_{\tau ,i}$ with respect to the induced basis 
from the monomial basis in Theorem \ref{th:main}. 
In Section \ref{sec:examples}, we introduce some examples of 
$\Gamma_{\tau ,i}$ and give the holonomic systems 
of differential equations characterizing Whittaker functions 
in Theorem \ref{th:submain}.

This is an enhanced version of the Master's thesis \cite{master_1_2006}. 
The author would like to express his gratitude to Takayuki Oda 
for valuable advice on this work. He also would like to thank Miki Hirano 
to provide a reference for the computation of Clebsch-Gordan coefficients.

\section{Preliminaries}\label{sec:preliminaries}
\subsection{Groups and algebras.}\label{subsec:groups_and_algebra}
We denote by $\mZ $,$\mR $ and $\mC $ the ring of rational integers, 
the real number field and the complex number field, respectively. 
Let $1_n$ (resp.\ $O_n$) be the unit (resp.\ the zero) matrix in the space 
$M_n (\mR )$ of real matrices of size $n$. 

The real symplectic group $G=Sp(3,\mR )$ of degree three is defined by 
\[
G=Sp(3,\mR )=\{g\in SL(6,\mR )\mid {}^t g J_3 g=J_3 \}, \quad 
J_3=\left( \begin{array}{cc}
O_3 & 1_3 \\
-1_3 & O_3
\end{array} \right) , 
\]
which is connected, semisimple, and split over $\mR $. Here ${}^t g$ and 
$g^{-1}$ mean the transpose and the inverse of $g$, respectively. 
Let 
$\theta \colon G\ni g \mapsto {}^t g^{-1}\in G$ be 
a Cartan involution of G. Then 
\begin{align*}
K & =\{ g \in G \mid \theta (g)=g \} \\
& =\left\{ \left. g = 
\left( \begin{array}{cc}
A & B \\
-B & A
\end{array} \right) 
\in G \ \right| \ A+\sqrt{-1} B\in U(3) \right\} 
\end{align*}
 is a maximal compact subgroup of $G$
which is isomorphic to the unitary group $U(3)$ of degree three.
The Lie algebra of $G$ is given by
\[
\g =\mathfrak{sp} (3,\mR )=\{ X \in M_6(\mR ) \mid J_3 X+{}^t X J_3 =0 \}.
\]
If we denote the differential of $\theta $ again by $\theta $, then we have 
$\theta (X)={-}^t X$ for $X\in \g $. Let $\gk$ and $\gp$ be the $+1$ and the 
$-1$ eigenspaces of $\theta \in \g $, respectively, that is, 
\begin{align*}
\gk &=\left\{ \left. X = 
\left( \begin{array}{cc}
A & B \\
-B & A
\end{array} \right) \in \g \ \right| \ 
A,B\in M_3 (\mR ),{}^t A=-A,{}^t B=B \right\} ,\\
\gp &=\left\{ \left.  X = 
\left( \begin{array}{cc}
A & B \\
B & -A
\end{array} \right) \in \g \ \right| \ 
A,B\in M_3 (\mR ),{}^t A=A,{}^t B=B \right\} .
\end{align*}
Then $\gk $ is the Lie algebra of $K$ which is isomorphic to 
the unitary algebra
\[
\gu (3)=\{ A+\sqrt{-1} B\in M_3(\mC )
\mid A,B\in M_3 (\mR ),{}^t A=-A,{}^t B=B \}
\]
of degree three, and $\g $ has the Cartan decomposition $\g =\gk \oplus \gp $.
We fix an isomorphism $\kappa \colon \gu (3)\to \gk $ which is given 
by the inverse of 
\[
\gk \ni \left( \begin{array}{cc}
A & B \\
-B & A
\end{array} \right) \mapsto A+\sqrt{-1} B\in \gu (3).
\]

For a Lie algebra $\gl $, we denote by $\gl_\mC =\gl \otimes_\mR \mC $ 
the complexification of $\gl $. For $1\leq i,j\leq 3$, we denote by $E_{ij}$
the matrix unit in $M_3(\mR )$ with entry $1$ at $(i,j)$-th component and $0$ 
at other entries. Take a compact Cartan subalgebra 
$\gh =\bigoplus_{i=1}^3 \mR T_i$ 
where $T_i =\kappa (\sqrt{-1} E_{ii})\in \gk $.
For $1\leq i\leq 3$, define a linear form $\beta_i $ on $\gh_\mC $
by $\beta_i (T_j)=\sqrt{-1} \delta_{ij} ,\ 1\leq j\leq 3$. Here $\delta_{ij} $
is the Kronecker's delta. Then the set $\Delta $ of the roots for 
$(\gh_\mC ,\g_\mC )$ is given by 
\[
\Delta =\Delta (\gh_\mC ,\g_\mC )=\{ \pm 2\beta_i \ (1\leq i\leq 3),\ 
\pm \beta_j \pm \beta_k \ (1\leq j<k\leq 3) \} ,
\]
and the subset $\Delta^+ =\{ 2\beta_i \ (1\leq i\leq 3),\ 
\beta_j \pm \beta_k \ (1\leq j<k\leq 3) \} $ form a positive root system. Let 
\begin{align*}
\Delta^+_c &=\{ \beta_j - \beta_k \ (1\leq j<k\leq 3) \}, \\
\Delta^+_n &=\{ 2\beta_i \ (1\leq i\leq 3),\ 
\beta_j +\beta_k \ (1\leq j<k\leq 3) \} 
\end{align*}
be the set of compact and non-compact positive roots, respectively.
If we denote the root space for $\beta \in \Delta $ by $\g_\beta $, then 
$\gk_\mC \simeq \mathfrak{gl} (3,\mC ) $ and $\gp_\mC $ have the decompositions
\begin{gather*}
\gk_\mC =\gh_\mC \oplus 
 \bigoplus_{\beta \in \Delta_c} \g_\beta ,\quad 
 \Delta_c=\Delta^+_c\cup (-\Delta^+_c), \\
\gp_\mC =\gp_+ \oplus \gp_- ,\quad
 \gp_\pm =\bigoplus_{\beta \in \Delta^+_n} \g_{\pm \beta } .
\end{gather*}
Now we take a basis of $\gk_\mC $ and $\gp_\pm $ consisting of root vectors. 
If we denote the extension of the isomorphism $\kappa $ to 
their complexifications again by $\kappa $, then we have 
$\kappa (E_{ij})\in \g_{\beta_i-\beta_j } $ for each $1\leq i\neq j\leq 3$ 
 and thus the set 
$\{ \kappa (E_{ij})\mid 1\leq i,j\leq 3\} $ forms a basis of $\gk_\mC $. 
On the other hand, if we define a map 
\[
p_\pm \colon \{ X\in M_3 (\mC ) \mid X={}^t X \} \ni X \mapsto 
\left( \begin{array}{cc}
X & \pm \sqrt{-1} X \\
\pm \sqrt{-1} X & -X
\end{array} \right) \in \gp_\pm ,
\]
then the element $X_{\pm ij} =p_\pm ((E_{ij}+E_{ji})/2)$ is a root vector in 
$\g_{\pm (\beta_i +\beta_j )}$ for $1\leq i\leq j\leq 3$ and the set 
$\{ X_{\pm ij} \mid 1\leq i\leq j\leq 3\} $ gives a basis of $\gp_\pm $.

Put $\ga =\bigoplus_{i=1}^3 \mR H_i$ with 
$H_1=\diag (1,0,0,-1,0,0),\ H_2=\diag (0,1,0,0,-1,0),
\ H_3=\diag (0,0,1,0,0,-1)$. 
Then $\ga $ is a maximal abelian subalgebra of $\gp $. 
For each $1\leq i\leq 3$, we define a linear form $e_i$ on $\ga$ by 
$e_i(H_j)=\delta_{ij} $ for $1\leq j\leq 3$. 
The set $\Sigma $ of the restricted roots for $(\ga ,\g )$ is given by 
\[
\Sigma =\Sigma (\ga ,\g )=\{ \pm 2e_i\ (1\leq i\leq 3) ,\ 
\pm e_j \pm e_k \ (1\leq j<k\leq 3) \} ,
\]
and the subset $\Sigma^+ =\{ 2e_i\ (1\leq i\leq 3) ,\ 
e_j \pm e_k \ (1\leq j<k\leq 3) \} $ forms a positive root system. 
For each $\alpha \in \Sigma $, we denote the restricted root space by 
$\g_\alpha $ and choose a restricted root vector $E_\alpha $ in $\g_\alpha $ 
as follows.
\begin{align*}
E_{2e_i }=\left( \begin{array}{cc}
O_3 & E_{ii} \\
O_3 & O_3  
\end{array} \right) ,\ 1\leq i\leq 3, & \\
E_{e_j +e_k }=\left( \begin{array}{cc}
O_3 & E_{jk}+E_{kj} \\
O_3 & O_3  
\end{array} \right) ,\ 
E_{e_j -e_k }=\left( \begin{array}{cc}
E_{jk} & O_3 \\
O_3 & -E_{kj} 
\end{array} \right) ,&\ 1\leq j<k \leq 3, 
\end{align*}
and $E_{-\alpha }=\theta (E_\alpha )$ for $\alpha \in \Sigma^+ $. 
If we put $\gn =\bigoplus_{\alpha \in \Sigma^+ } \g_\alpha $ then $\g $ has 
an Iwasawa decomposition $\g =\gn \oplus \ga \oplus \gk $. 
Also we have $G=N_{\min }A_{\min }K$, where $A_{\min }=\exp (\ga )$ 
and $N_{\min }=\exp (\gn )$.

\subsection{Definition of principal series representations of G}
\label{sec:principal-series}
Let $P_{\min }=M_{\min } A_{\min }N_{\min }$ be 
the minimal parabolic subgroup of $G$, where 
\[
M_{\min }=Z_K (A_{\min })=\{ \diag (\varepsilon_1, \varepsilon_2,\varepsilon_3,
\varepsilon_1 ,\varepsilon_2 ,\varepsilon_3 ) \mid \varepsilon_i \in \{ \pm 1\}
\ (1\leq i\leq 3)\} \simeq \{ \pm 1\}^{\oplus 3}. 
\]
For $\nu \in \Hom_\mR (\ga ,\mC)$, we define a coordinate 
$(\nu_1,\nu_2,\nu_3) \in \mC^3$ by $\nu_i=\nu (H_i),\ 1\leq i\leq 3$. 
Then the half sum 
$\rho =\frac{1}{2} \left( \sum_{\alpha \in \Sigma_+} \alpha \right)
=3e_1+2e_2+e_3$
of the positive roots has coordinate $(\rho_1,\rho_2,\rho_3)=(3,2,1)$.
We define a quasicharacter $e^\nu \colon A_{\min }\to \mC^\times $ by
\[
e^\nu (\diag (a_1,a_2,a_3,a_1^{-1},a_2^{-1},a_3^{-1}))
=a_1^{\nu_1}a_2^{\nu_2}a_3^{\nu_3} .
\]
Moreover, we fix a character $\sigma $ of $M_{\min }$. 
$\sigma $ is realized by 
$(\sigma_1,\sigma_2,\sigma_3)\in \{ 0,1\}^{\oplus 3} $ 
such that 
\[
\sigma (\diag (\varepsilon_1, \varepsilon_2,\varepsilon_3,
\varepsilon_1 ,\varepsilon_2 ,\varepsilon_3 ))
= \varepsilon_1^{\sigma_1} \varepsilon_2^{\sigma_2}
\varepsilon_3^{\sigma_3} \in \mC 
\ \text{ for } \ 
\varepsilon_i\in \{\pm 1\},\ 1\leq i\leq 3.
\]

With these data $(\sigma ,\nu )$, the parabolic induction 
\[
\pi_{(\sigma ,\nu )}=\Ind_{P_{\min }}^G (\sigma \otimes e^{\nu +\rho }
\otimes 1_{N_{\min }})
\]
is a Hilbert representation of $G$ by the right regular action on the 
Hilbert space $H_{(\sigma ,\nu )}$ which is the completion of 
\[
H_{(\sigma ,\nu )}^\infty=
\left\{ f\colon G\to \mC \text{ smooth } \left| 
\begin{array}{c} 
f(manx)=\sigma (m)e^{\nu +\rho } (a)f(x) \hphantom{=====} \\
\text{ for }\ m\in M_{\min },\ a\in A_{\min },\ n\in N_{\min },\ x\in G
\end{array} \right. \right\}
\]
with an inner product
\[
(f_1,f_2)=\int_K f_1(k) \overline{f_2(k)} dk\quad \text{ for } 
\quad f_1,f_2\in H_{(\sigma ,\nu )}^\infty .
\]
Here $dk$ is a Haar measure of $K$.

\section{The monomial basis for simple $\g \gl (3,\mC )$-modules and 
the adjoint representation of $\gk_\mC \simeq \g \gl (3,\mC )$ 
on $\gp_\pm $}
\label{sec:K-modules}
In this section, we recall some basic facts about the representations of 
$\gk_\mC \simeq \g \gl (3,\mC )$, and evaluate the adjoint 
representations of $\gk_\mC $ on $\gp_{\pm }$.

\subsection{The highest weight theory for $\g \gl (n,\mC )$}
\label{subsec:GZ-basis } 
We recall the highest weight theory for $\g \gl (n,\mC )$. 
A \textit{weight} of length $n$ is an 
integral vector $\gamma =(\gamma_1 , \gamma_2 ,\cdots ,\gamma_n )
\in \mZ^n $. The weight $\gamma $ is called 
\textit{dominant} if 
$\gamma_1 \geq \gamma_2 \geq \cdots \geq \gamma_n $. 
It is well-known that every irreducible finite dimensional 
representation $(\tau , V)$ of $\g \gl (n,\mC )$ 
has a \textit{weight space decomposition}
\[
V=\bigoplus_\gamma V(\gamma ),\quad 
V(\gamma )=\{ v\in V \mid E_{ii}v=\gamma_i v,\ 1\leq i\leq n \} .
\]
There is a dominant weight $\lambda $ which satisfies 
$\lambda \geq \gamma $ in the lexicographical order for any weight $\gamma $ 
such that $V(\gamma )\neq 0$. 
Such dominant weight is called 
\textit{the highest weight} and the representation $(\tau ,V)$ is labeled 
by the highest weight, i.e., $(\tau_\lambda ,V_\lambda )$. 

\subsection{Gelfand Tsetlin patterns} 
\label{subsec:G-pattern }
Recall that a \textit{Gelfand-Tsetlin pattern} 
(which simply we call \textit{G-pattern})
 of type $\mm_3 =(m_{13} ,m_{23} ,m_{33})$ is a triangular array
\[
M=\left( \begin{array}{c} \mm_3 \\ \mm_2 \\ \mm_1 \end{array} \right)
=\lgpt{m_{13}}{m_{23}}{m_{33}}{m_{12}}{m_{22}}{m_{11}}
\]
of integers satisfying the conditions
\begin{equation}\label{cdn:G-pattern}
m_{13}\geq m_{12}\geq m_{23}\geq m_{22}\geq m_{33},\quad 
m_{12}\geq m_{11}\geq m_{22}.
\end{equation}
The weight $\wt{M} =(\wt{M}_1,\wt{M}_2,\wt{M}_3)$ 
of a G-pattern $M$ is defined from the equations 
\[
\wt{M}_1=m_{11},\quad 
\wt{M}_1+\wt{M}_2=m_{12}+m_{22},\quad 
\wt{M}_1+\wt{M}_2+\wt{M}_3=m_{13}+m_{23}+m_{33}.
\]

For a G-pattern $M$, we define
\[
M\lgpt{i_{13}}{i_{23}}{i_{33}}{i_{12}}{i_{22}}{i_{11}}
=\lgpt{m_{13}+i_{13}\ }{m_{23}+i_{23}\ }{m_{33}+i_{33}}
{m_{12}+i_{12}\ }{m_{22}+i_{22}}{m_{11}+i_{11}} .
\]
If the vector $(i_{13},i_{23},i_{33})$ is zero, we omit the top row in the 
left hand side of the above defining equality. So the left hand side is 
written as
\[
M\lcgpt{i_{12}}{i_{22}}{i_{11}}.
\]
A convenient symbol is $M[k]$, which is defined by
\[
M\lcgpt{k}{-k}{0}.
\]
We note that G-patterns $M_1$ and $M_2$ have the same type and weight 
if and only if $M_1[k]=M_2$ for some $k\in \mZ $. 

We define some functions of G-patterns. We set 
\[
\delta (M)=\wt{M}_2-m_{23}=m_{12}+m_{22}-m_{11}-m_{23}.
\]
Let $\chi_+(M)$ and $\chi_-(M)$ be the characteristic functions of the sets 
$\{ M \mid \delta (M)>0 \} $ and $\{ M \mid \delta (M)<0 \} $, respectively. 
More generally we introduce functions $\chid{\pm }{i}(M)$ by
\[
\chid{+}{i}(M)=\left\{ \begin{array}{cc}
1, & \delta (M)>i \\
0, & \delta (M)\leq i
\end{array} \right. ,\quad 
\chid{-}{i}(M)=\left\{ \begin{array}{cc}
1, & \delta (M)<-i \\
0, & \delta (M)\geq -i
\end{array} \right. .
\]
Then we have $\chi_+(M)=\chid{+}{0}(M) $ and $\chi_-(M)=\chid{-}{0}(M)$.

We introduce `piecewise-linear' functions $C_1(M),\ \bar{C_1}(M)$ and 
$C_2(M)$ by
\[
C_1(M)=\left\{\begin{array}{ll}
m_{11}-m_{22},&\ \text{ if }\ \delta (M)\geq 0\\
m_{12}-m_{23},&\ \text{ if }\ \delta (M)\leq 0
\end{array}\right. ,\quad 
\bar{C_1}(M)=\left\{\begin{array}{ll}
m_{23}-m_{22},&\ \text{ if }\ \delta (M)\geq 0\\
m_{12}-m_{11},&\ \text{ if }\ \delta (M)\leq 0
\end{array}\right. ,
\]
and 
\[
C_2(M)=C_1(M)\bar{C_1} (M).
\]
Another expressions of $C_1(M)$ and $\bar{C_1}(M)$ are
\[
C_1(M)=\min \{ m_{11}-m_{22} ,\ m_{12}-m_{23} \} ,\quad
\bar{C_1} (M)=\min \{ m_{23}-m_{22} ,\ m_{12}-m_{11} \} .
\]
\subsection{The monomial basis in the sense of Gelfand-Zelevinsky}
\label{subsec:GZ-basis }

We recall the definition of the monomial basis in the sense of 
Gelfand-Zelevinsky. 

For a weight subspace $V_\lambda (\gamma ) $ and a dominant weight 
$\nu =(\nu_1,\nu_2,\cdots ,\nu_n)$, we set 
\[
V_\lambda (\gamma ,\nu )=\{ v\in V_\lambda (\gamma )\mid 
E^{\nu_i-\nu_{i+1}+1}_{i\hs i+1}v=0,\ 1\leq i\leq n-1 \} .
\]
A basis $B$ in $V_\lambda $ is called \textit{proper} if each of 
subspaces $V_\lambda (\gamma ,\nu )$ (for all possible $\gamma ,\nu $) 
is spanned by its subset $B\cap V_\lambda (\gamma ,\nu )$. 
It is known that the representation 
$(\tau_\lambda ,V_\lambda )$ of $\g \gl (3,\mC )$ has 
a proper basis, which is unique up to scalar multiple. 
Gelfand and Zelevinsky normalized the scalar factor somehow to get 
the formulas in Proposition \ref{prop:action_on_GZ-basis}. 
The normalized proper basis is called \textit{the monomial basis} 
because it is twisted dual of the crystal basis of Kashiwara or 
the canonical basis of Lusztig. 
For the representation 
$(\tau_{\mm_3}, V_{\mm_3})$, the monomial basis in $V_{\mm_3 } (\gamma )$ 
is parameterized by the G-patterns of type $\mm_3 $ whose weights are 
$\gamma $. 
We denote the monomial basis of $V_{\mm_3 } $ by 
$\{ f(M)\}_{M\in G(\mm_3)}$. Here $G(\mm_3)$ is the set of the G-patterns 
of type $\mm_3 $. 

The action of $\gk_\mC \simeq \g \gl (3,\mC )$ is given by following formulas.
\begin{prop}\label{prop:action_on_GZ-basis}
\textit 
(Gelfand-Zelevinsky) The action of simple root vectors on the 
monomial basis $\{ f(M)\}_{M\in G(\mm_3 )} $ of $V_{\mm_3}$ are 
given as follows. 
\begin{align*}
E_{12} f(M) & = (m_{12}-m_{11})f\bkt{M\cgpt{0}{0}{1}}
+(m_{23}-m_{22})\chi_+(M)f\bkt{M\cgpt{0}{0}{1}[-1]} \\
E_{21} f(M) & = (m_{11}-m_{22})f\bkt{M\cgpt{0}{0}{-1}}
+(m_{12}-m_{23})\chi_-(M)f\bkt{M\cgpt{0}{0}{-1}[-1]} \\
E_{23} f(M) & = (m_{13}-m_{12})f\bkt{M\cgpt{1}{0}{0}}
+\{ m_{13}-m_{12}-\delta (M) \} \chi_-(M)f\bkt{M\cgpt{1}{0}{0}[-1]} \\
E_{32} f(M) & = (m_{22}-m_{33})f\bkt{M\cgpt{0}{-1}{0}}
+\{ m_{22}-m_{33}+\delta (M) \} \chi_+(M)f\bkt{M\cgpt{0}{-1}{0}[-1]}
\end{align*}
In the right hand side of above formulas, we put $f(M')=0$ 
if $M'$ is a triangular array which does not 
satisfy the condition (\ref{cdn:G-pattern}) of G-patterns.
\end{prop}
In the later section, 
we need the action of root vectors on the monomial basis of $V_{\mm_3}$. 
So we compute the action of $E_{13}$ and $E_{31}$. 
Since $E_{13}=[E_{12},E_{23}]$ and $E_{31}=[E_{32},E_{21}]$, 
we obtain 
\begin{align}
E_{13}f(M)=&[E_{12},E_{23}]f(M)
=(m_{13}-m_{12})f\bkt{M\cgpt{1}{0}{1}}
-\bar{C_1}(M)f\bkt{M\cgpt{1}{0}{1}[-1]},
\label{eqn:a_act_root_vec} \\
E_{31}f(M)=&[E_{32},E_{21}]f(M)
=(m_{33}-m_{22})f\bkt{M\cgpt{0}{-1}{-1}}
+C_1(M)f\bkt{M\cgpt{0}{-1}{-1}[-1]}.
\label{eqn:b_act_root_vec} 
\end{align}
Here $[,]$ is a Lie bracket, i.e., $[X,Y]=XY-YX$ for $X,Y\in \g \gl (3,\mC )$.

The monomial basis has a interesting symmetry property. 
For each G-pattern $M$, we define the dual pattern $\hat{M} $ by
\[
\hat{M} =\lgpt{-m_{33}}{-m_{23}}{-m_{13}}{-m_{22}}{-m_{12}}{-m_{11}}. 
\]
If $M$ is a G-pattern of type $\mm_3$ and 
weight $\wt{M} $ then $\hat{M} $ is a G-pattern of 
type $\hat{\mm}_3 =(-m_{33},-m_{23},-m_{13})$ and weight 
$-\wt{M} =(-\wt{M}_1, -\wt{M}_2, -\wt{M}_3)$.
\begin{prop}\label{prop:sym_GZ-basis}
\textit
Let $\omega$ be the automorphism of $\g \gl (3,\mC )$ defined by 
$\omega (E_{ii})=-E_{ii}$ and $\omega (E_{jk} )=E_{kj}$ for 
$i,j,k\in \{ 1,2,3\}$ such that $\abs{j-k}=1$.
Let $T_{\mm_3}\colon V_{\mm_3}\to V_{\hat{\mm}_3 } $ be the linear map 
defined by $T_{\mm_3} (f(M))=f(\hat{M} )$ for $M\in G(\mm_3)$. 
Then $X\circ T_{\mm_3} =T_{\mm_3} \circ \omega(X)$.
\end{prop}
For the existence, uniqueness and properties of the monomial basis, 
we refer to the paper \cite{MR946886}.

In the later setions, we give the explicit expressions of 
various $K$-homomorphisms in terms of the monomial basis. 
However, the monomial basis have the ambiguity of scalar multiple. 
Thus, we have to fix the monomial basis for simple $K$-modules. 
\begin{defn}
\textit{ 
A simple $K$-module $V_\lambda$ equipped with the fixed monomial basis 
$\{f(M)\}_{M\in G(\lambda )}$
is called a simple $K$-module with the marking $\{f(M)\}_{M\in G(\lambda )}$. }
\end{defn}

\subsection{The adjoint representations of $\gk_\mC $ on $\gp_{\pm }$}
\label{subsec:K-modules}
We denote by $\me_i $ the unit vector of degree three 
with its $i$-th component 1 and the remaining component 0. 
It is known that both of $\gp_\pm $ become $K$-modules via the adjoint action 
of $K$. Concerning this, we have the following lemma.
\begin{lem}\label{lem:K-action}
\textit
We have isomorphisms $i_{\gp_+}\colon \gp_+ \to V_{2\me_{1}}$ and 
$i_{\gp_-}\colon \gp_- \to V_{-2\me_{3}}$ by the correspondences 
between their basis 
\begin{align*}
& (X_{+11},X_{+12},X_{+13},X_{+22},X_{+23},X_{+33}) \\
& \leftrightarrow (f\gpt{2}{0}{0}{2}{0}{2} ,f\gpt{2}{0}{0}{2}{0}{1} ,
f\gpt{2}{0}{0}{1}{0}{1} ,f\gpt{2}{0}{0}{2}{0}{0} ,f\gpt{2}{0}{0}{1}{0}{0} ,
f\gpt{2}{0}{0}{0}{0}{0} ), \\[5mm]
& (X_{-33},-X_{-23},X_{-22},X_{-13},-X_{-12},X_{-11}) \\
& \leftrightarrow ( f\gpt{0}{0}{-2}{0}{0}{0},f\gpt{0}{0}{-2}{0}{-1}{0} ,
f\gpt{0}{0}{-2}{0}{-2}{0} ,f\gpt{0}{0}{-2}{0}{-1}{-1} ,
f\gpt{0}{0}{-2}{0}{-2}{-1},f\gpt{0}{0}{-2}{0}{-2}{-2} )
\end{align*}
\end{lem}
\begin{proof}
By direct computation, 
we have the following tables of the adjoint actions of Cartan subalgebra 
and simple root vectors of $\gk_\mC $ on 
the basis $\{ X_{\pm ij}\}_{1\leq i\leq j\leq 3} $
of $\gp_{\pm }$.\\
\begin{gather*}
\begin{array}{|c|c|c|c|c|c|c|c|c|c|}
\hline
&\kappa (E_{11})&\kappa (E_{22})&\kappa (E_{33})
&\kappa (E_{12})&\kappa (E_{21})&\kappa (E_{23})&\kappa (E_{32})
&\kappa (E_{13})&\kappa (E_{31})
\\ \hline
X_{+11}&2X_{+11}&0&0&0&2X_{+12}&0&0&0&2X_{+13}\\ \hline
X_{+12}&X_{+12}&X_{+12}&0&X_{+11}&X_{+22}&0&X_{+13}&0&X_{+23}\\ \hline
X_{+13}&X_{+13}&0&X_{+13}&0&X_{+23}&X_{+12}&0&0&0\\ \hline
X_{+22}&0&2X_{+22}&0&2X_{+12}&0&0&2X_{+23}&X_{+11}&X_{+33}\\ \hline
X_{+23}&0&X_{+23}&X_{+23}&X_{+13}&0&X_{+22}&X_{+33}&X_{+12}&0\\ \hline
X_{+33}&0&0&2X_{+33}&0&0&2X_{+23}&0&2X_{+13}&0\\ \hline
\end{array}
\\
\text{TABLE 1. The adjoint actions of $\gk_\mC$ on 
the basis $\{ X_{+ij}\}_{1\leq i\leq j\leq 3} $
of $\gp_+$ }
\\\\
\begin{array}{|c|c|c|c|c|c|c|c|c|c|}
\hline
&\kappa (E_{11})&\kappa (E_{22})&\kappa (E_{33})
&\kappa (E_{12})&\kappa (E_{21})&\kappa (E_{23})&\kappa (E_{32})
&\kappa (E_{13})&\kappa (E_{31})
\\ \hline
-X_{-11}&2X_{-11}&0&0&2X_{-12}&0&0&0&2X_{-13}&0\\ \hline
-X_{-12}&X_{-12}&X_{-12}&0&X_{-22}&X_{-11}&X_{-13}&0&X_{-23}&0\\ \hline
-X_{-13}&X_{-13}&0&X_{-13}&X_{-23}&0&0&X_{-12}&0&0\\ \hline
-X_{-22}&0&2X_{-22}&0&0&2X_{-12}&2X_{-23}&0&X_{-33}&X_{-11}\\ \hline
-X_{-23}&0&X_{-23}&X_{-23}&0&X_{-13}&X_{-33}&X_{-22}&0&X_{-12}\\ \hline
-X_{-33}&0&0&2X_{-33}&0&0&0&2X_{-23}&0&2X_{-13}\\ \hline
\end{array}
\\
\text{TABLE 2. The adjoint actions of $\gk_\mC$ on 
the basis $\{ X_{-ij}\}_{1\leq i\leq j\leq 3} $
of $\gp_-$ }
\end{gather*}
Comparing the actions in above tables with the actions 
in Proposition \ref{prop:action_on_GZ-basis}, we have the assertion. 
\end{proof}
\begin{rem}
\textit{
The above lemma tells that 
$\gp_+$ and $\gp_-$ are simple $K$-modules with 
the markings $\{ X_{+ij}\}_{1\leq i\leq j\leq 3}$ 
and $\{ X_{-ij}\}_{1\leq i\leq j\leq 3}$, respectively. 
From now on we always take these markings for $\gp_\pm$.}
\end{rem}

\section{Clebsch-Gordan coefficients for the representations of 
$\g \gl (3,\mC ) $ with respect to the monomial basis}
\label{sec:Clebsch-Gordan}

In the later sections, we need irreducible decompositions of 
the tensor products $V\otimes_\mC \gp_{+} $ and 
$V\otimes_\mC \gp_{-} $ as $K$-modules 
for a $K$-type $(\tau ,V)$ of $\pi_{(\sigma ,\nu )}$. 
Since $\gp_+\simeq V_{2\me_{1}},\ \gp_-\simeq V_{-2\me_{3}}$ and 
$\gk_\mC \simeq \g \gl (3,\mC )$, it suffices to consider 
the irreducible decomposition of $V_{\lambda }\otimes_\mC V_{2\me_{1}}$ 
and $V_{\lambda }\otimes_\mC V_{-2\me_{3}}$ as $\g \gl (3,\mC )$-modules 
for arbitrary dominant weight $\lambda $. 
In this section, we take the marking $\{f(M)\}_{M\in G(\lambda )}$ for 
a simple $K$-module $V_\lambda $.

\subsection{The irreducible decomposition of 
$V_{\lambda }\otimes_\mC V_{\me_{1}}$}
\label{subsec:clebsh(1,0,0)}

Generically the tensor product $V_{\lambda }\otimes_\mC V_{\me_{1}}$ 
has three irreducible components: 
$V_{\lambda +\me_{1}},\ V_{\lambda +\me_{2}}$ and $V_{\lambda +\me_{3}}$. 
If $\lambda +\me_{i}\ (i=2,3)$ is not dominant, 
the corresponding irreducible component does not occur. 

For $1\leq i\leq 3$, let $i^{\lambda }_{\me_{i}}$ 
be a non-zero generator of 
$\Hom_K (V_{\lambda +\me_{i}},V_{\lambda }\otimes_\mC V_{\me_{1}})$, 
which is unique up to scalar multiple 
if $V_{\lambda +\me_{i}}$ is non-zero. 
Our purpose of this subsection is to give explicit expressions of 
these injectors $i^{\lambda }_{\me_{1}},\ i^{\lambda }_{\me_{2}}$ and 
$i^{\lambda }_{\me_{3}}$ in terms of the monomial basis. 
For this purpose, we prepare following equations of 
the functions of G-patterns.
\begin{lem}
\label{lem:eqn_G-pat_fct} 
\textit (i) We have another expressions of 
$C_1(M),\ \bar{C_1}(M)$ and $C_2(M)$, which is suitable for computation:
\begin{align}
C_1(M)	&=m_{11}-m_{22}+\delta (M)\chi_-(M)\label{eqn:G-fct001}\\
	&=m_{12}-m_{23}-\delta (M)\chi_+(M),\nonumber \\
\bar{C_1}(M)&=m_{23}-m_{22}+\delta (M)\chi_-(M)\label{eqn:G-fct002}\\
	&=m_{12}-m_{11}-\delta (M)\chi_+(M),\nonumber \\
C_2(M)	&=(m_{12}-m_{23})(m_{12}-m_{11})-(m_{12}-m_{22})\delta (M)\chi_+(M)
	\label{eqn:G-fct001s}\\
	&=(m_{11}-m_{22})(m_{23}-m_{22})+(m_{12}-m_{22})\delta (M)\chi_-(M).
	\nonumber
\end{align}
(ii) We have relations of the values of the functions 
$\delta $ and $\chid{\pm}{r} $ 
for another G-patterns as follows: 
\begin{align}
\delta \bkt{M\gpt{i_{13}}{i_{23}}{i_{33}}{i_{12}}{i_{22}}{i_{11}}[-l]}
&=\delta (M)+d,\label{eqn:G-fct003}\\
\chid{+}{r}(M\gpt{i_{13}}{i_{23}}{i_{33}}{i_{12}}{i_{22}}{i_{11}}[-l])
&=\chid{+}{r-d}(M),\label{eqn:G-fct004}\\
\chid{-}{r}(M\gpt{i_{13}}{i_{23}}{i_{33}}{i_{12}}{i_{22}}{i_{11}}[-l])
&=\chid{-}{r+d}(M).\label{eqn:G-fct005}
\end{align}
Here 
$d =\delta \gpt{i_{13}}{i_{23}}{i_{33}}{i_{12}}{i_{22}}{i_{11}}.$\\
(iii) We have shift relations of $\chid{\pm}{r}(M)$ as follows:
\begin{gather}
(\delta (M)-r)\chid{+}{r}(M)=(\delta (M)-r)\chid{+}{r-1}(M),
\label{eqn:G-fct006}\\
(\delta (M)+r)\chid{-}{r}(M)=(\delta (M)+r)\chid{-}{r-1}(M),
\label{eqn:G-fct007}\\
\chid{+}{r}(M)+\chid{-}{-r-1}(M)=1,\label{eqn:G-fct008}\\
\chid{+}{r_1}(M)\chid{-}{r_2}(M)=0\ \text{ if }\ r_1+r_2>-2,
\label{eqn:G-fct009}\\
\chid{+}{r_1}(M)\chid{+}{r_2}(M)=\chid{+}{r_1}(M)\ \text{ if }\ r_1>r_2,
\label{eqn:G-fct010}\\
\chid{-}{r_1}(M)\chid{-}{r_2}(M)=\chid{-}{r_1}(M)\ \text{ if }\ r_1>r_2.
\label{eqn:G-fct011}
\end{gather}
(iv) We have convenient relations of $C_1(M)\chid{\pm }{r}(M)$ 
and $\bar{C_1}(M)\chid{\pm}{r}(M)$ as follows:
\begin{align}
C_1(M)\chid{+}{r}(M)&=(m_{11}-m_{22})\chid{+}{r}(M)
&& (r\geq -1),\label{eqn:pf_clebsh(1,0,0)_003}\\
C_1(M)\chid{-}{r}(M)&=(m_{12}-m_{23})\chid{-}{r}(M) 
&& (r\geq -1),\label{eqn:pf_clebsh(1,0,0)_004}\\
\bar{C_1}(M)\chid{+}{r}(M)&=(m_{23}-m_{22})\chid{+}{r}(M)\ 
&& (r\geq -1),\label{eqn:pf_clebsh(1,0,0)_001}\\
\bar{C_1}(M)\chid{-}{r}(M)&=(m_{12}-m_{11})\chid{-}{r}(M)
&& (r\geq -1).\label{eqn:pf_clebsh(1,0,0)_002}
\end{align} 
\end{lem}
\begin{proof}
We can easily check these equations by direct computation. 
\end{proof}
The explicit expressions of 
the injectors $i^{\lambda }_{\me_{1}},\ i^{\lambda }_{\me_{2}}$ and 
$i^{\lambda }_{\me_{3}}$ are given as follows. 
\begin{prop}\label{prop:clebsh(1,0,0)}
\textit 
For $1\leq i\leq 3$, the image of the monomial basis by 
the injector $i^{\lambda }_{\me_{i}}\colon  
V_{\lambda +\me_{i}}\to V_{\lambda }\otimes_\mC V_{\me_{1}}$ 
is given by the form
\[
i^{\lambda }_{\me_{i}} (f(M)) 
=\sum_{0\leq j\leq k\leq 1}
\left\{ \sum_{l=0}^{\cpra{i}{jk}}
\cpa{c}{\lambda }{i;jk;l}{M} 
f\bkt{M\gpt{}{-\me_{i}}{}{0}{-k}{-j} [-l]} \right\} 
\otimes f\gpt{1}{0}{0}{k}{0}{j} 
\]
for a G-pattern $M$ of type $\lambda +\me_{i}$. 
In the right hand side of the above equation, we put $f(M')=0$ 
if $M'$ is a triangular array which does not 
satisfy the condition (\ref{cdn:G-pattern}) of G-patterns.

The explicit expressions of the coefficients are given by following formulas.\\
\noindent {\bf Formula 1:} The coefficients of the injector 
$i^{\lambda }_{\me_{1}}\colon  
V_{\lambda +\me_{1}}\to V_{\lambda }\otimes_\mC V_{\me_{1}}$ 
are given as follows:\\
$(\cpra{1}{11},\cpra{1}{01},\cpra{1}{00})=(1,2,1)$ and 
\begin{align*}
\cpa{c}{\lambda }{1;11;0}{M}
&=(m_{13} -m_{12})(m_{22}-m_{33}),&
\cpa{c}{\lambda }{1;11;1}{M}
&=-\bar{E} (M),\\
\cpa{c}{\lambda }{1;01;0}{M}
&=-(m_{13} -m_{12})(m_{22}-m_{33}),&
\cpa{c}{\lambda }{1;01;1}{M}
&=\bar{F} (M),\\
\cpa{c}{\lambda }{1;01;2}{M}
&=-C_2 (M)\chi_+ (M),&
\cpa{c}{\lambda }{1;00;0}{M}
&=-(m_{13} -m_{12})(m_{13}-m_{22}+1),\\
\cpa{c}{\lambda }{1;00;1}{M}
&=C_2 (M) .
\end{align*}
Here
\begin{align*}
& \bar{E} (M) =C_1 (M)\{ m_{13}-m_{33}+1-\bar{C_1} (M) \} ,\\
& \bar{F}(M) =-C_2 (M) -\chi_+ (M)\{ (m_{13}-m_{12})(m_{22}-m_{33})
+(m_{13}-m_{33}+1)\delta (M) \} .
\end{align*}
\noindent {\bf Formula 2:} The coefficients of the injector 
$i^{\lambda }_{\me_{2}}\colon  
V_{\lambda +\me_{2}}\to V_{\lambda }\otimes_\mC V_{\me_{1}}$ 
are given as follows: \\
$(\cpra{2}{11},\cpra{2}{01},\cpra{2}{00})=(1,1,1)$ and
\begin{align*}
\cpa{c}{\lambda }{2;11;0}{M}
&=m_{22} -m_{33},&
\cpa{c}{\lambda }{2;11;1}{M}
&=-\bar{D} (M)\chi_- (M),\\
\cpa{c}{\lambda }{2;01;0}{M}
&=-(m_{22} -m_{33}),&
\cpa{c}{\lambda }{2;01;1}{M}
&=\bar{C_1} (M),\\
\cpa{c}{\lambda }{2;00;0}{M}
&=-(m_{23} -m_{22}),&
\cpa{c}{\lambda }{2;00;1}{M}
&=-\bar{C_1} (M)\chi_- (M) .
\end{align*} 
Here $ \bar{D} (M) =-m_{22} +m_{33} +\delta (M).$ \\
\noindent {\bf Formula 3:} The coefficients of 
the injector $i^{\lambda }_{\me_{3}}\colon  
V_{\lambda +\me_{3}}\to V_{\lambda }\otimes_\mC V_{\me_{1}}$ are 
given as follows:\\
$(\cpra{2}{11},\cpra{2}{01},\cpra{2}{00})=(0,1,0)$ and
\begin{align*}
\cpa{c}{\lambda }{3;11;0}{M}
&=1,&
\cpa{c}{\lambda }{3;01;0}{M}
&=-1,&
\cpa{c}{\lambda }{3;01;1}{M}
&=-\chi_+ (M),&
\cpa{c}{\lambda }{3;00;0}{M}
&=1.
\end{align*}  
\end{prop}
\begin{proof}
For $1\leq i \leq 3$, let $i^{\lambda }_{\me_{i}} $ be a linear map 
which is defined by the equations in the statement of this proposition. 
In order to prove that $i^{\lambda }_{\me_{i}} $ is a $K$-homomorphism, 
it suffices to check the actions of 
the basis $E_{mm}\ (1\leq m\leq 3)$ of Cartan subalgebra and 
simple root vectors $E_{n\hs n+1},\ E_{n+1\hs n}\ (n=1,2)$ 
by direct computation. 

Since 
\begin{align*}
E_{mm}(f(M_1)\otimes f(M_2))
&=(E_{mm}f(M_1))\otimes f(M_2)+f(M_1)\otimes (E_{mm}f(M_2))\\
&=(\wt{M_1}_m+\wt{M_2}_m)f(M_1)\otimes f(M_2),
\end{align*}
we easily check 
$E_{mm}\circ i^{\lambda }_{\me_{i}} (f(M))
 =i^{\lambda }_{\me_{i}}\circ E_{mm} (f(M))$
for $1\leq m\leq 3$ and a G-pattern $M$ of type $\lambda +e_i$. 

Therefore the essential computation is those of simple root vectors. 
We have to confirm that 
$E_{mn}\circ i^{\lambda }_{\me_{i}}
=i^{\lambda }_{\me_{i}}\circ E_{mn}$
for four simple root vectors $E_{mn}\ (\abs{m-n}=1)$ and each $i=1,2,3$. 
We set $\cpa{c}{\lambda }{i;jk;l}{M}=0$ if $l>\cpra{i}{jk}$ or $l<0$.

First, we compute the image of 
the monomial basis $f(M)$ by $i^{\lambda }_{\me_{i}}\circ E_{mn}$. 
By using the equations in Proposition \ref{prop:action_on_GZ-basis}, we have
\begin{align*}
&i^{\lambda }_{\me_{i}}\bkt{E_{12}f(M)}
=(m_{12}-m_{11})i^{\lambda }_{\me_{i}}
\bkt{f\bkt{M\cgpt{0}{0}{1}}}
+(m_{23}-m_{22})\chi_+(M)i^{\lambda }_{\me_{i}}
\bkt{f\bkt{M\cgpt{0}{0}{1}[-1]}} \\
&=(m_{12}-m_{11})\sum_{0\leq j\leq k\leq 1}
\left\{ \sum_{l=0}^{\cpra{i}{jk}}
\cpa{c}{\lambda }{i;jk;l}{M\cgpt{0}{0}{1}} 
f\bkt{M\gpt{}{-\me_{i}}{}{0}{-k}{1-j} [-l]} \right\} 
\otimes f\gpt{1}{0}{0}{k}{0}{j} \\
&\hphantom{=}+(m_{23}-m_{22})\chi_+(M)
\sum_{0\leq j\leq k\leq 1}
\left\{ \sum_{l=0}^{\cpra{i}{jk}}
\cpa{c}{\lambda }{i;jk;l}{M\cgpt{0}{0}{1}[-1]} 
f\bkt{M\gpt{}{-\me_{i}}{}{0}{-k}{1-j} [-l-1]} \right\} \\
&\hphantom{==}\otimes f\gpt{1}{0}{0}{k}{0}{j} \\
&=\sum_{0\leq j\leq k\leq 1}
\left\{ \sum_{l=0}^{\cpra{i}{jk}+1}
\cpa{A}{\lambda }{12;i;jk;l}{M} 
f\bkt{M\gpt{}{-\me_{i}}{}{0}{-k}{1-j} [-l]} \right\} 
\otimes f\gpt{1}{0}{0}{k}{0}{j} ,
\end{align*}
where 
\begin{align*}
\cpa{A}{\lambda }{12;i;jk;l}{M}=
&(m_{12}-m_{11})\cpa{c}{\lambda }{i;jk;l}{M\cgpt{0}{0}{1}}\\
&+(m_{23}-m_{22})\chi_+(M)
\cpa{c}{\lambda }{i;jk;l-1}{M\cgpt{0}{0}{1}[-1]} .
\end{align*}
Similarly, we obtain the following equations:
\[
i^{\lambda }_{\me_{i}}\bkt{E_{21}f(M)}
=\sum_{0\leq j\leq k\leq 1}
\left\{ \sum_{l=0}^{\cpra{i}{jk}+1}
\cpa{A}{\lambda }{21;i;jk;l}{M} 
f\bkt{M\gpt{}{-\me_{i}}{}{0}{-k}{-1-j} [-l]} \right\} 
\otimes f\gpt{1}{0}{0}{k}{0}{j} ,
\]
where
\begin{align*}
\cpa{A}{\lambda }{21;i;jk;l}{M}=
&(m_{11}-m_{22})\cpa{c}{\lambda }{i;jk;l}{M\cgpt{0}{0}{-1}}\\
&+(m_{12}-m_{23})\chi_-(M)
\cpa{c}{\lambda }{i;jk;l-1}{M\cgpt{0}{0}{-1}[-1]} ,
\end{align*}
\[
i^{\lambda }_{\me_{i}}\bkt{E_{23}f(M)}
=\sum_{0\leq j\leq k\leq 1}
\left\{ \sum_{l=0}^{\cpra{i}{jk}+1}
\cpa{A}{\lambda }{23;i;jk;l}{M} 
f\bkt{M\gpt{}{-\me_{i}}{}{1}{-k}{-j} [-l]} \right\} 
\otimes f\gpt{1}{0}{0}{k}{0}{j} ,
\]
where
\begin{align*}
\cpa{A}{\lambda }{23;i;jk;l}{M}=
&(m_{13}-m_{12})\cpa{c}{\lambda }{i;jk;l}{M\cgpt{1}{0}{0}}\\
&+\{ m_{13}-m_{12}-\delta (M) \} \chi_-(M)
\cpa{c}{\lambda }{i;jk;l-1}{M\cgpt{1}{0}{0}[-1]} ,
\end{align*}
and 
\[
i^{\lambda }_{\me_{i}}\bkt{E_{32}f(M)}
=\sum_{0\leq j\leq k\leq 1}
\left\{ \sum_{l=0}^{\cpra{i}{jk}+1}
\cpa{A}{\lambda }{32;i;jk;l}{M} 
f\bkt{M\gpt{}{-\me_{i}}{}{0}{-1-k}{-j} [-l]} \right\} 
\otimes f\gpt{1}{0}{0}{k}{0}{j} ,
\]
where
\begin{align*}
\cpa{A}{\lambda }{32;i;jk;l}{M}=
&(m_{22}-m_{33})\cpa{c}{\lambda }{i;jk;l}{M\cgpt{0}{-1}{0}}\\
&+\{ m_{22}-m_{33}+\delta (M) \} \chi_+(M)
\cpa{c}{\lambda }{i;jk;l-1}{M\cgpt{0}{-1}{0}[-1]} .
\end{align*}

Next, we compute the image of the
the monomial basis $f(M)$ by $E_{mn}\circ i^{\lambda }_{\me_{i}}$ as follows: 
\begin{align*}
&E_{12}\circ i^{\lambda }_{\me_{i}} (f(M)) 
=\sum_{0\leq j\leq k\leq 1}
\sum_{l=0}^{\cpra{i}{jk}}
\cpa{c}{\lambda }{i;jk;l}{M} 
E_{12}\left\{ f\bkt{M\gpt{}{-\me_i}{}{0}{-k}{-j} [-l]}  
\otimes f\gpt{1}{0}{0}{k}{0}{j} \right\}\\
&=\sum_{0\leq j\leq k\leq 1}
\Biggl\{ \sum_{l=0}^{\cpra{i}{jk}}
(m_{12}-m_{11}-l+j)\cpa{c}{\lambda }{i;jk;l}{M} 
f\bkt{M\gpt{}{-\me_i}{}{0}{-k}{1-j} [-l]} \\
&\hphantom{==}+(m_{23}-m_{22}+k-l-\delta_{2i})
\chi_+\bkt{M\gpt{}{-\me_i}{}{0}{-k}{-j} [-l]}\\
&\hphantom{===}\times \cpa{c}{\lambda }{i;jk;l}{M} 
f\bkt{M\gpt{}{-\me_i}{}{0}{-k}{1-j} [-l-1]}  
 \Biggl\}\otimes f\gpt{1}{0}{0}{k}{0}{j}\\
&\hphantom{=}+\left\{\sum_{l=0}^{2}
\cpa{c}{\lambda }{i;01;l}{M} 
 f\bkt{M\gpt{}{-\me_i}{}{0}{-1}{0} [-l]}\right\}  
\otimes f\gpt{1}{0}{0}{1}{0}{1} \\
&=\sum_{0\leq j\leq k\leq 1}
\left\{ \sum_{l=0}^{\cpra{i}{jk}+1}
\cpa{B}{\lambda }{12;i;jk;l}{M} 
f\bkt{M\gpt{}{-\me_i}{}{0}{-k}{1-j} [-l]} \right\} 
\otimes f\gpt{1}{0}{0}{k}{0}{j} ,
\end{align*}
where
\begin{align*}
\cpa{B}{\lambda }{12;i;jk;l}{M}=
&(m_{12}-m_{11}-l+j)\cpa{c}{\lambda }{i;jk;l}{M}\\
&+(m_{23}-m_{22}+k-l+1-\delta_{2i})
\chi_+^{(k-j-\delta_{2i})}(M)
\cpa{c}{\lambda }{i;jk;l-1}{M}\\
&+\left\{\begin{array}{ll}
\cpa{c}{\lambda }{i;01;l}{M} &\text{if }\ (j,k)=(1,1),\\
0 & \text{otherwise}.
\end{array}\right. 
\end{align*}
Here we use the relation \ref{eqn:G-fct004} in Lemma \ref{lem:eqn_G-pat_fct}.
%\chi_+^{(k-j-\delta_{2i})}(M)
%\chi_-^{(-k+j+\delta_{2i})}(M)

Similarly, we obtain the following equations:
\begin{align*}
&E_{21}\circ i^{\lambda }_{\me_{i}} (f(M)) 
=\sum_{0\leq j\leq k\leq 1}
\left\{ \sum_{l=0}^{\cpra{i}{jk}+1}
\cpa{B}{\lambda }{i;jk;l}{M} 
f\bkt{M\gpt{}{-\me_i}{}{0}{-k}{-1-j} [-l]} \right\} 
\otimes f\gpt{1}{0}{0}{k}{0}{j} ,
\end{align*}
where
\begin{align*}
\cpa{B}{\lambda }{21;i;jk;l}{M}=
&(m_{11}-m_{22}+k-l-j)\cpa{c}{\lambda }{i;jk;l}{M}\\
&+(m_{12}-m_{23}-l+1+\delta_{2i})
\chi_-^{(-k+j+\delta_{2i})}(M)
\cpa{c}{\lambda }{i;jk;l-1}{M}\\
&+\left\{\begin{array}{ll}
\cpa{c}{\lambda }{i;11;l}{M} &\text{if }\ (j,k)=(0,1),\\
0 & \text{otherwise},
\end{array}\right. 
\end{align*}

\begin{align*}
&E_{23}\circ i^{\lambda }_{\me_{i}} (f(M)) 
=\sum_{0\leq j\leq k\leq 1}
\left\{ \sum_{l=0}^{\cpra{i}{jk}+1}
\cpa{B}{\lambda }{23;i;jk;l}{M} 
f\bkt{M\gpt{}{-\me_i}{}{1}{-k}{-j} [-l]} \right\} 
\otimes f\gpt{1}{0}{0}{k}{0}{j} ,
\end{align*}
where
\begin{align*}
\cpa{B}{\lambda }{23;i;jk;l}{M}=
&(m_{13}-m_{12}+l-\delta_{1i})\cpa{c}{\lambda }{i;jk;l}{M}\\
&+\{ m_{13}-m_{12}-\delta (M)+k-j+l-1-\delta_{1i}-\delta_{2i} \} \\
&\hphantom{=}\times \chi_-^{(-k+j+\delta_{2i})}(M)
\cpa{c}{\lambda }{i;jk;l-1}{M}\\
&+\left\{\begin{array}{ll}
\cpa{c}{\lambda }{i;00;l-1}{M} &\text{if }\ (j,k)=(0,1),\\
0 & \text{otherwise},
\end{array}\right. 
\end{align*}
and

\begin{align*}
&E_{32}\circ i^{\lambda }_{\me_{i}} (f(M)) 
=\sum_{0\leq j\leq k\leq 1}
\left\{ \sum_{l=0}^{\cpra{i}{jk}+1}
\cpa{B}{\lambda }{32;i;jk;l}{M} 
f\bkt{M\gpt{}{-\me_i}{}{0\hs }{-1-k}{-j} [-l]} \right\} 
\otimes f\gpt{1}{0}{0}{k}{0}{j} ,
\end{align*}
where
\begin{align*}
\cpa{B}{\lambda }{32;i;jk;l}{M}=
&(m_{22}-m_{33}-k+l+\delta_{3i})\cpa{c}{\lambda }{i;jk;l}{M}\\
&+\{ m_{22}-m_{33}+\delta (M)+j-2k+l-1+\delta_{2i}+\delta_{3i} \} \\
&\hphantom{=}\times \chi_+^{(k-j-\delta_{2i})}(M)
\cpa{c}{\lambda }{i;jk;l-1}{M}\\
&+\left\{\begin{array}{ll}
\cpa{c}{\lambda }{i;01;l}{M} &\text{if }\ (j,k)=(0,0),\\
0 & \text{otherwise}.
\end{array}\right. 
\end{align*}
Here we use the relations in Lemma \ref{lem:eqn_G-pat_fct} (ii). 

In order to complete the proof, we check the equations 
\begin{equation}
\cpa{A}{\lambda }{mn;i;jk;l}{M}=\cpa{B}{\lambda }{mn;i;jk;l}{M}
\label{eqn:pf_clebsh(1,0,0)}
\end{equation}
by direct computation. 

%%%%%%%%%%%%%%%%%%%%%%%%%%%%%%%%%%%%%%%%%%%%%%%%%%%%%%%
%                 begin formula3                      %
%%%%%%%%%%%%%%%%%%%%%%%%%%%%%%%%%%%%%%%%%%%%%%%%%%%%%%%
First, we check the equations (\ref{eqn:pf_clebsh(1,0,0)}) 
 for $i=3$, that is, 
the case of formula $3$.

\noindent $\bullet $ the proof of 
$E_{12}\circ i^{\lambda }_{\me_{3}}
=i^{\lambda }_{\me_{3}}\circ E_{12}$. 

We have 
\begin{align*}
\cpa{A}{\lambda }{12;3;11;0}{M}=&m_{12}-m_{11},\\
\cpa{A}{\lambda }{12;3;11;1}{M}=&(m_{23}-m_{22})\chi_+(M),\\
\cpa{A}{\lambda }{12;3;01;0}{M}=&-(m_{12}-m_{11}),\\
\cpa{A}{\lambda }{12;3;01;1}{M}=&-(m_{12}-m_{11})\chi_+\bkt{M\cgpt{0}{0}{1}}
	-(m_{23}-m_{22})\chi_+(M),\\
\cpa{A}{\lambda }{12;3;01;2}{M}=&-(m_{23}-m_{22})\chi_+(M)
	\chi_+\bkt{M\cgpt{0}{0}{1}[-1]},\\
\cpa{A}{\lambda }{12;3;00;0}{M}=&m_{12}-m_{11},\\ 
\cpa{A}{\lambda }{12;3;00;1}{M}=&(m_{23}-m_{22})\chi_+(M),
\end{align*}
and
\begin{align*}
\cpa{B}{\lambda }{12;3;11;0}{M}=&(m_{12}-m_{11}+1)-1,\\
\cpa{B}{\lambda }{12;3;11;1}{M}=&(m_{23}-m_{22}+1) \chi_+(M) -\chi_+ (M),\\
\cpa{B}{\lambda }{12;3;01;0}{M}=&-(m_{12}-m_{11}),\\ 
\cpa{B}{\lambda }{12;3;01;1}{M}=&-(m_{12}-m_{11}-1)\chi_+ (M)
	-(m_{23}-m_{22}+1) \chid{+}{1}(M),\\ 
\cpa{B}{\lambda }{12;3;01;2}{M}=&-(m_{23}-m_{22}) \chid{+}{1}(M)\chi_+ (M),\\ 
\cpa{B}{\lambda }{12;3;00;0}{M}=&m_{12}-m_{11},\\ 
\cpa{B}{\lambda }{12;3;00;1}{M}=&(m_{23}-m_{22}) \chi_+(M).
\end{align*}

By direct computation, we have
\begin{align*}
\cpa{A}{\lambda }{12;3;01;2}{M}-\cpa{B}{\lambda }{12;3;01;2}{M}
&=(m_{23}-m_{22})\left(\chid{+}{1}(M)\chi_+ (M)
	-\chi_+(M)\chi_+\bkt{M\cgpt{0}{0}{1}[-1]}\right)\\
&=(m_{23}-m_{22})(\chid{+}{1}(M)-\chid{+}{1}(M))=0,\\[3mm]
\cpa{A}{\lambda }{12;3;01;1}{M}-\cpa{B}{\lambda }{12;3;01;1}{M}
&=-(m_{12}-m_{11})\chi_+\bkt{M\cgpt{0}{0}{1}}-(m_{23}-m_{22})\chi_+(M)\\
&\hphantom{=}+(m_{12}-m_{11}-1)\chi_+ (M)+(m_{23}-m_{22}+1) \chid{+}{1}(M)\\
&=(\delta (M)-1)\chi_+(M)-(\delta (M)-1)\chid{+}{1}(M)=0.
\end{align*}
Hence we obtain the equations (\ref{eqn:pf_clebsh(1,0,0)}) 
 for $(j,k,l)=(0,1,1)$ and $(0,1,2)$. 
Here we use the relations (\ref{eqn:G-fct004}) and (\ref{eqn:G-fct006}). 

It is trivial that the equations (\ref{eqn:pf_clebsh(1,0,0)}) 
 hold for other $(j,k,l)$. \\

\noindent $\bullet $ the proof of 
$E_{21}\circ i^{\lambda }_{\me_{3}}
=i^{\lambda }_{\me_{3}}\circ E_{21}$. 

We have
\begin{align*}
\cpa{A}{\lambda }{21;3;11;0}{M}=&m_{11}-m_{22},\\
\cpa{A}{\lambda }{21;3;11;1}{M}=&(m_{12}-m_{23})\chi_-(M) ,\\
\cpa{A}{\lambda }{21;3;01;0}{M}=&-(m_{11}-m_{22}),\\ 
\cpa{A}{\lambda }{21;3;01;1}{M}=&-(m_{11}-m_{22})
	\chi_+ \bkt{M\cgpt{0}{0}{-1}}-(m_{12}-m_{23})\chi_-(M),\\
\cpa{A}{\lambda }{21;3;01;2}{M}=&-(m_{12}-m_{23})\chi_-(M) 
	\chi_+\bkt{M\cgpt{0}{0}{-1}[-1]},\\
\cpa{A}{\lambda }{21;3;00;0}{M}=&m_{11}-m_{22},\\ 
\cpa{A}{\lambda }{21;3;00;1}{M}=&(m_{12}-m_{23})\chi_-(M),
\end{align*}
and
\begin{align*}
\cpa{B}{\lambda }{21;3;11;0}{M}=&m_{11}-m_{22},\\ 
\cpa{B}{\lambda }{21;3;11;1}{M}=&(m_{12}-m_{23}) \chi_-(M),\\ 
\cpa{B}{\lambda }{21;3;01;0}{M}=&-(m_{11}-m_{22}+1)+ 1,\\
\cpa{B}{\lambda }{21;3;01;1}{M}=&-(m_{11}-m_{22}) \chi_+ (M)
	-(m_{12}-m_{23}) \chid{-}{-1}(M) ,\\
\cpa{B}{\lambda }{21;3;01;2}{M}=&-(m_{12}-m_{23}-1) 
	\chid{-}{-1}(M) \chi_+ (M),\\ 
\cpa{B}{\lambda }{21;3;00;0}{M}=&m_{11}-m_{22},\\ 
\cpa{B}{\lambda }{21;3;00;1}{M}=&(m_{12}-m_{23}) \chi_-(M) .
\end{align*}

We have
\begin{align*}
\cpa{A}{\lambda }{21;3;01;1}{M}
&=-(m_{11}-m_{22})\chid{+}{-1}(M)-(m_{12}-m_{23})(1-\chid{+}{-1}(M))\\
&=-m_{12}+m_{23}+\delta (M)\chid{+}{-1}(M)\\
&=-m_{12}+m_{23}+\delta (M)\chi_+(M),\\
\cpa{B}{\lambda }{21;3;01;1}{M}
&=-(m_{11}-m_{22})\chi_+ (M)-(m_{12}-m_{23}) (1-\chi_+ (M))\\
&=-m_{12}+m_{23}+\delta (M).
\end{align*}
Hence we obtain 
$\cpa{A}{\lambda }{21;3;01;1}{M}=\cpa{B}{\lambda }{21;3;01;1}{M}$.
Here we use the relations (\ref{eqn:G-fct004}), (\ref{eqn:G-fct006}) 
and (\ref{eqn:G-fct008}). 

We have
\begin{align*}
\cpa{A}{\lambda }{21;3;01;2}{M}
&=-(m_{12}-m_{23})\chi_-(M) \chid{+}{-1}(M)=0,\\
\cpa{B}{\lambda }{21;3;01;2}{M}
&=-(m_{12}-m_{23}-1) \chid{-}{-1}(M) \chi_+ (M)=0.
\end{align*}
Hence we obtain 
$\cpa{A}{\lambda }{21;3;01;2}{M}=\cpa{B}{\lambda }{21;3;01;2}{M}$.
Here we use the relations (\ref{eqn:G-fct004}) 
and (\ref{eqn:G-fct009}). 

It is trivial that the equations (\ref{eqn:pf_clebsh(1,0,0)}) 
 hold for other $(j,k,l)$. \\

\noindent $\bullet $ the proof of 
$E_{23}\circ i^{\lambda }_{\me_{3}}
=i^{\lambda }_{\me_{3}}\circ E_{23}$. 

We have
\begin{align*}
\cpa{A}{\lambda }{23;3;11;0}{M}=&(m_{13}-m_{12}),\\
\cpa{A}{\lambda }{23;3;11;1}{M}=&\{ m_{13}-m_{12}-\delta (M) \} \chi_-(M),\\
\cpa{A}{\lambda }{23;3;01;0}{M}=&-(m_{13}-m_{12}),\\ 
\cpa{A}{\lambda }{23;3;01;1}{M}=&-(m_{13}-m_{12})\chi_+\bkt{M\cgpt{1}{0}{0}}
	-\{ m_{13}-m_{12}-\delta (M) \} \chi_-(M),\\
\cpa{A}{\lambda }{23;3;01;2}{M}=&-\{ m_{13}-m_{12}-\delta (M) \} 
	\chi_-(M)\chi_+\bkt{M\cgpt{1}{0}{0}[-1]},\\
\cpa{A}{\lambda }{23;3;00;0}{M}=&m_{13}-m_{12},\\ 
\cpa{A}{\lambda }{23;3;00;1}{M}=&\{ m_{13}-m_{12}-\delta (M) \} \chi_-(M), 
\end{align*}
and
\begin{align*}
\cpa{B}{\lambda }{23;3;11;0}{M}=&m_{13}-m_{12},\\ 
\cpa{B}{\lambda }{23;3;11;1}{M}=&\{ m_{13}-m_{12}-\delta (M) \} \chi_-(M),\\ 
\cpa{B}{\lambda }{23;3;01;0}{M}=&-(m_{13}-m_{12}),\\
\cpa{B}{\lambda }{23;3;01;1}{M}=&-(m_{13}-m_{12}+1)\chi_+ (M)
	-\{ m_{13}-m_{12}-\delta (M)+1 \} \chid{-}{-1}(M)+1,\\
\cpa{B}{\lambda }{23;3;01;2}{M}=&-\{ m_{13}-m_{12}-\delta (M)+2\}
	\chid{-}{-1}(M)\chi_+ (M),\\ 
\cpa{B}{\lambda }{23;3;00;0}{M}=&m_{13}-m_{12},\\ 
\cpa{B}{\lambda }{23;3;00;1}{M}=&\{ m_{13}-m_{12}-\delta (M) \} \chi_-(M). 
\end{align*}

We have
\begin{align*}
\cpa{A}{\lambda }{23;3;01;1}{M}
&=-(m_{13}-m_{12})(\chid{+}{-1}(M)+\chi_-(M))+\delta (M)\chi_-(M)\\
&=-m_{13}+m_{12}+\delta (M)\chi_-(M),\\
\cpa{B}{\lambda }{23;3;01;1}{M}
&=-(m_{13}-m_{12}+1)(\chi_+(M)+\chid{-}{-1}(M))+1+\delta (M)\chid{-}{-1}(M)\\
&=-m_{13}+m_{12}+\delta (M)\chi_-(M).
\end{align*}
Hence $\cpa{A}{\lambda }{23;3;01;1}{M}=\cpa{B}{\lambda }{23;3;01;1}{M}$.
Here we use the relations (\ref{eqn:G-fct004}) and (\ref{eqn:G-fct008}). 

We have
\begin{align*}
\cpa{A}{\lambda }{23;3;01;2}{M}&=-\{ m_{13}-m_{12}-\delta (M) \} 
	\chi_-(M)\chid{+}{-1}(M)=0,\\
\cpa{B}{\lambda }{23;3;01;2}{M}&=-\{ m_{13}-m_{12}-\delta (M)+2\}
	\chid{-}{-1}(M)\chi_+ (M)=0.
\end{align*}
Hence $\cpa{A}{\lambda }{23;3;01;2}{M}=\cpa{B}{\lambda }{23;3;01;2}{M}$. 
Here we use the relations (\ref{eqn:G-fct004}) and (\ref{eqn:G-fct009}). 

It is trivial that the equations (\ref{eqn:pf_clebsh(1,0,0)}) 
hold for other $(j,k,l)$. \\

\noindent $\bullet $ the proof of 
$E_{32}\circ i^{\lambda }_{\me_{3}}
=i^{\lambda }_{\me_{3}}\circ E_{32}$. 
\begin{align*}
\cpa{A}{\lambda }{32;3;11;0}{M}=&m_{22}-m_{33},\\
\cpa{A}{\lambda }{32;3;11;1}{M}=&\{ m_{22}-m_{33}+\delta (M) \} \chi_+(M),\\
\cpa{A}{\lambda }{32;3;01;0}{M}=&-(m_{22}-m_{33}),\\ 
\cpa{A}{\lambda }{32;3;01;1}{M}=&-(m_{22}-m_{33})\chi_+\bkt{M\cgpt{0}{-1}{0}}
	-\{ m_{22}-m_{33}+\delta (M) \} \chi_+(M),\\
\cpa{A}{\lambda }{32;3;01;2}{M}=&-\{ m_{22}-m_{33}+\delta (M) \} 
	\chi_+(M) \chi_+\bkt{M\cgpt{0}{-1}{0}[-1]},\\
\cpa{A}{\lambda }{32;3;00;0}{M}=&(m_{22}-m_{33}),\\ 
\cpa{A}{\lambda }{32;3;00;1}{M}=&\{ m_{22}-m_{33}+\delta (M) \} \chi_+(M),
\end{align*}
and
\begin{align*}
\cpa{B}{\lambda }{32;3;11;0}{M}=&m_{22}-m_{33},\\ 
\cpa{B}{\lambda }{32;3;11;1}{M}=&\{ m_{22}-m_{33}+\delta (M)\} \chi_+(M) ,\\
\cpa{B}{\lambda }{32;3;01;0}{M}=&-(m_{22}-m_{33}),\\ 
\cpa{B}{\lambda }{32;3;01;1}{M}=&-(m_{22}-m_{33}+1)\chi_+ (M)
	-\{ m_{22}-m_{33}+\delta (M)-1 \}\chid{+}{1}(M) ,\\ 
\cpa{B}{\lambda }{32;3;01;2}{M}=&-\{ m_{22}-m_{33}+\delta (M)\} 
	\chid{+}{1}(M)\chi_+ (M),\\ 
\cpa{B}{\lambda }{32;3;00;0}{M}=&(m_{22}-m_{33}+1)-1,\\
\cpa{B}{\lambda }{32;3;00;1}{M}=&\{ m_{22}-m_{33}+\delta (M)+1 \} 
	\chi_+(M)-\chi_+ (M).
\end{align*}

We have
\begin{align*}
&\cpa{A}{\lambda }{32;3;01;1}{M}-\cpa{B}{\lambda }{32;3;01;1}{M}\\
&=-(m_{22}-m_{33})\chid{+}{1}(M)-\{ m_{22}-m_{33}+\delta (M) \} \chi_+(M)\\
&\phantom{=}+(m_{22}-m_{33}+1)\chi_+ (M)
	+\{ m_{22}-m_{33}+\delta (M)-1 \}\chid{+}{1}(M) \\
&=-(\delta (M)-1)\chi_+ (M)+(\delta (M)-1)\chid{+}{1}(M)=0.
\end{align*}
Hence we obtain the equation (\ref{eqn:pf_clebsh(1,0,0)}) 
for $(j,k,l)=(0,1,1)$. 
Here we use the relations (\ref{eqn:G-fct004}) and (\ref{eqn:G-fct006}). 

We have
\begin{align*}
&\cpa{A}{\lambda }{32;3;01;2}{M}-\cpa{B}{\lambda }{32;3;01;2}{M}\\
&=-\{ m_{22}-m_{33}+\delta (M) \} \left(
\chi_+(M) \chi_+\bkt{M\cgpt{0}{-1}{0}[-1]}-\chid{+}{1}(M)\chi_+ (M)\right)\\
&=-\{ m_{22}-m_{33}+\delta (M) \} 
\left(\chid{+}{1}(M)-\chid{+}{1}(M)\right)=0.
\end{align*}
Hence we obtain the equation (\ref{eqn:pf_clebsh(1,0,0)}) 
for $(j,k,l)=(0,1,2)$. 
Here we use the relations (\ref{eqn:G-fct004}) and (\ref{eqn:G-fct010}). 

It is trivial that the equations (\ref{eqn:pf_clebsh(1,0,0)}) 
 hold for other $(j,k,l)$. \\
\\
%%%%%%%%%%%%%%%%%%%%%%%%%%%%%%%%%%%%%%%%%%%%%%%%%%%%%%%
%                  end formula3                       %
%%%%%%%%%%%%%%%%%%%%%%%%%%%%%%%%%%%%%%%%%%%%%%%%%%%%%%%
In these computations, we use the relations in 
Lemma \ref{lem:eqn_G-pat_fct}, frequently. 
So we use these relations without notice in the proof of formula 1 and 
formula 2. 
%%%%%%%%%%%%%%%%%%%%%%%%%%%%%%%%%%%%%%%%%%%%%%%%%%%%%%%
%                  begin formula2                     %
%%%%%%%%%%%%%%%%%%%%%%%%%%%%%%%%%%%%%%%%%%%%%%%%%%%%%%%
Next, we check the equations \ref{eqn:pf_clebsh(1,0,0)} for $i=2$, that is, 
the case of formula $2$.

\noindent $\bullet $ the proof of 
$E_{12}\circ i^{\lambda }_{\me_{2}}
=i^{\lambda }_{\me_{2}}\circ E_{12}$. 

We have
\begin{align*}
\cpa{A}{\lambda }{12;2;11;0}{M}=&(m_{12}-m_{11})(m_{22} -m_{33}),\\
\cpa{A}{\lambda }{12;2;11;1}{M}= 
&-(m_{12}-m_{11})\bar{D}\bkt{M\cgpt{0}{0}{1}}\chi_-\bkt{M\cgpt{0}{0}{1}}\\ 
&+(m_{23}-m_{22})\chi_+(M)(m_{22} -m_{33}+1),\\
\cpa{A}{\lambda }{12;2;11;2}{M}=&-(m_{23}-m_{22})\chi_+(M)
\bar{D}\bkt{M\cgpt{0}{0}{1}[-1]}\chi_-\bkt{M\cgpt{0}{0}{1}[-1]},\\
\cpa{A}{\lambda }{12;2;01;0}{M}=&-(m_{12}-m_{11})(m_{22} -m_{33}),\\ 
\cpa{A}{\lambda }{12;2;01;1}{M}= 
&(m_{12}-m_{11})\bar{C_1}\bkt{M\cgpt{0}{0}{1}}
-(m_{23}-m_{22})\chi_+(M)(m_{22} -m_{33}+1),\\
\cpa{A}{\lambda }{12;2;01;2}{M}= 
&(m_{23}-m_{22})\chi_+(M)\bar{C_1}\bkt{M\cgpt{0}{0}{1}[-1]},\\
\cpa{A}{\lambda }{12;2;00;0}{M}=&-(m_{12}-m_{11})(m_{23} -m_{22}),\\ 
\cpa{A}{\lambda }{12;2;00;1}{M}= 
&-(m_{12}-m_{11})\bar{C_1}\bkt{M\cgpt{0}{0}{1}}\chi_-\bkt{M\cgpt{0}{0}{1}}\\ 
&-(m_{23}-m_{22})\chi_+(M)(m_{23} -m_{22}-1),\\
\cpa{A}{\lambda }{12;2;00;2}{M}= 
&-(m_{23}-m_{22})\chi_+(M) 
\bar{C_1}\bkt{M\cgpt{0}{0}{1}[-1]}\chi_-\bkt{M\cgpt{0}{0}{1}[-1]},
\end{align*}
and
\begin{align*}
\cpa{B}{\lambda }{12;2;11;0}{M}= 
&(m_{12}-m_{11}+1)(m_{22} -m_{33})-(m_{22} -m_{33}),\\
\cpa{B}{\lambda }{12;2;11;1}{M}= 
&-(m_{12}-m_{11}) \bar{D} (M)\chi_- (M)\\ 
&+(m_{23}-m_{22}) \chid{+}{-1}(M)(m_{22} -m_{33})+ \bar{C_1} (M),\\
\cpa{B}{\lambda }{12;2;11;2}{M}= 
&-(m_{23}-m_{22}-1)\chid{+}{-1}(M)\bar{D} (M)\chi_- (M),\\ 
\cpa{B}{\lambda }{12;2;01;0}{M}=&-(m_{12}-m_{11})(m_{22} -m_{33}),\\ 
\cpa{B}{\lambda }{12;2;01;1}{M}= 
&(m_{12}-m_{11}-1)\bar{C_1} (M)-(m_{23}-m_{22}) \chi_+(M)(m_{22} -m_{33}),\\ 
\cpa{B}{\lambda }{12;2;01;2}{M}=&(m_{23}-m_{22}-1) \chi_+(M)\bar{C_1} (M),\\ 
\cpa{B}{\lambda }{12;2;00;0}{M}=&-(m_{12}-m_{11})(m_{23} -m_{22}),\\ 
\cpa{B}{\lambda }{12;2;00;1}{M}=&-(m_{12}-m_{11}-1)\bar{C_1} (M)\chi_- (M)\\ 
&-(m_{23}-m_{22}-1) \chid{+}{-1}(M)(m_{23} -m_{22}),\\ 
\cpa{B}{\lambda }{12;2;00;2}{M}= 
&-(m_{23}-m_{22}-2) \chid{+}{-1}(M)\bar{C_1} (M)\chi_- (M).
\end{align*}

We have
\begin{align*}
\cpa{A}{\lambda }{12;2;11;1}{M}= 
&-(m_{12}-m_{11})(-m_{22} +m_{33} +\delta (M)-1)(1-\chi_+(M))\\ 
&+(m_{23}-m_{22})\chi_+(M)(m_{22} -m_{33}+1)\\
=&-(m_{12}-m_{11})(\bar{D}(M)-1)-(m_{12}-m_{11}+m_{22} -m_{33}+1)
\delta (M)\chi_+(M),\\[3mm]
\cpa{B}{\lambda }{12;2;11;1}{M}= 
&-(m_{12}-m_{11})(-m_{22}+m_{33}+\delta (M))(1-\chid{+}{-1}(M))\\ 
&+(m_{23}-m_{22}) \chid{+}{-1}(M)(m_{22} -m_{33})
+(m_{12}-m_{11}-\delta (M)\chi_+(M))\\
=&-(m_{12}-m_{11})(\bar{D}(M)-1)-(m_{12}-m_{11}+m_{22} -m_{33})
\delta (M)\chid{+}{-1}(M)\\
&-\delta (M)\chi_+(M)\\
=&-(m_{12}-m_{11})(\bar{D}(M)-1)-(m_{12}-m_{11}+m_{22} -m_{33}+1)
\delta (M)\chid{+}{-1}(M).
\end{align*}
Hence we obtain 
$\cpa{A}{\lambda }{12;2;11;1}{M}=\cpa{B}{\lambda }{12;2;11;1}{M}$.

We have
\begin{align*}
\cpa{A}{\lambda }{12;2;11;2}{M}=&-(m_{23}-m_{22})
\bar{D}\bkt{M\cgpt{0}{0}{1}[-1]}\chi_+(M)\chid{-}{-1}(M)=0,\\
\cpa{B}{\lambda }{12;2;11;2}{M}= 
&-(m_{23}-m_{22}-1)\bar{D} (M)\chid{+}{-1}(M)\chi_- (M)=0.
\end{align*}
Hence we obtain 
$\cpa{A}{\lambda }{12;2;11;2}{M}=\cpa{B}{\lambda }{12;2;11;2}{M}$.

We have
\begin{align*}
\cpa{A}{\lambda }{12;2;01;1}{M}= 
&(m_{12}-m_{11})\{m_{12}-m_{11}-1-(\delta (M)-1)\chi_+(M)\}\\
&-(m_{23}-m_{22})\chi_+(M)(m_{22} -m_{33}+1)\\
=&(m_{12}-m_{11})(m_{12}-m_{11}-1)\\
&-\chi_+(M)\{(m_{23}-m_{22})(m_{22} -m_{33}+1)+
(m_{12}-m_{11})(\delta (M)-1)\},\\[3mm]
\cpa{B}{\lambda }{12;2;01;1}{M}=
&(m_{12}-m_{11}-1)(m_{12}-m_{11}-\delta (M)\chi_+(M))\\
&-(m_{23}-m_{22}) \chi_+(M)(m_{22} -m_{33})\\
=& (m_{12}-m_{11})(m_{12}-m_{11}-1)\\
&-\chi_+(M)((m_{23}-m_{22})(m_{22} -m_{33})+(m_{12}-m_{11}-1)\delta (M)).
\end{align*}
Therefore
\begin{align*}
\cpa{A}{\lambda }{12;2;01;1}{M}-\cpa{B}{\lambda }{12;2;01;1}{M}
=-\chi_+(M)\{(m_{23}-m_{22})-(m_{12}-m_{11})+\delta (M)\}=0.
\end{align*}
Hence we obtain 
$\cpa{A}{\lambda }{12;2;01;1}{M}=\cpa{B}{\lambda }{12;2;01;1}{M}$.

We have
\begin{align*}
\cpa{A}{\lambda }{12;2;01;2}{M}= 
&(m_{23}-m_{22})\chi_+(M)\{m_{23}-m_{22}-1+(\delta (M)-1)\chid{-}{-1}(M)\}\\
=&(m_{23}-m_{22})(m_{23}-m_{22}-1)\chi_+(M),\\
\cpa{B}{\lambda }{12;2;01;2}{M}=& (m_{23}-m_{22}-1)(m_{23}-m_{22})\chi_+(M).
\end{align*}
Hence we obtain 
$\cpa{A}{\lambda }{12;2;01;2}{M}=\cpa{B}{\lambda }{12;2;01;2}{M}$.

We have
\begin{align*}
\cpa{A}{\lambda }{12;2;00;1}{M}= 
&-(m_{12}-m_{11})(m_{12}-m_{11}-1)\chid{-}{-1}(M)\\ 
&-(m_{23}-m_{22})(m_{23} -m_{22}-1)\chi_+(M)\\
=&-(\bar{C_1}(M))(\bar{C_1}(M)-1)(\chid{-}{-1}(M)+\chi_+(M))\\
=&-(\bar{C_1}(M))(\bar{C_1}(M)-1),\\[3mm]
\cpa{B}{\lambda }{12;2;00;1}{M}=&-(m_{12}-m_{11}-1)(m_{12}-m_{11})\chi_- (M)\\ 
&-(m_{23} -m_{22})(m_{23}-m_{22}-1) \chid{+}{-1}(M)\\
=&-(\bar{C_1}(M))(\bar{C_1}(M)-1)(\chi_-(M)+\chid{+}{-1}(M))\\
=&-(\bar{C_1}(M))(\bar{C_1}(M)-1).
\end{align*}
Hence we obtain 
$\cpa{A}{\lambda }{12;2;00;1}{M}=\cpa{B}{\lambda }{12;2;00;1}{M}$.

We have
\begin{align*}
\cpa{A}{\lambda }{12;2;00;2}{M}= 
&-(m_{23}-m_{22}) \bar{C_1}\bkt{M\cgpt{0}{0}{1}[-1]}\chi_+(M)\chid{-}{-1}(M)
=0,\\
\cpa{B}{\lambda }{12;2;00;2}{M}= 
&-(m_{23}-m_{22}-2) \bar{C_1} (M)\chid{+}{-1}(M)\chi_- (M)=0.
\end{align*}
Hence we obtain 
$\cpa{A}{\lambda }{12;2;00;2}{M}=\cpa{B}{\lambda }{12;2;00;2}{M}$. 

It is trivial that the equations (\ref{eqn:pf_clebsh(1,0,0)}) 
 hold for $(j,k,l)=(1,1,0),\ (0,1,0)$ and $(0,0,0)$. \\

\noindent $\bullet $ the proof of 
$E_{21}\circ i^{\lambda }_{\me_{2}}
=i^{\lambda }_{\me_{2}}\circ E_{21}$. 

We have
\begin{align*}
\cpa{A}{\lambda }{21;2;11;0}{M}=&(m_{11}-m_{22})(m_{22} -m_{33})\\
\cpa{A}{\lambda }{21;2;11;1}{M}=&-(m_{11}-m_{22})
	\bar{D}\bkt{M\cgpt{0}{0}{-1}}\chi_-\bkt{M\cgpt{0}{0}{-1}}\\ 
	&+(m_{12}-m_{23})\chi_-(M)(m_{22} -m_{33}+1),\\
\cpa{A}{\lambda }{21;2;11;2}{M}=&-(m_{12}-m_{23})\chi_-(M) 
	\bar{D}\bkt{M\cgpt{0}{0}{-1}[-1]}\chi_-\bkt{M\cgpt{0}{0}{-1}[-1]},\\
\cpa{A}{\lambda }{21;2;01;0}{M}=&-(m_{11}-m_{22})(m_{22} -m_{33}),\\ 
\cpa{A}{\lambda }{21;2;01;1}{M}=&(m_{11}-m_{22})\bar{C_1}\bkt{M\cgpt{0}{0}{-1}}
	-(m_{12}-m_{23})\chi_-(M)(m_{22} -m_{33}+1),\\
\cpa{A}{\lambda }{21;2;01;2}{M}=
	&(m_{12}-m_{23})\chi_-(M)\bar{C_1}\bkt{M\cgpt{0}{0}{-1}[-1]},\\
\cpa{A}{\lambda }{21;2;00;0}{M}=&-(m_{11}-m_{22}) (m_{23} -m_{22}),\\ 
\cpa{A}{\lambda }{21;2;00;1}{M}=&-(m_{11}-m_{22})
	\bar{C_1}\bkt{M\cgpt{0}{0}{-1}}\chi_-\bkt{M\cgpt{0}{0}{-1}}\\ 
	&-(m_{12}-m_{23})\chi_-(M)(m_{23} -m_{22}-1),\\
\cpa{A}{\lambda }{21;2;00;2}{M}=&-(m_{12}-m_{23})\chi_-(M) 
	\bar{C_1}\bkt{M\cgpt{0}{0}{-1}[-1]}\chi_-\bkt{M\cgpt{0}{0}{-1}[-1]},
\end{align*}
and
\begin{align*}
\cpa{B}{\lambda }{21;2;11;0}{M}=&(m_{11}-m_{22})(m_{22} -m_{33}),\\ 
\cpa{B}{\lambda }{21;2;11;1}{M}=&-(m_{11}-m_{22}-1)\bar{D} (M)\chi_- (M)\\ 
	&+(m_{12}-m_{23}+1) \chid{-}{1}(M)(m_{22} -m_{33}),\\ 
\cpa{B}{\lambda }{21;2;11;2}{M}=
	&-(m_{12}-m_{23}) \chid{-}{1}(M)\bar{D} (M)\chi_- (M),\\ 
\cpa{B}{\lambda }{21;2;01;0}{M}=
	&-(m_{11}-m_{22}+1)(m_{22} -m_{33})+ (m_{22} -m_{33}),\\
\cpa{B}{\lambda }{21;2;01;1}{M}=&(m_{11}-m_{22})\bar{C_1} (M)\\ 
	&-(m_{12}-m_{23}+1) \chi_-(M)(m_{22} -m_{33})-\bar{D} (M)\chi_- (M),\\
\cpa{B}{\lambda }{21;2;01;2}{M}=&(m_{12}-m_{23}) \chi_-(M)\bar{C_1} (M),\\ 
\cpa{B}{\lambda }{21;2;00;0}{M}=&-(m_{11}-m_{22})(m_{23} -m_{22}),\\ 
\cpa{B}{\lambda }{21;2;00;1}{M}=&-(m_{11}-m_{22}-1)\bar{C_1} (M)\chi_- (M)\\ 
	&-(m_{12}-m_{23}+1) \chid{-}{1}(M)(m_{23} -m_{22}),\\ 
\cpa{B}{\lambda }{21;2;00;2}{M}=
	&-(m_{12}-m_{23}) \chid{-}{1}(M)\bar{C_1}(M)\chi_-(M).
\end{align*}

We have
\begin{align*}
\cpa{A}{\lambda }{21;2;11;1}{M}=&-(m_{11}-m_{22})
	(-m_{22}+m_{33}+\delta (M)+1)\chid{-}{1}(M)\\ 
	&+(m_{12}-m_{23})\chi_-(M)(m_{22} -m_{33}+1)\\
=&-(m_{11}-m_{22})(\delta (M)+1)\chid{-}{1}(M)+(m_{12}-m_{23})\chi_-(M)\\
&+(m_{22} -m_{33})\{(m_{11}-m_{22})\chid{-}{1}(M)+(m_{12}-m_{23})\chi_-(M)\}\\
=&-(m_{11}-m_{22}-1)\delta (M)\chi_-(M)\\
&+(m_{22} -m_{33})\{(m_{11}-m_{22})\chid{-}{1}(M)+(m_{12}-m_{23})\chi_-(M)\},
\\[3mm]
\cpa{B}{\lambda }{21;2;11;1}{M}=&-(m_{11}-m_{22}-1)
	(-m_{22}+m_{33}+\delta (M))\chi_- (M)\\ 
	&+(m_{12}-m_{23}+1) \chid{-}{1}(M)(m_{22} -m_{33})\\
=&-(m_{11}-m_{22}-1)\delta (M)\chi_-(M)\\
&+(m_{22} -m_{33})\{(m_{11}-m_{22}-1)\chi_- (M)
+(m_{12}-m_{23}+1) \chid{-}{1}(M)\}.
\end{align*}
Therefore
\begin{align*}
&\cpa{A}{\lambda }{21;2;11;1}{M}-\cpa{B}{\lambda }{21;2;11;1}{M}\\
&=(m_{22} -m_{33})\{(\delta (M)+1)\chi_- (M)-(\delta (M)+1)\chid{-}{1}(M)\}=0.
\end{align*}
Hence we obtain 
$\cpa{A}{\lambda }{21;2;11;1}{M}=\cpa{B}{\lambda }{21;2;11;1}{M}$.

We have
\begin{align*}
\cpa{A}{\lambda }{21;2;11;2}{M}=&-(m_{12}-m_{23})\chi_-(M) 
	(-m_{22}+m_{33}+\delta (M))\chid{-}{1}(M)\\
=&-(m_{12}-m_{23})\bar{D}(M)\chid{-}{1}(M)\chi_- (M),\\[3mm]
\cpa{B}{\lambda }{21;2;11;2}{M}=
	&-(m_{12}-m_{23}) \bar{D} (M)\chid{-}{1}(M)\chi_- (M).
\end{align*}
Hence we obtain 
$\cpa{A}{\lambda }{21;2;11;2}{M}=\cpa{B}{\lambda }{21;2;11;2}{M}$.

We have
\begin{align*}
\cpa{A}{\lambda }{21;2;01;1}{M}=
&(m_{11}-m_{22})\{m_{23}-m_{22}+(\delta (M)+1)\chi_-(M)\}\\
	&-(m_{12}-m_{23})\chi_-(M)(m_{22} -m_{33}+1),\\[3mm]
\cpa{B}{\lambda }{21;2;01;1}{M}=
&(m_{11}-m_{22})(m_{23}-m_{22}+\delta (M)\chi_-(M))\\ 
	&-(m_{12}-m_{23}+1) \chi_-(M)(m_{22} -m_{33})
	-\bar{D}(M)\chi_- (M).
\end{align*}
Therefore
\begin{align*}
&\cpa{A}{\lambda }{21;2;01;1}{M}-\cpa{B}{\lambda }{21;2;01;1}{M}\\
&=(m_{11}-m_{22})\chi_-(M)-(m_{12}-m_{23})\chi_-(M)
+(m_{22} -m_{33})\chi_-(M)+\bar{D}(M)\chi_- (M)=0.
\end{align*}
Hence we obtain 
$\cpa{A}{\lambda }{21;2;01;1}{M}=\cpa{B}{\lambda }{21;2;01;1}{M}$.

We have
\begin{align*}
\cpa{A}{\lambda }{21;2;01;2}{M}=
	&(m_{12}-m_{23})\chi_-(M)
	\{m_{12}-m_{11}-(\delta (M)+1)\chid{+}{-1}(M)\}\\
	&=(m_{12}-m_{23})(m_{12}-m_{11})\chi_-(M),\\[3mm]
\cpa{B}{\lambda }{21;2;01;2}{M}=&(m_{12}-m_{23})(m_{12}-m_{11})\chi_-(M).
\end{align*}
Hence we obtain 
$\cpa{A}{\lambda }{21;2;01;2}{M}=\cpa{B}{\lambda }{21;2;01;2}{M}$.

We have
\begin{align*}
&\cpa{A}{\lambda }{21;2;00;1}{M}\\
&=-(m_{11}-m_{22})\bar{C_1}\bkt{M\cgpt{0}{0}{-1}}
\chi_-\bkt{M\cgpt{0}{0}{-1}}-(m_{12}-m_{23})\chi_-(M)(m_{23} -m_{22}-1)\\
&=\left\{ \begin{array}{ll}
0&\text{ if }\delta (M)>-1,\\
-(m_{12}-m_{23})(m_{23} -m_{22}-1)&\text{ if }\delta (M)=-1,\\
-(m_{11}-m_{22})(m_{12}-m_{11}+1)-(m_{12}-m_{23})(m_{23} -m_{22}-1)
&\text{ if }\delta (M)<-1,
\end{array}\right. \\
&\cpa{B}{\lambda }{21;2;00;1}{M}\\
&=-(m_{11}-m_{22}-1)\bar{C_1} (M)\chi_- (M)
-(m_{12}-m_{23}+1) \chid{-}{1}(M)(m_{23} -m_{22})\\ 
&=\left\{\begin{array}{ll}
0&\text{ if }\delta (M)>-1,\\
-(m_{11}-m_{22}-1)(m_{12}-m_{11})&\text{ if }\delta (M)=-1,\\
-(m_{11}-m_{22}-1)(m_{12}-m_{11})-(m_{12}-m_{23}+1)(m_{23} -m_{22})
&\text{ if }\delta (M)<-1.
\end{array}\right.
\end{align*}
Hence we obtain 
$\cpa{A}{\lambda }{21;2;00;1}{M}=\cpa{B}{\lambda }{21;2;00;1}{M}$.

We have
\begin{align*}
\cpa{A}{\lambda }{21;2;00;2}{M}=&-(m_{12}-m_{23})\chi_-(M) 
	(m_{12}-m_{11})\chi_-\bkt{M\cgpt{0}{0}{-1}[-1]}\\
=&-(m_{12}-m_{23})(m_{12}-m_{11})\chid{-}{1}(M)\chi_-(M),\\
\cpa{B}{\lambda }{21;2;00;2}{M}=
	&-(m_{12}-m_{23}) \chid{-}{1}(M)(m_{12}-m_{11})\chi_-(M). 
\end{align*}
Hence we obtain 
$\cpa{A}{\lambda }{21;2;00;2}{M}=\cpa{B}{\lambda }{21;2;00;2}{M}$.

It is trivial that the equations (\ref{eqn:pf_clebsh(1,0,0)}) 
 hold for $(j,k,l)=(1,1,0),\ (0,1,0)$ and $(0,0,0)$. \\

\noindent $\bullet $ the proof of 
$E_{23}\circ i^{\lambda }_{\me_{2}}
=i^{\lambda }_{\me_{2}}\circ E_{23}$. 

We have
\begin{align*}
\cpa{A}{\lambda }{23;2;11;0}{M}=&(m_{13}-m_{12})(m_{22} -m_{33}),\\
\cpa{A}{\lambda }{23;2;11;1}{M}= 
&-(m_{13}-m_{12})\bar{D}\bkt{M\cgpt{1}{0}{0}}\chi_-\bkt{M\cgpt{1}{0}{0}}\\ 
&+\{ m_{13}-m_{12}-\delta (M) \} \chi_-(M)(m_{22} -m_{33}+1),\\
\cpa{A}{\lambda }{23;2;11;2}{M}=&-\{ m_{13}-m_{12}-\delta (M) \} \chi_-(M) 
\bar{D}\bkt{M\cgpt{1}{0}{0}[-1]}\chi_-\bkt{M\cgpt{1}{0}{0}[-1]},\\
\cpa{A}{\lambda }{23;2;01;0}{M}=&-(m_{13}-m_{12})(m_{22} -m_{33}),\\ 
\cpa{A}{\lambda }{23;2;01;1}{M}=
&(m_{13}-m_{12})\bar{C_1}\bkt{M\cgpt{1}{0}{0}}\\ 
&-\{ m_{13}-m_{12}-\delta (M) \} \chi_-(M) (m_{22} -m_{33}+1),\\
\cpa{A}{\lambda }{23;2;01;2}{M}= 
&\{ m_{13}-m_{12}-\delta (M) \} \chi_-(M)\bar{C_1}\bkt{M\cgpt{1}{0}{0}[-1]},\\
\cpa{A}{\lambda }{23;2;00;0}{M}=&-(m_{13}-m_{12})(m_{23} -m_{22}),\\ 
\cpa{A}{\lambda }{23;2;00;1}{M}= 
&-(m_{13}-m_{12})\bar{C_1}\bkt{M\cgpt{1}{0}{0}}\chi_-\bkt{M\cgpt{1}{0}{0}}\\ 
&-\{ m_{13}-m_{12}-\delta (M) \} \chi_-(M) (m_{23} -m_{22}-1),\\
\cpa{A}{\lambda }{23;2;00;2}{M}= 
&-\{ m_{13}-m_{12}-\delta (M) \} \chi_-(M) 
 \bar{C_1}\bkt{M\cgpt{1}{0}{0}[-1]}\chi_-\bkt{M\cgpt{1}{0}{0}[-1]},
\end{align*}
and
\begin{align*}
\cpa{B}{\lambda }{23;2;11;0}{M}=&(m_{13}-m_{12})(m_{22} -m_{33}),\\ 
\cpa{B}{\lambda }{23;2;11;1}{M}=&-(m_{13}-m_{12}+1)\bar{D} (M)\chi_- (M)\\ 
&+\{ m_{13}-m_{12}-\delta (M)-1 \}\chid{-}{1}(M)(m_{22} -m_{33}),\\ 
\cpa{B}{\lambda }{23;2;11;2}{M}=&-\{ m_{13}-m_{12}-\delta (M)\} \chid{-}{1}(M) 
 \bar{D} (M)\chi_- (M),\\ 
\cpa{B}{\lambda }{23;2;01;0}{M}=&-(m_{13}-m_{12})(m_{22} -m_{33}),\\
\cpa{B}{\lambda }{23;2;01;1}{M}=&(m_{13}-m_{12}+1)\bar{C_1} (M)\\ 
&-\{ m_{13}-m_{12}-\delta (M) \} \chi_-(M) (m_{22} -m_{33})-(m_{23} -m_{22}),\\
\cpa{B}{\lambda }{23;2;01;2}{M}=
&\{ m_{13}-m_{12}-\delta (M)+1 \} \chi_-(M) \bar{C_1} (M)
-\bar{C_1} (M)\chi_- (M),\\ 
\cpa{B}{\lambda }{23;2;00;0}{M}=&-(m_{13}-m_{12})(m_{23} -m_{22}),\\ 
\cpa{B}{\lambda }{23;2;00;1}{M}=&-(m_{13}-m_{12}+1)\bar{C_1} (M)\chi_- (M)\\ 
&-\{ m_{13}-m_{12}-\delta (M)-1 \} \chid{-}{1}(M) (m_{23} -m_{22}),\\ 
\cpa{B}{\lambda }{23;2;00;2}{M}= 
&-\{ m_{13}-m_{12}-\delta (M)\}\chid{-}{1}(M)\bar{C_1} (M)\chi_- (M).
\end{align*}

We have
\begin{align*}
\cpa{A}{\lambda }{23;2;11;1}{M}= 
&-(m_{13}-m_{12})(-m_{22}+m_{33}+\delta (M)+1)\chid{-}{1}(M)\\ 
&+\{ m_{13}-m_{12}-\delta (M) \} \chi_-(M)(m_{22} -m_{33}+1),\\
\cpa{B}{\lambda }{23;2;11;1}{M}=
&-(m_{13}-m_{12}+1)(-m_{22}+m_{33}+\delta (M))\chi_- (M)\\ 
&+\{ m_{13}-m_{12}-\delta (M)-1 \}\chid{-}{1}(M)(m_{22} -m_{33}).
\end{align*}
Therefore
\begin{align*}
&\cpa{A}{\lambda }{23;2;11;1}{M}-\cpa{B}{\lambda }{23;2;11;1}{M}\\
&=(m_{13}-m_{12}-m_{22}+m_{33})(\delta (M)+1)(\chi_-(M)-\chid{-}{1}(M))=0.
\end{align*}
Hence we obtain 
$\cpa{A}{\lambda }{23;2;11;1}{M}=\cpa{B}{\lambda }{23;2;11;1}{M}$.

We have
\begin{align*}
\cpa{A}{\lambda }{23;2;11;2}{M}=&-\{ m_{13}-m_{12}-\delta (M) \} \chi_-(M) 
(-m_{22}+m_{33}+\delta (M))\chid{-}{1}(M)\\
=&-\{ m_{13}-m_{12}-\delta (M)\}  \bar{D} (M)\chid{-}{1}(M)\chi_- (M),\\[3mm]
\cpa{B}{\lambda }{23;2;11;2}{M}=&-\{ m_{13}-m_{12}-\delta (M)\}  
 \bar{D} (M)\chid{-}{1}(M)\chi_- (M). 
\end{align*}
Hence we obtain 
$\cpa{A}{\lambda }{23;2;11;2}{M}=\cpa{B}{\lambda }{23;2;11;2}{M}$.

We have
\begin{align*}
\cpa{A}{\lambda }{23;2;01;1}{M}=
&(m_{13}-m_{12})\{m_{23}-m_{22}+(\delta (M)+1)\chi_-(M)\}\\ 
&-\{ m_{13}-m_{12}-\delta (M) \} \chi_-(M) (m_{22} -m_{33}+1),\\[3mm]
\cpa{B}{\lambda }{23;2;01;1}{M}=
&(m_{13}-m_{12}+1)(m_{23}-m_{22}+\delta (M)\chi_-(M))\\ 
&-\{ m_{13}-m_{12}-\delta (M) \} \chi_-(M) (m_{22} -m_{33})-(m_{23} -m_{22})\\
=&(m_{13}-m_{12})\{m_{23}-m_{22}+(\delta (M)+1)\chi_-(M)\}\\ 
&-\{ m_{13}-m_{12}-\delta (M) \} \chi_-(M) (m_{22} -m_{33}+1).
\end{align*}
Hence we obtain  
$\cpa{A}{\lambda }{23;2;01;1}{M}=\cpa{B}{\lambda }{23;2;01;1}{M}$.

We have
\begin{align*}
\cpa{A}{\lambda }{23;2;01;2}{M}
=&\{ m_{13}-m_{12}-\delta (M) \} \chi_-(M)
\{m_{12}-m_{11}-(\delta (M)+1)\chid{+}{-1}(M)\}\\
=&\{ m_{13}-m_{12}-\delta (M) \}(m_{12}-m_{11})\chi_-(M),\\[3mm]
\cpa{B}{\lambda }{23;2;01;2}{M}
=&\{ m_{13}-m_{12}-\delta (M) \} \bar{C_1} (M)\chi_-(M) \\
=&\{ m_{13}-m_{12}-\delta (M) \}(m_{12}-m_{11})\chi_-(M).
\end{align*}
Hence we obtain 
$\cpa{A}{\lambda }{23;2;01;2}{M}=\cpa{B}{\lambda }{23;2;01;2}{M}$.

We have
\begin{align*}
\cpa{A}{\lambda }{23;2;00;1}{M}= 
&-(m_{13}-m_{12})(m_{12}-m_{11}+1)\chid{-}{1}(M)\\ 
&-\{ m_{13}-m_{12}-\delta (M) \} \chi_-(M) (m_{23} -m_{22}-1),\\[3mm]
\cpa{B}{\lambda }{23;2;00;1}{M}=&-(m_{13}-m_{12}+1)(m_{12}-m_{11})\chi_- (M)\\ 
&-\{ m_{13}-m_{12}-\delta (M)-1 \} \chid{-}{1}(M) (m_{23} -m_{22}).
\end{align*}
Therefore 
\begin{align*}
&\cpa{A}{\lambda }{23;2;00;1}{M}-\cpa{B}{\lambda }{23;2;00;1}{M}\\
&=(m_{13}-m_{12}+m_{23} -m_{22})(\delta (M)+1)(\chi_-(M)-\chid{-}{1}(M))=0.
\end{align*}
Hence we obtain 
$\cpa{A}{\lambda }{23;2;00;1}{M}=\cpa{B}{\lambda }{23;2;00;1}{M}$.

We have
\begin{align*}
\cpa{A}{\lambda }{23;2;00;2}{M}= 
&-\{ m_{13}-m_{12}-\delta (M) \} \chi_-(M) 
 (m_{12}-m_{11})\chi_-\bkt{M\cgpt{1}{0}{0}[-1]}\\
=&-\{ m_{13}-m_{12}-\delta (M) \} \chi_-(M) 
 (m_{12}-m_{11})\chid{-}{1}(M),\\[3mm]
\cpa{B}{\lambda }{23;2;00;2}{M}= 
&-\{ m_{13}-m_{12}-\delta (M)\}\chid{-}{1}(M)(m_{12}-m_{11})\chi_- (M).
\end{align*}
Hence $\cpa{A}{\lambda }{23;2;00;2}{M}=\cpa{B}{\lambda }{23;2;00;2}{M}$.

It is trivial that the equations (\ref{eqn:pf_clebsh(1,0,0)}) 
 hold for $(j,k,l)=(1,1,0),\ (0,1,0)$ and $(0,0,0)$. \\

\noindent $\bullet $ the proof of 
$E_{32}\circ i^{\lambda }_{\me_{2}}
=i^{\lambda }_{\me_{2}}\circ E_{32}$. 

We have
\begin{align*}
\cpa{A}{\lambda }{32;2;11;0}{M}=&(m_{22}-m_{33})(m_{22} -m_{33}-1),\\
\cpa{A}{\lambda }{32;2;11;1}{M}= 
&-(m_{22}-m_{33})\bar{D}\bkt{M\cgpt{0}{-1}{0}}\chi_-\bkt{M\cgpt{0}{-1}{0}}\\ 
&+\{ m_{22}-m_{33}+\delta (M) \} \chi_+(M) (m_{22} -m_{33}),\\
\cpa{A}{\lambda }{32;2;11;2}{M}=&-\{ m_{22}-m_{33}+\delta (M) \} \chi_+(M) 
 \bar{D}\bkt{M\cgpt{0}{-1}{0}[-1]}\chi_-\bkt{M\cgpt{0}{-1}{0}[-1]},\\
\cpa{A}{\lambda }{32;2;01;0}{M}=&-(m_{22}-m_{33})(m_{22} -m_{33}-1),\\ 
\cpa{A}{\lambda }{32;2;01;1}{M}= 
&(m_{22}-m_{33})\bar{C_1}\bkt{M\cgpt{0}{-1}{0}}\\ 
&-\{ m_{22}-m_{33}+\delta (M) \} \chi_+(M)(m_{22} -m_{33}),\\
\cpa{A}{\lambda }{32;2;01;2}{M}= 
&\{ m_{22}-m_{33}+\delta (M) \} \chi_+(M)\bar{C_1}\bkt{M\cgpt{0}{-1}{0}[-1]},\\
\cpa{A}{\lambda }{32;2;00;0}{M}=&-(m_{22}-m_{33})(m_{23} -m_{22}+1),\\ 
\cpa{A}{\lambda }{32;2;00;1}{M}= 
&-(m_{22}-m_{33})\bar{C_1}\bkt{M\cgpt{0}{-1}{0}}\chi_-\bkt{M\cgpt{0}{-1}{0}}\\ 
&-\{ m_{22}-m_{33}+\delta (M) \} \chi_+(M)(m_{23} -m_{22}),\\
\cpa{A}{\lambda }{32;2;00;2}{M}=&-\{ m_{22}-m_{33}+\delta (M) \} \chi_+(M) 
 \bar{C_1}\bkt{M\cgpt{0}{-1}{0}[-1]}\chi_-\bkt{M\cgpt{0}{-1}{0}[-1]},
\end{align*}
and
\begin{align*}
\cpa{B}{\lambda }{32;2;11;0}{M}=&(m_{22}-m_{33}-1)(m_{22} -m_{33}),\\ 
\cpa{B}{\lambda }{32;2;11;1}{M}=&-(m_{22}-m_{33})\bar{D} (M)\chi_- (M)\\ 
&+\{ m_{22}-m_{33}+\delta (M)\} \chid{+}{-1}(M) (m_{22} -m_{33}),\\
\cpa{B}{\lambda }{32;2;11;2}{M}=
&-\{ m_{22}-m_{33}+\delta (M)+1\} \chid{+}{-1}(M)\bar{D} (M)\chi_- (M),\\
\cpa{B}{\lambda }{32;2;01;0}{M}=&-(m_{22}-m_{33}-1)(m_{22} -m_{33}),\\ 
\cpa{B}{\lambda }{32;2;01;1}{M}=&(m_{22}-m_{33})\bar{C_1} (M)\\ 
&-\{ m_{22}-m_{33}+\delta (M)-1\} \chi_+(M) (m_{22} -m_{33}),\\ 
\cpa{B}{\lambda }{32;2;01;2}{M}= 
&\{ m_{22}-m_{33}+\delta (M) \} \chi_+(M) \bar{C_1} (M),\\ 
\cpa{B}{\lambda }{32;2;00;0}{M}=
&-(m_{22}-m_{33})(m_{23} -m_{22})-(m_{22} -m_{33}),\\
\cpa{B}{\lambda }{32;2;00;1}{M}= 
&-(m_{22}-m_{33}+1)\bar{C_1} (M)\chi_- (M)\\ 
&-\{ m_{22}-m_{33}+\delta (M)+1 \}\chid{+}{-1}(M) (m_{23} -m_{22})
+ \bar{C_1} (M),\\
\cpa{B}{\lambda }{32;2;00;2}{M}= 
&-\{ m_{22}-m_{33}+\delta (M)+2 \}\chid{+}{-1}(M)\bar{C_1} (M)\chi_- (M).\\ 
\end{align*}

We have
\begin{align*}
\cpa{A}{\lambda }{32;2;11;1}{M}= 
&-(m_{22}-m_{33})(-m_{22}+m_{33}+\delta (M))(1-\chi_+(M))\\ 
&+\{ m_{22}-m_{33}+\delta (M) \} \chi_+(M) (m_{22} -m_{33})\\
=&(m_{22}-m_{33})\{-\bar{D}(M)+2\delta (M)\chi_+(M)\},\\[3mm]
\cpa{B}{\lambda }{32;2;11;1}{M}=
&-(m_{22}-m_{33})(-m_{22}+m_{33}+\delta (M))(1-\chid{+}{-1}(M))\\ 
&+\{ m_{22}-m_{33}+\delta (M)\} \chid{+}{-1}(M) (m_{22} -m_{33})\\
=&(m_{22}-m_{33})\{-\bar{D}(M)+2\delta (M)\chid{+}{-1}(M)\}\\
=&(m_{22}-m_{33})\{-\bar{D}(M)+2\delta (M)\chi_+(M)\}.
\end{align*}
Hence we obtain 
$\cpa{A}{\lambda }{32;2;11;1}{M}=\cpa{B}{\lambda }{32;2;11;1}{M}$.

We have
\begin{align*}
\cpa{A}{\lambda }{32;2;11;2}{M}=&-\{ m_{22}-m_{33}+\delta (M) \}  
 \bar{D}\bkt{M\cgpt{0}{-1}{0}[-1]}\chi_+(M)\chid{-}{-1}(M)=0,\\[3mm]
\cpa{B}{\lambda }{32;2;11;2}{M}=
&-\{ m_{22}-m_{33}+\delta (M)+1\} \bar{D} (M)\chid{+}{-1}(M)\chi_- (M)=0.
\end{align*}
Hence we obtain 
$\cpa{A}{\lambda }{32;2;11;2}{M}=\cpa{B}{\lambda }{32;2;11;2}{M}$.

We have
\begin{align*}
\cpa{A}{\lambda }{32;2;01;1}{M}= 
&(m_{22}-m_{33})\{m_{12}-m_{11}-(\delta (M)-1)\chi_+(M)\}\\ 
&-\{ m_{22}-m_{33}+\delta (M) \} \chi_+(M)(m_{22} -m_{33})\\
=&(m_{22}-m_{33})\{(m_{12}-m_{11})-\chi_+(M)
(m_{22}-m_{33}+2\delta (M)-1)\},\\[3mm]
\cpa{B}{\lambda }{32;2;01;1}{M}=
&(m_{22}-m_{33})\{m_{12}-m_{11}-\delta (M)\chi_+(M)\}\\ 
&-\{ m_{22}-m_{33}+\delta (M)-1\} \chi_+(M) (m_{22} -m_{33})\\
=&(m_{22}-m_{33})\{(m_{12}-m_{11})-\chi_+(M) 
(m_{22}-m_{33}+2\delta (M)-1)\}.
\end{align*}
Hence we obtain 
$\cpa{A}{\lambda }{32;2;01;1}{M}=\cpa{B}{\lambda }{32;2;01;1}{M}$.

We have
\begin{align*}
\cpa{A}{\lambda }{32;2;01;2}{M}= 
&\{ m_{22}-m_{33}+\delta (M) \} 
\chi_+(M)(m_{23}-m_{22}+(\delta (M)-1)\chi_-(M))\\
=&\{ m_{22}-m_{33}+\delta (M) \}(m_{23}-m_{22})\chi_+(M),\\
\cpa{B}{\lambda }{32;2;01;2}{M}= 
&\{ m_{22}-m_{33}+\delta (M) \} (m_{23}-m_{22})\chi_+(M) .
\end{align*}
Hence we obtain 
$\cpa{A}{\lambda }{32;2;01;2}{M}=\cpa{B}{\lambda }{32;2;01;2}{M}$.

We have
\begin{align*}
\cpa{A}{\lambda }{32;2;00;1}{M}= 
&-(m_{22}-m_{33})(m_{12}-m_{11})\chid{-}{-1}(M)\\ 
&-\{ m_{22}-m_{33}+\delta (M) \} (m_{23} -m_{22})\chi_+(M)\\
=&-(m_{22}-m_{33})\bar{C_1}(M)\chid{-}{-1}(M)
-\{ m_{22}-m_{33}+\delta (M) \} \bar{C_1}(M)\chi_+(M)\\
=&-(m_{22}-m_{33})\bar{C_1}(M)(\chid{-}{-1}(M)+\chi_+(M))
-\bar{C_1}(M)\delta (M)\chi_+(M)\\
=&-(m_{22}-m_{33})\bar{C_1}(M)-\bar{C_1}(M)\delta (M)\chi_+(M),\\[3mm]
\cpa{B}{\lambda }{32;2;00;1}{M}= 
&-(m_{22}-m_{33}+1)\bar{C_1} (M)\chi_- (M)\\ 
&-\{ m_{22}-m_{33}+\delta (M)+1 \}\bar{C_1}(M)\chid{+}{-1}(M) +\bar{C_1} (M)\\
=&-(m_{22}-m_{33}+1)\bar{C_1} (M)(\chi_- (M)+\chid{+}{-1}(M))+\bar{C_1} (M)\\
&-\bar{C_1}(M)\delta (M)\chid{+}{-1}(M)\\
=&-(m_{22}-m_{33})\bar{C_1}(M)-\bar{C_1}(M)\delta (M)\chi_+(M).
\end{align*}
Hence we obtain 
$\cpa{A}{\lambda }{32;2;00;1}{M}=\cpa{B}{\lambda }{32;2;00;1}{M}$.

We have
\begin{align*}
\cpa{A}{\lambda }{32;2;00;2}{M}=&-\{ m_{22}-m_{33}+\delta (M) \}  
 \bar{C_1}\bkt{M\cgpt{0}{-1}{0}[-1]}\chi_+(M)\chid{-}{-1}(M)=0,\\[3mm]
\cpa{B}{\lambda }{32;2;00;2}{M}= 
&-\{ m_{22}-m_{33}+\delta (M)+2 \}\bar{C_1} (M)\chid{+}{-1}(M)\chi_- (M)=0.
\end{align*}
Hence we obtain 
$\cpa{A}{\lambda }{32;2;00;2}{M}=\cpa{B}{\lambda }{32;2;00;2}{M}$.

It is trivial that the equations (\ref{eqn:pf_clebsh(1,0,0)}) 
 hold for $(j,k,l)=(1,1,0),\ (0,1,0)$ and $(0,0,0)$. \\
%%%%%%%%%%%%%%%%%%%%%%%%%%%%%%%%%%%%%%%%%%%%%%%%%%%%%%%
%                  end formula2                       %
%%%%%%%%%%%%%%%%%%%%%%%%%%%%%%%%%%%%%%%%%%%%%%%%%%%%%%%

%%%%%%%%%%%%%%%%%%%%%%%%%%%%%%%%%%%%%%%%%%%%%%%%%%%%%%%
%                  begin formula1                     %
%%%%%%%%%%%%%%%%%%%%%%%%%%%%%%%%%%%%%%%%%%%%%%%%%%%%%%%
At last, we check the equations \ref{eqn:pf_clebsh(1,0,0)} 
for $i=1$, that is, the case of formula $1$ by direct computation.

\noindent $\bullet $ the proof of 
$E_{12}\circ i^{\lambda }_{\me_{1}}
=i^{\lambda }_{\me_{1}}\circ E_{12}$.

We have
\begin{align*}
\cpa{A}{\lambda }{12;1;11;0}{M}=
&(m_{12}-m_{11})(m_{13} -m_{12})(m_{22}-m_{33}),\\
\cpa{A}{\lambda }{12;1;11;1}{M}=
&-(m_{12}-m_{11})\bar{E}\bkt{M\cgpt{0}{0}{1}}\\ 
&+(m_{23}-m_{22})\chi_+(M)(m_{13} -m_{12}+1)(m_{22}-m_{33}+1), \\
\cpa{A}{\lambda }{12;1;11;2}{M}= 
&-(m_{23}-m_{22})\chi_+(M)\bar{E}\bkt{M\cgpt{0}{0}{1}[-1]},\\
\cpa{A}{\lambda }{12;1;01;0}{M}= 
&-(m_{12}-m_{11})(m_{13} -m_{12})(m_{22}-m_{33}),\\ 
\cpa{A}{\lambda }{12;1;01;1}{M}= 
&(m_{12}-m_{11})\bar{F}\bkt{M\cgpt{0}{0}{1}}\\ 
&-(m_{23}-m_{22})\chi_+(M)(m_{13} -m_{12}+1)(m_{22}-m_{33}+1),\\
\cpa{A}{\lambda }{12;1;01;2}{M}= 
&-(m_{12}-m_{11})C_2\bkt{M\cgpt{0}{0}{1}}\chi_+\bkt{M\cgpt{0}{0}{1}}\\ 
&+(m_{23}-m_{22})\chi_+(M)\bar{F}\bkt{M\cgpt{0}{0}{1}[-1]},\\
\cpa{A}{\lambda }{12;1;01;3}{M}= 
&-(m_{23}-m_{22})\chi_+(M)C_2\bkt{M\cgpt{0}{0}{1}[-1]}
\chi_+\bkt{M\cgpt{0}{0}{1}[-1]},\\
\cpa{A}{\lambda }{12;1;00;0}{M}= 
&-(m_{12}-m_{11})(m_{13} -m_{12})(m_{13}-m_{22}+1),\\ 
\cpa{A}{\lambda }{12;1;00;1}{M}= 
&(m_{12}-m_{11})C_2\bkt{M\cgpt{0}{0}{1}}\\ 
&-(m_{23}-m_{22})\chi_+(M)(m_{13} -m_{12}+1)(m_{13}-m_{22}),\\
\cpa{A}{\lambda }{12;1;00;2}{M}= 
&(m_{23}-m_{22})\chi_+(M)C_2\bkt{M\cgpt{0}{0}{1}[-1]},
\end{align*}
and
\begin{align*}
\cpa{B}{\lambda }{12;1;11;0}{M}= 
&(m_{12}-m_{11}+1)(m_{13} -m_{12})(m_{22}-m_{33})
-(m_{13} -m_{12})(m_{22}-m_{33}),\\
\cpa{B}{\lambda }{12;1;11;1}{M}= 
&-(m_{12}-m_{11})\bar{E} (M)\\ 
&+(m_{23}-m_{22}+1) \chi_+(M)(m_{13} -m_{12})(m_{22}-m_{33})+\bar{F}(M),\\
\cpa{B}{\lambda }{12;1;11;2}{M}= 
&-(m_{23}-m_{22})\chi_+(M)\bar{E}(M)-C_2 (M)\chi_+ (M),\\
\cpa{B}{\lambda }{12;1;01;0}{M}= 
&-(m_{12}-m_{11})(m_{13} -m_{12})(m_{22}-m_{33}),\\ 
\cpa{B}{\lambda }{12;1;01;1}{M}= 
&(m_{12}-m_{11}-1)\bar{F} (M)\\ 
&-(m_{23}-m_{22}+1)\chid{+}{1}(M)(m_{13} -m_{12})(m_{22}-m_{33}),\\ 
\cpa{B}{\lambda }{12;1;01;2}{M}= 
&-(m_{12}-m_{11}-2)C_2 (M)\chi_+ (M)
+(m_{23}-m_{22}) \chid{+}{1}(M)\bar{F} (M),\\ 
\cpa{B}{\lambda }{12;1;01;3}{M}= 
&-(m_{23}-m_{22}-1)\chid{+}{1}(M)C_2 (M)\chi_+(M),\\ 
\cpa{B}{\lambda }{12;1;00;0}{M}= 
&-(m_{12}-m_{11})(m_{13} -m_{12})(m_{13}-m_{22}+1),\\ 
\cpa{B}{\lambda }{12;1;00;1}{M}= 
&(m_{12}-m_{11}-1)C_2 (M)\\ 
&-(m_{23}-m_{22})\chi_+(M)(m_{13} -m_{12})(m_{13}-m_{22}+1),\\ 
\cpa{B}{\lambda }{12;1;00;2}{M}= 
&(m_{23}-m_{22}-1) \chi_+(M)C_2 (M). 
\end{align*}

We have
\begin{align*}
&\cpa{A}{\lambda }{12;1;11;1}{M}-\cpa{B}{\lambda }{12;1;11;1}{M}\\
&=-(m_{12}-m_{11})\{(m_{12}-m_{23})
+(m_{13}-m_{33}-m_{12}+m_{22}+1)\chi_+(M)\}\\ 
&\hphantom{=}+(m_{13}-m_{33}-m_{12}+m_{22}+1)(m_{23}-m_{22})\chi_+(M)\\
&\hphantom{=}-(m_{13} -m_{12})(m_{22}-m_{33})\chi_+(M)-\bar{F}(M)\\
&=-(m_{13}-m_{33}-m_{12}+m_{22}+1)\delta (M)\chi_+(M)\\
&\hphantom{=}-(m_{13} -m_{12})(m_{22}-m_{33})\chi_+(M)
-(m_{12}-m_{11})(m_{12}-m_{23})-\bar{F}(M)\\
&=-\{(m_{12}-m_{11})(m_{12}-m_{23})-(m_{12}-m_{22})\delta (M)\chi_+(M)\}\\
&\hphantom{=}-\chi_+(M)\{ (m_{13} -m_{12})(m_{22}-m_{33})
+(m_{13}-m_{33}+1)\delta (M)\}-\bar{F}(M)\\
&=0.
\end{align*}
Hence we obtain 
$\cpa{A}{\lambda }{12;1;11;1}{M}=\cpa{B}{\lambda }{12;1;11;1}{M}$.
Here we use the relations
\begin{align}
\bar{E}\bkt{M\cgpt{0}{0}{1}}=
\bar{E}(M)+(m_{12}-m_{23})+(m_{13}-m_{33}-m_{12}+m_{22}+1)\chi_+(M).
\end{align}

We have
\begin{align*}
\cpa{A}{\lambda }{12;1;11;2}{M}=
&-(m_{23}-m_{22})(\bar{E}(M)+C_1(M))\chi_+(M)\\
=&-(m_{23}-m_{22})\chi_+(M)\bar{E}(M)-C_1(M)(m_{23}-m_{22})\chi_+(M)\\
=&\cpa{B}{\lambda }{12;1;11;2}{M}.
\end{align*}
Hence we obtain 
$\cpa{A}{\lambda }{12;1;11;2}{M}=\cpa{B}{\lambda }{12;1;11;2}{M}$. 
Here we use the relation
\begin{align}
\bar{E}\bkt{M\cgpt{0}{0}{1}[-1]}\chi_+(M)=(\bar{E}(M)+C_1(M))\chi_+(M).
\end{align}

We have
\begin{align*}
&\cpa{A}{\lambda }{12;1;01;1}{M}-\cpa{B}{\lambda }{12;1;01;1}{M}\\
&=(m_{12}-m_{11})\{(m_{12}-m_{23})
-(m_{13} -m_{12})(m_{22}-m_{33})\chid{+}{1}(M)\\
&\hphantom{=}+(m_{13} -m_{12}+1)(m_{22}-m_{33}+1)\chi_+(M)\}+\bar{F}(M)\\
&\hphantom{=}-(m_{23}-m_{22})\chi_+(M)(m_{13} -m_{12}+1)(m_{22}-m_{33}+1)\\
&\hphantom{=}+(m_{23}-m_{22}+1)\chid{+}{1}(M)(m_{13} -m_{12})(m_{22}-m_{33})\\
&=(m_{12}-m_{11})(m_{12}-m_{23})
-(m_{13} -m_{12})(m_{22}-m_{33})(\delta (M)-1)\chid{+}{1}(M)\\
&\hphantom{=}+(m_{13} -m_{12}+1)(m_{22}-m_{33}+1)\delta (M)\chi_+(M)
+\bar{F}(M)\\
&=\{(m_{12}-m_{11})(m_{12}-m_{23})-(m_{12}-m_{22})\delta (M)\chi_+(M)\}\\
&\hphantom{=}+\chi_+(M)\{(m_{13} -m_{12})(m_{22}-m_{33})
+(m_{13}-m_{33}+1)\delta (M)\}
+\bar{F}(M)\\
&=0.
\end{align*}
Hence we obtain 
$\cpa{A}{\lambda }{12;1;01;1}{M}=\cpa{B}{\lambda }{12;1;01;1}{M}$. 
Here we use the relations
\begin{align}
\bar{F}\bkt{M\cgpt{0}{0}{1}}=&
\bar{F}(M)+(m_{12}-m_{23})
-(m_{13} -m_{12})(m_{22}-m_{33})\chid{+}{1}(M)
\label{eqn:pf_clebsh(1,0,0)_F1}\\
&+(m_{13} -m_{12}+1)(m_{22}-m_{33}+1)\chi_+(M).\nonumber
\end{align}

We have
\begin{align*}
\cpa{A}{\lambda }{12;1;01;2}{M}
=&(m_{23}-m_{22})\{-(m_{12}-m_{11})(m_{11}-m_{22}+1)\chid{+}{1}(M)\\
&+\bar{F}\bkt{M\cgpt{0}{0}{1}[-1]}\chi_+(M)\},\\
\cpa{B}{\lambda }{12;1;01;2}{M}
=&(m_{23}-m_{22})\{-(m_{12}-m_{11}-2)(m_{11}-m_{22})\chi_+ (M)\\
&+\bar{F} (M)\chid{+}{1}(M)\}.
\end{align*}
By the relation
\begin{align}
\bar{F}\bkt{M\cgpt{0}{0}{1}[-1]}\chi_+(M)
=&\bar{F} (M)\chid{+}{1}(M)+(m_{12}-m_{11})(m_{11}-m_{22}+1)\chid{+}{1}(M)\\
&-(m_{12}-m_{11}-2)(m_{11}-m_{22})\chi_+ (M),\nonumber
\end{align}
we have $\cpa{A}{\lambda }{12;1;01;2}{M}-\cpa{B}{\lambda }{12;1;01;2}{M}=0$. 
Hence we obtain 
$\cpa{A}{\lambda }{12;1;01;2}{M}=\cpa{B}{\lambda }{12;1;01;2}{M}$.

We have
\begin{align*}
\cpa{A}{\lambda }{12;1;01;3}{M}
=&-(m_{23}-m_{22})\chi_+(M)C_1(M)(m_{23}-m_{22}-1)
\chi_+\bkt{M\cgpt{0}{0}{1}[-1]}\\
=&-\bar{C_1}(M)\chi_+(M)C_1(M)(m_{23}-m_{22}-1)\chid{+}{1}(M)\\
=&\cpa{B}{\lambda }{12;1;01;3}{M}.
\end{align*}
Hence we obtain 
$\cpa{A}{\lambda }{12;1;01;3}{M}=\cpa{B}{\lambda }{12;1;01;3}{M}$.

We have
\begin{align*}
&\cpa{A}{\lambda }{12;1;00;1}{M}-\cpa{B}{\lambda }{12;1;00;1}{M}\\
&=(m_{12}-m_{11})\{(m_{12}-m_{22})\chi_+(M)-(m_{12}-m_{23})\}+C_2 (M)\\
&\hphantom{=}-(m_{23}-m_{22})\chi_+(M)(m_{12}-m_{22})\\
&=C_2 (M)-(m_{12}-m_{23})(m_{12}-m_{11})+(m_{12}-m_{22})\delta (M)\chi_+(M)\\
&=0. 
\end{align*}
Hence we obtain 
$\cpa{A}{\lambda }{12;1;00;1}{M}=\cpa{B}{\lambda }{12;1;00;1}{M}$. 
Here we use the relations
\begin{align}
C_2\bkt{M\cgpt{0}{0}{1}}
=C_2(M)+(m_{12}-m_{22})\chi_+(M)-(m_{12}-m_{23}).
\end{align} 

We have
\begin{align*}
&\cpa{A}{\lambda }{12;1;00;2}{M}\\
&=(m_{11}-m_{22})(m_{23}-m_{22})(m_{23}-m_{22}-1) \chi_+(M)C_2 (M)\\
&=\cpa{B}{\lambda }{12;1;00;2}{M}. 
\end{align*}
Hence we obtain 
$\cpa{A}{\lambda }{12;1;00;2}{M}=\cpa{B}{\lambda }{12;1;00;2}{M}$.

It is trivial that the equations (\ref{eqn:pf_clebsh(1,0,0)}) 
 hold for $(j,k,l)=(1,1,0),\ (0,1,0)$ and $(0,0,0)$. \\

\noindent $\bullet $ the proof of 
$E_{21}\circ i^{\lambda }_{\me_{1}}
=i^{\lambda }_{\me_{1}}\circ E_{21}$

We have
\begin{align*}
\cpa{A}{\lambda }{21;1;11;0}{M}=
&(m_{11}-m_{22})(m_{13} -m_{12})(m_{22}-m_{33}),\\
\cpa{A}{\lambda }{21;1;11;1}{M}= 
&-(m_{11}-m_{22})\bar{E}\bkt{M\cgpt{0}{0}{-1}}\\ 
&+(m_{12}-m_{23})\chi_-(M)(m_{13} -m_{12}+1)(m_{22}-m_{33}+1),\\
\cpa{A}{\lambda }{21;1;11;2}{M}= 
&-(m_{12}-m_{23})\chi_-(M)\bar{E}\bkt{M\cgpt{0}{0}{-1}[-1]},\\
\cpa{A}{\lambda }{21;1;01;0}{M}= 
&-(m_{11}-m_{22})(m_{13} -m_{12})(m_{22}-m_{33}),\\ 
\cpa{A}{\lambda }{21;1;01;1}{M}= 
&(m_{11}-m_{22})\bar{F}\bkt{M\cgpt{0}{0}{-1}}\\ 
&-(m_{12}-m_{23})\chi_-(M)(m_{13} -m_{12}+1)(m_{22}-m_{33}+1),\\
\cpa{A}{\lambda }{21;1;01;2}{M}= 
&-(m_{11}-m_{22})C_2\bkt{M\cgpt{0}{0}{-1}}\chi_+\bkt{M\cgpt{0}{0}{-1}}\\ 
&+(m_{12}-m_{23})\chi_-(M)\bar{F}\bkt{M\cgpt{0}{0}{-1}[-1]},\\
\cpa{A}{\lambda }{21;1;01;3}{M}=&-(m_{12}-m_{23})\chi_-(M)
C_2\bkt{M\cgpt{0}{0}{-1}[-1]}\chi_+\bkt{M\cgpt{0}{0}{-1}[-1]},\\
\cpa{A}{\lambda }{21;1;00;0}{M}= 
&-(m_{11}-m_{22})(m_{13} -m_{12})(m_{13}-m_{22}+1),\\ 
\cpa{A}{\lambda }{21;1;00;1}{M}= 
&(m_{11}-m_{22})C_2\bkt{M\cgpt{0}{0}{-1}}\\ 
&-(m_{12}-m_{23})\chi_-(M) (m_{13} -m_{12}+1)(m_{13}-m_{22}),\\
\cpa{A}{\lambda }{21;1;00;2}{M}= 
&(m_{12}-m_{23})\chi_-(M)C_2\bkt{M\cgpt{0}{0}{-1}[-1]},
\end{align*}
and
\begin{align*}
\cpa{B}{\lambda }{21;1;11;0}{M}= 
&(m_{11}-m_{22})(m_{13} -m_{12})(m_{22}-m_{33}),\\ 
\cpa{B}{\lambda }{21;1;11;1}{M}= 
&-(m_{11}-m_{22}-1)\bar{E}(M)\\ 
&+(m_{12}-m_{23})\chi_-(M)(m_{13} -m_{12})(m_{22}-m_{33}),\\ 
\cpa{B}{\lambda }{21;1;11;2}{M}= 
&-(m_{12}-m_{23}-1)\chi_-(M)\bar{E}(M),\\ 
\cpa{B}{\lambda }{21;1;01;0}{M}= 
&-(m_{11}-m_{22}+1)(m_{13} -m_{12})(m_{22}-m_{33})
+(m_{13} -m_{12})(m_{22}-m_{33}),\\
\cpa{B}{\lambda }{21;1;01;1}{M}= 
&(m_{11}-m_{22})\bar{F}(M)\\ 
&-(m_{12}-m_{23}) \chid{-}{-1}(M)(m_{13} -m_{12})(m_{22}-m_{33})-\bar{E}(M),\\
\cpa{B}{\lambda }{21;1;01;2}{M}= 
&-(m_{11}-m_{22}-1)C_2 (M)\chi_+ (M)\\ 
&+(m_{12}-m_{23}-1)\chid{-}{-1}(M)\bar{F}(M),\\ 
\cpa{B}{\lambda }{21;1;01;3}{M}= 
&-(m_{12}-m_{23}-2) \chid{-}{-1}(M)C_2 (M)\chi_+(M),\\ 
\cpa{B}{\lambda }{21;1;00;0}{M}= 
&-(m_{11}-m_{22})(m_{13} -m_{12})(m_{13}-m_{22}+1),\\ 
\cpa{B}{\lambda }{21;1;00;1}{M}=&(m_{11}-m_{22}-1)C_2 (M)\\ 
&-(m_{12}-m_{23}) \chi_-(M)(m_{13} -m_{12})(m_{13}-m_{22}+1),\\ 
\cpa{B}{\lambda }{21;1;00;2}{M}= 
&(m_{12}-m_{23}-1) \chi_-(M)C_2(M). 
\end{align*}

We have
\begin{align*}
&\cpa{A}{\lambda }{21;1;11;1}{M}-\cpa{B}{\lambda }{21;1;11;1}{M}\\
&=(m_{11}-m_{22})(m_{12}-m_{23})\chi_-(M)+\bar{E}(M)\chid{+}{-1}(M)\\
&\hphantom{=}+(m_{13}-m_{33}-m_{12}+m_{22}+1)(m_{12}-m_{23})
\chi_-(M)-\bar{E}(M)\\
&=(m_{12}-m_{23})(m_{13}-m_{33}+1-m_{12}+m_{11})\chi_-(M)
+\bar{E}(M)(\chid{+}{-1}(M)-1)\\
&=\bar{E}(M)(\chi_-(M)+\chid{+}{-1}(M)-1)=0.
\end{align*}
Hence we obtain 
$\cpa{A}{\lambda }{21;1;11;1}{M}=\cpa{B}{\lambda }{21;1;11;1}{M}$. 
Here we use the relations
\begin{align}
(m_{11}-m_{22})\bar{E}\bkt{M\cgpt{0}{0}{-1}}
=&(m_{11}-m_{22})\{\bar{E}(M)-(m_{12}-m_{23})\chi_-(M)\}\\
&-\bar{E}(M)\chid{+}{-1}(M).\nonumber
\end{align}

We have
\begin{align*}
\cpa{A}{\lambda }{21;1;11;2}{M}= 
&-(m_{12}-m_{23})\chi_-(M)(m_{12}-m_{23}-1)
\{ m_{13}-m_{33}+1+m_{12}-m_{11}\}\\
=&\cpa{B}{\lambda }{21;1;11;2}{M}.
\end{align*}
Hence we obtain 
$\cpa{A}{\lambda }{21;1;11;2}{M}=\cpa{B}{\lambda }{21;1;11;2}{M}$.

We have
\begin{align*}
&\cpa{A}{\lambda }{21;1;01;1}{M}-\cpa{B}{\lambda }{21;1;01;1}{M}\\
&=(m_{11}-m_{22})\{-(m_{13} -m_{12})(m_{22}-m_{33})\chid{-}{-1}(M)\\
&\hphantom{=}+(m_{13} -m_{12}+1)(m_{22}-m_{33}+1)\chi_-(M)
-(m_{13} -m_{23}-m_{33}+m_{22}+1)\}\\
&\hphantom{=}-(m_{12}-m_{23})\chi_-(M)(m_{13} -m_{12}+1)(m_{22}-m_{33}+1)\\
&\hphantom{=}+(m_{12}-m_{23}) \chid{-}{-1}(M)(m_{13} -m_{12})(m_{22}-m_{33})
+\bar{E}(M)\\
&=-\delta (M)\chi_-(M)(m_{13} -m_{12}+1)(m_{22}-m_{33}+1)\\
&\hphantom{=}+\delta (M) \chid{-}{-1}(M)(m_{13} -m_{12})(m_{22}-m_{33})\\
&\hphantom{=}-(m_{11}-m_{22})(m_{13} -m_{23}-m_{33}+m_{22}+1)+\bar{E}(M)\\
&=-(m_{11}-m_{22})(m_{13} -m_{23}-m_{33}+m_{22}+1)\\
&\hphantom{=}-(m_{13} -m_{33}-m_{12}+m_{22}+1)\delta (M)\chi_-(M)+\bar{E}(M)\\
&=0.
\end{align*}
Hence we obtain 
$\cpa{A}{\lambda }{21;1;01;1}{M}=\cpa{B}{\lambda }{21;1;01;1}{M}$. 
Here we use the relations
\begin{align}
\bar{F}\bkt{M\cgpt{0}{0}{-1}}=&
\bar{F}(M)-(m_{13} -m_{12})(m_{22}-m_{33})\chid{-}{-1}(M)\\
&+(m_{13} -m_{12}+1)(m_{22}-m_{33}+1)\chi_-(M)\nonumber \\
&-(m_{13} -m_{23}-m_{33}+m_{22}+1),\nonumber \\
\bar{E}(M)=&(m_{11}-m_{22})(m_{13} -m_{23}-m_{33}+m_{22}+1)\\
&+(m_{13} -m_{33}-m_{12}+m_{22}+1)\delta (M)\chi_-(M)\nonumber .
\end{align}

We have
\begin{align*}
\cpa{A}{\lambda }{21;1;01;2}{M}= 
&-(m_{11}-m_{22})(m_{11}-m_{22}-1)(m_{23}-m_{22})\chid{+}{-1}(M)\\ 
&-(m_{12}-m_{23})(m_{12}-m_{23}-1)(m_{12}-m_{11})\chi_-(M)\\
=&-C_2(M)(C_1(M)-1)\chid{+}{-1}(M)-C_2(M)(C_1(M)-1)\chi_-(M)\\
=&-C_2(M)(C_1(M)-1),\\
\cpa{B}{\lambda }{21;1;01;2}{M}= 
&-(C_1(M)-1)C_2 (M)\chi_+ (M)-(C_1(M)-1)C_2 (M)\chid{-}{-1}(M)\\
=&-C_2(M)(C_1(M)-1), 
\end{align*}
Hence we obtain 
$\cpa{A}{\lambda }{21;1;01;2}{M}=\cpa{B}{\lambda }{21;1;01;2}{M}$.
Here we use the relation
\begin{align}
\bar{F}(M)\chid{-}{r}(M)=-C_2(M)\chid{-}{r}(M)\quad (r\geq -1).
\label{eqn:pf_clebsh(1,0,0)_FC2}
\end{align}

We have
\begin{align*}
\cpa{A}{\lambda }{21;1;01;3}{M}=&-(m_{12}-m_{23})
C_2\bkt{M\cgpt{0}{0}{-1}[-1]}\chi_-(M)\chid{+}{-1}(M)=0,\\
\cpa{B}{\lambda }{21;1;01;3}{M}= 
&-(m_{12}-m_{23}-2) C_2 (M)\chid{-}{-1}(M)\chi_+(M)=0,
\end{align*}
Hence we obtain 
$\cpa{A}{\lambda }{21;1;01;3}{M}=\cpa{B}{\lambda }{21;1;01;3}{M}$.

We have
\begin{align*}
&\cpa{A}{\lambda }{21;1;00;1}{M}-\cpa{B}{\lambda }{21;1;00;1}{M}\\
&=(m_{11}-m_{22})\{(m_{12}-m_{22})\chi_-(M)-(m_{23}-m_{22})\}+C_2(M)\\
&\hphantom{=}-(m_{12}-m_{23})(m_{12}-m_{22})\chi_-(M)\\
&=-(m_{11}-m_{22})(m_{23}-m_{22})-(m_{12}-m_{22})\delta (M)\chi_-(M)+C_2(M)\\
&=0.
\end{align*}
Hence we obtain 
$\cpa{A}{\lambda }{21;1;00;1}{M}=\cpa{B}{\lambda }{21;1;00;1}{M}$.
Here we use the relations
\begin{align}
C_2\bkt{M\cgpt{0}{0}{-1}}=C_2(M)+(m_{12}-m_{22})\chi_-(M)-(m_{23}-m_{22}).
\end{align}

We have
\begin{align*}
\cpa{A}{\lambda }{21;1;00;2}{M}= 
&(m_{12}-m_{23})(m_{12}-m_{23}-1)(m_{12}-m_{11})\chi_-(M)\\
=&\cpa{B}{\lambda }{21;1;00;2}{M}.
\end{align*}
Hence we obtain 
$\cpa{A}{\lambda }{21;1;00;2}{M}=\cpa{B}{\lambda }{21;1;00;2}{M}$.

It is trivial that the equations (\ref{eqn:pf_clebsh(1,0,0)}) 
 hold for $(j,k,l)=(1,1,0),\ (0,1,0)$ and $(0,0,0)$. \\

\noindent $\bullet $ the proof of 
$E_{23}\circ i^{\lambda }_{\me_{1}}
=i^{\lambda }_{\me_{1}}\circ E_{23}$

We have
\begin{align*}
\cpa{A}{\lambda }{23;1;11;0}{M}= 
&(m_{13}-m_{12})(m_{13} -m_{12}-1)(m_{22}-m_{33}),\\
\cpa{A}{\lambda }{23;1;11;1}{M}= 
&-(m_{13}-m_{12})\bar{E}\bkt{M\cgpt{1}{0}{0}}\\ 
&+\{ m_{13}-m_{12}-\delta (M) \}\chi_-(M) (m_{13} -m_{12})(m_{22}-m_{33}+1),\\
\cpa{A}{\lambda }{23;1;11;2}{M}= 
&-\{ m_{13}-m_{12}-\delta (M)\}\chi_-(M)\bar{E}\bkt{M\cgpt{1}{0}{0}[-1]},\\
\cpa{A}{\lambda }{23;1;01;0}{M}= 
&-(m_{13}-m_{12})(m_{13} -m_{12}-1)(m_{22}-m_{33}),\\ 
\cpa{A}{\lambda }{23;1;01;1}{M}= 
&(m_{13}-m_{12})\bar{F}\bkt{M\cgpt{1}{0}{0}}\\ 
&-\{ m_{13}-m_{12}-\delta (M) \}\chi_-(M)(m_{13} -m_{12})(m_{22}-m_{33}+1),\\
\cpa{A}{\lambda }{23;1;01;2}{M}= 
&-(m_{13}-m_{12})C_2\bkt{M\cgpt{1}{0}{0}}\chi_+\bkt{M\cgpt{1}{0}{0}}\\ 
&+\{ m_{13}-m_{12}-\delta (M) \} \chi_-(M)\bar{F}\bkt{M\cgpt{1}{0}{0}[-1]},\\
\cpa{A}{\lambda }{23;1;01;3}{M}= 
&-\{ m_{13}-m_{12}-\delta (M) \} \chi_-(M) 
C_2\bkt{M\cgpt{1}{0}{0}[-1]}\chi_+\bkt{M\cgpt{1}{0}{0}[-1]},\\
\cpa{A}{\lambda }{23;1;00;0}{M}= 
&-(m_{13}-m_{12})(m_{13} -m_{12}-1)(m_{13}-m_{22}+1),\\ 
\cpa{A}{\lambda }{23;1;00;1}{M}= 
&(m_{13}-m_{12})C_2\bkt{M\cgpt{1}{0}{0}}\\ 
&-\{ m_{13}-m_{12}-\delta (M) \} \chi_-(M)(m_{13} -m_{12})(m_{13}-m_{22}),\\
\cpa{A}{\lambda }{23;1;00;2}{M}= 
&\{ m_{13}-m_{12}-\delta (M) \} \chi_-(M)C_2\bkt{M\cgpt{1}{0}{0}[-1]},
\end{align*}
and
\begin{align*}
\cpa{B}{\lambda }{23;1;11;0}{M}= 
&(m_{13}-m_{12}-1)(m_{13} -m_{12})(m_{22}-m_{33}),\\ 
\cpa{B}{\lambda }{23;1;11;1}{M}=&-(m_{13}-m_{12})\bar{E} (M)\\ 
&+\{ m_{13}-m_{12}-\delta (M)-1\}\chi_-(M)(m_{13} -m_{12})(m_{22}-m_{33}),\\ 
\cpa{B}{\lambda }{23;1;11;2}{M}= 
&-\{ m_{13}-m_{12}-\delta (M)\}\chi_-(M)\bar{E} (M),\\ 
\cpa{B}{\lambda }{23;1;01;0}{M}= 
&-(m_{13}-m_{12}-1)(m_{13} -m_{12})(m_{22}-m_{33}),\\
\cpa{B}{\lambda }{23;1;01;1}{M}=&(m_{13}-m_{12})\bar{F} (M)\\ 
&-\{ m_{13}-m_{12}-\delta (M)\}\chid{-}{-1}(M)(m_{13} -m_{12})(m_{22}-m_{33})\\
&-(m_{13} -m_{12})(m_{13}-m_{22}+1),\\
\cpa{B}{\lambda }{23;1;01;2}{M}= 
&-(m_{13}-m_{12}+1)C_2 (M)\chi_+ (M)\\ 
&+\{ m_{13}-m_{12}-\delta (M)+1 \}\chid{-}{-1}(M) \bar{F} (M)+C_2 (M),\\ 
\cpa{B}{\lambda }{23;1;01;3}{M}= 
&-\{ m_{13}-m_{12}-\delta (M)+2 \}\chid{-}{-1}(M)C_2 (M)\chi_+ (M),\\ 
\cpa{B}{\lambda }{23;1;00;0}{M}= 
&-(m_{13}-m_{12}-1)(m_{13} -m_{12})(m_{13}-m_{22}+1),\\ 
\cpa{B}{\lambda }{23;1;00;1}{M}=&(m_{13}-m_{12})C_2 (M)\\ 
&-\{ m_{13}-m_{12}-\delta (M)-1 \}\chi_-(M)
(m_{13} -m_{12})(m_{13}-m_{22}+1),\\ 
\cpa{B}{\lambda }{23;1;00;2}{M}= 
&\{ m_{13}-m_{12}-\delta (M)\}\chi_-(M) C_2 (M).
\end{align*}

We have
\begin{align*}
\cpa{A}{\lambda }{23;1;11;1}{M}-\cpa{B}{\lambda }{23;1;11;1}{M}
=&-(m_{13}-m_{12})(m_{13}+m_{23}-m_{33}-2m_{12}+m_{11})\chi_-(M)\\
&+(m_{13} -m_{12})\chi_-(M)(m_{13}-m_{12}-
\delta (M)+m_{22}-m_{33})\\
=&0.
\end{align*}
Hence we obtain 
$\cpa{A}{\lambda }{23;1;11;1}{M}=\cpa{B}{\lambda }{23;1;11;1}{M}$. 
Here we use the relation
\begin{align}
\bar{E}\bkt{M\cgpt{1}{0}{0}}
=\bar{E}(M)+(m_{13}+m_{23}-m_{33}-2m_{12}+m_{11})\chi_-(M).
\end{align}

By the relation 
\begin{align}
\bar{E}\bkt{M\cgpt{1}{0}{0}[-1]}\chi_-(M)=\bar{E} (M)\chi_-(M),
\end{align}
Hence we obtain
\begin{align*}
\cpa{A}{\lambda }{23;1;11;2}{M}=\cpa{B}{\lambda }{23;1;11;2}{M}.
\end{align*}

We have
\begin{align*}
&\cpa{A}{\lambda }{23;1;01;1}{M}-\cpa{B}{\lambda }{23;1;01;1}{M}\\
&=(m_{13}-m_{12})\{(m_{13}-m_{12}-\delta (M))\chi_-(M)\\
&\hphantom{=}-(m_{13}-m_{12})(m_{22}-m_{33})(\chid{-}{-1}-\chi_-(M))
+(m_{13}-m_{22}+1)\}\\
&\hphantom{=}-\{ m_{13}-m_{12}-\delta (M) \}\chi_-(M)
	(m_{13} -m_{12})(m_{22}-m_{33}+1)\\
&\hphantom{=}+\{ m_{13}-m_{12}-\delta (M)\}\chid{-}{-1}(M)
	(m_{13} -m_{12})(m_{22}-m_{33})\\
&\hphantom{=}+(m_{13} -m_{12})(m_{13}-m_{22}+1)\\
&=(m_{13} -m_{12})(m_{22}-m_{33})\delta (M)(\chi_-(M)-\chid{-}{-1}(M))\\
&=0.
\end{align*}
Hence we obtain 
$\cpa{A}{\lambda }{23;1;01;1}{M}=\cpa{B}{\lambda }{23;1;01;1}{M}$. 
Here we use the relation
\begin{align}
\bar{F}\bkt{M\cgpt{1}{0}{0}}=&\bar{F}(M)+(m_{13}-m_{12}-\delta (M))\chi_-(M)\\
&-(m_{13}-m_{12})(m_{22}-m_{33})(\chid{-}{-1}-\chi_-(M))
+(m_{13}-m_{22}+1).\nonumber
\end{align}

We have
\begin{align*}
\cpa{A}{\lambda }{23;1;01;2}{M}
=&-(m_{13}-m_{12})(m_{11}-m_{22})(m_{23}-m_{22})\chid{+}{-1}(M)\\
&-\{ m_{13}-m_{12}-\delta (M) \} (m_{12}-m_{23})(m_{12}-m_{11})\chi_-(M)\\
=&-(m_{13}-m_{12})C_2(M)\chid{+}{-1}(M)
-\{ m_{13}-m_{12}-\delta (M) \}C_2(M)\chi_-(M)\\
=&-(m_{13}-m_{12})C_2(M)+\delta (M)C_2(M)\chi_-(M),\\
\cpa{B}{\lambda }{23;1;01;2}{M}= 
&-(m_{13}-m_{12}+1)C_2 (M)\chi_+ (M)\\ 
&-\{ m_{13}-m_{12}-\delta (M)+1 \}C_2 (M)\chid{-}{-1}(M) +C_2 (M)\\
=&-(m_{13}-m_{12}+1)C_2 (M)+C_2 (M)\delta (M)\chid{-}{-1}(M)+C_2 (M)\\
=&-(m_{13}-m_{12})C_2(M)+\delta (M)C_2(M)\chi_-(M).
\end{align*}
Hence we obtain 
$\cpa{A}{\lambda }{23;1;01;2}{M}=\cpa{B}{\lambda }{23;1;01;2}{M}$.
Here we use the relation (\ref{eqn:pf_clebsh(1,0,0)_FC2}).

We have
\begin{align*}
\cpa{A}{\lambda }{23;1;01;3}{M}= 
&-\{ m_{13}-m_{12}-\delta (M) \} C_2\bkt{M\cgpt{1}{0}{0}[-1]}
\chi_-(M)\chid{+}{-1}(M)=0,\\
\cpa{B}{\lambda }{23;1;01;3}{M}= 
&-\{ m_{13}-m_{12}-\delta (M)+2 \}C_2 (M)\chid{-}{-1}(M)\chi_+ (M)=0.
\end{align*}
Hence we obtain 
$\cpa{A}{\lambda }{23;1;01;3}{M}=\cpa{B}{\lambda }{23;1;01;3}{M}$.

We have
\begin{align*}
&\cpa{A}{\lambda }{23;1;00;1}{M}-\cpa{B}{\lambda }{23;1;00;1}{M}\\
&=(m_{13}-m_{12})(m_{12}-m_{22}+1+\delta (M))\chi_-(M)\\
&\hphantom{=}-\{ m_{13}-m_{12}-\delta (M) \} 
	\chi_-(M)(m_{13} -m_{12})(m_{13}-m_{22})\\
&\hphantom{=}+\{ m_{13}-m_{12}-\delta (M)-1 \}
	\chi_-(M)(m_{13} -m_{12})(m_{13}-m_{22}+1)\\
&=0.
\end{align*}
Hence we obtain 
$\cpa{A}{\lambda }{23;1;00;1}{M}=\cpa{B}{\lambda }{23;1;00;1}{M}$. 
Here we use the relation
\begin{align}
C_2\bkt{M\cgpt{1}{0}{0}}=C_2(M)+(m_{12}-m_{22}+1+\delta (M))\chi_-(M).
\end{align}

We have
\begin{align*}
&\cpa{A}{\lambda }{23;1;00;2}{M}\\
&=\{ m_{13}-m_{12}-\delta (M)\}\chi_-(M) C_2 (M)\\
&=\cpa{B}{\lambda }{23;1;00;2}{M}.
\end{align*}
Hence we obtain 
$\cpa{A}{\lambda }{23;1;00;2}{M}=\cpa{B}{\lambda }{23;1;00;2}{M}$. 

It is trivial that the equations (\ref{eqn:pf_clebsh(1,0,0)}) 
 hold for $(j,k,l)=(1,1,0),\ (0,1,0)$ and $(0,0,0)$. \\

\noindent $\bullet $ the proof of 
$E_{32}\circ i^{\lambda }_{\me_{1}}
=i^{\lambda }_{\me_{1}}\circ E_{32}$

We have
\begin{align*}
\cpa{A}{\lambda }{32;1;11;0}{M}= 
&(m_{22}-m_{33})(m_{13} -m_{12})(m_{22}-m_{33}-1),\\
\cpa{A}{\lambda }{32;1;11;1}{M}= 
&-(m_{22}-m_{33})\bar{E}\bkt{M\cgpt{0}{-1}{0}}\\ 
&+\{ m_{22}-m_{33}+\delta (M) \} \chi_+(M) (m_{13} -m_{12}+1)(m_{22}-m_{33}),\\
\cpa{A}{\lambda }{32;1;11;2}{M}= 
&-\{ m_{22}-m_{33}+\delta (M) \} \chi_+(M) \bar{E}\bkt{M\cgpt{0}{-1}{0}[-1]},\\
\cpa{A}{\lambda }{32;1;01;0}{M}= 
&-(m_{22}-m_{33})(m_{13} -m_{12})(m_{22}-m_{33}-1),\\ 
\cpa{A}{\lambda }{32;1;01;1}{M}= 
&(m_{22}-m_{33})\bar{F}\bkt{M\cgpt{0}{-1}{0}}\\ 
&-\{ m_{22}-m_{33}+\delta (M) \} \chi_+(M)(m_{13} -m_{12}+1)(m_{22}-m_{33}),\\
\cpa{A}{\lambda }{32;1;01;2}{M}= 
&-(m_{22}-m_{33})C_2\bkt{M\cgpt{0}{-1}{0}}\chi_+\bkt{M\cgpt{0}{-1}{0}}\\ 
&+\{ m_{22}-m_{33}+\delta (M) \} \chi_+(M)\bar{F}\bkt{M\cgpt{0}{-1}{0}[-1]},\\
\cpa{A}{\lambda }{32;1;01;3}{M}=&-\{ m_{22}-m_{33}+\delta (M) \} \chi_+(M)
C_2\bkt{M\cgpt{0}{-1}{0}[-1]}\chi_+\bkt{M\cgpt{0}{-1}{0}[-1]},\\
\cpa{A}{\lambda }{32;1;00;0}{M}= 
&-(m_{22}-m_{33})(m_{13} -m_{12})(m_{13}-m_{22}+2),\\ 
\cpa{A}{\lambda }{32;1;00;1}{M}= 
&(m_{22}-m_{33})C_2\bkt{M\cgpt{0}{-1}{0}}\\ 
&-\{ m_{22}-m_{33}+\delta (M) \} \chi_+(M)
(m_{13} -m_{12}+1)(m_{13}-m_{22}+1),\\
\cpa{A}{\lambda }{32;1;00;2}{M}= 
&\{ m_{22}-m_{33}+\delta (M) \} \chi_+(M)C_2\bkt{M\cgpt{0}{-1}{0}[-1]},
\end{align*}
and
\begin{align*}
\cpa{B}{\lambda }{32;1;11;0}{M}= 
&(m_{22}-m_{33}-1)(m_{13} -m_{12})(m_{22}-m_{33}),\\ 
\cpa{B}{\lambda }{32;1;11;1}{M}= 
&-(m_{22}-m_{33})\bar{E} (M)\\ 
&+\{ m_{22}-m_{33}+\delta (M)-1\}\chi_+(M)(m_{13} -m_{12})(m_{22}-m_{33}),\\
\cpa{B}{\lambda }{32;1;11;2}{M}= 
&-\{ m_{22}-m_{33}+\delta (M)\}\chi_+(M) \bar{E} (M),\\
\cpa{B}{\lambda }{32;1;01;0}{M}= 
&-(m_{22}-m_{33}-1)(m_{13} -m_{12})(m_{22}-m_{33}),\\ 
\cpa{B}{\lambda }{32;1;01;1}{M}= 
&(m_{22}-m_{33})\bar{F} (M)\\ 
&-\{ m_{22}-m_{33}+\delta (M)-2\}\chid{+}{1}(M)
(m_{13} -m_{12})(m_{22}-m_{33}),\\ 
\cpa{B}{\lambda }{32;1;01;2}{M}= 
&-(m_{22}-m_{33}+1)C_2 (M)\chi_+ (M)\\ 
&+\{ m_{22}-m_{33}+\delta (M)-1 \}\chid{+}{1}(M)\bar{F} (M),\\ 
\cpa{B}{\lambda }{32;1;01;3}{M}= 
&-\{ m_{22}-m_{33}+\delta (M)\}\chid{+}{1}(M)C_2 (M)\chi_+ (M),\\ 
\cpa{B}{\lambda }{32;1;00;0}{M}= 
&-(m_{22}-m_{33})(m_{13} -m_{12})(m_{13}-m_{22}+1)
-(m_{13} -m_{12})(m_{22}-m_{33}),\\
\cpa{B}{\lambda }{32;1;00;1}{M}=&(m_{22}-m_{33}+1)C_2 (M)\\ 
&-\{ m_{22}-m_{33}+\delta (M)\}\chi_+(M)(m_{13} -m_{12})(m_{13}-m_{22}+1)
+ \bar{F} (M),\\
\cpa{B}{\lambda }{32;1;00;2}{M}= 
&\{ m_{22}-m_{33}+\delta (M)+1\}\chi_+(M) C_2 (M)-C_2 (M)\chi_+ (M).
\end{align*}

We have
\begin{align*}
&\cpa{A}{\lambda }{32;1;11;1}{M}-\cpa{B}{\lambda }{32;1;11;1}{M}\\
&=-(m_{22}-m_{33})(m_{13}-m_{33}-m_{12}+m_{22}+\delta (M))\chi_+(M)\\
&\hphantom{=}+\{ m_{22}-m_{33}+\delta (M) \} 
	\chi_+(M) (m_{13} -m_{12}+1)(m_{22}-m_{33})\\
&\hphantom{=}-\{ m_{22}-m_{33}+\delta (M)-1\}
	\chi_+(M)(m_{13} -m_{12})(m_{22}-m_{33})\\
&=0.
\end{align*}
Hence we obtain 
$\cpa{A}{\lambda }{32;1;11;1}{M}=\cpa{B}{\lambda }{32;1;11;1}{M}$. 
Here we use the relation
\begin{align}
\bar{E}\bkt{M\cgpt{0}{-1}{0}}
=\bar{E}(M)+(m_{13}-m_{33}-m_{12}+m_{22}+\delta (M))\chi_+(M).
\end{align}

By the relation 
\begin{align}
\bar{E}\bkt{M\cgpt{0}{-1}{0}[-1]}\chi_+(M)=\bar{E} (M)\chi_+(M),
\end{align}
we obtain
\begin{align*}
\cpa{A}{\lambda }{32;1;11;2}{M}=\cpa{B}{\lambda }{32;1;11;2}{M}.
\end{align*}

We have
\begin{align*}
&\cpa{A}{\lambda }{32;1;01;1}{M}-\cpa{B}{\lambda }{32;1;01;1}{M}\\
&=(m_{22}-m_{33})\{-(m_{13}-m_{12})(m_{22}-m_{33}-1)\chid{+}{1}(M)\\
&\hphantom{=}+(m_{13}-m_{12}+1)(m_{22}-m_{33}+1)\chi_+(M)
	+(\delta (M)-1)\chi_+(M)\\
&\hphantom{=}-( m_{22}-m_{33}+\delta (M) )\chi_+(M)(m_{13} -m_{12}+1)\\
&\hphantom{=}+( m_{22}-m_{33}+\delta (M)-2)\chid{+}{1}(M)(m_{13} -m_{12})\}\\
&=(m_{22}-m_{33})(m_{13} -m_{12})(\delta (M)-1)(\chid{+}{1}(M)-\chi_+(M))\\
&=0.
\end{align*}
Hence we obtain 
$\cpa{A}{\lambda }{32;1;01;1}{M}=\cpa{B}{\lambda }{32;1;01;1}{M}$. 
Here we use the relation
\begin{align}
\bar{F}\bkt{M\cgpt{0}{-1}{0}}=&
\bar{F}(M)-(m_{13}-m_{12})(m_{22}-m_{33}-1)\chid{+}{1}\\
&+\{(m_{13}-m_{12}+1)(m_{22}-m_{33}+1)+\delta (M)-1\}\chi_+(M).\nonumber
\end{align}

We have
\begin{align*}
&\cpa{A}{\lambda }{32;1;01;2}{M}-\cpa{B}{\lambda }{32;1;01;2}{M}\\
&=( m_{22}-m_{33}+\delta (M) )
(\bar{F}\bkt{M\cgpt{0}{-1}{0}[-1]}\chi_+(M)-\bar{F}(M)\chid{+}{1}(M))\\
&\hphantom{=}+\bar{F}(M)\chid{+}{1}(M)+(m_{22}-m_{33}+1)C_2 (M)\chi_+ (M)
-(m_{22}-m_{33})C_2\bkt{M\cgpt{0}{-1}{0}}\chid{+}{1}(M)\\
&=( m_{22}-m_{33}+\delta (M) )
\{(m_{13}-m_{22}+1)\chid{+}{1}(M)+C_2(M)\chid{+}{1}(M)\}\\
&\hphantom{=}-(\delta (M)-1)C_2 (M)\chi_+ (M)+\bar{F}(M)\chid{+}{1}(M)
-(m_{22}-m_{33})C_2\bkt{M\cgpt{0}{-1}{0}}\chid{+}{1}(M)\\
&=( m_{22}-m_{33}+\delta (M) )(m_{13}-m_{22}+1)\chid{+}{1}(M)\\
&\hphantom{=}-(m_{22}-m_{33})(C_2\bkt{M\cgpt{0}{-1}{0}}-C_2(M))\chid{+}{1}(M)
+C_2 (M)\chid{+}{1}(M)+\bar{F}(M)\chid{+}{1}(M)\\
&=\chid{+}{1}(M)\{C_2 (M)+(m_{13}-m_{12})(m_{22}-m_{33})
+(m_{13}-m_{33}+1)\delta (M)+\bar{F}(M)\chid{+}{1}(M)\}\\
&=0.
\end{align*}
Hence we obtain 
$\cpa{A}{\lambda }{32;1;01;2}{M}=\cpa{B}{\lambda }{32;1;01;2}{M}$. 
Here we use the relations 
\begin{align}
\bar{F}\bkt{M\cgpt{0}{-1}{0}[-1]}\chi_+(M)=&
\bar{F}(M)\chid{+}{1}(M)+(m_{13}-m_{22}+1)\chid{+}{1}(M)\\
&+C_2(M)(\chid{+}{1}(M)-\chi_+(M))\nonumber ,\\
C_2\bkt{M\cgpt{0}{-1}{0}}\chid{+}{1}(M)
=&(C_2(M)+m_{12}-m_{22}+1-\delta (M))\chid{+}{1}(M) .\nonumber
\end{align}

We have
\begin{align*}
\cpa{A}{\lambda }{32;1;01;3}{M}=&-\{ m_{22}-m_{33}+\delta (M) \} \chi_+(M)
(m_{11}-m_{22})(m_{23}-m_{22})\chid{+}{1}(M)\\
=&\cpa{B}{\lambda }{32;1;01;3}{M}.
\end{align*}
Hence we obtain 
$\cpa{A}{\lambda }{32;1;01;3}{M}=\cpa{B}{\lambda }{32;1;01;3}{M}$. 

We have
\begin{align*}
&\cpa{A}{\lambda }{32;1;00;1}{M}-\cpa{B}{\lambda }{32;1;00;1}{M}\\
&= (m_{22}-m_{33})(m_{12}-m_{22}+1-\delta (M))\chi_+(M)-C_2 (M)\\
&\hphantom{=}-\{ m_{22}-m_{33}+\delta (M) \} 
	\chi_+(M)(m_{13}-m_{22}+1)- \bar{F} (M)\\
&=-C_2 (M)-\chi_+(M)\{(m_{13}-m_{12})(m_{22}-m_{33})
	+(m_{13}-m_{33}+1)\delta (M) \}- \bar{F} (M)\\
&=0.
\end{align*}
Hence we obtain 
$\cpa{A}{\lambda }{32;1;00;1}{M}=\cpa{B}{\lambda }{32;1;00;1}{M}$. 
Here we use the relation
\begin{align}
C_2\bkt{M\cgpt{0}{-1}{0}}=C_2(M)+(m_{12}-m_{22}+1-\delta (M))\chi_+(M).
\end{align}

We have
\begin{align*}
\cpa{A}{\lambda }{32;1;00;2}{M}=\cpa{B}{\lambda }{32;1;00;2}{M}.\\
\end{align*}
Hence we obtain 
$\cpa{A}{\lambda }{32;1;00;2}{M}=\cpa{B}{\lambda }{32;1;00;2}{M}$. 

It is trivial that the equations (\ref{eqn:pf_clebsh(1,0,0)}) 
 hold for $(j,k,l)=(1,1,0),\ (0,1,0)$ and $(0,0,0)$. \\

%%%%%%%%%%%%%%%%%%%%%%%%%%%%%%%%%%%%%%%%%%%%%%%%%%%%%%%
%                  end formula1                       %
%%%%%%%%%%%%%%%%%%%%%%%%%%%%%%%%%%%%%%%%%%%%%%%%%%%%%%%

\end{proof}

\subsection{Irreducible decompositions of 
$V_{\lambda }\otimes_\mC V_{2\me_{1}}$ and 
$V_{\lambda }\otimes_\mC V_{-2\me_{3}}$}
\label{subsec:clebsh(2,0,0)}
For a vector space $W$, we denote by $\id_W$ the identity map of $W$. 
We denote $\lambda \pm (\me_i+\me_j)\quad (1\leq i\leq j\leq 3)$ 
by $\lambda [\pm ij]$ for the sake of simplicity.
Generically the tensor product $V_{\lambda }\otimes_\mC V_{2\me_{1}}$ 
has six irreducible components: 
$V_{\lambda \epe{i}{j}},\ 1\leq i\leq j\leq 3$. 
Here some components may vanish. 

For $1\leq i\leq j\leq 3$, let $i^{\lambda }_{\me_{i}+\me_{j}}$ be 
a non-zero generator of 
$\Hom_K (V_{\lambda \epe{i}{j}},V_{\lambda }\otimes_\mC V_{2\me_{1}})$, 
which is unique up to scalar multiple 
if $V_{\lambda }\otimes_\mC V_{2\me_{1}}$ has 
a non-zero $\tau_{\lambda \epe{i}{j}}$-isotypic component. 

\begin{lem}\label{lem:clebsh}
\textit
We define a linear map $P_{\me_{1}}\colon  
V_{\me_{1}}\otimes_\mC V_{\me_{1}}\to V_{2\me_{1}}$ by
\begin{align*}
P_{\me_{1}}\bkt{f\gpt{1}{0}{0}{1}{0}{1} 
	\otimes f\gpt{1}{0}{0}{1}{0}{1}} 
	=&f\gpt{2}{0}{0}{2}{0}{2}, &
P_{\me_{1}}\bkt{f\gpt{1}{0}{0}{1}{0}{1} 
	\otimes f\gpt{1}{0}{0}{1}{0}{0}} 
	=&f\gpt{2}{0}{0}{2}{0}{1}, \\
P_{\me_{1}}\bkt{f\gpt{1}{0}{0}{1}{0}{1} 
	\otimes f\gpt{1}{0}{0}{0}{0}{0}} 
	=&f\gpt{2}{0}{0}{1}{0}{1}, &
P_{\me_{1}}\bkt{f\gpt{1}{0}{0}{1}{0}{0} 
	\otimes f\gpt{1}{0}{0}{1}{0}{1}} 
	=&f\gpt{2}{0}{0}{2}{0}{1}, \\
P_{\me_{1}}\bkt{f\gpt{1}{0}{0}{1}{0}{0} 
	\otimes f\gpt{1}{0}{0}{1}{0}{0}} 
	=&f\gpt{2}{0}{0}{2}{0}{0}, &
P_{\me_{1}}\bkt{f\gpt{1}{0}{0}{1}{0}{0} 
	\otimes f\gpt{1}{0}{0}{0}{0}{0}} 
	=&f\gpt{2}{0}{0}{1}{0}{0}, \\
P_{\me_{1}}\bkt{f\gpt{1}{0}{0}{0}{0}{0} 
	\otimes f\gpt{1}{0}{0}{1}{0}{1}} 
	=&f\gpt{2}{0}{0}{1}{0}{1}, &
P_{\me_{1}}\bkt{f\gpt{1}{0}{0}{0}{0}{0} 
	\otimes f\gpt{1}{0}{0}{1}{0}{0}} 
	=&f\gpt{2}{0}{0}{1}{0}{0}, \\
P_{\me_{1}}\bkt{f\gpt{1}{0}{0}{0}{0}{0} 
	\otimes f\gpt{1}{0}{0}{0}{0}{0}} 
	=&f\gpt{2}{0}{0}{0}{0}{0}. 
\end{align*}
Then $P_{\me_{1}}$ is a non-zero generator of 
$\Hom_K (V_{\me_{1}}\otimes_\mC V_{\me_{1}},V_{2\me_{1}})$. 
\end{lem}
\begin{proof}
Since $V_{\me_1}\otimes_\mC V_{\me_1}\simeq V_{2\me_1}\oplus V_{\me_1+\me_2}$, 
in order to prove this lemma, 
it suffices to check $P_{\me_1}\circ i^{\me_1}_{\me_2}=0$ and 
that $P_{\me_1}\circ i^{\me_1}_{\me_1}$ agree with the identity map 
up to scalar multiple. 
To consider the case $\lambda =\me_1$ in Proposition \ref{prop:clebsh(1,0,0)}, 
we obtain the explicit expressions of $i^{\me_1}_{\me_1}$ and 
$i^{\me_1}_{\me_2}$ in terms of the monomial basis as follows: 
\begin{align*}
i^{\me_1}_{\me_1}\bkt{f\gpt{2}{0}{0}{2}{0}{2}}=&
-6f\gpt{1}{0}{0}{1}{0}{1}\otimes f\gpt{1}{0}{0}{1}{0}{1},\\
i^{\me_1}_{\me_1}\bkt{f\gpt{2}{0}{0}{2}{0}{1}}=&
-3\left\{f\gpt{1}{0}{0}{1}{0}{1}\otimes f\gpt{1}{0}{0}{1}{0}{0}
+f\gpt{1}{0}{0}{1}{0}{0}\otimes f\gpt{1}{0}{0}{1}{0}{1}\right\}, \\
i^{\me_1}_{\me_1}\bkt{f\gpt{2}{0}{0}{2}{0}{0}}=&
-6f\gpt{1}{0}{0}{1}{0}{0}\otimes f\gpt{1}{0}{0}{1}{0}{0},\\
i^{\me_1}_{\me_1}\bkt{f\gpt{2}{0}{0}{1}{0}{1}}=&
-3\left\{f\gpt{1}{0}{0}{1}{0}{1}\otimes f\gpt{1}{0}{0}{0}{0}{0}
+f\gpt{1}{0}{0}{0}{0}{0}\otimes f\gpt{1}{0}{0}{1}{0}{1}\right\},\\
i^{\me_1}_{\me_1}\bkt{f\gpt{2}{0}{0}{1}{0}{0}}=&
-3\left\{f\gpt{1}{0}{0}{1}{0}{0}\otimes f\gpt{1}{0}{0}{0}{0}{0}
+f\gpt{1}{0}{0}{0}{0}{0}\otimes f\gpt{1}{0}{0}{1}{0}{0}\right\},\\
i^{\me_1}_{\me_1}\bkt{f\gpt{2}{0}{0}{0}{0}{0}}=&
-6f\gpt{1}{0}{0}{0}{0}{0}\otimes f\gpt{1}{0}{0}{0}{0}{0}, \\[3mm]
i^{\me_1}_{\me_2}\bkt{f\gpt{1}{1}{0}{1}{1}{1}}=&
f\gpt{1}{0}{0}{1}{0}{0}\otimes f\gpt{1}{0}{0}{1}{0}{1}
-f\gpt{1}{0}{0}{1}{0}{1}\otimes f\gpt{1}{0}{0}{1}{0}{0}, \\
i^{\me_1}_{\me_2}\bkt{f\gpt{1}{1}{0}{1}{0}{1}}=&
f\gpt{1}{0}{0}{0}{0}{0}\otimes f\gpt{1}{0}{0}{1}{0}{1}
-f\gpt{1}{0}{0}{1}{0}{1}\otimes f\gpt{1}{0}{0}{0}{0}{0}, \\
i^{\me_1}_{\me_2}\bkt{f\gpt{1}{1}{0}{1}{0}{0}}=&
f\gpt{1}{0}{0}{0}{0}{0}\otimes f\gpt{1}{0}{0}{1}{0}{0}
-f\gpt{1}{0}{0}{1}{0}{0}\otimes f\gpt{1}{0}{0}{0}{0}{0}.
\end{align*}
By direct computation, we can easily check 
$P_{\me_1}\circ i^{\me_1}_{\me_1}=-6\id_{V_{2\me_1}}$ 
and $P_{\me_1}\circ i^{\me_1}_{\me_2}=0$.
\end{proof}

By a composition of the projectors in Lemma \ref{lem:clebsh} 
and the injectors in Proposition \ref{prop:clebsh(1,0,0)}, 
we obtain following formulas.

\begin{prop}\label{prop:clebsh(2,0,0)}
\textit
For $1\leq i\leq j\leq 3$ and a G-pattern $M$ of type $\lambda \epe{i}{j}$, 
the image of the monomial basis $f(M)$ by 
the injector $i^{\lambda }_{\me_{i}+\me_{j}}\colon  
V_{\lambda \epe{i}{j}}\to V_{\lambda }\otimes_\mC V_{2\me_{1}}$
 is given by 
\[
i^{\lambda }_{\me_{i}+\me_{j}} (f(M)) 
=\sum_{0\leq k\leq l\leq 2}
\left\{ \sum_{m=0}^{\cpr{i}{j}{k}{l}}
\cp{i}{j}{M}{l}{k}{m} 
f\bkt{M\gpt{}{-\me_{i}-\me_{j}}{}{0}{-l}{-k} [-m]} \right\} 
\otimes f\gpt{2}{0}{0}{l}{0}{k} .
\]
In the right hand side of the above formula, we put $f(M')=0$ 
if $M'$ is a triangular array which does not 
satisfy the condition (\ref{cdn:G-pattern}) of G-patterns.

The explicit expressions of the coefficients are given by 
the following formulas.

\noindent {\bf Formula 1:} The coefficients of 
the injector $i^{\lambda }_{2\me_{1}}\colon  
V_{\lambda \epe{1}{1}}\to V_{\lambda }\otimes_\mC V_{2\me_{1}}$ 
are given as follows: \\
$(\cpr{1}{1}{2}{2},\cpr{1}{1}{1}{2},\cpr{1}{1}{1}{1},
\cpr{1}{1}{0}{2},\cpr{1}{1}{0}{1},\cpr{1}{1}{0}{0})
=(2,3,2,4,3,2)$ and
\begin{align*}
\cp{1}{1}{M}{2}{2}{0}=&(m_{13} -m_{12})(m_{13} -m_{12}-1)
	(m_{22}-m_{33})(m_{22}-m_{33}-1),\\
\cp{1}{1}{M}{2}{2}{1}=&-2(m_{13}-m_{12})(m_{22}-m_{33})(\bar{E}(M)-C_1(M)),\\
\cp{1}{1}{M}{2}{2}{2}=&\bar{E} (M)(C_1 (M)-1)(m_{13} -m_{33}-\bar{C_1} (M)),\\
\cp{1}{1}{M}{2}{1}{0}=&-2(m_{13} -m_{12})(m_{13} -m_{12}-1)
	(m_{22}-m_{33})(m_{22}-m_{33}-1),\\
\cp{1}{1}{M}{2}{1}{1}=&2(m_{13} -m_{12})(m_{22}-m_{33})\left\{ \bar{F}
	\bkt{M\gpt{-1}{0}{0}{0}{-1}{-1}}+\bar{E}(M)\right\},\\
\cp{1}{1}{M}{2}{1}{2}=&-2\left\{\bar{E} (M)\bar{F}
	\bkt{M\gpt{-1}{0}{0}{-1}{0}{-1}} \right.\\
	&\left.+(m_{13} -m_{12})(m_{22}-m_{33})
	C_1 (M)(\bar{C_1} (M)+1)\chi_+ (M)\right\} ,\\
\cp{1}{1}{M}{2}{1}{3}=&2(C_1(M)-1)\bar{C_1} (M)\bar{E} (M)\chi_+ (M),\\
\cp{1}{1}{M}{1}{1}{0}=&-2(m_{13} -m_{12})(m_{13} -m_{12}-1)
	(m_{22}-m_{33})(m_{13}-m_{22}+1),\\
\cp{1}{1}{M}{1}{1}{1}=&2(m_{13} -m_{12})\left\{(m_{22}-m_{33})C_1(M)
	(\bar{C_1}(M)+1)+(m_{13}-m_{22})\bar{E}(M) \right\} ,\\
\cp{1}{1}{M}{1}{1}{2}=&-2\bar{E}(M)(C_1(M)-1)\bar{C_1}(M),\\
\cp{1}{1}{M}{2}{0}{0}=&(m_{13} -m_{12})(m_{13} -m_{12}-1)
	(m_{22}-m_{33})(m_{22}-m_{33}-1),\\
\cp{1}{1}{M}{2}{0}{1}=&-(m_{13} -m_{12})(m_{22}-m_{33})
	\bkt{\bar{F} (M) + \bar{F}\bkt{M\gpt{-1}{0}{0}{0}{-1}{0}} },\\
\cp{1}{1}{M}{2}{0}{2}=&(m_{13} -m_{12}+1)(m_{22}-m_{33}+1)
	C_2 (M)\chi_+ (M) \\
	&+(m_{13} -m_{12})(m_{22}-m_{33})(C_1 (M)+1)(\bar{C_1} (M)+1)
	\chid{+}{1} (M)\\
	&+\bar{F} (M)\bar{F}\bkt{M\gpt{-1}{0}{0}{-1}{0}{0}} ,\\
\cp{1}{1}{M}{2}{0}{3}=&-C_2 (M) \left\{\chid{+}{1} (M)\bar{F} (M) 
	+\chi_+ (M) \bar{F}\bkt{M\gpt{-1}{0}{0}{-2}{1}{0}} \right\},\\
\cp{1}{1}{M}{2}{0}{4}=&C_2 (M) (C_1 (M)-1)(\bar{C_1} (M)-1)\chid{+}{1}(M),\\
\cp{1}{1}{M}{1}{0}{0}=&2(m_{13} -m_{12})(m_{13} -m_{12}-1)
	(m_{22}-m_{33})(m_{13}-m_{22}+1),\\
\cp{1}{1}{M}{1}{0}{1}=&-2(m_{13}-m_{12})\left\{ (m_{13}-m_{22})
	\bar{F}(M)\right. 
	\left. +(m_{22}-m_{33})C_2\bkt{M\gpt{-1}{0}{0}{0}{-1}{0}}\right\} ,\\
\cp{1}{1}{M}{1}{0}{2}=&2 \left\{ C_2\bkt{M\gpt{-1}{0}{0}{-1}{0}{0}}
	\bar{F}(M)\right. \\
	&\left. +\chi_+(M)(m_{13}-m_{12}+1)(m_{13}-m_{22}-1)
	C_2 (M) \right\} ,\\
\cp{1}{1}{M}{1}{0}{3}=&-2C_2 (M)(C_1(M)-1)(\bar{C_1} (M)-1) \chi_+ (M),\\
\cp{1}{1}{M}{0}{0}{0}=&(m_{13} -m_{12})(m_{13} -m_{12}-1)
	(m_{13}-m_{22}+1)(m_{13}-m_{22}),\\
\cp{1}{1}{M}{0}{0}{1}=&-2(m_{13} -m_{12})(m_{13}-m_{22})C_2(M),\\
\cp{1}{1}{M}{0}{0}{2}=&C_2 (M)(C_1(M)-1)(\bar{C_1} (M)-1) .
\end{align*}

\noindent {\bf Formula 2:} The coefficients of 
the injector $i^{\lambda }_{2\me_{2}}\colon  
V_{\lambda \epe{2}{2}}\to V_{\lambda }\otimes_\mC V_{2\me_{1}}$ 
are given as follows:\\
$(\cpr{2}{2}{2}{2},\cpr{2}{2}{1}{2},\cpr{2}{2}{1}{1},
\cpr{2}{2}{0}{2},\cpr{2}{2}{0}{1},\cpr{2}{2}{0}{0})
=(2,2,2,2,2,2)$ and
\begin{align*}
\cp{2}{2}{M}{2}{2}{0}=&(m_{22}-m_{33})(m_{22}-m_{33}-1),\\
\cp{2}{2}{M}{2}{2}{1}=&-(m_{22}-m_{33})\{ \bar{D} (M)\chi_-(M)
	+(\bar{D}(M)+2)\chid{-}{1}(M)\} ,\\
\cp{2}{2}{M}{2}{2}{2}=&\bar{D} (M)(\bar{D}(M)+1)\chid{-}{1}(M),\\
\cp{2}{2}{M}{2}{1}{0}=&-2(m_{22}-m_{33})(m_{22}-m_{33}-1),\\
\cp{2}{2}{M}{2}{1}{1}=&2(m_{22}-m_{33})\{ \bar{C_1} (M)
	+(\bar{D}(M)+1)\chi_-(M)\} ,\\
\cp{2}{2}{M}{2}{1}{2}=&-2\bar{C_1}(M)\bar{D} (M)\chi_-(M),\\
\cp{2}{2}{M}{1}{1}{0}=&-2(m_{22}-m_{33})(m_{23}-m_{22}),\\
\cp{2}{2}{M}{1}{1}{1}=&2\{ \bar{D}(M)(m_{23}-m_{22}-1)\chi_-(M)
	-(m_{22}-m_{33})(\bar{C_1}(M)+1)\chid{-}{1}(M) \},\\
\cp{2}{2}{M}{1}{1}{2}=&2\bar{C_1}(M)\bar{D}(M)\chid{-}{1}(M),\\
\cp{2}{2}{M}{2}{0}{0}=&(m_{22}-m_{33})(m_{22}-m_{33}-1),\\
\cp{2}{2}{M}{2}{0}{1}=&-2(m_{22}-m_{33})\bar{C_1}(M),\\
\cp{2}{2}{M}{2}{0}{2}=&\bar{C_1}(M)(\bar{C_1}(M)-1),\\
\cp{2}{2}{M}{1}{0}{0}=&2(m_{22}-m_{33})(m_{23}-m_{22}),\\
\cp{2}{2}{M}{1}{0}{1}=&2\bar{C_1}(M) \{ (m_{22}-m_{33})\chi_-(M)
	-(m_{23}-m_{22}-1)\} ,\\
\cp{2}{2}{M}{1}{0}{2}=&-2\bar{C_1}(M)(\bar{C_1}(M)-1)\chi_-(M),\\
\cp{2}{2}{M}{0}{0}{0}=&(m_{23}-m_{22})(m_{23}-m_{22}-1),\\
\cp{2}{2}{M}{0}{0}{1}=&\bar{C_1}(M)\{ (m_{23}-m_{22}-2)\chi_-(M)
	+(m_{23}-m_{22})\chid{-}{1}(M)\} ,\\
\cp{2}{2}{M}{0}{0}{2}=&\bar{C_1}(M)(\bar{C_1}(M)-1)\chid{-}{1}(M).
\end{align*}

\noindent {\bf Formula 3:} The coefficients of 
the injector $i^{\lambda }_{2\me_{3}}\colon  
V_{\lambda \epe{3}{3}}\to V_{\lambda }\otimes_\mC V_{2\me_{1}}$ 
are given as follows:\\
$(\cpr{3}{3}{2}{2},\cpr{3}{3}{1}{2},\cpr{3}{3}{1}{1},
\cpr{3}{3}{0}{2},\cpr{3}{3}{0}{1},\cpr{3}{3}{0}{0})
=(0,1,0,2,1,0)$ and
\begin{align*}
\cp{3}{3}{M}{2}{2}{0}=&1,&
\cp{3}{3}{M}{2}{1}{0}=&-2,\\
\cp{3}{3}{M}{2}{1}{1}=&-2\chi_+(M),&
\cp{3}{3}{M}{1}{1}{0}=&2,\\
\cp{3}{3}{M}{2}{0}{0}=&1,&
\cp{3}{3}{M}{2}{0}{1}=&\chid{+}{1}(M)+\chi_+(M),\\ 
\cp{3}{3}{M}{2}{0}{2}=&\chid{+}{1}(M),&
\cp{3}{3}{M}{1}{0}{0}=&-2,\\
\cp{3}{3}{M}{1}{0}{1}=&-2\chi_+(M),&
\cp{3}{3}{M}{0}{0}{0}=&1.
\end{align*}

\noindent {\bf Formula 4:} The coefficients of 
the injector $i^{\lambda }_{\me_{1}+\me_{2}}\colon  
V_{\lambda \epe{1}{2}}\to V_{\lambda }\otimes_\mC V_{2\me_{1}}$ 
are given as follows\\
$(\cpr{1}{2}{2}{2},\cpr{1}{2}{1}{2},\cpr{1}{2}{1}{1},
\cpr{1}{2}{0}{2},\cpr{1}{2}{0}{1},\cpr{1}{2}{0}{0})
=(2,2,2,3,2,2)$ and
\begin{align*}
\cp{1}{2}{M}{2}{2}{0}=&(m_{13} -m_{12})(m_{22}-m_{33})(m_{22}-m_{33}-1),\\
\cp{1}{2}{M}{2}{2}{1}=&-(m_{22}-m_{33})\left\{\bar{E} (M)
	+\chi_- (M) (m_{13} -m_{12})(\bar{D} (M)+1)\right\} ,\\
\cp{1}{2}{M}{2}{2}{2}=&\bar{D} (M) \bar{E} (M) \chi_- (M),\\
\cp{1}{2}{M}{2}{1}{0}=&-2(m_{13}-m_{12})(m_{22}-m_{33})(m_{22}-m_{33}-1),\\
\cp{1}{2}{M}{2}{1}{1}=&(m_{22}-m_{33})\left\{\bar{E}(M)+\bar{F}(M)\right. \\
	&\left. +(m_{13}-m_{12})\left\{ \bar{C_1}(M)+1+\bar{D}(M)(1-\chi_+(M)) 
	\right\} \right\},\\
\cp{1}{2}{M}{2}{1}{2}=&-\bar{C_1} (M)\bar{E} (M)-C_2(M)\bkt{ 1-\bar{D} (M)
	+\delta (M) \chi_+ (M)} ,\\
\cp{1}{2}{M}{1}{1}{0}=&(m_{13} -m_{12})(m_{22}-m_{33})
	(2m_{22}-m_{13}-m_{23}-2),\\
\cp{1}{2}{M}{1}{1}{1}=&\bar{E} (M)(m_{23} -m_{22})+
	C_2 (M)(m_{22}-m_{33}+1)\\
	&+(m_{13}-m_{12})\chi_-(M)\left\{ \bar{D}(M)(m_{13}-m_{22}+1)
	-(m_{22}-m_{33})(\bar{C_1} (M)+1) \right\},\\
\cp{1}{2}{M}{1}{1}{2}=&C_2(M)\chi_-(M)(m_{13}-m_{33}+2-\bar{C_1}(M)
	-\bar{D}(M)),\\
\cp{1}{2}{M}{2}{0}{0}=&(m_{13} -m_{12})(m_{22}-m_{33})(m_{22}-m_{33}-1),\\
\cp{1}{2}{M}{2}{0}{1}=&-(m_{22} -m_{33})\left\{ \bar{F} (M) 
	+ (m_{13} -m_{12})(\bar{C_1} (M)+\chi_+ (M))\right\} ,\\
\cp{1}{2}{M}{2}{0}{2}=&(\bar{C_1} (M)+\chi_+ (M)-1)\bar{F} (M)
	+(m_{22} -m_{33} +1)C_2(M)\chi_+ (M),\\
\cp{1}{2}{M}{2}{0}{3}=&-C_2 (M)(\bar{C_1} (M)-1)\chi_+ (M),\\
\cp{1}{2}{M}{1}{0}{0}=&-(m_{13} -m_{12})(m_{22}-m_{33})
	(2m_{22}-m_{13}-m_{23}-2),\\
\cp{1}{2}{M}{1}{0}{1}=&(m_{13} -m_{12})\bar{C_1} (M)
	\left\{ (m_{22}-m_{33})(1-\chi_+ (M))-(m_{13}-m_{22}+1) \right\}\\
	 &-(m_{23}-m_{22})\bar{F} (M)-(m_{22}-m_{33}+1)C_2(M),\\
\cp{1}{2}{M}{1}{0}{2}=&2C_2(M) (\bar{C_1} (M)-1),\\
\cp{1}{2}{M}{0}{0}{0}=&(m_{13} -m_{12})(m_{13}-m_{22}+1)(m_{23}-m_{22}),\\
\cp{1}{2}{M}{0}{0}{1}=&(m_{13}-m_{12})(m_{13}-m_{22}+1)\bar{C_1}(M)\chi_-(M)
	-(m_{23}-m_{22}-1)C_2(M),\\
\cp{1}{2}{M}{0}{0}{2}=&-C_2 (M)(\bar{C_1} (M)-1)\chi_- (M) .
\end{align*}

\noindent {\bf Formula 5:} The coefficients of 
the injector $i^{\lambda }_{\me_{1}+\me_{3}}\colon  
V_{\lambda \epe{1}{3}}\to V_{\lambda }\otimes_\mC V_{2\me_{1}}$ 
are given as follows:\\
$(\cpr{1}{3}{2}{2},\cpr{1}{3}{1}{2},\cpr{1}{3}{1}{1},
\cpr{1}{3}{0}{2},\cpr{1}{3}{0}{1},\cpr{1}{3}{0}{0})
=(1,2,1,3,2,1)$ and
\begin{align*}
\cp{1}{3}{M}{2}{2}{0}=&(m_{13} -m_{12})(m_{22}-m_{33}),\\
\cp{1}{3}{M}{2}{2}{1}=&-\bar{E} (M) ,\\
\cp{1}{3}{M}{2}{1}{0}=&-2(m_{13} -m_{12})(m_{22}-m_{33}),\\
\cp{1}{3}{M}{2}{1}{1}=&\bar{E} (M)+\bar{F} (M)
	-(m_{13}-m_{12})(m_{22}-m_{33})\chi_+(M) ,\\
\cp{1}{3}{M}{2}{1}{2}=&(\bar{E} (M)-C_2(M))\chi_+ (M) ,\\
\cp{1}{3}{M}{1}{1}{0}=&(m_{13} -m_{12})(2m_{22}-m_{13} -m_{33} -1),\\
\cp{1}{3}{M}{1}{1}{1}=&C_2(M) -\bar{E} (M),\\
\cp{1}{3}{M}{2}{0}{0}=&(m_{13} -m_{12})(m_{22}-m_{33}),\\
\cp{1}{3}{M}{2}{0}{1}=&(m_{13}-m_{12})(m_{22}-m_{33})\chid{+}{1}(M)
	-\bar{F}(M),\\
\cp{1}{3}{M}{2}{0}{2}=&C_2 (M)\chi_+ (M)-\bar{F} (M)\chid{+}{1}(M),\\
\cp{1}{3}{M}{2}{0}{3}=&C_2 (M)\chid{+}{1}(M),\\
\cp{1}{3}{M}{1}{0}{0}=&-(m_{13} -m_{12})(2m_{22}-m_{13} -m_{23} -1),\\
\cp{1}{3}{M}{1}{0}{1}=&\bar{F}(M)-C_2(M)
	+(m_{13}-m_{12})(m_{13}-m_{22}+1)\chi_+ (M),\\
\cp{1}{3}{M}{1}{0}{2}=&-2C_2(M)\chi_+(M) ,\\
\cp{1}{3}{M}{0}{0}{0}=&-(m_{13}-m_{12})(m_{13}-m_{22}+1),\\
\cp{1}{3}{M}{0}{0}{1}=&C_2(M).
\end{align*}

\noindent {\bf Formula 6:} The coefficients of 
the injector $i^{\lambda }_{\me_{2}+\me_{3}}\colon  
V_{\lambda \epe{2}{3}}\to V_{\lambda }\otimes_\mC V_{2\me_{1}}$ 
are given as follows:\\
$(\cpr{2}{3}{2}{2},\cpr{2}{3}{1}{2},\cpr{2}{3}{1}{1},
\cpr{2}{3}{0}{2},\cpr{2}{3}{0}{1},\cpr{2}{3}{0}{0})
=(1,1,1,2,1,1)$ and
\begin{align*}
\cp{2}{3}{M}{2}{2}{0}=&m_{22}-m_{33},&
\cp{2}{3}{M}{2}{2}{1}=&-\bar{D}(M)\chi_-(M),\\
\cp{2}{3}{M}{2}{1}{0}=&-2(m_{22}-m_{33}),&
\cp{2}{3}{M}{2}{1}{1}=&\bar{C_1}(M) -(m_{22}-m_{33})+\delta (M)\chi_-(M),\\
\cp{2}{3}{M}{1}{1}{0}=&2m_{22}-m_{23}-m_{33},&
\cp{2}{3}{M}{1}{1}{1}=&-(\bar{C_1}(M)+\bar{D}(M))\chi_-(M),\\
\cp{2}{3}{M}{2}{0}{0}=&m_{22}-m_{33},&
\cp{2}{3}{M}{2}{0}{1}=&-\{\bar{C_1}(M) -(m_{22}-m_{33})\chi_+(M)\},\\
\cp{2}{3}{M}{2}{0}{2}=&-\bar{C_1}(M)\chi_+(M),&
\cp{2}{3}{M}{1}{0}{0}=&-(2m_{22}-m_{23}-m_{33}),\\
\cp{2}{3}{M}{1}{0}{1}=&2\bar{C_1}(M),&
\cp{2}{3}{M}{0}{0}{0}=&-(m_{23}-m_{22}),\\
\cp{2}{3}{M}{0}{0}{1}=&-\bar{C_1}(M)\chi_-(M).
\end{align*}
\end{prop}

\begin{proof}
For $1\leq i\leq j\leq 3$, let 
$i^{\lambda}_{\me_i+\me_j}\colon  V_{\lambda \epe{i}{j}}
\to V_{\lambda }\otimes_\mC V_{2\me_{1}}$ be
the composite of three $K$-homomorphisms
\begin{align*}
i^{\lambda +\me_j}_{\me_i}\colon & 
V_{\lambda \epe{i}{j}}
\to V_{\lambda +\me_j}\otimes_\mC V_{\me_{1}} \\
i^{\lambda }_{\me_j}\otimes_\mC \id_{V_{\me_{1}}}\colon &
V_{\lambda +\me_j}\otimes_\mC V_{\me_{1}} 
\to V_{\lambda }\otimes_\mC V_{\me_{1}}\otimes_\mC V_{\me_{1}} \\
\id_{V_{\lambda }}\otimes_\mC P_{\me_1}\colon &
V_{\lambda }\otimes_\mC V_{\me_{1}}\otimes_\mC V_{\me_{1}} 
\to V_{\lambda }\otimes_\mC V_{2\me_{1}} .
\end{align*}
Then $i^{\lambda}_{\me_i+\me_j}$ is an element of 
$\Hom_K(V_{\lambda \epe{i}{j}},V_{\lambda }\otimes_\mC V_{2\me_{1}})$. 
By direct computation, we confirm that $i^{\lambda}_{\me_i+\me_j}$ is non-zero 
and obtain the explicit expression of this $K$-homomorphism. 
\begin{align*}
&i^{\lambda}_{\me_i+\me_j}(f(M))=
(\id_{V_{\lambda }}\otimes_\mC P_{\me_1})
\circ (i^{\lambda }_{\me_j}\otimes_\mC \id_{V_{\me_{1}}})
\circ i^{\lambda +\me_j}_{\me_i}(f(M))\\
&=\sum_{0\leq k\leq l\leq 1}
\sum_{m=0}^{\cpra{i}{kl}}
\cpa{c}{\lambda }{i;kl;m}{M} 
(\id_{V_{\lambda }}\otimes_\mC P_{\me_1})
\circ (i^{\lambda }_{\me_j}\otimes_\mC \id_{V_{\me_{1}}})
\left( f\bkt{M\gpt{}{-\me_{i}}{}{0}{-l}{-k} [-m]}  
\otimes f\gpt{1}{0}{0}{l}{0}{k} \right)\\
&=\sum_{0\leq k\leq l\leq 1}
\sum_{m=0}^{\cpra{i}{kl}}
\cpa{c}{\lambda }{i;kl;m}{M} 
\Biggl\{\sum_{0\leq p\leq q\leq 1}
 \sum_{r=0}^{\cpra{j}{pq}}
\cpa{c}{\lambda }{j;pq;r}{M\gpt{}{-\me_{i}}{}{0}{-l}{-k} [-m]} \\
&\hphantom{==}
\times f\bkt{M\gpt{}{-\me_{i}-\me_{j}}{}{0}{-l-q}{-k-p} [-m-r]} \Biggl\} 
\otimes P_{\me_1}\left(f\gpt{1}{0}{0}{q}{0}{p}
\otimes f\gpt{1}{0}{0}{l}{0}{k} \right)\\
&=\sum_{0\leq s\leq t\leq 2}
\left\{ \sum_{u=0}^{R_{[ij;st]}}
\cp{i}{j}{M}{t}{s}{u} 
f\bkt{M\gpt{}{-\me_{i}-\me_{j}}{}{0}{-t}{-s} [-u]} \right\} 
\otimes f\gpt{2}{0}{0}{t}{0}{s}, 
\end{align*}
where
\[
\cp{i}{j}{M}{t}{s}{u}=
\underset{p+k=s,\ q+l=t}{\underset{0\leq k\leq l\leq 1,}
{\sum_{0\leq p\leq q\leq 1,}}}
\hspace{1mm} \underset{m+r=u}{\underset{0\leq r\leq \cpra{j}{pq}.}
{\sum_{0\leq m\leq  \cpra{i}{kl},}}}
\cpa{c}{\lambda }{i;kl;m}{M} 
\cpa{c}{\lambda }{j;pq;r}{M\gpt{}{-\me_{i}}{}{0}{-l}{-k} [-m]}, 
\]
and
\[
R_{[ij;st]}=\max{\{ \cpra{i}{kl}+\cpra{j}{pq}\mid 
0\leq k\leq l\leq 1,\ 0\leq p\leq q\leq 1,\ p+k=s,\ q+l=t\} }.
\]
Now we simplify each coefficients by direct computation. 

First, we compute the case of formula 1.\\
$\bullet $ the case $(s,t)=(2,2)$.

We have
\begin{align*}
\cp{1}{1}{M}{2}{2}{0}&=\cpa{c}{\lambda }{1;11;0}{M} 
\cpa{c}{\lambda }{1;11;0}{M\gpt{-1}{0}{0}{0}{-1}{-1}}\\
&=(m_{13}-m_{12})(m_{22}-m_{33})(m_{13}-m_{12}-1)(m_{22}-m_{33}-1),\\
\cp{1}{1}{M}{2}{2}{1}&=\cpa{c}{\lambda }{1;11;0}{M} 
\cpa{c}{\lambda }{1;11;1}{M\gpt{-1}{0}{0}{0}{-1}{-1}}
+\cpa{c}{\lambda }{1;11;1}{M} 
\cpa{c}{\lambda }{1;11;0}{M\gpt{-1}{0}{0}{0}{-1}{-1}[-1]}\\
&=-(m_{13}-m_{12})(m_{22}-m_{33})\bar{E}\bkt{M\gpt{-1}{0}{0}{0}{-1}{-1}}
-\bar{E}(M)(m_{13}-m_{12})(m_{22}-m_{33})\\
&=-2(m_{13}-m_{12})(m_{22}-m_{33})(\bar{E}(M)-C_1(M)),\\
\cp{1}{1}{M}{2}{2}{2}&=\cpa{c}{\lambda }{1;11;1}{M} 
\cpa{c}{\lambda }{1;11;1}{M\gpt{-1}{0}{0}{0}{-1}{-1}[-1]}\\
&=\bar{E}(M)\bar{E}\bkt{M\gpt{-1}{0}{0}{0}{-1}{-1}[-1]}\\
&=\bar{E}(M)(C_1(M)-1)(m_{13}-m_{33}-\bar{C_1}(M)).
\end{align*}
Here we use the relations:
\begin{gather}
\bar{E}\bkt{M\gpt{-1}{0}{0}{0}{-1}{-1}}=\bar{E}(M)-2C_1(M),
\label{eqn:G-fct201}\\
C_1\bkt{M\gpt{-1}{0}{0}{0}{-1}{-1}}=C_1(M),\quad 
\bar{C_1}\bkt{M\gpt{-1}{0}{0}{0}{-1}{-1}}=\bar{C_1}(M)+1,
\label{eqn:G-fct202}\\
C_1\bkt{M[-1]}=C_1(M)-1,\quad 
\bar{C_1}\bkt{M[-1]}=\bar{C_1}(M)-1.\label{eqn:G-fct203}
\end{gather}
\\
\noindent $\bullet $ the case $(s,t)=(1,2)$.

We have
\begin{align*}
\cp{1}{1}{M}{2}{1}{0}&=\cpa{c}{\lambda }{1;11;0}{M} 
\cpa{c}{\lambda }{1;01;0}{M\gpt{-1}{0}{0}{0}{-1}{-1}}
+\cpa{c}{\lambda }{1;01;0}{M} 
\cpa{c}{\lambda }{1;11;0}{M\gpt{-1}{0}{0}{0}{-1}{0}}\\
&=-2(m_{13}-m_{12})(m_{22}-m_{33})(m_{13}-m_{12}-1)(m_{22}-m_{33}-1),\\
\cp{1}{1}{M}{2}{1}{1}&=\cpa{c}{\lambda }{1;11;0}{M} 
\cpa{c}{\lambda }{1;01;1}{M\gpt{-1}{0}{0}{0}{-1}{-1}}
+\cpa{c}{\lambda }{1;01;1}{M} 
\cpa{c}{\lambda }{1;11;0}{M\gpt{-1}{0}{0}{0}{-1}{0}[-1]}\\
&\hphantom{=}+\cpa{c}{\lambda }{1;01;0}{M} 
\cpa{c}{\lambda }{1;11;1}{M\gpt{-1}{0}{0}{0}{-1}{0}}
+\cpa{c}{\lambda }{1;11;1}{M} 
\cpa{c}{\lambda }{1;01;0}{M\gpt{-1}{0}{0}{0}{-1}{-1}[-1]}\\
&=(m_{13}-m_{12})(m_{22}-m_{33})\bar{F}\bkt{M\gpt{-1}{0}{0}{0}{-1}{-1}}
+\bar{F}(M)(m_{13}-m_{12})(m_{22}-m_{33})\\
&\hphantom{=}+(m_{13}-m_{12})(m_{22}-m_{33})
\bar{E}\bkt{M\gpt{-1}{0}{0}{0}{-1}{0}}
+\bar{E}(M)(m_{13}-m_{12})(m_{22}-m_{33})\\
&=2(m_{13}-m_{12})(m_{22}-m_{33})
\bkt{\bar{F}\bkt{M\gpt{-1}{0}{0}{0}{-1}{-1}}+\bar{E}(M)},\\
\cp{1}{1}{M}{2}{1}{2}&=\cpa{c}{\lambda }{1;11;1}{M} 
\cpa{c}{\lambda }{1;01;1}{M\gpt{-1}{0}{0}{0}{-1}{-1}[-1]}
+\cpa{c}{\lambda }{1;01;1}{M} 
\cpa{c}{\lambda }{1;11;1}{M\gpt{-1}{0}{0}{0}{-1}{0}[-1]}\\
&\hphantom{=}+\cpa{c}{\lambda }{1;11;0}{M} 
\cpa{c}{\lambda }{1;01;2}{M\gpt{-1}{0}{0}{0}{-1}{-1}}
+\cpa{c}{\lambda }{1;01;2}{M} 
\cpa{c}{\lambda }{1;11;0}{M\gpt{-1}{0}{0}{0}{-1}{0}[-2]}\\
&=-\bar{E}(M)\bar{F}\bkt{M\gpt{-1}{0}{0}{0}{-1}{-1}[-1]}
-\bar{F}(M)\bar{E}\bkt{M\gpt{-1}{0}{0}{0}{-1}{0}[-1]}\\
&\hphantom{=}-(m_{13}-m_{12})(m_{22}-m_{33})
C_2\bkt{M\gpt{-1}{0}{0}{0}{-1}{-1}}
\chi_+\bkt{M\gpt{-1}{0}{0}{0}{-1}{-1}}\\
&\hphantom{=}-C_2(M)\chi_+(M)(m_{13}-m_{12}+1)(m_{22}-m_{33}+1)\\
&=-2\bar{E}(M)\bar{F}\bkt{M\gpt{-1}{0}{0}{0}{-1}{-1}[-1]}
-(m_{13}-m_{12})(m_{22}-m_{33})C_1(M)\chi_+(M)\\
&\hphantom{=}+C_2(M)\chi_+(M)(m_{13}-m_{33}+1-m_{12}+m_{22}) \\
&\hphantom{=}-(m_{13}-m_{12})(m_{22}-m_{33})
C_1(M)(\bar{C_1}(M)+1)\chi_+(M)\\
&\hphantom{=}-C_2(M)\chi_+(M)(m_{13}-m_{12}+1)(m_{22}-m_{33}+1)\\
&=-2\biggl\{\bar{E}(M)\bar{F}\bkt{M\gpt{-1}{0}{0}{0}{-1}{-1}[-1]}\\
&\hphantom{=}+(m_{13}-m_{12})(m_{22}-m_{33})C_1(M)(\bar{C_1}(M)+1)\chi_+(M)
\biggl\},\\
\cp{1}{1}{M}{2}{1}{3}&=\cpa{c}{\lambda }{1;11;1}{M} 
\cpa{c}{\lambda }{1;01;2}{M\gpt{-1}{0}{0}{0}{-1}{0}[-1]}
+\cpa{c}{\lambda }{1;01;2}{M} 
\cpa{c}{\lambda }{1;11;1}{M\gpt{-1}{0}{0}{0}{-1}{0}[-2]}\\
&=\bar{E}(M)C_2\bkt{M\gpt{-1}{0}{0}{0}{-1}{0}[-1]}
\chi_+\bkt{M\gpt{-1}{0}{0}{0}{-1}{0}[-1]}\\
&\hphantom{=}+C_2(M)\chi_+(M)\bar{E}\bkt{M\gpt{-1}{0}{0}{0}{-1}{0}[-2]}\\
&=2(C_1(M)-1)\bar{C_1}(M)\bar{E}(M)\chi_+(M).
\end{align*}
Here we use the relations:
\begin{align}
&\bar{F}\bkt{M\gpt{-1}{0}{0}{0}{-1}{-1}}-\bar{F}(M)
=\bar{E}\bkt{M\gpt{-1}{0}{0}{0}{-1}{0}}-\bar{E}(M),
\label{eqn:G-fct204}\\
&\bar{F}(M)\bar{E}\bkt{M\gpt{-1}{0}{0}{0}{-1}{0}[-1]}\label{eqn:G-fct205}\\
&=\bar{E}(M)\bar{F}\bkt{M\gpt{-1}{0}{0}{0}{-1}{-1}[-1]}
+(m_{13}-m_{12})(m_{22}-m_{33})C_1(M)\chi_+(M)\nonumber \\
&\hphantom{=}-C_2(M)\chi_+(M)(m_{13}-m_{33}+1-m_{12}+m_{22}) ,\nonumber \\
&C_2\bkt{M\gpt{-1}{0}{0}{0}{-1}{-1}}\chi_+\bkt{M\gpt{-1}{0}{0}{0}{-1}{-1}}
=C_1(M)(\bar{C_1}(M)+1)\chi_+(M),\label{eqn:G-fct206}\\
&C_2\bkt{M\gpt{-1}{0}{0}{0}{-1}{0}[-1]}
\chi_+\bkt{M\gpt{-1}{0}{0}{0}{-1}{0}[-1]}
=(C_1(M)-1)\bar{C_1}(M)\chi_+(M),\label{eqn:G-fct207}\\
&C_2(M)\chi_+(M)\bar{E}\bkt{M\gpt{-1}{0}{0}{0}{-1}{0}[-2]}
=(C_1(M)-1)\bar{C_1}(M)\bar{E}(M)\chi_+(M).\label{eqn:G-fct208}
\end{align}
\\
\noindent $\bullet $ the case $(s,t)=(1,1)$.

We have
\begin{align*}
\cp{1}{1}{M}{1}{1}{0}&=\cpa{c}{\lambda }{1;11;0}{M} 
\cpa{c}{\lambda }{1;00;0}{M\gpt{-1}{0}{0}{0}{-1}{-1}}
+\cpa{c}{\lambda }{1;00;0}{M} 
\cpa{c}{\lambda }{1;11;0}{M\gpt{-1}{0}{0}{0}{0}{0}}\\
&=-(m_{13}-m_{12})(m_{22}-m_{33})(m_{13}-m_{12}-1)(m_{13}-m_{22}+1)\\
&\hphantom{=}-(m_{13}-m_{12})(m_{13}-m_{22}+1)
(m_{13}-m_{12}-1)(m_{22}-m_{33})\\
&=-2(m_{13}-m_{12})(m_{13}-m_{12}-1)(m_{22}-m_{33})(m_{13}-m_{22}+1),\\
\cp{1}{1}{M}{1}{1}{1}&=\cpa{c}{\lambda }{1;11;0}{M} 
\cpa{c}{\lambda }{1;00;1}{M\gpt{-1}{0}{0}{0}{-1}{-1}}
+\cpa{c}{\lambda }{1;00;1}{M} 
\cpa{c}{\lambda }{1;11;0}{M\gpt{-1}{0}{0}{0}{0}{0}[-1]}\\
&\hphantom{=}+\cpa{c}{\lambda }{1;00;0}{M} 
\cpa{c}{\lambda }{1;11;1}{M\gpt{-1}{0}{0}{0}{0}{0}}
+\cpa{c}{\lambda }{1;11;1}{M} 
\cpa{c}{\lambda }{1;00;0}{M\gpt{-1}{0}{0}{0}{-1}{-1}[-1]}\\
&=(m_{13}-m_{12})(m_{22}-m_{33})C_2\bkt{M\gpt{-1}{0}{0}{0}{-1}{-1}}\\
&\hphantom{=}+C_2(M)(m_{13}-m_{12})(m_{22}-m_{33}+1)\\
&\hphantom{=}+(m_{13}-m_{12})(m_{13}-m_{22}+1)
\bar{E}\bkt{M\gpt{-1}{0}{0}{0}{0}{0}}\\
&\hphantom{=}+\bar{E}(M)(m_{13}-m_{12})(m_{13}-m_{22})\\
&=(m_{13}-m_{12})\big\{ (m_{22}-m_{33})C_1(M)(\bar{C_1}(M)+1)
+C_2(M)(m_{22}-m_{33}+1)\\
&\hphantom{=}+(m_{13}-m_{22}+1)(\bar{E}(M)-C_1(M))
+\bar{E}(M)(m_{13}-m_{22})\big\}\\
&=(m_{13}-m_{12})\big\{ 2(m_{22}-m_{33})C_1(M)(\bar{C_1}(M)+1)\\
&\hphantom{=}+2(m_{13}-m_{22})\bar{E}(M)
+\bar{E}(M)-C_1(M)(m_{13}-m_{33}+1-\bar{C_1}(M))\big\}\\
&=2(m_{13}-m_{12})\big\{ (m_{22}-m_{33})C_1(M)(\bar{C_1}(M)+1)
+(m_{13}-m_{22})\bar{E}(M)\big\},\\
\cp{1}{1}{M}{1}{1}{2}&=\cpa{c}{\lambda }{1;11;1}{M} 
\cpa{c}{\lambda }{1;00;1}{M\gpt{-1}{0}{0}{0}{-1}{-1}[-1]}
+\cpa{c}{\lambda }{1;00;1}{M} 
\cpa{c}{\lambda }{1;11;1}{M\gpt{-1}{0}{0}{0}{0}{0}[-1]}\\
&=-\bar{E}(M)C_2\bkt{M\gpt{-1}{0}{0}{0}{-1}{-1}[-1]}
-C_2(M)\bar{E}\bkt{M\gpt{-1}{0}{0}{0}{0}{0}[-1]}\\
&=-2\bar{E}(M)(C_1(M)-1)\bar{C_1}(M).
\end{align*}
Here we use the relations:
\begin{align}
\bar{E}\bkt{M\gpt{-1}{0}{0}{0}{0}{0}}&=\bar{E}(M)-C_1(M),
\label{eqn:G-fct209}\\
C_2(M)\bar{E}\bkt{M\gpt{-1}{0}{0}{0}{0}{0}[-1]}&
=(C_1(M)-1)\bar{C_1}(M)\bar{E}(M),\label{eqn:G-fct210}
\end{align}
(\ref{eqn:G-fct202}) and (\ref{eqn:G-fct203}).\\

\noindent $\bullet $ the case $(s,t)=(0,2)$.

We have
\begin{align*}
\cp{1}{1}{M}{2}{0}{0}&=\cpa{c}{\lambda }{1;01;0}{M} 
\cpa{c}{\lambda }{1;01;0}{M\gpt{-1}{0}{0}{0}{-1}{0}}\\
&=(m_{13}-m_{12})(m_{22}-m_{33})(m_{13}-m_{12}-1)(m_{22}-m_{33}-1),\\
\cp{1}{1}{M}{2}{0}{1}&=\cpa{c}{\lambda }{1;01;0}{M} 
\cpa{c}{\lambda }{1;01;1}{M\gpt{-1}{0}{0}{0}{-1}{0}}
+\cpa{c}{\lambda }{1;01;1}{M} 
\cpa{c}{\lambda }{1;01;0}{M\gpt{-1}{0}{0}{0}{-1}{0}[-1]}\\
&=-(m_{13}-m_{12})(m_{22}-m_{33})
\bkt{\bar{F}\bkt{M\gpt{-1}{0}{0}{0}{-1}{0}}+\bar{F}(M)},\\
\cp{1}{1}{M}{2}{0}{2}&=\cpa{c}{\lambda }{1;01;0}{M} 
\cpa{c}{\lambda }{1;01;2}{M\gpt{-1}{0}{0}{0}{-1}{0}}
+\cpa{c}{\lambda }{1;01;1}{M} 
\cpa{c}{\lambda }{1;01;1}{M\gpt{-1}{0}{0}{0}{-1}{0}[-1]}\\
&\hphantom{=}+\cpa{c}{\lambda }{1;01;2}{M} 
\cpa{c}{\lambda }{1;01;0}{M\gpt{-1}{0}{0}{0}{-1}{0}[-2]}\\
&=(m_{13}-m_{12})(m_{22}-m_{33})C_2\bkt{M\gpt{-1}{0}{0}{0}{-1}{0}}
\chi_+\bkt{M\gpt{-1}{0}{0}{0}{-1}{0}}\\
&\hphantom{=}+\bar{F}(M)\bar{F}\bkt{M\gpt{-1}{0}{0}{0}{-1}{0}[-1]}
+C_2(M)\chi_+(M)(m_{13}-m_{12}+1)(m_{22}-m_{33}+1)\\
&=(m_{13}-m_{12}+1)(m_{22}-m_{33}+1)C_2(M)\chi_+(M)\\
&\hphantom{=}+(m_{13}-m_{12})(m_{22}-m_{33})(C_1(M)+1)(\bar{C_1}(M)+1)
\chid{+}{1}(M)\\
&\hphantom{=}+\bar{F}(M)\bar{F}\bkt{M\gpt{-1}{0}{0}{0}{-1}{0}[-1]},\\
\cp{1}{1}{M}{2}{0}{3}&=\cpa{c}{\lambda }{1;01;1}{M} 
\cpa{c}{\lambda }{1;01;2}{M\gpt{-1}{0}{0}{0}{-1}{0}[-1]}
+\cpa{c}{\lambda }{1;01;2}{M} 
\cpa{c}{\lambda }{1;01;1}{M\gpt{-1}{0}{0}{0}{-1}{0}[-2]}\\
&=-\bar{F}(M)C_2\bkt{M\gpt{-1}{0}{0}{0}{-1}{0}[-1]}
\chi_+\bkt{M\gpt{-1}{0}{0}{0}{-1}{0}[-1]}\\
&\hphantom{=}-C_2(M)\chi_+(M)\bar{F}\bkt{M\gpt{-1}{0}{0}{0}{-1}{0}[-2]}\\
&=-C_2(M)\bkt{\chid{+}{1}(M)\bar{F}(M)+
\chi_+(M)\bar{F}\bkt{M\gpt{-1}{0}{0}{0}{-1}{0}[-2]}},\\
\cp{1}{1}{M}{2}{0}{4}&=\cpa{c}{\lambda }{1;01;2}{M} 
\cpa{c}{\lambda }{1;01;2}{M\gpt{-1}{0}{0}{0}{-1}{0}[-2]}\\
&=C_2(M)\chi_+(M)C_2\bkt{M\gpt{-1}{0}{0}{0}{-1}{0}[-2]}
\chi_+\bkt{M\gpt{-1}{0}{0}{0}{-1}{0}[-2]}\\
&=C_2(M)(C_1(M)-1)(\bar{C_1}(M)-1)\chid{+}{1}(M).
\end{align*}
Here we use the equations:
\begin{align}
C_2\bkt{M\gpt{-1}{0}{0}{0}{-1}{0}}
\chi_+\bkt{M\gpt{-1}{0}{0}{0}{-1}{0}}
&=(C_1(M)+1)(\bar{C_1}(M)+1)\chid{+}{1}(M),
\label{eqn:G-fct211}
\end{align}
(\ref{eqn:G-fct004}), (\ref{eqn:G-fct010}) and (\ref{eqn:G-fct203}).\\
  
\noindent $\bullet $ the case $(s,t)=(0,1)$.

We have
\begin{align*}
\cp{1}{1}{M}{1}{0}{0}&=\cpa{c}{\lambda }{1;01;0}{M} 
\cpa{c}{\lambda }{1;00;0}{M\gpt{-1}{0}{0}{0}{-1}{0}}
+\cpa{c}{\lambda }{1;00;0}{M} 
\cpa{c}{\lambda }{1;01;0}{M\gpt{-1}{0}{0}{0}{0}{0}}\\
&=(m_{13}-m_{12})(m_{22}-m_{33})(m_{13}-m_{12}-1)(m_{13}-m_{22}+1)\\
&\hphantom{=}+(m_{13}-m_{12})(m_{13}-m_{22}+1)
(m_{13}-m_{12}-1)(m_{22}-m_{33})\\
&=2(m_{13}-m_{12})(m_{13}-m_{12}-1)(m_{22}-m_{33})(m_{13}-m_{22}+1),\\
\cp{1}{1}{M}{1}{0}{1}&=\cpa{c}{\lambda }{1;01;0}{M} 
\cpa{c}{\lambda }{1;00;1}{M\gpt{-1}{0}{0}{0}{-1}{0}}
+\cpa{c}{\lambda }{1;00;1}{M} 
\cpa{c}{\lambda }{1;01;0}{M\gpt{-1}{0}{0}{0}{0}{0}[-1]}\\
&\hphantom{=}+\cpa{c}{\lambda }{1;00;0}{M} 
\cpa{c}{\lambda }{1;01;1}{M\gpt{-1}{0}{0}{0}{0}{0}}
+\cpa{c}{\lambda }{1;01;1}{M} 
\cpa{c}{\lambda }{1;00;0}{M\gpt{-1}{0}{0}{0}{-1}{0}[-1]}\\
&=-(m_{13}-m_{12})(m_{22}-m_{33})C_2\bkt{M\gpt{-1}{0}{0}{0}{-1}{0}}\\
&\hphantom{=}-C_2(M)(m_{13}-m_{12})(m_{22}-m_{33}+1)\\
&\hphantom{=}-(m_{13}-m_{12})(m_{13}-m_{22}+1)
\bar{F}\bkt{M\gpt{-1}{0}{0}{0}{0}{0}}\\
&\hphantom{=}-\bar{F}(M)(m_{13}-m_{12})(m_{13}-m_{22})\\
&=-(m_{13}-m_{12})\Big\{(m_{22}-m_{33})C_2\bkt{M\gpt{-1}{0}{0}{0}{-1}{0}}
+C_2(M)(m_{22}-m_{33}+1)\\
&\hphantom{=}+(m_{13}-m_{22}+1)
\{\bar{F}(M)+(m_{22}-m_{33}+\delta (M))\chi_+(M)\}\\
&\hphantom{=}+\bar{F}(M)(m_{13}-m_{22})\Big\}\\
&=-(m_{13}-m_{12})\Big\{(m_{22}-m_{33})
\Big(C_2\bkt{M\gpt{-1}{0}{0}{0}{-1}{0}}\\
&\hphantom{=}+C_2(M)+(m_{12}-m_{22}+1-\delta (M))\chi_+(M)\Big)
+2(m_{13}-m_{22})\bar{F}(M)+\bar{F}(M)\\
&\hphantom{=}+C_2(M)+\chi_+(M)\{(m_{13}-m_{12})(m_{22}-m_{33})
+(m_{13}-m_{33}+1)\delta (M)\}\Big\}\\
&=-2(m_{13}-m_{12})\Big\{(m_{22}-m_{33})C_2\bkt{M\gpt{-1}{0}{0}{0}{-1}{0}}
+(m_{13}-m_{22})\bar{F}(M)\Big\},\\
\cp{1}{1}{M}{1}{0}{2}&=\cpa{c}{\lambda }{1;01;1}{M} 
\cpa{c}{\lambda }{1;00;1}{M\gpt{-1}{0}{0}{0}{-1}{0}[-1]}
+\cpa{c}{\lambda }{1;00;1}{M} 
\cpa{c}{\lambda }{1;01;1}{M\gpt{-1}{0}{0}{0}{0}{0}[-1]}\\
&\hphantom{=}+\cpa{c}{\lambda }{1;01;2}{M} 
\cpa{c}{\lambda }{1;00;0}{M\gpt{-1}{0}{0}{0}{-1}{0}[-2]}
+\cpa{c}{\lambda }{1;00;0}{M} 
\cpa{c}{\lambda }{1;01;2}{M\gpt{-1}{0}{0}{0}{0}{0}}\\
&=\bar{F}(M)C_2\bkt{M\gpt{-1}{0}{0}{0}{-1}{0}[-1]}
+C_2(M)\bar{F}\bkt{M\gpt{-1}{0}{0}{0}{0}{0}[-1]}\\
&\hphantom{=}+C_2(M)\chi_+(M)(m_{13}-m_{12}+1)(m_{13}-m_{22}-1)\\
&\hphantom{=}+(m_{13}-m_{12})(m_{13}-m_{22}+1)
C_2\bkt{M\gpt{-1}{0}{0}{0}{0}{0}}\chi_+\bkt{M\gpt{-1}{0}{0}{0}{0}{0}}\\
&=\bar{F}(M)C_2\bkt{M\gpt{-1}{0}{0}{0}{-1}{0}[-1]}\\
&\hphantom{=}+C_2(M)\big\{\bar{F}(M)+m_{12}-m_{22}-1+\delta (M)-\chi_+(M)
(m_{13}-m_{12}+\delta (M))\big\}\\
&\hphantom{=}+C_2(M)\chi_+(M)\{(m_{13}-m_{12}+1)(m_{13}-m_{22}-1)\\
&\hphantom{=}+(m_{13}-m_{12})(m_{13}-m_{22}+1)\}\\
&=2\bar{F}(M)C_2\bkt{M\gpt{-1}{0}{0}{0}{-1}{0}[-1]}
-(m_{12}-m_{22}-1+\delta (M))C_2(M)(1-\chi_+(M))\\
&\hphantom{=}+C_2(M)\big\{m_{12}-m_{22}-1+\delta (M)-\chi_+(M)
(m_{13}-m_{12}+\delta (M))\big\}\\
&\hphantom{=}+C_2(M)\chi_+(M)\{(m_{13}-m_{12}+1)(m_{13}-m_{22}-1)\\
&\hphantom{=}+(m_{13}-m_{12})(m_{13}-m_{22}+1)\}\\
&=2\{\bar{F}(M)C_2\bkt{M\gpt{-1}{0}{0}{0}{-1}{0}[-1]}\\
&\hphantom{=}+C_2(M)\chi_+(M)(m_{13}-m_{12}+1)(m_{13}-m_{22}-1)\},\\
\cp{1}{1}{M}{1}{0}{3}&=\cpa{c}{\lambda }{1;00;1}{M} 
\cpa{c}{\lambda }{1;01;2}{M\gpt{-1}{0}{0}{0}{0}{0}[-1]}
+\cpa{c}{\lambda }{1;01;2}{M} 
\cpa{c}{\lambda }{1;00;1}{M\gpt{-1}{0}{0}{0}{-1}{0}[-2]}\\
&=-C_2(M)C_2\bkt{M\gpt{-1}{0}{0}{0}{0}{0}[-1]}
\chi_+\bkt{M\gpt{-1}{0}{0}{0}{0}{0}[-1]}\\
&\hphantom{=}-C_2(M)\chi_+(M)C_2\bkt{M\gpt{-1}{0}{0}{0}{-1}{0}[-2]}\\
&=-2C_2(M)(C_1(M)-1)(\bar{C_1}(M)-1)\chi_+(M).
\end{align*}
Here we use the relations:
\begin{align}
\bar{F}\bkt{M\gpt{-1}{0}{0}{0}{0}{0}}&=
\bar{F}(M)+(m_{22}-m_{33}+\delta (M))\chi_+(M),
\label{eqn:G-fct212}\\
C_2\bkt{M\gpt{-1}{0}{0}{0}{-1}{0}}&=
C_2(M)+(m_{12}-m_{22}+1-\delta (M))\chi_+(M),\label{eqn:G-fct213}\\
&=(C_1(M)+\chi_+(M))(\bar{C_1}(M)+\chi_+(M)),\nonumber \\
C_2\bkt{M\gpt{-1}{0}{0}{0}{0}{0}}&=C_2(M),
\label{eqn:G-fct214}\\
\bar{F}\bkt{M\gpt{-1}{0}{0}{0}{0}{0}}
&=\bar{F}(M)+m_{12}-m_{22}-1+\delta (M)\label{eqn:G-fct215}\\
&\hphantom{=}-\chi_+(M)(m_{13}-m_{12}+\delta (M)),\nonumber \\
C_2\bkt{M\gpt{-1}{0}{0}{0}{-1}{0}[-1]}&=
C_2(M)-(m_{12}-m_{22}-1+\delta (M))(1-\chi_+(M)),
\label{eqn:G-fct216}\\
\bar{F}(M)(1-\chi_+(M))&=-C_2(M)(1-\chi_+(M))\label{eqn:G-fct217},
\end{align}
(\ref{eqn:G-fct004}) and (\ref{eqn:G-fct203}). \\

\noindent $\bullet $ the case $(s,t)=(0,0)$.

We have
\begin{align*}
\cp{1}{1}{M}{0}{0}{0}&=\cpa{c}{\lambda }{1;00;0}{M} 
\cpa{c}{\lambda }{1;00;0}{M\gpt{-1}{0}{0}{0}{0}{0}}\\
&=(m_{13}-m_{12})(m_{13}-m_{22}+1)(m_{13}-m_{12}-1)(m_{13}-m_{22}),\\
\cp{1}{1}{M}{0}{0}{1}&=\cpa{c}{\lambda }{1;00;0}{M} 
\cpa{c}{\lambda }{1;00;1}{M\gpt{-1}{0}{0}{0}{0}{0}}
+\cpa{c}{\lambda }{1;00;1}{M} 
\cpa{c}{\lambda }{1;00;0}{M\gpt{-1}{0}{0}{0}{0}{0}[-1]}\\
&=-(m_{13}-m_{12})(m_{13}-m_{22}+1)
C_2\bkt{M\gpt{-1}{0}{0}{0}{0}{0}}\\
&\hphantom{=}-C_2(M)(m_{13}-m_{12})(m_{13}-m_{22}-1)\\
&=-2(m_{13}-m_{12})(m_{13}-m_{22})C_2(M),\\
\cp{1}{1}{M}{0}{0}{2}&=+\cpa{c}{\lambda }{1;00;1}{M} 
\cpa{c}{\lambda }{1;00;1}{M\gpt{-1}{0}{0}{0}{0}{0}[-1]}\\
&=C_2(M)C_2\bkt{M\gpt{-1}{0}{0}{0}{0}{0}[-1]}\\
&=C_2(M)(C_1(M)-1)(\bar{C_1}(M)-1).
\end{align*}
Here we use the relations 
(\ref{eqn:G-fct203}) 
and (\ref{eqn:G-fct214}).\\

Next, we compute the case of formula 2.\\

\noindent $\bullet $ the case $(s,t)=(2,2)$.

We have
\begin{align*}
\cp{2}{2}{M}{2}{2}{0}
=&\cpa{c}{\lambda }{2;11;0}{M} 
\cpa{c}{\lambda }{2;11;0}{M\gpt{0}{-1}{0}{0}{-1}{-1} }
=(m_{22} -m_{33})(m_{22} -m_{33}-1), \\
\cp{2}{2}{M}{2}{2}{1}
=&\cpa{c}{\lambda }{2;11;0}{M} 
\cpa{c}{\lambda }{2;11;1}{M\gpt{0}{-1}{0}{0}{-1}{-1} }
+\cpa{c}{\lambda }{2;11;1}{M} 
\cpa{c}{\lambda }{2;11;0}{M\gpt{0}{-1}{0}{0}{-1}{-1}[-1] }\\
=&-(m_{22} -m_{33})\bar{D}\bkt{M\gpt{0}{-1}{0}{0}{-1}{-1} }
\chi_-\bkt{M\gpt{0}{-1}{0}{0}{-1}{-1} }
-\bar{D} (M)\chi_- (M)(m_{22} -m_{33})\\
=&-(m_{22} -m_{33})\{\bar{D} (M)\chi_- (M)+(\bar{D}(M)+2)\chid{-}{1}(M)\},\\
\cp{2}{2}{M}{2}{2}{2}
=&\cpa{c}{\lambda }{2;11;1}{M} 
\cpa{c}{\lambda }{2;11;1}{M\gpt{0}{-1}{0}{0}{-1}{-1}[-1] } \\
=&\bar{D} (M)\chi_- (M) 
\bar{D}\bkt{M\gpt{0}{-1}{0}{0}{-1}{-1}[-1] }
\chi_-\bkt{M\gpt{0}{-1}{0}{0}{-1}{-1}[-1] } \\
=&\bar{D} (M)(\bar{D} (M)+1)\chid{-}{1}(M). 
\end{align*}
Here we use the relations 
\begin{align}
\bar{D}\bkt{M\gpt{0}{-1}{0}{0}{-1}{-1} }
&=\bar{D}(M)+2,
\end{align}
(\ref{eqn:G-fct005}), (\ref{eqn:G-fct011}) and 
(\ref{eqn:G-fct_pf(2,0,0)001}).\\

\noindent $\bullet $ the case $(s,t)=(1,2)$.

We have
\begin{align*}
\cp{2}{2}{M}{2}{1}{0}
=&\cpa{c}{\lambda }{2;11;0}{M} 
\cpa{c}{\lambda }{2;01;0}{M\gpt{0}{-1}{0}{0}{-1}{-1}} 
+\cpa{c}{\lambda }{2;01;0}{M} 
\cpa{c}{\lambda }{2;11;0}{M\gpt{0}{-1}{0}{0}{-1}{0}} \\
=&-(m_{22} -m_{33})(m_{22} -m_{33}-1) 
-(m_{22} -m_{33}))(m_{22} -m_{33}-1) \\
=&-2(m_{22} -m_{33})(m_{22} -m_{33}-1), \\
\cp{2}{2}{M}{2}{1}{1}
=&\cpa{c}{\lambda }{2;11;0}{M} 
\cpa{c}{\lambda }{2;01;1}{M\gpt{0}{-1}{0}{0}{-1}{-1} } 
+\cpa{c}{\lambda }{2;01;1}{M} 
\cpa{c}{\lambda }{2;11;0}{M\gpt{0}{-1}{0}{0}{-1}{0} [-1]} \\
&+\cpa{c}{\lambda }{2;01;0}{M} 
\cpa{c}{\lambda }{2;11;1}{M\gpt{0}{-1}{0}{0}{-1}{0} } 
+\cpa{c}{\lambda }{2;11;1}{M} 
\cpa{c}{\lambda }{2;01;0}{M\gpt{0}{-1}{0}{0}{-1}{-1} [-1]} \\
=&(m_{22} -m_{33})\bar{C_1}\bkt{M\gpt{0}{-1}{0}{0}{-1}{-1} } 
+\bar{C_1}(M)(m_{22} -m_{33})\\
&+(m_{22} -m_{33})\bar{D}\bkt{M\gpt{0}{-1}{0}{0}{-1}{0} }
\chi_-\bkt{M\gpt{0}{-1}{0}{0}{-1}{0} } 
+\bar{D} (M)\chi_- (M)(m_{22} -m_{33})\\
=&2(m_{22} -m_{33})\{\bar{C_1}(M)+(\bar{D} (M)+1)\chi_- (M)\},\\
\cp{2}{2}{M}{2}{1}{2}
=&\cpa{c}{\lambda }{2;11;1}{M} 
\cpa{c}{\lambda }{2;01;1}{M\gpt{0}{-1}{0}{0}{-1}{-1} [-1]} 
+\cpa{c}{\lambda }{2;01;1}{M} 
\cpa{c}{\lambda }{2;11;1}{M\gpt{0}{-1}{0}{0}{-1}{0} [-1]} \\
=&-\bar{D} (M)\chi_- (M)
\bar{C_1}\bkt{M\gpt{0}{-1}{0}{0}{-1}{-1} [-1]} \\
&-\bar{C_1}(M)\bar{D}\bkt{M\gpt{0}{-1}{0}{0}{-1}{0} [-1]}
\chi_-\bkt{M\gpt{0}{-1}{0}{0}{-1}{0} [-1]} \\
=&-2\bar{C_1}(M)\bar{D} (M)\chi_- (M).
\end{align*}
Here we use the relations 
\begin{align}
\bar{D}\bkt{M\gpt{0}{-1}{0}{0}{-1}{0} }
&=\bar{D}(M)+1,\\
\bar{C_1}\bkt{M\gpt{0}{-1}{0}{0}{-1}{-1} }
&=\bar{C_1}(M)+chi_-(M),\label{eqn:pf(0,2,0)001}\\
\bar{C_1}\bkt{M\gpt{0}{-1}{0}{0}{-1}{-1} [-1]}\chi_- (M)
&=\bar{C_1}(M)\chi_- (M),
\end{align}
(\ref{eqn:G-fct005}) and (\ref{eqn:G-fct_pf(2,0,0)001}).\\

\noindent $\bullet $ the case $(s,t)=(1,1)$.

We have
\begin{align*}
\cp{2}{2}{M}{1}{1}{0}
=&\cpa{c}{\lambda }{2;11;0}{M} 
\cpa{c}{\lambda }{2;00;0}{M\gpt{0}{-1}{0}{0}{-1}{-1}} 
+\cpa{c}{\lambda }{2;00;0}{M} 
\cpa{c}{\lambda }{2;11;0}{M\gpt{0}{-1}{0}{0}{0}{0}} \\
=&-(m_{22} -m_{33})(m_{23} -m_{22}) 
-(m_{23} -m_{22})(m_{22} -m_{33})\\
=&-2(m_{22} -m_{33})(m_{23} -m_{22}), \\
\cp{2}{2}{M}{1}{1}{1}
=&\cpa{c}{\lambda }{2;11;0}{M} 
\cpa{c}{\lambda }{2;00;1}{M\gpt{0}{-1}{0}{0}{-1}{-1}} 
+\cpa{c}{\lambda }{2;00;1}{M} 
\cpa{c}{\lambda }{2;11;0}{M\gpt{0}{-1}{0}{0}{0}{0} [-1]} \\
&+\cpa{c}{\lambda }{2;11;1}{M} 
\cpa{c}{\lambda }{2;00;0}{M\gpt{0}{-1}{0}{0}{-1}{-1} [-1]} 
+\cpa{c}{\lambda }{2;00;0}{M} 
\cpa{c}{\lambda }{2;11;1}{M\gpt{0}{-1}{0}{0}{0}{0}} \\
=&-(m_{22} -m_{33})\bar{C_1}\bkt{M\gpt{0}{-1}{0}{0}{-1}{-1}}
\chi_-\bkt{M\gpt{0}{-1}{0}{0}{-1}{-1}} \\
&-\bar{C_1} (M)\chi_- (M)(m_{22} -m_{33}+1)
+\bar{D} (M)\chi_- (M)(m_{23} -m_{22}-1)\\
&+(m_{23} -m_{22}) \bar{D}\bkt{M\gpt{0}{-1}{0}{0}{0}{0}}
\chi_-\bkt{M\gpt{0}{-1}{0}{0}{0}{0}} \\
=&-(m_{22} -m_{33})(m_{12}-m_{11}+1)\chid{-}{1}(M) 
-(m_{12}-m_{11})\chi_- (M)(m_{22} -m_{33}+1)\\
&+(-m_{22}+m_{33}+\delta (M))\chi_- (M)(m_{23} -m_{22}-1)\\
&+(m_{23} -m_{22}) (-m_{22}+m_{33}+\delta (M)+1)\chid{-}{1}(M) \\
=&-2(m_{22} -m_{33})(m_{12}-m_{11}+1)\chid{-}{1}(M)\\
&+2(-m_{22}+m_{33}+\delta (M))\chi_- (M)(m_{23} -m_{22}-1)\\
&+(m_{23} -m_{33})(\delta (M)+1)(\chid{-}{1}(M)-\chi_- (M))\\
=&2\{\bar{D}(M)(m_{23} -m_{22}-1)\chi_- (M)
-(m_{22} -m_{33})(\bar{C_1}(M)+1)\chid{-}{1}(M)\},\\
\cp{2}{2}{M}{1}{1}{2}
=&\cpa{c}{\lambda }{2;11;1}{M} 
\cpa{c}{\lambda }{2;00;1}{M\gpt{0}{-1}{0}{0}{-1}{-1} [-1]} 
+\cpa{c}{\lambda }{2;00;1}{M} 
\cpa{c}{\lambda }{2;11;1}{M\gpt{0}{-1}{0}{0}{0}{0} [-1]} \\
=&\bar{D} (M)\chi_- (M)\bar{C_1}\bkt{M\gpt{0}{-1}{0}{0}{-1}{-1} [-1]}
\chi_-\bkt{M\gpt{0}{-1}{0}{0}{-1}{-1} [-1]} \\
&+\bar{C_1} (M)\chi_- (M)\bar{D}\bkt{M\gpt{0}{-1}{0}{0}{0}{0} [-1]}
\chi_-\bkt{M\gpt{0}{-1}{0}{0}{0}{0} [-1]} \\
=&\bar{D} (M)(m_{12}-m_{11})\chid{-}{1}(M)
+\bar{C_1} (M)\bar{D}(M)\chid{-}{1}(M)\\
=&2\bar{C_1} (M)\bar{D}(M)\chid{-}{1}(M).
\end{align*}
Here we use the relations
\begin{align}
\bar{D}\bkt{M\gpt{0}{-1}{0}{0}{0}{0}}
&=\bar{D}(M)+1,
\end{align}
(\ref{eqn:G-fct003}), (\ref{eqn:G-fct005}), (\ref{eqn:G-fct011}), 
(\ref{eqn:pf_clebsh(1,0,0)_002}), (\ref{eqn:pf(0,2,0)001}) and 
(\ref{eqn:G-fct_pf(2,0,0)001}).\\

\noindent $\bullet $ the case $(s,t)=(0,2)$.

We have
\begin{align*}
\cp{2}{2}{M}{2}{0}{0}
=&\cpa{c}{\lambda }{2;01;0}{M} 
\cpa{c}{\lambda }{2;01;0}{M\gpt{0}{-1}{0}{0}{-1}{0}} \\
=&(m_{22} -m_{33})(m_{22} -m_{33}-1)\\
\cp{2}{2}{M}{2}{0}{1}
=&\cpa{c}{\lambda }{2;01;0}{M} 
\cpa{c}{\lambda }{2;01;1}{M\gpt{0}{-1}{0}{0}{-1}{0}} 
+\cpa{c}{\lambda }{2;01;1}{M} 
\cpa{c}{\lambda }{2;01;0}{M\gpt{0}{-1}{0}{0}{-1}{0} [-1]} \\
=&-(m_{22} -m_{33}) \bar{C_1}\bkt{M\gpt{0}{-1}{0}{0}{-1}{0}} 
-\bar{C_1}(M)(m_{22} -m_{33}) \\
=&-2(m_{22} -m_{33})\bar{C_1}(M),\\
\cp{2}{2}{M}{2}{0}{2}
=&\cpa{c}{\lambda }{2;01;1}{M} 
\cpa{c}{\lambda }{2;01;1}{M\gpt{0}{-1}{0}{0}{-1}{0} [-1]} \\
=&\bar{C_1}(M)\bar{C_1}\bkt{M\gpt{0}{-1}{0}{0}{-1}{0} [-1]} \\
=&\bar{C_1}(M)(\bar{C_1}(M)-1).
\end{align*}
Here we use the relations
\begin{align}
\bar{C_1}\bkt{M\gpt{0}{-1}{0}{0}{-1}{0}} =\bar{C_1}(M),
\label{eqn:pf(0,2,0)002}
\end{align}
and (\ref{eqn:G-fct203}).\\

\noindent $\bullet $ the case $(s,t)=(0,1)$.

We have
\begin{align*}
\cp{2}{2}{M}{1}{0}{0}
=&\cpa{c}{\lambda }{2;01;0}{M} 
\cpa{c}{\lambda }{2;00;0}{M\gpt{0}{-1}{0}{0}{-1}{0}} 
+\cpa{c}{\lambda }{2;00;0}{M} 
\cpa{c}{\lambda }{2;01;0}{M\gpt{0}{-1}{0}{0}{0}{0}} \\
=&(m_{22} -m_{33}) (m_{23} -m_{22})
+(m_{23} -m_{22}) (m_{22} -m_{33})\\
=&2(m_{22} -m_{33}) (m_{23} -m_{22}),\\
\cp{2}{2}{M}{1}{0}{1}
=&\cpa{c}{\lambda }{2;01;0}{M} 
\cpa{c}{\lambda }{2;00;1}{M\gpt{0}{-1}{0}{0}{-1}{0}} 
+\cpa{c}{\lambda }{2;00;1}{M} 
\cpa{c}{\lambda }{2;01;0}{M\gpt{0}{-1}{0}{0}{0}{0} [-1]} \\
&+\cpa{c}{\lambda }{2;01;1}{M} 
\cpa{c}{\lambda }{2;00;0}{M\gpt{0}{-1}{0}{0}{-1}{0} [-1]} 
+\cpa{c}{\lambda }{2;00;0}{M} 
\cpa{c}{\lambda }{2;01;1}{M\gpt{0}{-1}{0}{0}{0}{0}} \\
=&(m_{22} -m_{33})\bar{C_1}\bkt{M\gpt{0}{-1}{0}{0}{-1}{0}} 
\chi_-\bkt{M\gpt{0}{-1}{0}{0}{-1}{0}} 
+\bar{C_1} (M)\chi_- (M)(m_{22} -m_{33}+1) \\
&-\bar{C_1}(M)(m_{23} -m_{22}-1)
-(m_{23} -m_{22})\bar{C_1}\bkt{M\gpt{0}{-1}{0}{0}{0}{0}} \\
=&(m_{22} -m_{33})\bar{C_1}(M)\chi_-(M)
+\bar{C_1} (M)\chi_- (M)(m_{22} -m_{33}+1) \\
&-\bar{C_1}(M)(m_{23} -m_{22}-1)
-(m_{23} -m_{22})(\bar{C_1}(M)-\chid{+}{-1}(M)) \\
=&2(m_{22} -m_{33})\bar{C_1}(M)\chi_-(M)\\
&-2(m_{23} -m_{22})\bar{C_1}(M)+\bar{C_1}(M)(1+\chi_- (M)+\chid{+}{-1}(M))\\
=&2\bar{C_1}(M)\{(m_{22} -m_{33})\chi_-(M)-(m_{23} -m_{22}-1)\},\\
\cp{2}{2}{M}{1}{0}{2}
=&\cpa{c}{\lambda }{2;00;1}{M} 
\cpa{c}{\lambda }{2;01;1}{M\gpt{0}{-1}{0}{0}{0}{0} [-1]} 
+\cpa{c}{\lambda }{2;01;1}{M} 
\cpa{c}{\lambda }{2;00;1}{M\gpt{0}{-1}{0}{0}{-1}{0} [-1]} \\
=&-\bar{C_1} (M)\chi_- (M) 
\bar{C_1}\bkt{M\gpt{0}{-1}{0}{0}{0}{0} [-1]} \\
&-\bar{C_1}(M)\bar{C_1}\bkt{M\gpt{0}{-1}{0}{0}{-1}{0} [-1]} 
\chi_-\bkt{M\gpt{0}{-1}{0}{0}{-1}{0} [-1]} \\
=&-\bar{C_1} (M)\chi_- (M) (\bar{C_1} (M)-\chid{+}{-1}(M)-1) 
-\bar{C_1}(M)(\bar{C_1}(M)-1)\chi_-(M) \\
=&-2\bar{C_1}(M)(\bar{C_1}(M)-1)\chi_-(M).
\end{align*}
Here we use the equations
\begin{align}
\bar{C_1}\bkt{M\gpt{0}{-1}{0}{0}{0}{0}}&=
\bar{C_1}(M)-\chid{+}{-1}(M),\label{eqn:pf(0,2,0)003}
\end{align}
(\ref{eqn:G-fct005}), (\ref{eqn:G-fct009}), 
(\ref{eqn:G-fct008}), 
(\ref{eqn:pf_clebsh(1,0,0)_001}), 
(\ref{eqn:pf(0,2,0)002}) and 
(\ref{eqn:G-fct203}).\\

\noindent $\bullet $ the case $(s,t)=(0,0)$.

We have
\begin{align*}
\cp{2}{2}{M}{0}{0}{0}
=&\cpa{c}{\lambda }{2;00;0}{M} 
\cpa{c}{\lambda }{2;00;0}{M\gpt{0}{-1}{0}{0}{0}{0}} \\
=&(m_{23} -m_{22}))(m_{23} -m_{22}-1) \\
\cp{2}{2}{M}{0}{0}{1}
=&\cpa{c}{\lambda }{2;00;0}{M} 
\cpa{c}{\lambda }{2;00;1}{M\gpt{0}{-1}{0}{0}{0}{0}}
+\cpa{c}{\lambda }{2;00;1}{M} 
\cpa{c}{\lambda }{2;00;0}{M\gpt{0}{-1}{0}{0}{0}{0} [-1]} \\
=&(m_{23} -m_{22})\bar{C_1}\bkt{M\gpt{0}{-1}{0}{0}{0}{0}}
\chi_-\bkt{M\gpt{0}{-1}{0}{0}{0}{0}}
+\bar{C_1} (M)\chi_- (M) (m_{23} -m_{22}-2)\\
=&\bar{C_1} (M)\{(m_{23} -m_{22}-2)\chi_- (M)
+(m_{23} -m_{22})\chid{-}{1}(M)\},\\
\cp{2}{2}{M}{0}{0}{2}
=&\cpa{c}{\lambda }{2;00;1}{M} 
\cpa{c}{\lambda }{2;00;1}{M\gpt{0}{-1}{0}{0}{0}{0} [-1]} \\
=&\bar{C_1} (M)\chi_- (M)\bar{C_1}\bkt{M\gpt{0}{-1}{0}{0}{0}{0} [-1]}
\chi_-\bkt{M\gpt{0}{-1}{0}{0}{0}{0} [-1]} \\
=&\bar{C_1} (M)(\bar{C_1} (M)-1)\chid{-}{1}(M).
\end{align*}
Here we use the relations 
(\ref{eqn:G-fct005}), (\ref{eqn:G-fct009}), (\ref{eqn:G-fct011}), 
(\ref{eqn:G-fct203}) and (\ref{eqn:pf(0,2,0)003}).\\

Next, we compute the case of formula 3.\\
\noindent $\bullet $ the case $(s,t)=(2,2)$.

We have
\begin{align*}
\cp{3}{3}{M}{2}{2}{0}
=&\cpa{c}{\lambda }{3;11;0}{M} 
\cpa{c}{\lambda }{3;11;0}{M\gpt{0}{0}{-1}{0}{-1}{-1} }=1.\\
\end{align*}

\noindent $\bullet $ the case $(s,t)=(1,2)$.

We have
\begin{align*}
\cp{3}{3}{M}{2}{1}{0}
=&\cpa{c}{\lambda }{3;11;0}{M} 
\cpa{c}{\lambda }{3;01;0}{M\gpt{0}{0}{-1}{0}{-1}{-1}} 
+\cpa{c}{\lambda }{3;01;0}{M} 
\cpa{c}{\lambda }{3;11;0}{M\gpt{0}{0}{-1}{0}{-1}{0}} \\
=&-1-1=-2,\\
\cp{3}{3}{M}{2}{1}{1}
=&\cpa{c}{\lambda }{3;11;0}{M} 
\cpa{c}{\lambda }{3;01;1}{M\gpt{0}{0}{-1}{0}{-1}{-1} } 
+\cpa{c}{\lambda }{3;01;1}{M} 
\cpa{c}{\lambda }{3;11;0}{M\gpt{0}{0}{-1}{0}{-1}{0} [-1]} \\
=&-\chi_+\bkt{M\gpt{0}{0}{-1}{0}{-1}{-1} } -\chi_+ (M)=-2\chi_+(M).
\end{align*}
Here we use the relation (\ref{eqn:G-fct004}).\\

\noindent $\bullet $ the case $(s,t)=(1,1)$.

We have
\begin{align*}
\cp{3}{3}{M}{1}{1}{0}
=&\cpa{c}{\lambda }{3;11;0}{M} 
\cpa{c}{\lambda }{3;00;0}{M\gpt{0}{0}{-1}{0}{-1}{-1}} 
+\cpa{c}{\lambda }{3;00;0}{M} 
\cpa{c}{\lambda }{3;11;0}{M\gpt{0}{0}{-1}{0}{0}{0}} \\
=&1+1=2.\\
\end{align*}

\noindent $\bullet $ the case $(s,t)=(0,2)$.

We have
\begin{align*}
\cp{3}{3}{M}{2}{0}{0}
=&\cpa{c}{\lambda }{3;01;0}{M} 
\cpa{c}{\lambda }{3;01;0}{M\gpt{0}{0}{-1}{0}{-1}{0}}
=1.\\
\cp{3}{3}{M}{2}{0}{1}
=&\cpa{c}{\lambda }{3;01;0}{M} 
\cpa{c}{\lambda }{3;01;1}{M\gpt{0}{0}{-1}{0}{-1}{0}} 
+\cpa{c}{\lambda }{3;01;1}{M} 
\cpa{c}{\lambda }{3;01;0}{M\gpt{0}{0}{-1}{0}{-1}{0} [-1]} \\
=&\chi_+\bkt{M\gpt{0}{0}{-1}{0}{-1}{0}}+\chi_+ (M)
=\chid{+}{1}(M)+\chi_+ (M)\\
\cp{3}{3}{M}{2}{0}{2}
=&\cpa{c}{\lambda }{3;01;1}{M} 
\cpa{c}{\lambda }{3;01;1}{M\gpt{0}{0}{-1}{0}{-1}{0} [-1]}
=\chi_+ (M)\chi_+ \bkt{M\gpt{0}{0}{-1}{0}{-1}{0} [-1]}\\
=&\chid{+}{1}(M).
\end{align*}
Here we use the relation (\ref{eqn:G-fct004}).\\

\noindent $\bullet $ the case $(s,t)=(0,1)$.

We have
\begin{align*}
\cp{3}{3}{M}{1}{0}{0}
=&\cpa{c}{\lambda }{3;01;0}{M} 
\cpa{c}{\lambda }{3;00;0}{M\gpt{0}{0}{-1}{0}{-1}{0}} 
+\cpa{c}{\lambda }{3;00;0}{M} 
\cpa{c}{\lambda }{3;01;0}{M\gpt{0}{0}{-1}{0}{0}{0}} \\
=&-1-1=-2, \\
\cp{3}{3}{M}{1}{0}{1}
=&\cpa{c}{\lambda }{3;01;1}{M} 
\cpa{c}{\lambda }{3;00;0}{M\gpt{0}{0}{-1}{0}{-1}{0} [-1]} 
+\cpa{c}{\lambda }{3;00;0}{M} 
\cpa{c}{\lambda }{3;01;1}{M\gpt{0}{0}{-1}{0}{0}{0}} \\
=&-\chi_+ (M)-\chi_+\bkt{M\gpt{0}{0}{-1}{0}{0}{0}} 
=-2\chi_+ (M).
\end{align*}
Here we use the relation (\ref{eqn:G-fct004}).\\

\noindent $\bullet $ the case $(s,t)=(0,0)$.

We have
\begin{align*}
\cp{3}{3}{M}{0}{0}{0}
=&\cpa{c}{\lambda }{3;00;0}{M} 
\cpa{c}{\lambda }{3;00;0}{M\gpt{0}{0}{-1}{0}{0}{0}} =1.\\
\end{align*}

Next, we compute the case of formula 4.\\
$\bullet $ the case $(s,t)=(2,2)$.

We have
\begin{align*}
\cp{1}{2}{M}{2}{2}{0}=&
\cpa{c}{\lambda }{1;11;0}{M} 
\cpa{c}{\lambda }{2;11;0}{M\gpt{-1}{0}{0}{0}{-1}{-1} } \\
=&(m_{13} -m_{12})(m_{22}-m_{33})(m_{22} -m_{33}-1),\\
\cp{1}{2}{M}{2}{2}{1}=&
\cpa{c}{\lambda }{1;11;0}{M} 
\cpa{c}{\lambda }{2;11;1}{M\gpt{-1}{0}{0}{0}{-1}{-1} }
+\cpa{c}{\lambda }{1;11;1}{M} 
\cpa{c}{\lambda }{2;11;0}{M\gpt{-1}{0}{0}{0}{-1}{-1}[-1] }\\
=&-(m_{13} -m_{12})(m_{22}-m_{33})\bar{D}\bkt{M\gpt{-1}{0}{0}{0}{-1}{-1}}
	\chi_-\bkt{M\gpt{-1}{0}{0}{0}{-1}{-1}}\\
&-\bar{E} (M)(m_{22} -m_{33})\\
=&-(m_{22} -m_{33})\{\bar{E}(M)+(m_{13} -m_{12})(\bar{D}(M)+1)\chi_-(M)\},\\
\cp{1}{2}{M}{2}{2}{2}=&
\cpa{c}{\lambda }{1;11;1}{M} 
\cpa{c}{\lambda }{2;11;1}{M\gpt{-1}{0}{0}{0}{-1}{-1}[-1] } \\
=&\bar{E}(M)\bar{D}\bkt{M\gpt{-1}{0}{0}{0}{-1}{-1}[-1]}
	\chi_-\bkt{M\gpt{-1}{0}{0}{0}{-1}{-1}[-1]} \\
=&\bar{E}(M)\bar{D}(M)\chi_-(M).
\end{align*}
Here we use the relations:
\begin{align}
\bar{D}\bkt{M\gpt{-1}{0}{0}{0}{-1}{-1}}=&\bar{D}(M)+1,\\
\bar{D}\bkt{M[-1]}=&\bar{D}(M)-1,\label{eqn:G-fct_pf(2,0,0)001}
\end{align}
and (\ref{eqn:G-fct005}).\\

\noindent $\bullet $ the case $(s,t)=(1,2)$.

We have
\begin{align*}
\cp{1}{2}{M}{2}{1}{0}
=&\cpa{c}{\lambda }{1;11;0}{M} 
\cpa{c}{\lambda }{2;01;0}{M\gpt{-1}{0}{0}{0}{-1}{-1}} 
+\cpa{c}{\lambda }{1;01;0}{M} 
\cpa{c}{\lambda }{2;11;0}{M\gpt{-1}{0}{0}{0}{-1}{0}} \\
=&-(m_{13} -m_{12})(m_{22}-m_{33})(m_{22} -m_{33}-1)\\
&-(m_{13} -m_{12})(m_{22}-m_{33})(m_{22} -m_{33}-1)\\
=&-2(m_{13} -m_{12})(m_{22}-m_{33})(m_{22} -m_{33}-1),\\
\cp{1}{2}{M}{2}{1}{1}
=&\cpa{c}{\lambda }{1;11;0}{M} 
\cpa{c}{\lambda }{2;01;1}{M\gpt{-1}{0}{0}{0}{-1}{-1} } 
+\cpa{c}{\lambda }{1;01;1}{M} 
\cpa{c}{\lambda }{2;11;0}{M\gpt{-1}{0}{0}{0}{-1}{0} [-1]} \\
&+\cpa{c}{\lambda }{1;01;0}{M} 
\cpa{c}{\lambda }{2;11;1}{M\gpt{-1}{0}{0}{0}{-1}{0} } 
+\cpa{c}{\lambda }{1;11;1}{M} 
\cpa{c}{\lambda }{2;01;0}{M\gpt{-1}{0}{0}{0}{-1}{-1} [-1]} \\
=& (m_{13} -m_{12})(m_{22}-m_{33})\bar{C_1}\bkt{M\gpt{-1}{0}{0}{0}{-1}{-1} }\\
&+\bar{F}(M)(m_{22} -m_{33})\\
&+(m_{13} -m_{12})(m_{22}-m_{33})\bar{D}\bkt{M\gpt{-1}{0}{0}{0}{-1}{0} }
	\chi_-\bkt{M\gpt{-1}{0}{0}{0}{-1}{0} }\\
&+\bar{E} (M)(m_{22} -m_{33})\\
=&(m_{22}-m_{33})\big\{\bar{E} (M)+\bar{F}(M)\\
&+(m_{13} -m_{12})\{\bar{C_1}(M)+1+\bar{D}(M)(1-\chi_+(M))\} \big\},\\
\cp{1}{2}{M}{2}{1}{2}
=&\cpa{c}{\lambda }{1;11;1}{M} 
\cpa{c}{\lambda }{2;01;1}{M\gpt{-1}{0}{0}{0}{-1}{-1} [-1]} 
+\cpa{c}{\lambda }{1;01;1}{M} 
\cpa{c}{\lambda }{2;11;1}{M\gpt{-1}{0}{0}{0}{-1}{0} [-1]} \\
&+\cpa{c}{\lambda }{1;01;2}{M} 
\cpa{c}{\lambda }{2;11;0}{M\gpt{-1}{0}{0}{0}{-1}{0} [-2]} \\
=& -\bar{E} (M)\bar{C_1}\bkt{M\gpt{-1}{0}{0}{0}{-1}{-1} [-1]} \\
&-\bar{F} (M)\bar{D}\bkt{M\gpt{-1}{0}{0}{0}{-1}{0} [-1]}
	\chi_-\bkt{M\gpt{-1}{0}{0}{0}{-1}{0} [-1]} \\
&-C_2 (M)\chi_+ (M)(m_{22} -m_{33}+1)\\
=&-\bar{E}(M)\bar{C_1}(M)-C_2(M)(1-\bar{D}(M)+\delta (M)\chi_+(M)),\\
\cp{1}{2}{M}{2}{1}{3}
=&\cpa{c}{\lambda }{1;01;2}{M} 
\cpa{c}{\lambda }{2;11;1}{M\gpt{-1}{0}{0}{0}{-1}{0} [-2]} \\
=&C_2 (M)\chi_+ (M)\bar{D}\bkt{M\gpt{-1}{0}{0}{0}{-1}{0} [-2]}
	\chi_-\bkt{M\gpt{-1}{0}{0}{0}{-1}{0} [-2]} \\
=&0.
\end{align*}
Here we use the relations:
\begin{align}
\bar{D}\bkt{M\gpt{-1}{0}{0}{0}{-1}{0} }=&\bar{D}(M),\\
\bar{C_1}\bkt{M\gpt{-1}{0}{0}{0}{-1}{-1} }=&\bar{C_1}(M)+1,
\label{eqn:G-fct_pf(2,0,0)002}
\end{align}
(\ref{eqn:G-fct005}), (\ref{eqn:G-fct008}), 
(\ref{eqn:G-fct009}), (\ref{eqn:G-fct_pf(2,0,0)001}), 
(\ref{eqn:pf_clebsh(1,0,0)_FC2}) and (\ref{eqn:G-fct203}).\\

\noindent $\bullet $ the case $(s,t)=(1,1)$.

We have
\begin{align*}
\cp{1}{2}{M}{1}{1}{0}
=&\cpa{c}{\lambda }{1;11;0}{M} 
\cpa{c}{\lambda }{2;00;0}{M\gpt{-1}{0}{0}{0}{-1}{-1}} 
+\cpa{c}{\lambda }{1;00;0}{M} 
\cpa{c}{\lambda }{2;11;0}{M\gpt{-1}{0}{0}{0}{0}{0}} \\
=& -(m_{13} -m_{12})(m_{22}-m_{33})(m_{23} -m_{22}+1) \\
&-(m_{13} -m_{12})(m_{13}-m_{22}+1)(m_{22} -m_{33})\\
=&(m_{13} -m_{12})(m_{22}-m_{33})(2m_{22}-m_{13}-m_{23}-2),\\
\cp{1}{2}{M}{1}{1}{1}
=&\cpa{c}{\lambda }{1;11;0}{M} 
\cpa{c}{\lambda }{2;00;1}{M\gpt{-1}{0}{0}{0}{-1}{-1}} 
+\cpa{c}{\lambda }{1;00;1}{M} 
\cpa{c}{\lambda }{2;11;0}{M\gpt{-1}{0}{0}{0}{0}{0} [-1]} \\
&+\cpa{c}{\lambda }{1;11;1}{M} 
\cpa{c}{\lambda }{2;00;0}{M\gpt{-1}{0}{0}{0}{-1}{-1} [-1]} 
+\cpa{c}{\lambda }{1;00;0}{M} 
\cpa{c}{\lambda }{2;11;1}{M\gpt{-1}{0}{0}{0}{0}{0}} \\
=&-(m_{13} -m_{12})(m_{22}-m_{33})\bar{C_1}\bkt{M\gpt{-1}{0}{0}{0}{-1}{-1}}
	\chi_-\bkt{M\gpt{-1}{0}{0}{0}{-1}{-1}} \\
&+C_2 (M)(m_{22} -m_{33}+1)
	+\bar{E} (M)(m_{23} -m_{22})\\
&+(m_{13} -m_{12})(m_{13}-m_{22}+1)\bar{D}\bkt{M\gpt{-1}{0}{0}{0}{0}{0}}
	\chi_-\bkt{M\gpt{-1}{0}{0}{0}{0}{0}}\\
=&\bar{E} (M)(m_{23} -m_{22})+C_2 (M)(m_{22} -m_{33}+1)\\
&+(m_{13} -m_{12})\chi_-(M)
\{(m_{13}-m_{22}+1)\bar{D}(M)-(m_{22}-m_{33})(\bar{C_1}(M)+1)\},\\
\cp{1}{2}{M}{1}{1}{2}
=&\cpa{c}{\lambda }{1;11;1}{M} 
\cpa{c}{\lambda }{2;00;1}{M\gpt{-1}{0}{0}{0}{-1}{-1} [-1]} 
+\cpa{c}{\lambda }{1;00;1}{M} 
\cpa{c}{\lambda }{2;11;1}{M\gpt{-1}{0}{0}{0}{0}{0} [-1]} \\
=&\bar{E} (M)\bar{C_1}\bkt{M\gpt{-1}{0}{0}{0}{-1}{-1} [-1]}
	\chi_-\bkt{M\gpt{-1}{0}{0}{0}{-1}{-1} [-1]} \\
&-C_2 (M)\bar{D}\bkt{M\gpt{-1}{0}{0}{0}{0}{0} [-1]}
	\chi_-\bkt{M\gpt{-1}{0}{0}{0}{0}{0} [-1]} \\
=&C_2(M)(m_{13}-m_{33}+2-\bar{C_1}(M)-\bar{D}(M)).
\end{align*}
Here we use the relations:
\begin{align}
\bar{D}\bkt{M\gpt{-1}{0}{0}{0}{0}{0}}=\bar{D}(M),
\end{align}
(\ref{eqn:G-fct005}), (\ref{eqn:G-fct_pf(2,0,0)001}), 
(\ref{eqn:G-fct_pf(2,0,0)002}) and (\ref{eqn:G-fct203}).\\

\noindent $\bullet $ the case $(s,t)=(0,2)$.

We have
\begin{align*}
\cp{1}{2}{M}{2}{0}{0}
=&\cpa{c}{\lambda }{1;01;0}{M} 
\cpa{c}{\lambda }{2;01;0}{M\gpt{-1}{0}{0}{0}{-1}{0}} \\
=&(m_{13} -m_{12})(m_{22}-m_{33})(m_{22} -m_{33}-1),\\
\cp{1}{2}{M}{2}{0}{1}
=&\cpa{c}{\lambda }{1;01;0}{M} 
\cpa{c}{\lambda }{2;01;1}{M\gpt{-1}{0}{0}{0}{-1}{0}} 
+\cpa{c}{\lambda }{1;01;1}{M} 
\cpa{c}{\lambda }{2;01;0}{M\gpt{-1}{0}{0}{0}{-1}{0} [-1]} \\
=& -(m_{13} -m_{12})(m_{22}-m_{33})\bar{C_1} \bkt{M\gpt{-1}{0}{0}{0}{-1}{0}} 
	-\bar{F} (M)(m_{22} -m_{33})\\
=&-(m_{22}-m_{33})\{(m_{13} -m_{12})(\bar{C_1}(M)+\chi_+(M))+\bar{F} (M)\},\\
\cp{1}{2}{M}{2}{0}{2}
=&\cpa{c}{\lambda }{1;01;2}{M} 
\cpa{c}{\lambda }{2;01;0}{M\gpt{-1}{0}{0}{0}{-1}{0} [-2]} 
+\cpa{c}{\lambda }{1;01;1}{M} 
\cpa{c}{\lambda }{2;01;1}{M\gpt{-1}{0}{0}{0}{-1}{0} [-1]} \\
=&C_2 (M)\chi_+ (M)(m_{22} -m_{33}+1)
	+\bar{F} (M)\bar{C_1}\bkt{M\gpt{-1}{0}{0}{0}{-1}{0} [-1]} \\
=&(\bar{C_1}(M)+\chi_+(M)-1)\bar{F} (M)+(m_{22} -m_{33}+1)C_2 (M)\chi_+ (M),\\
\cp{1}{2}{M}{2}{0}{3}
=&\cpa{c}{\lambda }{1;01;2}{M} 
\cpa{c}{\lambda }{2;01;1}{M\gpt{-1}{0}{0}{0}{-1}{0} [-2]} \\
=& -C_2 (M)\chi_+ (M)\bar{C_1}\bkt{M\gpt{-1}{0}{0}{0}{-1}{0} [-2]} \\
=&-C_2 (M)(\bar{C_1}(M)-1)\chi_+ (M).
\end{align*}
Here we use the relations:
\begin{align}
\bar{C_1} \bkt{M\gpt{-1}{0}{0}{0}{-1}{0}}=\bar{C_1}(M)+\chi_+(M),
\label{eqn:G-fct_pf(2,0,0)004}
\end{align}
and (\ref{eqn:G-fct203}).\\

\noindent $\bullet $ the case $(s,t)=(0,1)$.

We have
\begin{align*}
\cp{1}{2}{M}{1}{0}{0}
=&\cpa{c}{\lambda }{1;01;0}{M} 
\cpa{c}{\lambda }{2;00;0}{M\gpt{-1}{0}{0}{0}{-1}{0}} 
+\cpa{c}{\lambda }{1;00;0}{M} 
\cpa{c}{\lambda }{2;01;0}{M\gpt{-1}{0}{0}{0}{0}{0}} \\
=& (m_{13} -m_{12})(m_{22}-m_{33})(m_{23} -m_{22}+1) \\
&+ (m_{13} -m_{12})(m_{13}-m_{22}+1)(m_{22} -m_{33})\\
=&-(m_{13} -m_{12})(m_{22}-m_{33})(2m_{22}-m_{13}-m_{23}-2),\\
\cp{1}{2}{M}{1}{0}{1}
=&\cpa{c}{\lambda }{1;01;0}{M} 
\cpa{c}{\lambda }{2;00;1}{M\gpt{-1}{0}{0}{0}{-1}{0}} 
+\cpa{c}{\lambda }{1;00;1}{M} 
\cpa{c}{\lambda }{2;01;0}{M\gpt{-1}{0}{0}{0}{0}{0} [-1]} \\
&+\cpa{c}{\lambda }{1;01;1}{M} 
\cpa{c}{\lambda }{2;00;0}{M\gpt{-1}{0}{0}{0}{-1}{0} [-1]} 
+\cpa{c}{\lambda }{1;00;0}{M} 
\cpa{c}{\lambda }{2;01;1}{M\gpt{-1}{0}{0}{0}{0}{0}} \\
=&(m_{13} -m_{12})(m_{22}-m_{33})\bar{C_1}\bkt{M\gpt{-1}{0}{0}{0}{-1}{0}}
	\chi_-\bkt{M\gpt{-1}{0}{0}{0}{-1}{0}} \\
& -C_2 (M)(m_{22} -m_{33}+1)-\bar{F} (M)(m_{23} -m_{22})\\
& -(m_{13} -m_{12})(m_{13}-m_{22}+1)\bar{C_1}\bkt{M\gpt{-1}{0}{0}{0}{0}{0}} \\
=&(m_{13} -m_{12})\bar{C_1}(M)\{(m_{22}-m_{33})(1-\chi_+(M))
	-(m_{13}-m_{22}+1)\}\\
&-C_2 (M)(m_{22} -m_{33}+1)-\bar{F} (M)(m_{23} -m_{22}),\\
\cp{1}{2}{M}{1}{0}{2}
=&\cpa{c}{\lambda }{1;00;1}{M} 
\cpa{c}{\lambda }{2;01;1}{M\gpt{-1}{0}{0}{0}{0}{0} [-1]} 
+\cpa{c}{\lambda }{1;01;1}{M} 
\cpa{c}{\lambda }{2;00;1}{M\gpt{-1}{0}{0}{0}{-1}{0} [-1]} \\
&+\cpa{c}{\lambda }{1;01;2}{M} 
\cpa{c}{\lambda }{2;00;0}{M\gpt{-1}{0}{0}{0}{-1}{0} [-2]}\\
=& C_2 (M)\bar{C_1}\bkt{M\gpt{-1}{0}{0}{0}{0}{0} [-1]} \\
&-\bar{F} (M)\bar{C_1}\bkt{M\gpt{-1}{0}{0}{0}{-1}{0} [-1]}
	\chi_-\bkt{M\gpt{-1}{0}{0}{0}{-1}{0} [-1]} \\
&+C_2 (M)\chi_+ (M)(m_{23} -m_{22}-1)\\
=& C_2 (M)(\bar{C_1}(M)-1)(1+\chid{-}{-1}(M)+\chi_+ (M)) \\
=& 2C_2 (M)(\bar{C_1}(M)-1),\\
\cp{1}{2}{M}{1}{0}{3}
=&\cpa{c}{\lambda }{1;01;2}{M} 
\cpa{c}{\lambda }{2;00;1}{M\gpt{-1}{0}{0}{0}{-1}{0} [-2]}\\
=&C_2 (M)\chi_+ (M)\bar{C_1}\bkt{M\gpt{-1}{0}{0}{0}{-1}{0} [-2]}
	\chi_-\bkt{M\gpt{-1}{0}{0}{0}{-1}{0} [-2]}\\
=&0.
\end{align*}
Here we use the relations: 
\begin{align}
\bar{C_1}\bkt{M\gpt{-1}{0}{0}{0}{0}{0}}=\bar{C_1}(M),
\label{eqn:G-fct_pf(2,0,0)005}
\end{align}
(\ref{eqn:G-fct005}), (\ref{eqn:G-fct008}), (\ref{eqn:G-fct009}), 
(\ref{eqn:pf_clebsh(1,0,0)_001}), (\ref{eqn:pf_clebsh(1,0,0)_FC2}), 
(\ref{eqn:G-fct203}) and (\ref{eqn:G-fct_pf(2,0,0)004}).\\

\noindent $\bullet $ the case $(s,t)=(0,0)$.

We have
\begin{align*}
\cp{1}{2}{M}{0}{0}{0}=&
\cpa{c}{\lambda }{1;00;0}{M} 
\cpa{c}{\lambda }{2;00;0}{M\gpt{-1}{0}{0}{0}{0}{0}} \\
=&(m_{13} -m_{12})(m_{13}-m_{22}+1)(m_{23} -m_{22}) ,\\
\cp{1}{2}{M}{0}{0}{1}=&
\cpa{c}{\lambda }{1;00;0}{M} 
\cpa{c}{\lambda }{2;00;1}{M\gpt{-1}{0}{0}{0}{0}{0}}
+\cpa{c}{\lambda }{1;00;1}{M} 
\cpa{c}{\lambda }{2;00;0}{M\gpt{-1}{0}{0}{0}{0}{0} [-1]} \\
=&(m_{13} -m_{12})(m_{13}-m_{22}+1)\bar{C_1}\bkt{M\gpt{-1}{0}{0}{0}{0}{0}}
	\chi_-\bkt{M\gpt{-1}{0}{0}{0}{0}{0}}\\
&-C_2 (M)(m_{23} -m_{22}-1)\\
=&(m_{13} -m_{12})(m_{13}-m_{22}+1)\bar{C_1}(M)\chi_-(M)
-C_2 (M)(m_{23} -m_{22}-1),\\
\cp{1}{2}{M}{0}{0}{2}=&
\cpa{c}{\lambda }{1;00;1}{M} 
\cpa{c}{\lambda }{2;00;1}{M\gpt{-1}{0}{0}{0}{0}{0} [-1]} \\
=&-C_2 (M)\bar{C_1}\bkt{M\gpt{-1}{0}{0}{0}{0}{0} [-1]}
	\chi_-\bkt{M\gpt{-1}{0}{0}{0}{0}{0} [-1]} \\
=&-C_2 (M)(\bar{C_1}(M)-1)\chi_-(M).
\end{align*}
Here we use the relations 
(\ref{eqn:G-fct005}), (\ref{eqn:G-fct203}) and 
(\ref{eqn:G-fct_pf(2,0,0)005}).\\

Next, we compute the case of formula 5.\\
\noindent $\bullet $ the case $(s,t)=(2,2)$.

We have
\begin{align*}
\cp{1}{3}{M}{2}{2}{0}
=&\cpa{c}{\lambda }{1;11;0}{M} 
\cpa{c}{\lambda }{3;11;0}{M\gpt{-1}{0}{0}{0}{-1}{-1} } 
=(m_{13} -m_{12})(m_{22}-m_{33}),\\
\cp{1}{3}{M}{2}{2}{1}
=&\cpa{c}{\lambda }{1;11;1}{M} 
\cpa{c}{\lambda }{3;11;0}{M\gpt{-1}{0}{0}{0}{-1}{-1}[-1] }
=-\bar{E} (M).\\
\end{align*}

\noindent $\bullet $ the case $(s,t)=(1,2)$.

We have
\begin{align*}
\cp{1}{3}{M}{2}{1}{0}
=&\cpa{c}{\lambda }{1;11;0}{M} 
\cpa{c}{\lambda }{3;01;0}{M\gpt{-1}{0}{0}{0}{-1}{-1}} 
+\cpa{c}{\lambda }{1;01;0}{M} 
\cpa{c}{\lambda }{3;11;0}{M\gpt{-1}{0}{0}{0}{-1}{0}} \\
=&-(m_{13} -m_{12})(m_{22}-m_{33})-(m_{13} -m_{12})(m_{22}-m_{33})\\
=&-2(m_{13} -m_{12})(m_{22}-m_{33}),\\
\cp{1}{3}{M}{2}{1}{1}
=&\cpa{c}{\lambda }{1;11;0}{M} 
\cpa{c}{\lambda }{3;01;1}{M\gpt{-1}{0}{0}{0}{-1}{-1} } 
+\cpa{c}{\lambda }{1;01;1}{M} 
\cpa{c}{\lambda }{3;11;0}{M\gpt{-1}{0}{0}{0}{-1}{0} [-1]} \\
&+\cpa{c}{\lambda }{1;11;1}{M} 
\cpa{c}{\lambda }{3;01;0}{M\gpt{-1}{0}{0}{0}{-1}{-1} [-1]} \\
=&-(m_{13} -m_{12})(m_{22}-m_{33})\chi_+ (M)+\bar{F}(M)+\bar{E}(M),\\
\cp{1}{3}{M}{2}{1}{2}
=&\cpa{c}{\lambda }{1;11;1}{M} 
\cpa{c}{\lambda }{3;01;1}{M\gpt{-1}{0}{0}{0}{-1}{-1} [-1]} 
+\cpa{c}{\lambda }{1;01;2}{M} 
\cpa{c}{\lambda }{3;11;0}{M\gpt{-1}{0}{0}{0}{-1}{0} [-2]} \\
=&\bar{E} (M)\chi_+ \bkt{M\gpt{-1}{0}{0}{0}{-1}{-1} [-1]}-C_2 (M)\chi_+ (M)\\
=&(\bar{E} (M)-C_2(M))\chi_+ (M).
\end{align*}
Here we use the relation (\ref{eqn:G-fct004}).\\

\noindent $\bullet $ the case $(s,t)=(1,1)$.

We have
\begin{align*}
\cp{1}{3}{M}{1}{1}{0}
=&\cpa{c}{\lambda }{1;11;0}{M} 
\cpa{c}{\lambda }{3;00;0}{M\gpt{-1}{0}{0}{0}{-1}{-1}} 
+\cpa{c}{\lambda }{1;00;0}{M} 
\cpa{c}{\lambda }{3;11;0}{M\gpt{-1}{0}{0}{0}{0}{0}} \\
=&(m_{13} -m_{12})(m_{22}-m_{33})-(m_{13} -m_{12})(m_{13}-m_{22}+1)\\
=&(m_{13} -m_{12})(2m_{22}-m_{13}-m_{33}-1),\\
\cp{1}{3}{M}{1}{1}{1}
=&\cpa{c}{\lambda }{1;00;1}{M} 
\cpa{c}{\lambda }{3;11;0}{M\gpt{-1}{0}{0}{0}{0}{0} [-1]}
+\cpa{c}{\lambda }{1;11;1}{M} 
\cpa{c}{\lambda }{3;00;0}{M\gpt{-1}{0}{0}{0}{-1}{-1} [-1]} \\
=&C_2 (M)-\bar{E} (M).\\
\end{align*}

\noindent $\bullet $ the case $(s,t)=(0,2)$.

We have
\begin{align*}
\cp{1}{3}{M}{2}{0}{0}
=&\cpa{c}{\lambda }{1;01;0}{M} 
\cpa{c}{\lambda }{3;01;0}{M\gpt{-1}{0}{0}{0}{-1}{0}}
=(m_{13} -m_{12})(m_{22}-m_{33}),\\
\cp{1}{3}{M}{2}{0}{1}
=&\cpa{c}{\lambda }{1;01;0}{M} 
\cpa{c}{\lambda }{3;01;1}{M\gpt{-1}{0}{0}{0}{-1}{0}} 
+\cpa{c}{\lambda }{1;01;1}{M} 
\cpa{c}{\lambda }{3;01;0}{M\gpt{-1}{0}{0}{0}{-1}{0} [-1]} \\
=&(m_{13} -m_{12})(m_{22}-m_{33}) \chi_+\bkt{M\gpt{-1}{0}{0}{0}{-1}{0}} 
-\bar{F}(M)\\
=&(m_{13} -m_{12})(m_{22}-m_{33}) \chid{+}{1}(M)-\bar{F}(M),\\
\cp{1}{3}{M}{2}{0}{2}
=&\cpa{c}{\lambda }{1;01;2}{M} 
\cpa{c}{\lambda }{3;01;0}{M\gpt{-1}{0}{0}{0}{-1}{0} [-2]} 
+\cpa{c}{\lambda }{1;01;1}{M} 
\cpa{c}{\lambda }{3;01;1}{M\gpt{-1}{0}{0}{0}{-1}{0} [-1]} \\
=&C_2 (M)\chi_+ (M)
-\bar{F}(M)\chi_+\bkt{M\gpt{-1}{0}{0}{0}{-1}{0} [-1]} \\
=&C_2 (M)\chi_+ (M)-\bar{F}(M)\chid{+}{1}(M), \\
\cp{1}{3}{M}{2}{0}{3}
=&\cpa{c}{\lambda }{1;01;2}{M} 
\cpa{c}{\lambda }{3;01;1}{M\gpt{-1}{0}{0}{0}{-1}{0} [-2]} \\
=&C_2 (M)\chi_+ (M)\chi_+\bkt{M\gpt{-1}{0}{0}{0}{-1}{0} [-2]} \\
=&C_2 (M)\chid{+}{1}(M). 
\end{align*}
Here we use the relations (\ref{eqn:G-fct004}) and (\ref{eqn:G-fct010}).\\

\noindent $\bullet $ the case $(s,t)=(0,1)$.

We have
\begin{align*}
\cp{1}{3}{M}{1}{0}{0}
=&\cpa{c}{\lambda }{1;01;0}{M} 
\cpa{c}{\lambda }{3;00;0}{M\gpt{-1}{0}{0}{0}{-1}{0}} 
+\cpa{c}{\lambda }{1;00;0}{M} 
\cpa{c}{\lambda }{3;01;0}{M\gpt{-1}{0}{0}{0}{0}{0}} \\
=&-(m_{13} -m_{12})(m_{22}-m_{33})+(m_{13} -m_{12})(m_{13}-m_{22}+1)\\
=&-(m_{13} -m_{12})(2m_{22}-m_{13}-m_{33}-1),\\
\cp{1}{3}{M}{1}{0}{1}
=&\cpa{c}{\lambda }{1;00;1}{M} 
\cpa{c}{\lambda }{3;01;0}{M\gpt{-1}{0}{0}{0}{0}{0} [-1]} \\
&+\cpa{c}{\lambda }{1;01;1}{M} 
\cpa{c}{\lambda }{3;00;0}{M\gpt{-1}{0}{0}{0}{-1}{0} [-1]} 
+\cpa{c}{\lambda }{1;00;0}{M} 
\cpa{c}{\lambda }{3;01;1}{M\gpt{-1}{0}{0}{0}{0}{0}} \\
=&-C_2 (M)+\bar{F}(M) 
+(m_{13} -m_{12})(m_{13}-m_{22}+1)\chi_+ \bkt{M\gpt{-1}{0}{0}{0}{0}{0}} \\
=&\bar{F}(M)-C_2 (M)+(m_{13} -m_{12})(m_{13}-m_{22}+1)\chi_+ (M),\\
\cp{1}{3}{M}{1}{0}{2}
=&\cpa{c}{\lambda }{1;00;1}{M} 
\cpa{c}{\lambda }{3;01;1}{M\gpt{-1}{0}{0}{0}{0}{0} [-1]} 
+\cpa{c}{\lambda }{1;01;2}{M} 
\cpa{c}{\lambda }{3;00;0}{M\gpt{-1}{0}{0}{0}{-1}{0} [-2]}\\
=&-C_2 (M)\chi_+ \bkt{M\gpt{-1}{0}{0}{0}{0}{0} [-1]} -C_2 (M)\chi_+ (M)\\
=&-2C_2 (M)\chi_+ (M).
\end{align*}
Here we use the relation (\ref{eqn:G-fct004}).\\

\noindent $\bullet $ the case $(s,t)=(0,0)$.

We have
\begin{align*}
\cp{1}{3}{M}{0}{0}{0}
=&\cpa{c}{\lambda }{1;00;0}{M} 
\cpa{c}{\lambda }{3;00;0}{M\gpt{-1}{0}{0}{0}{0}{0}} 
=-(m_{13} -m_{12})(m_{13}-m_{22}+1),\\
\cp{1}{3}{M}{0}{0}{1}
=&\cpa{c}{\lambda }{1;00;1}{M} 
\cpa{c}{\lambda }{3;00;0}{M\gpt{-1}{0}{0}{0}{0}{0} [-1]}=C_2 (M). \\
\end{align*}

At last, we compute the case of formula 6.\\
\noindent $\bullet $ the case $(s,t)=(2,2)$.

We have
\begin{align*}
\cp{2}{3}{M}{2}{2}{0}
=&\cpa{c}{\lambda }{2;11;0}{M} 
\cpa{c}{\lambda }{3;11;0}{M\gpt{0}{-1}{0}{0}{-1}{-1} }
=m_{22} -m_{33},\\
\cp{2}{3}{M}{2}{2}{1}
=&\cpa{c}{\lambda }{2;11;1}{M} 
\cpa{c}{\lambda }{3;11;0}{M\gpt{0}{-1}{0}{0}{-1}{-1}[-1] }
=-\bar{D} (M)\chi_- (M).\\
\end{align*}

\noindent $\bullet $ the case $(s,t)=(1,2)$.

We have
\begin{align*}
\cp{2}{3}{M}{2}{1}{0}
=&\cpa{c}{\lambda }{2;11;0}{M} 
\cpa{c}{\lambda }{3;01;0}{M\gpt{0}{-1}{0}{0}{-1}{-1}} 
+\cpa{c}{\lambda }{2;01;0}{M} 
\cpa{c}{\lambda }{3;11;0}{M\gpt{0}{-1}{0}{0}{-1}{0}} \\
=&-(m_{22} -m_{33})-(m_{22} -m_{33})=-2(m_{22} -m_{33}),\\
\cp{2}{3}{M}{2}{1}{1}
=&\cpa{c}{\lambda }{2;11;0}{M} 
\cpa{c}{\lambda }{3;01;1}{M\gpt{0}{-1}{0}{0}{-1}{-1} } 
+\cpa{c}{\lambda }{2;01;1}{M} 
\cpa{c}{\lambda }{3;11;0}{M\gpt{0}{-1}{0}{0}{-1}{0} [-1]} \\
&+\cpa{c}{\lambda }{2;11;1}{M} 
\cpa{c}{\lambda }{3;01;0}{M\gpt{0}{-1}{0}{0}{-1}{-1} [-1]} \\
=&-(m_{22} -m_{33})\chi_+\bkt{M\gpt{0}{-1}{0}{0}{-1}{-1} } 
+\bar{C_1}(M)+\bar{D} (M)\chi_- (M)\\
=&\bar{C_1}(M)-(m_{22} -m_{33})(\chi_-(M)+\chid{+}{-1}(M))
+\delta (M)\chi_- (M)\\
=&\bar{C_1}(M)-(m_{22} -m_{33})+\delta (M)\chi_- (M),\\
\cp{2}{3}{M}{2}{1}{2}
=&\cpa{c}{\lambda }{2;11;1}{M} 
\cpa{c}{\lambda }{3;01;1}{M\gpt{0}{-1}{0}{0}{-1}{-1} [-1]} \\
=&\bar{D} (M)\chi_- (M)
\chi_+\bkt{M\gpt{0}{-1}{0}{0}{-1}{-1} [-1]}=0 .
\end{align*}
Here we use the relations (\ref{eqn:G-fct004}), 
(\ref{eqn:G-fct008}) and (\ref{eqn:G-fct009}).\\

\noindent $\bullet $ the case $(s,t)=(1,1)$.

We have
\begin{align*}
\cp{2}{3}{M}{1}{1}{0}
=&\cpa{c}{\lambda }{2;11;0}{M} 
\cpa{c}{\lambda }{3;00;0}{M\gpt{0}{-1}{0}{0}{-1}{-1}} 
+\cpa{c}{\lambda }{2;00;0}{M} 
\cpa{c}{\lambda }{3;11;0}{M\gpt{0}{-1}{0}{0}{0}{0}} \\
=&(m_{22} -m_{33})-(m_{23} -m_{22})
=2m_{22}-m_{23}-m_{33},\\
\cp{2}{3}{M}{1}{1}{1}
=&\cpa{c}{\lambda }{2;00;1}{M} 
\cpa{c}{\lambda }{3;11;0}{M\gpt{0}{-1}{0}{0}{0}{0} [-1]} 
+\cpa{c}{\lambda }{2;11;1}{M} 
\cpa{c}{\lambda }{3;00;0}{M\gpt{0}{-1}{0}{0}{-1}{-1} [-1]} \\
=&-\bar{C_1} (M)\chi_- (M)-\bar{D} (M)\chi_- (M)
=-(\bar{C_1} (M)+\bar{D} (M))\chi_- (M).\\
\end{align*}

\noindent $\bullet $ the case $(s,t)=(0,2)$.

We have
\begin{align*}
\cp{2}{3}{M}{2}{0}{0}
=&\cpa{c}{\lambda }{2;01;0}{M} 
\cpa{c}{\lambda }{3;01;0}{M\gpt{0}{-1}{0}{0}{-1}{0}}
=m_{22} -m_{33},\\
\cp{2}{3}{M}{2}{0}{1}
=&\cpa{c}{\lambda }{2;01;0}{M} 
\cpa{c}{\lambda }{3;01;1}{M\gpt{0}{-1}{0}{0}{-1}{0}} 
+\cpa{c}{\lambda }{2;01;1}{M} 
\cpa{c}{\lambda }{3;01;0}{M\gpt{0}{-1}{0}{0}{-1}{0} [-1]} \\
=&(m_{22} -m_{33})\chi_+\bkt{M\gpt{0}{-1}{0}{0}{-1}{0}} 
-\bar{C_1}(M)
=-\{\bar{C_1}(M)-(m_{22} -m_{33})\chi_+(M)\},\\
\cp{2}{3}{M}{2}{0}{2}
=&\cpa{c}{\lambda }{2;01;1}{M} 
\cpa{c}{\lambda }{3;01;1}{M\gpt{0}{-1}{0}{0}{-1}{0} [-1]} 
=-\bar{C_1}(M)\chi_+\bkt{M\gpt{0}{-1}{0}{0}{-1}{0} [-1]} \\
=&-\bar{C_1}(M)\chi_+(M).
\end{align*}
Here we use the relation (\ref{eqn:G-fct004}).\\

\noindent $\bullet $ the case $(s,t)=(0,1)$.

We have
\begin{align*}
\cp{2}{3}{M}{1}{0}{0}
=&\cpa{c}{\lambda }{2;01;0}{M} 
\cpa{c}{\lambda }{3;00;0}{M\gpt{0}{-1}{0}{0}{-1}{0}} 
+\cpa{c}{\lambda }{2;00;0}{M} 
\cpa{c}{\lambda }{3;01;0}{M\gpt{0}{-1}{0}{0}{0}{0}} \\
=&-(m_{22} -m_{33})+(m_{23} -m_{22})
=-(2m_{22} -m_{23}-m_{33}),\\
\cp{2}{3}{M}{1}{0}{1}
=&\cpa{c}{\lambda }{2;00;1}{M} 
\cpa{c}{\lambda }{3;01;0}{M\gpt{0}{-1}{0}{0}{0}{0} [-1]} \\
&+\cpa{c}{\lambda }{2;01;1}{M} 
\cpa{c}{\lambda }{3;00;0}{M\gpt{0}{-1}{0}{0}{-1}{0} [-1]} 
+\cpa{c}{\lambda }{2;00;0}{M} 
\cpa{c}{\lambda }{3;01;1}{M\gpt{0}{-1}{0}{0}{0}{0}} \\
=&\bar{C_1} (M)\chi_- (M)+\bar{C_1}(M)
+(m_{23} -m_{22})\chi_+ \bkt{M\gpt{0}{-1}{0}{0}{0}{0}} \\
=&\bar{C_1} (M)(1+\chi_- (M)+\chid{+}{-1}(M))=2\bar{C_1} (M),\\
\cp{2}{3}{M}{1}{0}{2}
=&\cpa{c}{\lambda }{2;00;1}{M} 
\cpa{c}{\lambda }{3;01;1}{M\gpt{0}{-1}{0}{0}{0}{0} [-1]} \\
=&\bar{C_1} (M)\chi_- (M)\chi_+\bkt{M\gpt{0}{-1}{0}{0}{0}{0} [-1]} 
=0.
\end{align*}
(\ref{eqn:G-fct004}), (\ref{eqn:G-fct008}), (\ref{eqn:G-fct009}) and 
(\ref{eqn:pf_clebsh(1,0,0)_001}).

\noindent $\bullet $ the case $(s,t)=(0,0)$.

We have
\begin{align*}
\cp{2}{3}{M}{0}{0}{0}
=&\cpa{c}{\lambda }{2;00;0}{M} 
\cpa{c}{\lambda }{3;00;0}{M\gpt{0}{-1}{0}{0}{0}{0}}
=-(m_{23} -m_{22}),\\
\cp{2}{3}{M}{0}{0}{1}
=&\cpa{c}{\lambda }{2;00;1}{M} 
\cpa{c}{\lambda }{3;00;0}{M\gpt{0}{-1}{0}{0}{0}{0} [-1]}
=-\bar{C_1} (M)\chi_- (M).\\
\end{align*}

\end{proof}

For $1\leq i\leq j\leq 3$, we define a linear map 
$i^{\lambda }_{-\me_{i}-\me_{j}}\colon 
V_{\lambda \eme{i}{j}}\to V_{\lambda }\otimes_\mC V_{-2\me_{3}} $ 
by 
\[
i^{\lambda }_{-\me_{i}-\me_{j}}
=(T_{\hat{\lambda }}\otimes_\mC T_{2\me_1})
\circ i^{\widehat{\lambda }}_{\me_{4-j}+\me_{4-i}}
\circ T_{\lambda -\me_{i}-\me_{j}}.
\]
By Proposition \ref{prop:sym_GZ-basis}, 
we see that 
\[
X\circ i^{\lambda }_{-\me_{i}-\me_{j}}
=i^{\lambda }_{-\me_{i}-\me_{j}}\circ \omega^2(X),\quad X\in \g .
\]
Since $\omega^2=\id_{\g}$, $i^{\lambda }_{-\me_{i}-\me_{j}} $
is a non-zero generator of 
$\Hom_K (V_{\lambda \eme{i}{j}},V_{\lambda }\otimes_\mC V_{-2\me_{3}})$, 
which is unique up to scalar multiple. 
So, we obtain the following proposition 
from Proposition \ref{prop:clebsh(2,0,0)} easily.

\begin{prop}\label{prop:clebsh(0,0,-2)}
\textit
For $1\leq i\leq j\leq 3$ and 
G-pattern $M$ of type $\lambda \eme{i}{j}$, 
the image of the monomial basis 
$f(M)$ by the injector $i^{\lambda }_{-\me_{i}-\me_{j}}\colon  
V_{\lambda \eme{i}{j}}\to V_{\lambda }\otimes_\mC V_{-2\me_{3}}$
 is given by 
\[
i^{\lambda }_{-\me_{i}-\me_{j}} (f(M)) 
=\sum_{0\leq k\leq l\leq 2}
\left\{ \sum_{m=0}^{\cpr{4-j\hs }{4-i}{k}{l}}
\cpd{4-j\hs }{4-i}{(\hat{M})}{l}{k}{m} 
f\bkt{M\gpt{}{\me_{i}+\me_{j}}{}{l}{0}{k} [-m]} \right\} 
\otimes f\gpt{0}{0}{-2}{0}{-l}{-k} 
\]
In the right hand side of the above formula, we put $f(M')=0$ 
if $M'$ is a triangular array which does not 
satisfy the condition (\ref{cdn:G-pattern}) of G-patterns.
\end{prop}

\section{$(\g ,K)$-module structure of 
principal series representation of $Sp(3,\mR )$}
\label{sec:structure}

\subsection{Irreducible decomposition of 
$(\pi_{(\sigma ,\nu )}|_K,H_{(\sigma ,\nu)})$}
\label{subsec:peter-weyl}

For an irreducible finite dimensional representation $(\tau ,V)$ of $K$,
we denote by $(\tau^*,V^*)$ its contragradient representation. 
For a basis $\{ v_i\}_{i\in I}$ of $V$, 
we denote by $\{ v_i^*\}_{i\in I}$ 
its dual basis of $V^*$. Let $L^+$ be the set of dominant weights 
of length three.

We set 
\[
L^2_{(M_{\min },\sigma )}(K)=\{ f\in L^2(K)\mid
f(mx)=\sigma (m)f(x)\ \text{for a.e. } m\in M_{\min },\ x\in K\} 
\]
and give $K$-module structure by the right regular action of $K$. 
Then the restriction map 
$r_K \colon H_{(\sigma ,\nu)}\ni f\mapsto f|_K\in L^2_{(M_{\min },\sigma )}(K)$ 
is an isomorphism of $K$-modules. 

$L^2(K)$ has a  $K\times K$-bimodule structure 
by the two sided regular action:
\[
((k_1,k_2)f)(x)=f(k_1^{-1}xk_2),\quad 
x\in K,\ f\in L^2(K),\ (k_1,k_2)\in K\times K,
\]
and a homomorphism $\Phi_\lambda \colon 
V_\lambda^*\otimes_\mC V_\lambda \to L^2(K)$ 
of $K\times K$-bimodules by 
\[
w\otimes v\mapsto (x\mapsto \ip{w}{\tau_\lambda (x)v})
\] 
where $\ip{}{} $ is a canonical pairing on 
$V_\lambda^*\otimes_\mC V_\lambda $.
Then the Peter-Weyl theorem tells that 
\[
\hatoplus_{\lambda \in L^+} \Phi_\lambda \colon 
\hatoplus_{\lambda \in L^+} 
V_\lambda^*\otimes_\mC V_\lambda \to L^2(K)
\] 
is an isomorphism as $K\times K$-bimodules. 
Here $\hatoplus $ means a Hilbert space direct sum. 

Since $L^2_{(M_{\min },\sigma )}(K)\subset L^2(K)$, 
we have an irreducible decomposition of 
$L^2_{(M_{\min },\sigma )}(K)$: 
\[
L^2_{(M_{\min },\sigma )}(K)\simeq \hatoplus_{\lambda \in L^+} 
(V_\lambda^*[\sigma ])\otimes_\mC V_\lambda .
\]
Here $V_\lambda^*[\sigma ]$ is the $\sigma $-isotypic component in 
$(\tau_\lambda^*|_{M_{\min }},V^*)$, i.e., 
\[
V_\lambda^*[\sigma ]=
\bigoplus_{\gamma \equiv (\sigma_1,\sigma_2,\sigma_3)\bmod 2}
V_\lambda^*(\gamma)=
\bigoplus_{M\in G_\sigma (\lambda )}
\mC f(M)^*
\]
where $G_\sigma (\lambda )=\{ M\in G(\lambda )\mid 
\wt{M} \equiv (\sigma_1,\sigma_2,\sigma_3)\bmod 2\} $.

From above arguments, We obtain an isomorphism 
\[
r_K^{-1}\circ \hatoplus_{\lambda \in L^+} \Phi_\lambda \colon 
\hatoplus_{\lambda \in L^+} 
(V_\lambda^*[\sigma ])\otimes_\mC V_\lambda \to 
H_{(\sigma ,\nu )}.
\]
Now we define elementary function $\efct{M}{N}\in H_{(\sigma ,\nu )}$ by 
$\efct{M}{N}=r_K^{-1}\circ \Phi_\lambda (f(M)^*\otimes f(N))$ for 
$M\in G_{\sigma }(\lambda )$ and $N\in G(\lambda )$. 
Let $l(M)$ be the order of the set 
$\{ N\in G(\lambda )\mid 
\wt{N}>\wt{M} \ \text{ or } \ 
\ N=M[k] \ \text{ for some }\ k\geq 0\}$ for a G-pattern $M$ of type $\lambda $. 
For a G-pattern $M$ of type $\lambda $, let 
$\evec{\lambda }{M}$ be a column vector of degree 
$d_\lambda =\dim V_{\lambda }$ 
with its $l(N)$-th component $\efct{M}{N} \ (\ N\in G(\lambda )\ )$, i.e., 
\[
\evec{\lambda }{M} =
\overset{t}{\overset{ }{\overset{ }{\overset{ }{ }}}}\biggl(
s\bkt{M,\gpt{\lambda_1\hs }{\lambda_2\hs }{\lambda_3}{\lambda_1}{\lambda_2}{\lambda_1}}, \cdots ,
\overset{l(N)\text{-th}}{\overset{\downarrow}{ \efct{M}{N}}}
,\cdots ,s\bkt{M,\gpt{\lambda_1\hs }{\lambda_2\hs }{\lambda_3}
{\lambda_2}{\lambda_3}{\lambda_3}}\biggl).
\hphantom{=====}
\]
Moreover we denote by $\langle \evec{\lambda }{M}\rangle $ 
the subspace of $H_{(\sigma ,\nu )}$ generated by the functions in the entries 
of the vector $\evec{\lambda }{M}$, i.e., 
$\langle \evec{\lambda }{M}\rangle =
\bigoplus_{N\in G(\lambda )}\mC \efct{M}{N}
\ \bkt{\ \simeq V_\lambda \ }$. We take the marking 
$\{\efct{M}{N}\}_{N\in G(\lambda )}$ for 
the simple $K$-module $\langle \evec{\lambda }{M}\rangle $.

\begin{prop}
As a unitary representation of $K$, it has an irreducible 
decomposition:
\[
H_{(\sigma ,\nu )}
\simeq \hatoplus_{\lambda \in L^+} 
(V_\lambda^*[\sigma ])\otimes_\mC V_\lambda .
\]
The direct sum 
\[
\bigoplus_{M\in G_{\sigma }(\lambda )}
\langle \evec{\lambda }{M} \rangle 
\]
is the $\tau_{\lambda }$-isotypic component 
$(V_\lambda^*[\sigma ])\otimes_\mC V_\lambda $ 
in $H_{(\sigma ,\nu )}$.
\end{prop}

\subsection{General setting }
\label{subsec:setting}
The $K$-finite part $H_{(\sigma ,\nu ),K}$ of $H_{(\sigma ,\nu )}$ is 
also a $\gk $-module. 
Because of the Cartan decomposition $\g =\gk \oplus \gp $, 
in order to describe the action of $\g $ or $\g_\mC $ it suffices to 
investigate the action of $\gp $ or $\gp_\mC =\gp_+\oplus \gp_-$. 

For a $K$-type $(\tau_\lambda ,V_\lambda )$ of 
$(\pi_{(\sigma ,\nu )}, H_{(\sigma ,\nu ),K})$ and 
a nonzero $K$-homomorphism $\eta \colon V_\lambda \to H_{(\sigma ,\nu ),K}$,
we define linear map 
$\tilde{\eta }\colon \gp_\mC \otimes_\mC V_\lambda \to H_{(\sigma ,\nu ),K}$ 
by $X\otimes v \mapsto X\cdot \eta (v) $. 
Then $\tilde{\eta }$ is $K$-homomorphism with $\gp_\mC $ endowed with 
the adjoint action $\Ad$ of $K$. 

Since 
\begin{align*}
&V_\lambda \otimes_\mC  \gp_+ \simeq 
V_\lambda \otimes_\mC  V_{2\me_1} \simeq 
\bigoplus_{1\leq i\leq j\leq 3} V_{\lambda \epe{i}{j}}\\ 
&\left( \text{ resp.\ } V_\lambda \otimes_\mC  \gp_- \simeq 
V_\lambda \otimes_\mC  V_{-2\me_3} \simeq 
\bigoplus_{1\leq i\leq j\leq 3} V_{\lambda \eme{i}{j}}
\right) ,
\end{align*}
there are six injective $K$-homomorphisms 
\begin{align*}
&I_{+ij}^\lambda =(\id_{V_\lambda }\otimes_\mC i_{\gp_+})
\circ i_{\me_i+\me_j}^\lambda 
\colon  V_{\lambda \epe{i}{j}} \to V_\lambda \otimes_\mC \gp_+ ,\quad 
1\leq i\leq j\leq 3 \\
&\left( \text{ resp.\ } I_{-ij}^\lambda =(\id_{V_\lambda }
\otimes_\mC i_{\gp_-})
\circ i_{-\me_i-\me_j}^\lambda 
\colon  V_{\lambda \eme{i}{j}} \to V_\lambda \otimes_\mC \gp_-,\quad 
1\leq i\leq j\leq 3 \right)
\end{align*}
for general $\lambda $. 
Then we define $\mC $-linear maps 
\[
\Gamma_{\pm ij}^\lambda \colon 
\Hom_K(V_\lambda ,H_{(\sigma ,\nu ),K})
\to \Hom_K(V_{\lambda [\pm ij]},H_{(\sigma ,\nu ),K})\quad 
,\ 1\leq i\leq j\leq 3
\]
by $\eta \mapsto \tilde{\eta }\circ I_{\pm ij}^\lambda $.

Now we settle two purposes of this paper:
\begin{description}
\item[(i)] Describe the injective $K$-homomorphism $I_{\pm ij}^\lambda $ 
in terms of the monomial basis.
\item[(ii)] Determine the matrix representations of 
the linear homomorphisms $\Gamma_{\pm ij}^\lambda $ 
with respect to the induced basis defined in the next subsection. 
\end{description}
We have already accomplished the first purpose in Lemma \ref{lem:K-action}, 
Proposition \ref{prop:clebsh(2,0,0)} and \ref{prop:clebsh(0,0,-2)}. 
We accomplish the second purpose 
in the subsection \ref{subsec:matrix_rep}. 
As a result, we obtain infinite number of 'contiguous relations', a kind 
of system of differential-difference relations among vectors in 
$H_{(\sigma ,\nu )}[\tau_{\lambda }]$ and 
$H_{(\sigma ,\nu )}[\tau_{\lambda [\pm ij]}]$. 
Here $H_{(\sigma ,\nu )}[\tau ] $ is $\tau$-isotypic component of 
$H_{(\sigma ,\nu )}$.

\subsection{The canonical blocks of elementary functions}
\label{subsec:canonical_blocks}

Let $\eta \colon V_\lambda \to H_{(\sigma ,\nu ),K}$ be a non-zero 
$K$-homomorphism, where $V_\lambda $ is a simple $K-module$ 
with the marking $\{f(M)\}_{M\in G(\lambda )}$. 
Then we identify $\eta $ with the column vector of degree 
$d_\lambda =\dim V_{\lambda }$ whose $l(N)$-th component is $\eta \bkt{f(N)}$ 
for $N\in G(\lambda )$, i.e., 
\[
\overset{t}{\overset{ }{\overset{ }{\overset{ }{ }}}}\biggl(
\eta \bkt{f\gpt{\lambda_1\hs }{\lambda_2\hs }{\lambda_3}{\lambda_1}{\lambda_2}{\lambda_1}}, \cdots ,
\overset{l(N)\text{-th}}{\overset{\downarrow}{ \eta \bkt{f(N)}}}
,\cdots ,\eta \bkt{f\gpt{\lambda_1\hs }{\lambda_2\hs }{\lambda_3}
{\lambda_2}{\lambda_3}{\lambda_3}}\biggl).
\hphantom{=========}
\]
By this identification, 
we identify $\evec{\lambda }{M}$ with
the injective $K$-homomorphism 
\[
V_\lambda \ni f(N)\mapsto \efct{M}{N}\in H_{(\sigma ,\nu ),K},\ 
N\in G(\lambda )
\]
for $M\in G_\sigma (\lambda )$. 
We note that $\{ \evec{\lambda}{M} \}_{M\in G_\sigma (\lambda )} $ 
is a basis of $\Hom_K (V_\lambda ,H_{(\sigma ,\nu ),K})$ and we call 
it \textit{the induced basis from the monomial basis}. 

We define a certain matrix of elementary functions corresponding to 
the induced basis $\{ \evec{\lambda}{M} \}_{M\in G_\sigma (\lambda )} $ 
of $\Hom_K(V_\lambda ,H_{(\sigma ,\nu ),K})$
for each $K$-type $\tau_\lambda $ of 
our principal series representation $\pi_{(\sigma ,\nu )}$.
\begin{defn}
\textit{
For $M\in G_\sigma (\lambda )$, 
let $d_\lambda^\sigma $ and 
$l^\sigma (M)$ be the orders of the set 
$G_\sigma (\lambda )$ and 
$\{ N\in G_\sigma (\lambda )\mid l(M)\leq l(N)\}$, respectively. 
The $d_\lambda \times d_\lambda^\sigma $ matrix $\eblock{\lambda }$ whose 
$l^\sigma (M)$-th column is $\evec{\lambda }{M}$ for 
$M\in G_\sigma (\lambda )$ is called
\textit{the canonical block of elementary functions} 
for $\tau_\lambda  $-isotypic component.
}
\end{defn}

\subsection{The $\gp_\pm $-matrix corresponding to 
$I_{\pm ij}^\lambda $}
\label{subsec:p-matrix}

In this subsection, we define $\gp_\pm $-matrix 
$\pmat{\lambda }{\pm }{ij}$ of size $d_{\lambda [\pm ij]}\times d_\lambda $ 
corresponding to $I_{\pm ij}^\lambda $ 
with respect to the monomial basis. 

\begin{defn}
\textit{
We define a $\gp_\pm $-matrix  
$\pmat{\lambda }{\pm }{ij}$ of size $d_{\lambda [\pm ij]}\times d_\lambda $
as follows.\\
(i) For $1\leq i\leq j\leq 3$, we define a $\gp_+$-matrix 
$\pmat{\lambda }{+}{ij}\in M(d_{\lambda \epe{i}{j}},d_\lambda ,\mC)\otimes \gp_+$ 
by 
\begin{align*}
\pmat{\lambda }{+}{ij}=
&\left\{ \sum_{m=0}^{\cpr{i}{j}{2}{2}}
L^{\lambda }_{+ij}\bkt{\cgpt{0}{-2}{-2} [-m]} \right\} 
\otimes X_{+11} 
+\left\{ \sum_{m=0}^{\cpr{i}{j}{1}{2}}
L^{\lambda }_{+ij}\bkt{\cgpt{0}{-2}{-1} [-m]} \right\} 
\otimes X_{+12} \\
&+\left\{ \sum_{m=0}^{\cpr{i}{j}{1}{1}}
L^{\lambda }_{+ij}\bkt{\cgpt{0}{-1}{-1} [-m]} \right\} 
\otimes X_{+13} 
+\left\{ \sum_{m=0}^{\cpr{i}{j}{0}{2}}
L^{\lambda }_{+ij}\bkt{\cgpt{0}{-2}{0} [-m]} \right\} 
\otimes X_{+22} \\
&+\left\{ \sum_{m=0}^{\cpr{i}{j}{0}{1}}
L^{\lambda }_{+ij}\bkt{\cgpt{0}{-1}{0} [-m]} \right\} 
\otimes X_{+23} 
+\left\{ \sum_{m=0}^{\cpr{i}{j}{0}{0}}
L^{\lambda }_{+ij}\bkt{\cgpt{0}{0}{0} [-m]} \right\} 
\otimes X_{+33} .
\end{align*}
Here $L^{\lambda }_{+ij}\left( \cgpt{0}{-l}{-k}[-m] \right) $ 
is a matrix of size $d_{\lambda \epe{i}{j}}\times d_\lambda $ 
whose $l(M)$-th row is given by  
\[
\left\{
\begin{array}{lll}
(\overbrace{0,\cdots ,0}^{l-1}, \cp{i}{j}{M}{l}{k}{m}, 
\overbrace{0,\cdots ,0}^{d_\lambda -l})
&\left( \ l=l\left(M\gpt{}{-\me_i-\me_j}{}{0}{-l}{-k}[-m]\right) \ \right) \\
&\quad \text{if }M\gpt{}{-\me_i-\me_j}{}{0}{-l}{-k}[-m]
\in G(\lambda )\\
\mathbf{0}&\quad \text{otherwise }
\end{array}
\right. 
\]
for $M\in G(\lambda \epe{i}{j})$.\\
(ii) For $1\leq i\leq j\leq 3$, we define a $\gp_-$-matrix 
$\pmat{\lambda }{-}{ij}\in M(d_{\lambda \eme{i}{j}},d_\lambda ,\mC)\otimes \gp_-$ 
by 
\begin{align*}
\pmat{\lambda }{-}{ij}=
&\left\{ \sum_{m=0}^{\cpr{4-j\hs }{4-i}{2}{2}}
L^{\lambda }_{-ij}\bkt{\cgpt{2}{0}{2} [-m]} \right\} 
\otimes X_{-11} 
-\left\{ \sum_{m=0}^{\cpr{4-j\hs }{4-i}{1}{2}}
L^{\lambda }_{-ij}\bkt{\cgpt{2}{0}{1} [-m]} \right\} 
\otimes X_{-12} \\
&+\left\{ \sum_{m=0}^{\cpr{4-j\hs }{4-i}{1}{1}}
L^{\lambda }_{-ij}\bkt{\cgpt{1}{0}{1} [-m]} \right\} 
\otimes X_{-13} 
+\left\{ \sum_{m=0}^{\cpr{4-j\hs }{4-i}{0}{2}}
L^{\lambda }_{-ij}\bkt{\cgpt{2}{0}{0} [-m]} \right\} 
\otimes X_{-22} \\
&-\left\{ \sum_{m=0}^{\cpr{4-j\hs }{4-i}{0}{1}}
L^{\lambda }_{-ij}\bkt{\cgpt{1}{0}{0} [-m]} \right\} 
\otimes X_{-23} 
+\left\{ \sum_{m=0}^{\cpr{4-j\hs }{4-i}{0}{0}}
L^{\lambda }_{-ij}\bkt{\cgpt{0}{0}{0} [-m]} \right\} 
\otimes X_{-33} .
\end{align*}
Here $L^{\lambda }_{-ij}\left( \cgpt{l}{0}{k} [-m]\right)$ 
is a matrix of size $d_{\lambda \eme{i}{j}}\times d_\lambda $ 
whose $l(M)$-th row is given by  
\[
\left\{
\begin{array}{lll}
(\overbrace{0,\cdots ,0}^{l-1}, \cpd{4-j\hs }{4-i}{(\hat{M})}{l}{k}{m} , 
\overbrace{0,\cdots ,0}^{d_\lambda -l})\ 
&\left(\ l=l\left(M\gpt{}{\me_i+\me_j}{}{l}{0}{k}[-m]\right)\ \right) \\
&\quad \text{if }M\gpt{}{\me_i+\me_j}{}{l}{0}{k}[-m]\in G(\lambda )\\
\mathbf{0}& \quad \text{otherwise }
\end{array}
\right. 
\]
for $M\in G(\lambda \eme{i}{j})$.
}
\end{defn}

Now we define 
$\pmat{\lambda }{\pm }{ij}\evec{\lambda }{M}\in 
\bkt{H_{(\sigma ,\nu ),K}}^{\oplus d_{\lambda [\pm ij]}}
\simeq \mC^{d_{\lambda [\pm ij]}}\otimes_\mC H_{(\sigma ,\nu ),K}$ 
by the action 
\begin{gather*}
(L\otimes X)(v\otimes f)=L(v)\otimes Xf,\\ 
\big(\ L\otimes X\in \Hom_\mC (\mC^{d_\lambda }, 
\mC^{d_{\lambda [\pm ij]}})
\otimes_\mC \gp_\pm ,\quad  
v\otimes f\in \mC^{d_\lambda }\otimes_\mC H_{(\sigma ,\nu ),K}\ \big)
\end{gather*}
for 
\begin{align*}
&\evec{\lambda }{M}\in \bkt{H_{(\sigma ,\nu ),K}}^{\oplus d_\lambda }
\simeq \mC^{d_\lambda }\otimes_\mC H_{(\sigma ,\nu ),K},\\
&\pmat{\lambda }{\pm }{ij}\in M(d_{\lambda [\pm ij]},d_\lambda ,\mC )\otimes_\mC \gp_\pm 
\simeq \Hom_\mC (\mC^{d_\lambda }, \mC^{d_{\lambda [\pm ij]}})
\otimes_\mC \gp_\pm .
\end{align*}
By the definition of $\pmat{\lambda }{\pm }{ij}$, 
we note that the vector $\pmat{\lambda }{\pm }{ij}\evec{\lambda }{M}$ 
is identified with the image of $\evec{\lambda }{M}$ by 
$\Gamma^\lambda_{\pm ij}$.

\subsection{The contiguous relations}
\label{subsec:matrix_rep}

To compute the matrix representation of $\Gamma^\lambda_{\pm ij}$ 
with respect to the induced basis, we prepare the following lemmas. 
\begin{lem}
\label{lem:Iwasawa}
\textit 
The root vectors $X_{\pm ij}\ (0\leq i\leq j\leq 3)$ in $\gp_{\pm }$ 
have the following expressions according to the Iwasawa decomposition of $\g $.
\begin{align*}
& X_{+ij}=\left\{ \begin{array}{ll}
2\sqrt{-1}E_{2e_i}+H_i+\kappa (E_{ii}),
&i=j\\
(E_{e_i-e_j}+\sqrt{-1}E_{e_i+e_j})+\kappa (E_{ji}),
&i<j
\end{array} \right. ,\\
& X_{-ij}=\left\{ \begin{array}{ll}
-2\sqrt{-1}E_{2e_i}+H_i-\kappa (E_{ii}),
&i=j\\
(E_{e_i-e_j}-\sqrt{-1}E_{e_i+e_j})-\kappa (E_{ij}),
&i<j
\end{array} \right. .
\end{align*}
\end{lem}
\begin{proof}
These are obtained by direct computation.
\end{proof}

%check
\begin{lem}
\label{lem:rel_clebsh}
\textit{
The coefficients $\cp{i}{j}{M}{l}{k}{m}$ 
in Proposition \ref{prop:clebsh(2,0,0)} 
have the following relations:
\begin{align*}
k_{ij}(M)\hs \cp{i}{j}{M\gpt{}{\me_i+\me_j}{}{0}{2}{2}[m]}{2}{2}{m}=
&(m_{12}-m_{23}+1)\chid{-}{-1}(M)
\cp{i}{j}{M\gpt{}{\me_i+\me_j}{}{0}{2}{2}[m]}{2}{1}{m-1}\\
&+(m_{11}-m_{22}+1)
\cp{i}{j}{M\gpt{}{\me_i+\me_j}{}{0}{2}{2}[m]}{2}{1}{m}\\
&+(C_1(M)+1)
\cp{i}{j}{M\gpt{}{\me_i+\me_j}{}{0}{2}{2}[m]}{1}{1}{m-1}\\
&+(m_{33}-m_{22}-1)
\cp{i}{j}{M\gpt{}{\me_i+\me_j}{}{0}{2}{2}[m]}{1}{1}{m}
\end{align*}
for $1\leq i\leq j\leq 3$, $m\in \mZ$ and $M\in G(\lambda )$. 
Here 
\begin{align*}
\cp{i}{j}{M\gpt{}{\me_i+\me_j}{}{0}{2}{2}[m]}{l}{k}{m}=&0
& \textit{ if } & \cpr{i}{j}{k}{l}<m \ \textit{ or }\ m<0.  
\end{align*}
and
\begin{align*}
&k_{11}(M)=-2m_{11}+2m_{13},  &k_{12}(M)=-2m_{11}+m_{13}+m_{23}-2,\\
&k_{22}(M)=-2m_{11}+2m_{23}-2,&k_{13}(M)=-2m_{11}+m_{13}+m_{33}-3,\\
&k_{33}(M)=-2m_{11}+2m_{33}-4,&k_{23}(m)=-2m_{11}+m_{23}+m_{33}-4.
\end{align*}}
\end{lem}
\begin{proof}
In order to prove the assertion, it suffices to confirm the equations 
\begin{align}
&k_{ij}\bkt{M\gpt{}{-\me_i-\me_j}{}{0}{-2}{-2}[-m]}\hs \cp{i}{j}{M}{2}{2}{m}
\label{eqn:as1_relcl}\\
&=(m_{12}-m_{23}-m+1+\delta_{2i}+\delta_{2j})
\chid{-}{-1}\bkt{M\gpt{}{-\me_i-\me_j}{}{0}{-2}{-2}[-m]}
\cp{i}{j}{M}{2}{1}{m-1}\nonumber \\
&+(m_{11}-m_{22}-m+1)\cp{i}{j}{M}{2}{1}{m}
+\bkt{C_1\bkt{M\gpt{}{-\me_i-\me_j}{}{0}{-2}{-2}[-m]}+1}
\cp{i}{j}{M}{1}{1}{m-1}\nonumber \\
&+(m_{33}-m_{22}-m+1-\delta_{3i}-\delta_{3j})
\cp{i}{j}{M}{1}{1}{m}\nonumber 
\end{align}
for $1\leq i\leq j\leq 3$, $m\in \mZ$ and $M\in G(\lambda [+ij])$. 
Since
\begin{align*}
k_{ij}\bkt{M\gpt{}{-\me_i-\me_j}{}{0}{-2}{-2}[-m]}
&=k_{ij}(M)+2(1-\delta_{ij}),\\
\chid{-}{-1}\bkt{M\gpt{}{-\me_i-\me_j}{}{0}{-2}{-2}[-m]}
&=\chi_-^{(-1+\delta_{2i}+\delta_{2j})}(M),\\
C_1\bkt{M\gpt{}{-\me_i-\me_j}{}{0}{-2}{-2}[-m]}+1
&=C_1(M)-m+1+\delta_{2i}\cdot \chi_-^{(\delta_{2j})}(M)
+\delta_{2j}\cdot \chi_-(M),
\end{align*}
(\ref{eqn:G-fct001}), (\ref{eqn:G-fct003}) and (\ref{eqn:G-fct005}), 
the equations (\ref{eqn:as1_relcl}) are equivalent to 
the equations 
\begin{align}
&\{k_{ij}(M)+2(1-\delta_{ij})\}\cp{i}{j}{M}{2}{2}{m}
\label{eqn:as2_relcl}\\
&=(m_{12}-m_{23}-m+1+\delta_{2i}+\delta_{2j})
\chi_-^{(-1+\delta_{2i}+\delta_{2j})}(M)
\cp{i}{j}{M}{2}{1}{m-1}\nonumber \\
&\hphantom{=}+(m_{11}-m_{22}-m+1)\cp{i}{j}{M}{2}{1}{m}\nonumber \\
&\hphantom{=}+\bkt{C_1(M)-m+1+\delta_{2i}\cdot \chi_-^{(\delta_{2j})}(M)
+\delta_{2j}\cdot \chi_-(M)}
\cp{i}{j}{M}{1}{1}{m-1}\nonumber \\
&\hphantom{=}+(m_{33}-m_{22}-m+1-\delta_{3i}-\delta_{3j})
\cp{i}{j}{M}{1}{1}{m}\nonumber 
\end{align}
for $1\leq i\leq j\leq 3$, $m\in \mZ$ and $M\in G(\lambda [+ij])$. 

We prove the equations (\ref{eqn:as2_relcl}) by direct computation. \\

%%%%%%%%%%%%%%%%%%%%%%%%%%%%%%%%%%%%%%%%%%%%%%%%%%%%%%%%%%%%%%%%%%%%%%%%
%check
\noindent $\bullet $ the proof of the case of the $(i,j)=(1,1)$.

Since $(\cpr{1}{1}{2}{2},\cpr{1}{1}{1}{2},\cpr{1}{1}{1}{1})=(2,3,2)$, 
we have to confirm the following equations: 
\begin{align}
&(m_{11}-m_{22}+1)\cp{1}{1}{M}{2}{1}{0}
+(m_{33}-m_{22}+1)\cp{1}{1}{M}{1}{1}{0}\label{eqn:relcl11_0}\\
&=(-2m_{11}+2m_{13})\cp{1}{1}{M}{2}{2}{0}\nonumber ,\\[2mm]
&(m_{12}-m_{23})\chid{-}{-1}(M)\cp{1}{1}{M}{2}{1}{0}\label{eqn:relcl11_1}
+(m_{11}-m_{22})\cp{1}{1}{M}{2}{1}{1}\\
&+C_1(M)\cp{1}{1}{M}{1}{1}{0}
+(m_{33}-m_{22})\cp{1}{1}{M}{1}{1}{1}\nonumber \\
&=(-2m_{11}+2m_{13})\cp{1}{1}{M}{2}{2}{1}\nonumber ,\\[2mm]
&(m_{12}-m_{23}-1)\chid{-}{-1}(M)\cp{1}{1}{M}{2}{1}{1}\label{eqn:relcl11_2}
+(m_{11}-m_{22}-1)\cp{1}{1}{M}{2}{1}{2} \\
&+\bkt{C_1(M)-1}\cp{1}{1}{M}{1}{1}{1}
+(m_{33}-m_{22}-1)\cp{1}{1}{M}{1}{1}{2}\nonumber \\
&=(-2m_{11}+2m_{13})\cp{1}{1}{M}{2}{2}{2},\nonumber \\[2mm]
&(m_{12}-m_{23}-2)\chid{-}{-1}(M)\cp{1}{1}{M}{2}{1}{2}\label{eqn:relcl11_3}
+(m_{11}-m_{22}-2)\cp{1}{1}{M}{2}{1}{3} \\
&+\bkt{C_1(M)-2}\cp{1}{1}{M}{1}{1}{2}=0\nonumber ,\\[2mm]
&(m_{12}-m_{23}-3)\chid{-}{-1}(M)\cp{1}{1}{M}{2}{1}{3}=0.\label{eqn:relcl11_4}
\\ \nonumber
\end{align}

We have 
\begin{align*}
&(m_{11}-m_{22}+1)\cp{1}{1}{M}{2}{1}{0}
+(m_{33}-m_{22}+1)\cp{1}{1}{M}{1}{1}{0}\\
&=-2(m_{11}-m_{22}+1)(m_{13} -m_{12})(m_{13} -m_{12}-1)
(m_{22}-m_{33})(m_{22}-m_{33}-1)\\
&\hphantom{=}-2(m_{33}-m_{22}+1)(m_{13} -m_{12})(m_{13} -m_{12}-1)
(m_{22}-m_{33})(m_{13}-m_{22}+1)\\
&=-2(m_{11}-m_{13})
(m_{13} -m_{12})(m_{13} -m_{12}-1)(m_{22}-m_{33})(m_{22}-m_{33}-1)\\
&=(-2m_{11}+2m_{13})\cp{1}{1}{M}{2}{2}{0}.
\end{align*}
Hence we obtain the equation (\ref{eqn:relcl11_0}).\\

We have
\begin{align*}
&(m_{12}-m_{23})\chid{-}{-1}(M)\cp{1}{1}{M}{2}{1}{0}
+(m_{11}-m_{22})\cp{1}{1}{M}{2}{1}{1}\\
&=-2(m_{12}-m_{23})\chid{-}{-1}(M)
(m_{13} -m_{12})(m_{13} -m_{12}-1)(m_{22}-m_{33})(m_{22}-m_{33}-1)\\
&\hphantom{=}+2(m_{11}-m_{22})(m_{13} -m_{12})(m_{22}-m_{33})
\left\{ \bar{F}\bkt{M\gpt{-1}{0}{0}{0}{-1}{-1}}+\bar{E}(M)\right\}\\
&=2(m_{13} -m_{12})(m_{22}-m_{33})\Big\{ 
-(m_{12}-m_{23})(m_{13} -m_{12}-1)(m_{22}-m_{33}-1)(1-\chi_+(M))\\
&\hphantom{=}+(m_{11}-m_{22})\big\{-C_1(M)(\bar{C_1}(M)+1)\\
&\hphantom{=}-\chi_+(M)\{(m_{13}-m_{12}-1)(m_{22}-m_{33}-1)
+(m_{13}-m_{33})\delta (M)\}+\bar{E}(M)\big\}\Big\}\\
&=2(m_{13} -m_{12})(m_{22}-m_{33})\Big\{ 
-(m_{12}-m_{23}-\delta (M)\chi_+(M))(m_{13} -m_{12}-1)(m_{22}-m_{33}-1)\\
&\hphantom{=}-(m_{11}-m_{22})C_1(M)(\bar{C_1}(M)+1)
-(m_{11}-m_{22})\chi_+(M)(m_{13}-m_{33})\delta (M)\\
&\hphantom{=}+(m_{11}-m_{22})\bar{E}(M)\Big\}\\
&=2(m_{13} -m_{12})(m_{22}-m_{33})\Big\{ 
-C_1(M)(m_{13} -m_{12}-1)(m_{22}-m_{33}-1)\\
&\hphantom{=}-(m_{11}-m_{22})C_1(M)(\bar{C_1}(M)+1)
-C_1(M)\chi_+(M)(m_{13}-m_{33})\delta (M)+(m_{11}-m_{22})\bar{E}(M)\Big\}\\
&=2(m_{13} -m_{12})(m_{22}-m_{33})\Big\{
-C_1(M)\big\{(m_{13} -m_{12}-1)(m_{22}-m_{33}-1)\\
&\hphantom{=}+(m_{13}-m_{33})\delta (M)\chi_+(M)
+(m_{11}-m_{22})(\bar{C_1}(M)+1)\big\}+(m_{11}-m_{22})\bar{E}(M)\Big\},
\\[2mm]
&C_1(M)\cp{1}{1}{M}{1}{1}{0}+(m_{33}-m_{22})\cp{1}{1}{M}{1}{1}{1} \\
&=-2C_1(M)(m_{13} -m_{12})(m_{13} -m_{12}-1)(m_{22}-m_{33})(m_{13}-m_{22}+1)\\
&\hphantom{=}+2(m_{33}-m_{22})(m_{13} -m_{12})\Big\{(m_{22}-m_{33})C_1(M)
	(\bar{C_1}(M)+1) +(m_{13}-m_{22})\bar{E}(M) \Big\} \\
&=2(m_{13} -m_{12})(m_{22}-m_{33})
\Big\{-C_1(M)\big\{(m_{13} -m_{12}-1)(m_{13}-m_{22}+1)\\
&\hphantom{=}+(m_{22}-m_{33})(\bar{C_1}(M)+1)\big\}
-(m_{13}-m_{22})\bar{E}(M) \Big\}.
\end{align*}
Therefore 
\begin{align*}
&(m_{12}-m_{23})\chid{-}{-1}(M)\cp{1}{1}{M}{2}{1}{0}
+(m_{11}-m_{22})\cp{1}{1}{M}{2}{1}{1}\\
&+C_1(M)\cp{1}{1}{M}{1}{1}{0}
+(m_{33}-m_{22})\cp{1}{1}{M}{1}{1}{1}\nonumber \\
&=2(m_{13} -m_{12})(m_{22}-m_{33})\Big\{
-C_1(M)\big\{(m_{13} -m_{12}-1)(m_{13}-m_{33})\\
&\hphantom{=}+(m_{13}-m_{33})\delta (M)\chi_+(M)
+(m_{11}-m_{33})(\bar{C_1}(M)+1)\big\}+(m_{11}-m_{13})\bar{E}(M)\Big\}\\
&=2(m_{13} -m_{12})(m_{22}-m_{33})\Big\{
-C_1(M)\big\{(m_{13} -m_{11})(m_{13}-m_{33})\\
&\hphantom{=}-(m_{13}-m_{33})(m_{12}-m_{11}-\delta (M)\chi_+(M)+1)
+(m_{11}-m_{33})(\bar{C_1}(M)+1)\big\}\\
&\hphantom{=}+(m_{11}-m_{13})\bar{E}(M)\Big\}\\
&=2(m_{13} -m_{12})(m_{22}-m_{33})\Big\{
(m_{11} -m_{13})C_1(M)(m_{13}-m_{33}-\bar{C_1}(M)-1)\\
&\hphantom{=}+(m_{11}-m_{13})\bar{E}(M)\Big\}\\
&=4(m_{11}-m_{13})(m_{13} -m_{12})(m_{22}-m_{33})(\bar{E}(M)-C_1(M))\\
&=(-2m_{11}+2m_{13})\cp{1}{1}{M}{2}{2}{1}.
\end{align*}
Hence we obtain the equation (\ref{eqn:relcl11_1}). 
Here we use the relations
\begin{align}
\bar{F}\bkt{M\gpt{-1}{0}{0}{0}{-1}{-1}}
&=-C_1(M)(\bar{C_1}(M)+1)\label{eqn:relcl11_pf01}\\
&\hphantom{=}-\chi_+(M)\{(m_{13}-m_{12}-1)(m_{22}-m_{33}-1)
+(m_{13}-m_{33})\delta (M)\},\nonumber
\end{align}
(\ref{eqn:G-fct002}) and (\ref{eqn:G-fct008}).\\

We have
\begin{align*}
&(m_{12}-m_{23}-1)\chid{-}{-1}(M)\cp{1}{1}{M}{2}{1}{1}
+(m_{11}-m_{22}-1)\cp{1}{1}{M}{2}{1}{2} \\
&=2(m_{12}-m_{23}-1)\chid{-}{-1}(M)(m_{13} -m_{12})(m_{22}-m_{33})
\Big\{ \bar{F}\bkt{M\gpt{-1}{0}{0}{0}{-1}{-1}}+\bar{E}(M)\Big\}\\
&\hphantom{=}-2(m_{11}-m_{22}-1)\Big\{\bar{E} (M)\bar{F}
	\bkt{M\gpt{-1}{0}{0}{-1}{0}{-1}} \\
&\hphantom{=}+(m_{13} -m_{12})(m_{22}-m_{33})
	C_1 (M)(\bar{C_1} (M)+1)\chi_+ (M)\Big\}\\
&=2(m_{12}-m_{23}-1)\chid{-}{-1}(M)(m_{13} -m_{12})(m_{22}-m_{33})
\big\{ -C_1(M)(\bar{C_1}(M)+1)+\bar{E}(M)\big\}\\
&\hphantom{=}-2(m_{11}-m_{22}-1)\Big\{\bar{E} (M)
\big\{-(C_1(M)-1)\bar{C_1}(M)
-\chi_+(M)\{(m_{13}-m_{12})(m_{22}-m_{33})\\
&\hphantom{=}+(m_{13}-m_{33})\delta (M)\} \big\}
+(m_{13} -m_{12})(m_{22}-m_{33})
	C_1 (M)(\bar{C_1} (M)+1)\chi_+ (M)\Big\}\\
&=2(C_1 (M)-1)\chid{-}{-1}(M)(m_{13} -m_{12})(m_{22}-m_{33})
\big\{ -C_1(M)(\bar{C_1}(M)+1)+\bar{E}(M)\big\}\\
&\hphantom{=}-2\bar{E} (M)
\big\{-(C_1(M)-1)\bar{C_1}(M)(m_{11}-m_{22}-1)\\
&\hphantom{=}-(C_1 (M)-1)\chi_+(M)\{(m_{13}-m_{12})(m_{22}-m_{33})
+(m_{13}-m_{33})\delta (M)\} \big\}\\
&\hphantom{=}-2(C_1 (M)-1)(m_{13} -m_{12})(m_{22}-m_{33})
	C_1 (M)(\bar{C_1} (M)+1)\chi_+ (M)\\
&=-2(C_1 (M)-1)\Big\{(m_{13} -m_{12})(m_{22}-m_{33})
	C_1(M)(\bar{C_1}(M)+1)(\chid{-}{-1}(M)+\chi_+(M))\\
&\hphantom{=}-\bar{E} (M)\big\{\bar{C_1}(M)(m_{11}-m_{22}-1)
+(m_{13} -m_{12})(m_{22}-m_{33})(\chid{-}{-1}(M)+\chi_+(M))\\
&\hphantom{=}+(m_{13}-m_{33})\delta (M)\chi_+(M) \big\}\Big\}\\
&=-2(C_1 (M)-1)\Big\{(m_{13} -m_{12})(m_{22}-m_{33})C_1(M)(\bar{C_1}(M)+1)\\
&\hphantom{=}-\bar{E} (M)\big\{\bar{C_1}(M)(m_{11}-m_{22}-1)
+(m_{13} -m_{12})(m_{22}-m_{33})\\
&\hphantom{=}+(m_{13}-m_{33})\delta (M)\chi_+(M) \big\}\Big\},\\[2mm]
&(C_1(M)-1)\cp{1}{1}{M}{1}{1}{1}
+(m_{33}-m_{22}-1)\cp{1}{1}{M}{1}{1}{2}\nonumber \\
&=2(C_1(M)-1)(m_{13} -m_{12})\{(m_{22}-m_{33})C_1(M)(\bar{C_1}(M)+1)
+(m_{13}-m_{22})\bar{E}(M) \} \\
&\hphantom{=}-2(m_{33}-m_{22}-1)\bar{E}(M)(C_1(M)-1)\bar{C_1}(M)\\
&=-2(C_1(M)-1)\Big\{-(m_{13} -m_{12})(m_{22}-m_{33})C_1(M)(\bar{C_1}(M)+1)\\
&\hphantom{=}-\bar{E}(M)\{(m_{13} -m_{12})(m_{13}-m_{22})
+(m_{22}-m_{33}+1)\bar{C_1}(M) \Big\}.
\end{align*}
Therefore
\begin{align*}
&(m_{12}-m_{23}-1)\chid{-}{-1}(M)\cp{1}{1}{M}{2}{1}{1}
+(m_{11}-m_{22}-1)\cp{1}{1}{M}{2}{1}{2} \\
&+\bkt{C_1(M)-1}\cp{1}{1}{M}{1}{1}{1}
+(m_{33}-m_{22}-1)\cp{1}{1}{M}{1}{1}{2}\nonumber \\
&=2(C_1 (M)-1)\bar{E} (M)\big\{\bar{C_1}(M)(m_{11}-m_{33})
+(m_{13} -m_{12})(m_{13}-m_{33})\\
&\hphantom{=}+(m_{13}-m_{33})\delta (M)\chi_+(M) \big\}\\
&=2(C_1 (M)-1)\bar{E} (M)\big\{\bar{C_1}(M)(m_{11}-m_{33})
+(m_{13} -m_{11})(m_{13}-m_{33})\\
&\hphantom{=}-(m_{13}-m_{33})(m_{12}-m_{11}-\delta (M)\chi_+(M)) \big\}\\
&=2(m_{13} -m_{11})\bar{E} (M)(C_1 (M)-1)(m_{13}-m_{33}-\bar{C_1}(M))\\
&=(-2m_{11}+2m_{13})\cp{1}{1}{M}{2}{2}{2}. 
\end{align*}
Hence we obtain the equation (\ref{eqn:relcl11_2}). 
Here we use the relations
\begin{align}
\bar{F}\bkt{M\gpt{-1}{0}{0}{-1}{0}{-1}}
&=-(C_1(M)-1)\bar{C_1}(M)\label{eqn:relcl11_pf02}\\
&\hphantom{=}-\chi_+(M)\{(m_{13}-m_{12})(m_{22}-m_{33})
+(m_{13}-m_{33})\delta (M)\},\nonumber
\end{align}
(\ref{eqn:pf_clebsh(1,0,0)}), (\ref{eqn:pf_clebsh(1,0,0)_003}), 
(\ref{eqn:pf_clebsh(1,0,0)_004}), (\ref{eqn:G-fct009}) and 
(\ref{eqn:relcl11_pf01}).\\

We have
\begin{align*}
&(m_{12}-m_{23}-2)\chid{-}{-1}(M)\cp{1}{1}{M}{2}{1}{2}
+(m_{11}-m_{22}-2)\cp{1}{1}{M}{2}{1}{3} \\
&+\bkt{C_1(M)-2}\cp{1}{1}{M}{1}{1}{2}\\
&=-2(m_{12}-m_{23}-2)\chid{-}{-1}(M)\Big\{\bar{E} (M)\bar{F}
	\bkt{M\gpt{-1}{0}{0}{-1}{0}{-1}} \\
&\hphantom{=}+(m_{13} -m_{12})(m_{22}-m_{33})
	C_1 (M)(\bar{C_1} (M)+1)\chi_+ (M)\Big\}\\
&\hphantom{=}+2(m_{11}-m_{22}-2)(C_1(M)-1)\bar{C_1} (M)\bar{E} (M)\chi_+ (M) \\
&\hphantom{=}-2(C_1(M)-2)\bar{E}(M)(C_1(M)-1)\bar{C_1}(M)\\
&=2(C_1(M)-2)\chid{-}{-1}(M)\bar{E} (M)(C_1(M)-1)\bar{C_1}(M)\\
&\hphantom{=}+2(C_1(M)-2)(C_1(M)-1)\bar{C_1} (M)\bar{E} (M)\chi_+ (M) \\
&\hphantom{=}-2(C_1(M)-2)\bar{E}(M)(C_1(M)-1)\bar{C_1}(M)\\
&=2(C_1(M)-2)(C_1(M)-1)\bar{C_1}(M)\bar{E}(M)(\chid{-}{-1}(M)+\chi_+ (M)-1)=0.
\end{align*}
Hence we obtain the equation (\ref{eqn:relcl11_3}). 
Here we obtain the relations (\ref{eqn:G-fct008}), 
(\ref{eqn:G-fct009}), (\ref{eqn:pf_clebsh(1,0,0)_003}), 
(\ref{eqn:pf_clebsh(1,0,0)_004}) and (\ref{eqn:relcl11_pf02}).\\

We have
\begin{align*}
&(m_{12}-m_{23}-3)\chid{-}{-1}(M)\cp{1}{1}{M}{2}{1}{3}\\
&=2(m_{12}-m_{23}-3)(C_1(M)-1)
\bar{C_1}(M)\bar{E}(M)\chid{-}{-1}(M)\chi_+(M)=0.
\end{align*}
Hence we obtain the equation (\ref{eqn:relcl11_4}). 
Here we use the relation (\ref{eqn:G-fct009}).\\

%%%%%%%%%%%%%%%%%%%%%%%%%%%%%%%%%%%%%%%%%%%%%%%%%%%%%%%%%%%%%%%%%%%%%%%%
%check
\noindent $\bullet $ the proof of the case of the $(i,j)=(2,2)$.

Since $(\cpr{2}{2}{2}{2},\cpr{2}{2}{1}{2},\cpr{2}{2}{1}{1})=(2,2,2)$, 
we have to confirm the following equations: 
\begin{align}
&(m_{11}-m_{22}+1)\cp{2}{2}{M}{2}{1}{0}
+(m_{33}-m_{22}+1)\cp{2}{2}{M}{1}{1}{0}\label{eqn:relcl22_0}\\
&=(-2m_{11}+2m_{23}-2)\cp{2}{2}{M}{2}{2}{0}\nonumber ,\\[2mm]
&(m_{12}-m_{23}+2)\chid{-}{1}(M)\cp{2}{2}{M}{2}{1}{0}\label{eqn:relcl22_1}
+(m_{11}-m_{22})\cp{2}{2}{M}{2}{1}{1}\\
&+\bkt{C_1(M)+\chid{-}{1}(M)+\chi_-(M)}\cp{2}{2}{M}{1}{1}{0}
+(m_{33}-m_{22})\cp{2}{2}{M}{1}{1}{1}\nonumber \\
&=(-2m_{11}+2m_{23}-2)\cp{2}{2}{M}{2}{2}{1}\nonumber ,\\[2mm]
&(m_{12}-m_{23}+1)\chid{-}{1}(M)\cp{2}{2}{M}{2}{1}{1}\label{eqn:relcl22_2}
+(m_{11}-m_{22}-1)\cp{2}{2}{M}{2}{1}{2} \\
&+\bkt{C_1(M)-1+\chid{-}{1}(M)+\chi_-(M)}\cp{2}{2}{M}{1}{1}{1}
+(m_{33}-m_{22}-1)\cp{2}{2}{M}{1}{1}{2}\nonumber \\
&=(-2m_{11}+2m_{23}-2)\cp{2}{2}{M}{2}{2}{2},\nonumber \\[2mm]
&(m_{12}-m_{23})\chid{-}{1}(M)
\cp{2}{2}{M}{2}{1}{2}\label{eqn:relcl22_3}
+\bkt{C_1(M)-2+\chid{-}{1}(M)+\chi_-(M)}
\cp{2}{2}{M}{1}{1}{2}=0.\\ \nonumber
\end{align}

We have
\begin{align*}
&(m_{11}-m_{22}+1)\cp{2}{2}{M}{2}{1}{0}
+(m_{33}-m_{22}+1)\cp{2}{2}{M}{1}{1}{0}\\
&=-2(m_{11}-m_{22}+1)(m_{22}-m_{33})(m_{22}-m_{33}-1)\\
&\hphantom{=}-2(m_{33}-m_{22}+1)(m_{22}-m_{33})(m_{23}-m_{22})\\
&=-2(m_{11}-m_{23}+1)(m_{22}-m_{33})(m_{22}-m_{33}-1)\\
&=(-2m_{11}+2m_{23}-2)\cp{2}{2}{M}{2}{2}{0}\nonumber .
\end{align*}
Hence we obtain the equation (\ref{eqn:relcl22_0}).\\

We have
\begin{align*}
&(m_{12}-m_{23}+2)\chid{-}{1}(M)\cp{2}{2}{M}{2}{1}{0}
+(m_{11}-m_{22})\cp{2}{2}{M}{2}{1}{1}\\
&+\bkt{C_1(M)+\chid{-}{1}(M)+\chi_-(M)}\cp{2}{2}{M}{1}{1}{0}
+(m_{33}-m_{22})\cp{2}{2}{M}{1}{1}{1}\nonumber \\
&=-2(m_{12}-m_{23}+2)\chid{-}{1}(M)(m_{22}-m_{33})(m_{22}-m_{33}-1)\\
&\hphantom{=}
+2(m_{11}-m_{22})(m_{22}-m_{33})\{ \bar{C_1} (M)+(\bar{D}(M)+1)\chi_-(M)\} \\
&\hphantom{=}
-2\bkt{C_1(M)+\chid{-}{1}(M)+\chi_-(M)}(m_{22}-m_{33})(m_{23}-m_{22})\\
&\hphantom{=}+2(m_{33}-m_{22})\{ \bar{D}(M)(m_{23}-m_{22}-1)\chi_-(M)
-(m_{22}-m_{33})(\bar{C_1}(M)+1)\chid{-}{1}(M) \}\\
&=2(m_{22}-m_{33})\Big\{-(m_{12}-m_{23}+2)(m_{22}-m_{33}-1)\chid{-}{1}(M)\\
&\hphantom{=}+(m_{11}-m_{22})
\{ m_{23}-m_{22}+(\delta (M)+\bar{D}(M)+1)\chi_-(M)\}\\
&\hphantom{=}
-\{m_{11}-m_{22}+(\delta (M)+1)\chi_-(M)+\chid{-}{1}(M)\}(m_{23}-m_{22})\\
&\hphantom{=}-\{ \bar{D}(M)(m_{23}-m_{22}-1)\chi_-(M)
-(m_{22}-m_{33})(m_{12}-m_{11}+1)\chid{-}{1}(M) \}\Big\}\\
&=2(m_{22}-m_{33})\Big\{-(m_{11}-m_{23}+1)(m_{22}-m_{33}-1)\chid{-}{1}(M)\\
&\hphantom{=}+(\delta (M)+1)\{(m_{11}-m_{23})\chi_-(M)+\chid{-}{1}(M)\}
+(m_{11}-m_{23}+1)\bar{D}(M)\chi_-(M)\Big\}\\
&=2(m_{22}-m_{33})
\Big\{(m_{11}-m_{23}+1)(-m_{22}+m_{33}+\delta (M)+2)\chid{-}{1}(M)\\
&\hphantom{=}+(m_{11}-m_{23}+1)\bar{D}(M)\chi_-(M)\Big\}\\
&=2(m_{11}-m_{23}+1)(m_{22}-m_{33})
\{(\bar{D}(M)+2)\chid{-}{1}(M)+\bar{D} (M)\chi_-(M)\}\\
&=(-2m_{11}+2m_{23}-2)\cp{2}{2}{M}{2}{2}{1}\nonumber .
\end{align*}
Hence we obtain the equation (\ref{eqn:relcl22_1}). 
Here we use the relations (\ref{eqn:G-fct001}), (\ref{eqn:G-fct002}), 
(\ref{eqn:G-fct007}) and (\ref{eqn:pf_clebsh(1,0,0)_002}).\\

We have
\begin{align*}
&\bkt{C_1(M)-1+\chid{-}{1}(M)+\chi_-(M)}\cp{2}{2}{M}{1}{1}{1}\\
&=2\bkt{C_1(M)-1+\chid{-}{1}(M)+\chi_-(M)}
\{ \bar{D}(M)(m_{23}-m_{22}-1)\chi_-(M)\\
&\hphantom{=}-(m_{22}-m_{33})(\bar{C_1}(M)+1)\chid{-}{1}(M) \}\\
&=2C_1(M)\bar{D}(M)(m_{23}-m_{22}-1)\chi_-(M)\\
&\hphantom{=}-2\{\bkt{C_1(M)+1}(m_{22}-m_{33})(\bar{C_1}(M)+1)
-\bar{D}(M)(m_{23}-m_{22}-1)\}\chid{-}{1}(M).
\end{align*}
Therefore
\begin{align*}
&(m_{12}-m_{23}+1)\chid{-}{1}(M)\cp{2}{2}{M}{2}{1}{1}
+(m_{11}-m_{22}-1)\cp{2}{2}{M}{2}{1}{2} \\
&+\bkt{C_1(M)-1+\chid{-}{1}(M)+\chi_-(M)}\cp{2}{2}{M}{1}{1}{1}
+(m_{33}-m_{22}-1)\cp{2}{2}{M}{1}{1}{2}\nonumber \\
&=2(m_{12}-m_{23}+1)\chid{-}{1}(M)(m_{22}-m_{33})
\{ \bar{C_1} (M)+(\bar{D}(M)+1)\chi_-(M)\}\\
&\hphantom{=}-2(m_{11}-m_{22}-1)\bar{C_1}(M)\bar{D} (M)\chi_-(M)\\
&\hphantom{=}+2C_1(M)\bar{D}(M)(m_{23}-m_{22}-1)\chi_-(M)\\
&\hphantom{=}-2\{\bkt{C_1(M)+1}(m_{22}-m_{33})(\bar{C_1}(M)+1)
-\bar{D}(M)(m_{23}-m_{22}-1)\}\chid{-}{1}(M)\\
&\hphantom{=}+2(m_{33}-m_{22}-1)\bar{C_1}(M)\bar{D}(M)\chid{-}{1}(M) \\
&=2(m_{12}-m_{23}+1)(m_{22}-m_{33})
( m_{12}-m_{11}+\bar{D}(M)+1)\chid{-}{1}(M)\\
&\hphantom{=}-2(m_{11}-m_{22}-1)(m_{12}-m_{11})\bar{D} (M)\chi_-(M)\\
&\hphantom{=}+2(m_{12}-m_{23})\bar{D}(M)(m_{23}-m_{22}-1)\chi_-(M)\\
&\hphantom{=}-2\{(m_{12}-m_{23}+1)(m_{22}-m_{33})(m_{12}-m_{11}+1)
-\bar{D}(M)(m_{23}-m_{22}-1)\}\chid{-}{1}(M)\\
&\hphantom{=}+2(m_{33}-m_{22}-1)(m_{12}-m_{11})\bar{D}(M)\chid{-}{1}(M) \\
&=2(m_{11}-m_{23})(m_{22}-m_{33})\bar{D}(M)\chid{-}{1}(M)
-2(-m_{22}+m_{33}+\delta (M)+1)\bar{D}(M)\chid{-}{1}(M)\\
&\hphantom{=}-2(m_{11}-m_{23})(\delta (M)+1)\bar{D} (M)\chi_-(M)\\
&=-2(m_{11}-m_{23}+1)(-m_{22}+m_{33}+\delta (M)+1)\bar{D}(M)\chid{-}{1}(M)\\
&=(-2m_{11}+2m_{23}-2)\cp{2}{2}{M}{2}{2}{2}.
\end{align*}
Hence we obtain the equation (\ref{eqn:relcl22_2}). 
Here we use the relations (\ref{eqn:G-fct001}), (\ref{eqn:G-fct002}), 
(\ref{eqn:G-fct007}), (\ref{eqn:G-fct011}), (\ref{eqn:pf_clebsh(1,0,0)_002})
 and (\ref{eqn:pf_clebsh(1,0,0)_004}). \\

We have
\begin{align*}
&(m_{12}-m_{23})\chid{-}{1}(M)\cp{2}{2}{M}{2}{1}{2}
+\bkt{C_1(M)-2+\chid{-}{1}(M)+\chi_-(M)}
\cp{2}{2}{M}{1}{1}{2}\\
&=-2(m_{12}-m_{23})\chid{-}{1}(M)\bar{C_1}(M)\bar{D} (M)\chi_-(M)\\
&\hphantom{=}+2\bkt{C_1(M)-2+\chid{-}{1}(M)+\chi_-(M)}
\bar{C_1}(M)\bar{D}(M)\chid{-}{1}(M)\\
&=-2C_1(M)\bar{C_1}(M)\bar{D} (M)\chid{-}{1}(M)
+2C_1(M)\bar{C_1}(M)\bar{D}(M)\chid{-}{1}(M)=0.
\end{align*}
Hence we obtain the equation (\ref{eqn:relcl22_3}). 
Here we use the relations (\ref{eqn:pf_clebsh(1,0,0)_004}) 
and (\ref{eqn:G-fct011}).\\

%%%%%%%%%%%%%%%%%%%%%%%%%%%%%%%%%%%%%%%%%%%%%%%%%%%%%%%%%%%%%%%%%%%%%%%%
%check
\noindent $\bullet $ the proof of the case of the $(i,j)=(3,3)$.

Since $(\cpr{3}{3}{2}{2},\cpr{3}{3}{1}{2},\cpr{3}{3}{1}{1})=(0,1,0)$, 
we have to confirm the following equations: 
\begin{align}
&(m_{11}-m_{22}+1)\cp{3}{3}{M}{2}{1}{0}
+(m_{33}-m_{22}-1)\cp{3}{3}{M}{1}{1}{0}\label{eqn:relcl33_0}\\
&=(-2m_{11}+2m_{33}-4)\cp{3}{3}{M}{2}{2}{0}\nonumber ,\\[2mm]
&(m_{12}-m_{23})\chid{-}{-1}(M)\cp{3}{3}{M}{2}{1}{0}\label{eqn:relcl33_1}
+(m_{11}-m_{22})\cp{3}{3}{M}{2}{1}{1}
+C_1(M)\cp{3}{3}{M}{1}{1}{0}\\
&=(-2m_{11}+2m_{33}-4)\cp{3}{3}{M}{2}{2}{1}\nonumber ,\\[2mm]
&(m_{12}-m_{23}-1)\chid{-}{-1}(M)\cp{3}{3}{M}{2}{1}{1}\label{eqn:relcl33_2}=0.
\\ \nonumber
\end{align}

We have 
\begin{align*}
&(m_{11}-m_{22}+1)\cp{3}{3}{M}{2}{1}{0}
+(m_{33}-m_{22}-1)\cp{3}{3}{M}{1}{1}{0}\\
&=-2(m_{11}-m_{22}+1)+2(m_{33}-m_{22}-1)\\
&=-2m_{11}+2m_{33}-4
=(-2m_{11}+2m_{33}-4)\cp{3}{3}{M}{2}{2}{0}.
\end{align*}
Hence we obtain the equation (\ref{eqn:relcl33_0}).\\

We have
\begin{align*}
&(m_{12}-m_{23})\chid{-}{-1}(M)\cp{3}{3}{M}{2}{1}{0}
+(m_{11}-m_{22})\cp{3}{3}{M}{2}{1}{1}+C_1(M)\cp{3}{3}{M}{1}{1}{0}\\
&=-2(m_{12}-m_{23})\chid{-}{-1}(M)-2(m_{11}-m_{22})\chi_+(M)+2C_1(M)\\
&=-2(m_{12}-m_{23})(1-\chi_+(M))-2(m_{11}-m_{22})\chi_+(M)+2C_1(M)\\
&=-2(m_{12}-m_{23}-\delta (M)\chi_+(M)-C_1(M))=0\nonumber .
\end{align*}
Hence we obtain the equation (\ref{eqn:relcl33_1}).
Here we use the relations (\ref{eqn:G-fct001}) and (\ref{eqn:G-fct008}).\\

We have
\begin{align*}
&(m_{12}-m_{23}-1)\chid{-}{-1}(M)\cp{3}{3}{M}{2}{1}{1}
=-2(m_{12}-m_{23}-1)\chid{-}{-1}(M)\chi_+(M)=0.
\end{align*}
Hence we obtain the equation (\ref{eqn:relcl33_2}).
Here we use the relation (\ref{eqn:G-fct009}).\\

%%%%%%%%%%%%%%%%%%%%%%%%%%%%%%%%%%%%%%%%%%%%%%%%%%%%%%%%%%%%%%%%%%%%%%%%
%check
\noindent $\bullet $ the proof of the case of the $(i,j)=(1,2)$.

Since $(\cpr{1}{2}{2}{2},\cpr{1}{2}{1}{2},\cpr{1}{2}{1}{1})=(2,2,2)$, 
we have to confirm the following equations: 
\begin{align}
&(m_{11}-m_{22}+1)\cp{1}{2}{M}{2}{1}{0}
+(m_{33}-m_{22}+1)\cp{1}{2}{M}{1}{1}{0}\label{eqn:relcl12_0}\\
&=(-2m_{11}+m_{13}+m_{23})\cp{1}{2}{M}{2}{2}{0}\nonumber ,\\
\nonumber \\
&(m_{12}-m_{23}+1)\chi_-(M)\cp{1}{2}{M}{2}{1}{0}\label{eqn:relcl12_1}
+(m_{11}-m_{22})\cp{1}{2}{M}{2}{1}{1}\\
&+\bkt{C_1(M)+\chi_-(M)}\cp{1}{2}{M}{1}{1}{0}\nonumber 
+(m_{33}-m_{22})\cp{1}{2}{M}{1}{1}{1} \\
&=(-2m_{11}+m_{13}+m_{23})\cp{1}{2}{M}{2}{2}{1}\nonumber ,\\
\nonumber \\
&(m_{12}-m_{23})\chi_-(M)\cp{1}{2}{M}{2}{1}{1}\label{eqn:relcl12_2}
+(m_{11}-m_{22}-1)\cp{1}{2}{M}{2}{1}{2}\\
&+\bkt{C_1(M)-1+\chi_-(M)}\cp{1}{2}{M}{1}{1}{1}
+(m_{33}-m_{22}-1)\cp{1}{2}{M}{1}{1}{2}\nonumber \\
&=(-2m_{11}+m_{13}+m_{23})\cp{1}{2}{M}{2}{2}{2},\nonumber \\
\nonumber \\
&(m_{12}-m_{23}-1)\chi_-(M)\cp{1}{2}{M}{2}{1}{2}\label{eqn:relcl12_3}
+\bkt{C_1(M)-2+\chi_-(M)}\cp{1}{2}{M}{1}{1}{2}=0.\\ \nonumber
\end{align}

We have 
\begin{align*}
&(m_{11}-m_{22}+1)\cp{1}{2}{M}{2}{1}{0}
+(m_{33}-m_{22}+1)\cp{1}{2}{M}{1}{1}{0}\\
&=-2(m_{11}-m_{22}+1)(m_{13}-m_{12})(m_{22}-m_{33})(m_{22}-m_{33}-1)\\
&\hphantom{=}+(m_{33}-m_{22}+1)(m_{13} -m_{12})(m_{22}-m_{33})
(2m_{22}-m_{13}-m_{23}-2)\\
&=-(2m_{11}-m_{13}-m_{23})
(m_{13}-m_{12})(m_{22}-m_{33})(m_{22}-m_{33}-1)\\
&=(-2m_{11}+m_{13}+m_{23})\cp{1}{2}{M}{2}{2}{0}.
\end{align*}
Hence we obtain the equation (\ref{eqn:relcl12_0}).\\

We have
\begin{align*}
&(m_{11}-m_{22})\cp{1}{2}{M}{2}{1}{1}+(m_{33}-m_{22})\cp{1}{2}{M}{1}{1}{1} \\
&=(m_{11}-m_{22})(m_{22}-m_{33})\big\{\bar{E}(M)\\
&\hphantom{=}-C_2(M)-\chi_+(M)\{(m_{13}-m_{12})(m_{22}-m_{33})
	+(m_{13}-m_{33}+1)\delta (M)\}\\
&\hphantom{=}+(m_{13}-m_{12})\{\bar{C_1}(M)+1+\bar{D}(M)(1-\chi_+(M))\}\big\}\\
&\hphantom{=}
+(m_{33}-m_{22})\Big\{\bar{E} (M)(m_{23} -m_{22})+C_2 (M)(m_{22}-m_{33}+1)\\
&\hphantom{=}+(m_{13}-m_{12})\chi_-(M)\{ \bar{D}(M)(m_{13}-m_{22}+1) 
-(m_{22}-m_{33})(\bar{C_1} (M)+1) \}\Big\} \\
&=(m_{22}-m_{33})\Big\{(m_{11}-m_{23})\bar{E} (M)\\
&\hphantom{=}-(m_{11}-m_{33}+1)C_2(M)
-(m_{11}-m_{22})(m_{13}-m_{33}+1)\delta (M)\chi_+(M)\\
&\hphantom{=}+(m_{13}-m_{12})\big\{(m_{11}-m_{22})\{\bar{C_1}(M)+1+
\bar{D}(M)(1-\chi_+(M))-(m_{22}-m_{33})\chi_+(M)\}\\
&\hphantom{=}-\chi_-(M)\{ \bar{D}(M)(m_{13}-m_{22}+1)
-(m_{22}-m_{33})(\bar{C_1} (M)+1) \}\big\}\Big\}\\
&=(m_{22}-m_{33})\Big\{(2m_{11}-m_{13}-m_{23})\bar{E} (M)
+(m_{13}-m_{11})C_1(M)(m_{13}-m_{33}+1-\bar{C_1}(M))\\
&\hphantom{=}-(m_{11}-m_{33}+1)C_2(M)
-C_1(M)(m_{13}-m_{33}+1)\delta (M)\chi_+(M)\\
&\hphantom{=}+(m_{13}-m_{12})\big\{(m_{11}-m_{22})\{\bar{C_1}(M)+1+
\bar{D}(M)-(-m_{22}+m_{33}+\delta (M))\chi_+(M)\\
&\hphantom{=}-(m_{22}-m_{33})\chi_+(M)\}
-\chi_-(M)\{ \bar{D}(M)(m_{13}-m_{22}+1)
-(m_{22}-m_{33})(\bar{C_1} (M)+1) \}\big\}\Big\}\\
&=(m_{22}-m_{33})\Big\{(2m_{11}-m_{13}-m_{23})\bar{E} (M)\\
&\hphantom{=}+(m_{13}-m_{11}-\bar{C_1}(M)
-\delta (M)\chi_+(M))C_1(M)(m_{13}-m_{33}+1)\\
&\hphantom{=}+(m_{13}-m_{12})\big\{(m_{11}-m_{22})
\{\bar{C_1}(M)+1+\bar{D}(M)-\delta (M)\chi_+(M)\}\\
&\hphantom{=}-\chi_-(M)\{ \bar{D}(M)(m_{13}-m_{22}+1)
-(m_{22}-m_{33})(\bar{C_1} (M)+1) \}\big\}\Big\}\\
&=(m_{22}-m_{33})\Big\{(2m_{11}-m_{13}-m_{23})\bar{E} (M)
+(m_{13}-m_{12})\big\{C_1(M)(m_{13}-m_{33}+1)\\
&\hphantom{=}+(m_{11}-m_{22})
\{\bar{C_1}(M)+1+\bar{D}(M)-\delta (M)\chi_+(M)\}\\
&\hphantom{=}-\chi_-(M)\{ \bar{D}(M)(m_{13}-m_{22}+1)
-(m_{22}-m_{33})(\bar{C_1} (M)+1) \}\big\}\Big\}.
\end{align*}
Therefore
\begin{align*}
&(m_{12}-m_{23}+1)\chi_-(M)\cp{1}{2}{M}{2}{1}{0}
+(m_{11}-m_{22})\cp{1}{2}{M}{2}{1}{1}\\
&+\bkt{C_1(M)+\chi_-(M)}\cp{1}{2}{M}{1}{1}{0}\nonumber 
+(m_{33}-m_{22})\cp{1}{2}{M}{1}{1}{1} \\
&=-2(m_{12}-m_{23}+1)\chi_-(M)(m_{13}-m_{12})(m_{22}-m_{33})(m_{22}-m_{33}-1)\\
&\hphantom{=}+\bkt{C_1(M)+\chi_-(M)}
	(m_{13} -m_{12})(m_{22}-m_{33})(2m_{22}-m_{13}-m_{23}-2)\\
&\hphantom{=}+(m_{22}-m_{33})\Big\{(2m_{11}-m_{13}-m_{23})\bar{E} (M)
+(m_{13}-m_{12})\big\{C_1(M)(m_{13}-m_{33}+1)\\
&\hphantom{=}+(m_{11}-m_{22})
\{\bar{C_1}(M)+1+\bar{D}(M)-\delta (M)\chi_+(M)\}\\
&\hphantom{=}-\chi_-(M)\{ \bar{D}(M)(m_{13}-m_{22}+1)
-(m_{22}-m_{33})(\bar{C_1} (M)+1) \}\big\}\Big\}\\
&=(m_{22}-m_{33})\Big\{(2m_{11}-m_{13}-m_{23})\bar{E} (M)\\
&\hphantom{=}+(m_{13}-m_{12})\big\{
-2(m_{12}-m_{23}+1)(m_{22}-m_{33}-1)\chi_-(M)\\
&\hphantom{=}+\{m_{11}-m_{22}+(\delta (M)+1)\chi_-(M)\}
(2m_{22}-m_{13}-m_{23}-2)\\
&\hphantom{=}+(m_{11}-m_{22}+\delta (M)\chi_-(M))(m_{13}-m_{33}+1)
+(m_{11}-m_{22})\\
&\hphantom{==}\times \{m_{23}-m_{22}+\delta (M)\chi_-(M)
+1+(-m_{22}+m_{33}+\delta (M))-\delta (M)(1-\chi_-(M))\}\\
&\hphantom{=}-\chi_-(M)\{ (-m_{22}+m_{33}+\delta (M))(m_{13}-m_{22}+1)
-(m_{22}-m_{33})(m_{12}-m_{11}+1) \}\big\}\Big\}\\
&=(m_{22}-m_{33})\Big\{(2m_{11}-m_{13}-m_{23})\bar{E} (M)\\
&\hphantom{=}+(2m_{11}-m_{13}-m_{23})(m_{13}-m_{12})\chi_-(M)
(-m_{22}+m_{33}+\delta (M)+1)\Big\}\\
&=(2m_{11}-m_{13}-m_{23})(m_{22}-m_{33})
\{\bar{E} (M)+(m_{13}-m_{12})\chi_-(M)(\bar{D}(M)+1)\}\\
&=(-2m_{11}+m_{13}+m_{23})\cp{1}{2}{M}{2}{2}{1}\nonumber .\\
\end{align*}
Hence we obtain the equation (\ref{eqn:relcl12_1}). 
Here we use the relations (\ref{eqn:G-fct001}), (\ref{eqn:G-fct002}) and 
(\ref{eqn:pf_clebsh(1,0,0)_004}).\\

We have 
\begin{align*}
&(m_{12}-m_{23})\chi_-(M)\cp{1}{2}{M}{2}{1}{1}\\
&=(m_{12}-m_{23})\chi_-(M)(m_{22}-m_{33})\big\{\bar{E}(M)+\bar{F}(M) \\
&\hphantom{=}+(m_{13}-m_{12})\{ \bar{C_1}(M)+1+\bar{D}(M)(1-\chi_+(M)) \} 
\big\}\\
&=(m_{22}-m_{33})\chi_-(M)\big\{(m_{12}-m_{23})(\bar{E}(M)-C_2(M)) \\
&\hphantom{=}+(m_{13}-m_{12})C_1(M)( \bar{C_1}(M)+1+\bar{D}(M) ) \big\}\\[2mm]
&(m_{11}-m_{22}-1)\cp{1}{2}{M}{2}{1}{2}\\
&=-(m_{11}-m_{22}-1)\big\{\bar{C_1} (M)\bar{E} (M)+C_2(M)\bkt{ 1-\bar{D} (M)
+\delta (M) \chi_+ (M)}\big\}\\
&=-(m_{11}-m_{22}-1)\big\{\bar{C_1}(M)\bar{E}(M)+C_2(M)\bkt{m_{22}-m_{33}+1
-\delta (M)\chi_-(M)}\big\}\\[2mm]
&\bkt{C_1(M)-1+\chi_-(M)}\cp{1}{2}{M}{1}{1}{1}\\
&=\bkt{C_1(M)-1+\chi_-(M)}\big\{\bar{E} (M)(m_{23} -m_{22})+
C_2 (M)(m_{22}-m_{33}+1)\\
&\hphantom{=}+(m_{13}-m_{12})\chi_-(M)\{\bar{D}(M)(m_{13}-m_{22}+1)
-(m_{22}-m_{33})(\bar{C_1}(M)+1) \}\big\}\\
&=(m_{11}-m_{22}-1)\big\{\bar{C_1}(M)\bar{E}(M)-\delta (M)\chi_-(M)\bar{E}(M)
+C_2 (M)(m_{22}-m_{33}+1)\big\}\\
&\hphantom{=}+(\delta (M)+1)\chi_-(M)
\big\{\bar{E} (M)(m_{23} -m_{22})+C_2 (M)(m_{22}-m_{33}+1)\big\}\\
&\hphantom{=}+(m_{13}-m_{12})C_1(M)\chi_-(M)\{\bar{D}(M)(m_{13}-m_{22}+1)
-(m_{22}-m_{33})(\bar{C_1}(M)+1) \}\\[2mm]
&(m_{33}-m_{22}-1)\cp{1}{2}{M}{1}{1}{2}\nonumber \\
&=(m_{33}-m_{22}-1)C_2(M)\chi_-(M)(m_{13}-m_{33}+2-\bar{C_1}(M)-\bar{D}(M))\\
&=-(m_{22}-m_{33}+1)\chi_-(M)\{(m_{12}-m_{11})\bar{E}(M)+(1-\bar{D}(M))C_2(M)\}
\end{align*}
Therefore
\begin{align*}
&(m_{11}-m_{22}-1)\cp{1}{2}{M}{2}{1}{2}
+\bkt{C_1(M)-1+\chi_-(M)}\cp{1}{2}{M}{1}{1}{1}\\
&\hphantom{=}+(m_{12}-m_{23})\chi_-(M)\cp{1}{2}{M}{2}{1}{1}
+(m_{33}-m_{22}-1)\cp{1}{2}{M}{1}{1}{2}\\
&=\chi_-(M)\big\{-(m_{11}-m_{23})\bar{E}(M)\bar{D} (M)
+(m_{11}-m_{33})\delta (M)C_2(M)\\
&\hphantom{=}+(m_{13}-m_{33}+1)(m_{13}-m_{12})C_1(M)\bar{D}(M)
-(m_{12}-m_{23})(m_{22}-m_{33})C_2(M)\\
&\hphantom{=}+(m_{22}-m_{33}+1)\bar{D}(M)C_2(M)\big\}\\
&=\chi_-(M)\big\{-(2m_{11}-m_{13}-m_{23})\bar{E}(M)\bar{D} (M)\\
&\hphantom{=}+(m_{11}-m_{13})(m_{12}-m_{23})
(m_{13}-m_{33}+1-m_{12}+m_{11})\bar{D}(M)\\
&\hphantom{=}+(m_{11}-m_{33})\delta (M)(m_{12}-m_{23})(m_{12}-m_{11})\\
&\hphantom{=}+(m_{13}-m_{33}+1)(m_{13}-m_{12})(m_{12}-m_{23})\bar{D}(M)\\
&\hphantom{=}-(m_{12}-m_{23})(m_{22}-m_{33})(m_{12}-m_{23})(m_{12}-m_{11})\\
&\hphantom{=}+(m_{22}-m_{33}+1)(m_{12}-m_{23})(m_{12}-m_{11})\bar{D}(M)\big\}\\
&=\chi_-(M)\big\{-(2m_{11}-m_{13}-m_{23})\bar{E}(M)\bar{D} (M)
+(m_{12}-m_{23})(m_{12}-m_{11})\{(m_{11}-m_{33})\delta (M)\\
&\hphantom{=}-(m_{22}-m_{33})(m_{12}-m_{23})
+(m_{22}-m_{11})\bar{D}(M)\}\big\}\\
&=\chi_-(M)\big\{-(2m_{11}-m_{13}-m_{23})\bar{E}(M)\bar{D} (M)
+(m_{12}-m_{23})(m_{12}-m_{11})\{(m_{11}-m_{33})\delta (M)\\
&\hphantom{=}-(m_{22}-m_{33})(m_{12}-m_{23})
+(m_{22}-m_{11})(-m_{22}+m_{33}-\delta (M))\}\big\}\\
&=\chi_-(M)-(2m_{11}-m_{13}-m_{23})\bar{E}(M)\bar{D} (M)\\
&=(-2m_{11}+m_{13}+m_{23})\cp{1}{2}{M}{2}{2}{2}.
\end{align*}
Hence we obtain the equation \ref{eqn:relcl12_2}. 
Here we use the relations 
(\ref{eqn:G-fct009}), 
(\ref{eqn:G-fct011}), 
(\ref{eqn:G-fct008}), 
(\ref{eqn:pf_clebsh(1,0,0)_FC2}), 
(\ref{eqn:pf_clebsh(1,0,0)_002}) and 
(\ref{eqn:pf_clebsh(1,0,0)_004}). \\

We have
\begin{align*}
&(m_{12}-m_{23}-1)\chi_-(M)\cp{1}{2}{M}{2}{1}{2}\\
&=(m_{12}-m_{23}-1)\chi_-(M)\big\{-\bar{C_1} (M)\bar{E} (M)-C_2(M)
	\bkt{ 1-\bar{D} (M)+\delta (M) \chi_+ (M)}\big\}\\
&=-(C_1(M)-1)\chi_-(M)(m_{13}-m_{33}+2-\bar{C_1}(M)-\bar{D}(M)),\\[2mm]
&\bkt{C_1(M)-2+\chi_-(M)}\cp{1}{2}{M}{1}{1}{2}\\
&=\bkt{C_1(M)-2+\chi_-(M)}
C_2(M)\chi_-(M)(m_{13}-m_{33}+2-\bar{C_1}(M)-\bar{D}(M))\\
&=(C_1(M)-1)C_2(M)\chi_-(M)(m_{13}-m_{33}+2-\bar{C_1}(M)-\bar{D}(M)).
\end{align*}
Therefore
\begin{align*}
&(m_{12}-m_{23}-1)\chi_-(M)\cp{1}{2}{M}{2}{1}{2}
+\bkt{C_1(M)-2+\chi_-(M)}\cp{1}{2}{M}{1}{1}{2}=0.
\end{align*}
Hence we obtain the equation (\ref{eqn:relcl12_3}). 
Here we use the relations (\ref{eqn:G-fct009}), (\ref{eqn:G-fct011}) 
and (\ref{eqn:pf_clebsh(1,0,0)_004}).\\

%\cp{1}{2}{M}{2}{2}{2}=&\bar{D} (M) \bar{E} (M) \chi_- (M),\\
%\cp{1}{2}{M}{2}{1}{0}=&-2(m_{13}-m_{12})(m_{22}-m_{33})(m_{22}-m_{33}-1),\\
%\cp{1}{2}{M}{2}{1}{1}=&(m_{22}-m_{33})\big\{\bar{E}(M)+\bar{F}(M) \\
%	&+(m_{13}-m_{12})\{ \bar{C_1}(M)+1+\bar{D}(M)(1-\chi_+(M)) \} \big\},\\
%\cp{1}{2}{M}{2}{1}{2}=&-\bar{C_1} (M)\bar{E} (M)-C_2(M)\bkt{ 1-\bar{D} (M)
%	+\delta (M) \chi_+ (M)} ,\\
%\cp{1}{2}{M}{1}{1}{0}=&(m_{13} -m_{12})(m_{22}-m_{33})
%	(2m_{22}-m_{13}-m_{23}-2),\\
%\cp{1}{2}{M}{1}{1}{1}=&\bar{E} (M)(m_{23} -m_{22})+
%	C_2 (M)(m_{22}-m_{33}+1)\\
%	&+(m_{13}-m_{12})\chi_-(M)\{ \bar{D}(M)(m_{13}-m_{22}+1) \\
%	& -(m_{22}-m_{33})(\bar{C_1} (M)+1) \},\\
%\cp{1}{2}{M}{1}{1}{2}=&C_2(M)\chi_-(M)(m_{13}-m_{33}+2-\bar{C_1}(M)
%	-\bar{D}(M)),\\

%%%%%%%%%%%%%%%%%%%%%%%%%%%%%%%%%%%%%%%%%%%%%%%%%%%%%%%%%%%%%%%%%%%%%%%%
%check
\noindent $\bullet $ the proof of the case of the $(i,j)=(1,3)$.

Since $(\cpr{1}{3}{2}{2},\cpr{1}{3}{1}{2},\cpr{1}{3}{1}{1})=(1,2,1)$, 
we have to confirm the following equations: 
\begin{align}
&(m_{11}-m_{22}+1)\cp{1}{3}{M}{2}{1}{0}
+(m_{33}-m_{22})\cp{1}{3}{M}{1}{1}{0}\label{eqn:relcl13_0}\\
&=(-2m_{11}+m_{13}+m_{33}-1)\cp{1}{3}{M}{2}{2}{0}\nonumber ,\\[2mm]
&(m_{12}-m_{23})\chid{-}{-1}(M)\cp{1}{3}{M}{2}{1}{0}\label{eqn:relcl13_1}
+(m_{11}-m_{22})\cp{1}{3}{M}{2}{1}{1}\\
&+C_1(M)\cp{1}{3}{M}{1}{1}{0}
+(m_{33}-m_{22}-1)\cp{1}{3}{M}{1}{1}{1}\nonumber \\
&=(-2m_{11}+m_{13}+m_{33}-1)\cp{1}{3}{M}{2}{2}{1}\nonumber ,\\[2mm]
&(m_{12}-m_{23}-1)\chid{-}{-1}(M)\cp{1}{3}{M}{2}{1}{1}\label{eqn:relcl13_2}
+(m_{11}-m_{22}-1)\cp{1}{3}{M}{2}{1}{2}\\
&+\bkt{C_1(M)-1}\cp{1}{3}{M}{1}{1}{1}=0,\nonumber \\[2mm]
&(m_{12}-m_{23}-2)\chid{-}{-1}(M)\cp{1}{3}{M}{2}{1}{2}=0.\label{eqn:relcl13_3}
\end{align}

We have
\begin{align*}
&(m_{11}-m_{22}+1)\cp{1}{3}{M}{2}{1}{0}
+(m_{33}-m_{22})\cp{1}{3}{M}{1}{1}{0}\\
&=-2(m_{11}-m_{22}+1)(m_{13} -m_{12})(m_{22}-m_{33})\\
&\hphantom{=}+(m_{33}-m_{22})(m_{13} -m_{12})(2m_{22}-m_{13} -m_{23} -1)\\
&=-(2m_{11}-m_{13} -m_{23} +1)(m_{13} -m_{12})(m_{22}-m_{33})\\
&=(-2m_{11}+m_{13}+m_{33}-1)\cp{1}{3}{M}{2}{2}{0}.
\end{align*}
Hence we obtain the equation (\ref{eqn:relcl13_0}).\\

We have
\begin{align*}
&(m_{12}-m_{23})\chid{-}{-1}(M)\cp{1}{3}{M}{2}{1}{0}
+(m_{11}-m_{22})\cp{1}{3}{M}{2}{1}{1}\\
&=-2(m_{12}-m_{23})\chid{-}{-1}(M)(m_{13} -m_{12})(m_{22}-m_{33})\\
&\hphantom{=}+(m_{11}-m_{22})\{\bar{E} (M)+\bar{F} (M)
-(m_{13}-m_{12})(m_{22}-m_{33})\chi_+(M) \}\\
&=-2(m_{12}-m_{23})(m_{13} -m_{12})(m_{22}-m_{33})(1-\chi_+(M))
+(m_{11}-m_{22})\big\{\bar{E} (M)-C_2(M)\\
&\hphantom{=}
-\chi_+(M)\{2(m_{13}-m_{12})(m_{22}-m_{33})
+(m_{13}-m_{33}+1)\delta (M) \}\big\}\\
&=-2(m_{13} -m_{12})(m_{22}-m_{33})(m_{12}-m_{23}-\delta (M)\chi_+(M))\\
&\hphantom{=}
+(m_{11}-m_{22})(\bar{E} (M)-C_2(M))
-(m_{11}-m_{22})(m_{13}-m_{33}+1)\delta (M) \chi_+(M)\\
&=-2(m_{13} -m_{12})(m_{22}-m_{33})C_1(M)\\
&\hphantom{=}
+(m_{11}-m_{22})(\bar{E} (M)-C_2(M))
-C_1(M)(m_{13}-m_{33}+1)\delta (M) \chi_+(M)\\
&=-C_1(M)\{(m_{13}-m_{33}+1)\delta (M) \chi_+(M)
+2(m_{13} -m_{12})(m_{22}-m_{33})\}\\
&\hphantom{=}
+(m_{11}-m_{22})(\bar{E} (M)-C_2(M)),\\[2mm]
&C_1(M)\cp{1}{3}{M}{1}{1}{0}
+(m_{33}-m_{22}-1)\cp{1}{3}{M}{1}{1}{1}\\
&=C_1(M)(m_{13} -m_{12})(2m_{22}-m_{13} -m_{33} -1)
+(m_{33}-m_{22}-1)(C_2(M)-\bar{E} (M)).
\end{align*}
Therefore
\begin{align*}
&(m_{12}-m_{23})\chid{-}{-1}(M)\cp{1}{3}{M}{2}{1}{0}
+(m_{11}-m_{22})\cp{1}{3}{M}{2}{1}{1}\\
&+C_1(M)\cp{1}{3}{M}{1}{1}{0}
+(m_{33}-m_{22}-1)\cp{1}{3}{M}{1}{1}{1}\\
&=C_1(M)\{-(m_{13}-m_{33}+1)\delta (M) \chi_+(M)
+(m_{13} -m_{12})(m_{33}-m_{13} -1)\}\\
&\hphantom{=}+(m_{33}-m_{11}-1)(C_2(M)-\bar{E} (M))\\
&=C_1(M)\{-(m_{13}-m_{33}+1)\delta (M) \chi_+(M)
+(m_{13} -m_{12})(m_{33}-m_{13} -1)\\
&\hphantom{=}+(m_{33}-m_{11}-1)\bar{C_1}(M)\}-(m_{33}-m_{11}-1)\bar{E} (M)\\
&=C_1(M)\{(m_{13}-m_{33}+1)(m_{12}-m_{11}-\delta (M)\chi_+(M))
+(-m_{13}+m_{11})(m_{13}-m_{33}+1)\\
&\hphantom{=}+(m_{33}-m_{11}-1)\bar{C_1}(M)\}-(m_{33}-m_{11}-1)\bar{E} (M)\\
&=(-m_{13}+m_{11})C_1(M)(m_{13}-m_{33}+1-\bar{C_1}(M))\}
-(m_{33}-m_{11}-1)\bar{E} (M)\\
&=(2m_{11}-m_{13}-m_{33}+1)\bar{E} (M)\\
&=(-2m_{11}+m_{13}+m_{33}-1)\cp{1}{3}{M}{2}{2}{1}.
\end{align*}
Hence we obtain the equation (\ref{eqn:relcl13_1}). 
Here we use the relations (\ref{eqn:G-fct001}), (\ref{eqn:G-fct002}), 
(\ref{eqn:G-fct008}) and (\ref{eqn:pf_clebsh(1,0,0)_003}). \\

We have
\begin{align*}
&(m_{12}-m_{23}-1)\chid{-}{-1}(M)\cp{1}{3}{M}{2}{1}{1}\\
&=(m_{12}-m_{23}-1)\chid{-}{-1}(M)\{\bar{E} (M)+\bar{F} (M)
-(m_{13}-m_{12})(m_{22}-m_{33})\chi_+(M) \}\\
&=(m_{12}-m_{23}-1)(1-\chi_+(M))(\bar{E} (M)-C_2(M)).
\end{align*}
Therefore
\begin{align*}
&(m_{12}-m_{23}-1)\chid{-}{-1}(M)\cp{1}{3}{M}{2}{1}{1}
+(m_{11}-m_{22}-1)\cp{1}{3}{M}{2}{1}{2}\\
&\hphantom{=}+\bkt{C_1(M)-1}\cp{1}{3}{M}{1}{1}{1}\\
&=(m_{12}-m_{23}-1)(1-\chi_+(M))(\bar{E} (M)-C_2(M))\\
&\hphantom{=}+(m_{11}-m_{22}-1)(\bar{E} (M)-C_2(M))\chi_+ (M) \\
&\hphantom{=}+\bkt{C_1(M)-1}(C_2(M) -\bar{E} (M))\\
&=(m_{12}-m_{23}-\delta (M)\chi_+(M)-1)(\bar{E} (M)-C_2(M))\\
&\hphantom{=}-\bkt{C_1(M)-1}(\bar{E} (M)-C_2(M))=0.
\end{align*}
Hence we obtain the equation (\ref{eqn:relcl13_2}). 
Here we use the relations 
(\ref{eqn:G-fct001}), 
(\ref{eqn:G-fct008}), 
(\ref{eqn:G-fct009}) and 
(\ref{eqn:pf_clebsh(1,0,0)_FC2}).\\

We have
\begin{align*}
&(m_{12}-m_{23}-2)\chid{-}{-1}(M)\cp{1}{3}{M}{2}{1}{2}\\
&=(m_{12}-m_{23}-2)(\bar{E} (M)-C_2(M))\chid{-}{-1}(M)\chi_+ (M)=0.
\end{align*}
Hence we obtain the equation (\ref{eqn:relcl13_3}). 
Here we use the relation (\ref{eqn:G-fct009}).\\

%%%%%%%%%%%%%%%%%%%%%%%%%%%%%%%%%%%%%%%%%%%%%%%%%%%%%%%%%%%%%%%%%%%%%%%%
%check
\noindent $\bullet $ the proof of the case of the $(i,j)=(2,3)$.

Since $(\cpr{2}{3}{2}{2},\cpr{2}{3}{1}{2},\cpr{2}{3}{1}{1})=(1,1,1)$, 
we have to confirm the following equations: 
\begin{align}
&(m_{11}-m_{22}+1)\cp{2}{3}{M}{2}{1}{0}
+(m_{33}-m_{22})\cp{2}{3}{M}{1}{1}{0}\label{eqn:relcl23_0}\\
&=(-2m_{11}+m_{23}+m_{33}-2)\cp{2}{3}{M}{2}{2}{0}\nonumber ,\\[2mm]
&(m_{12}-m_{23}+1)\chi_-(M)\cp{2}{3}{M}{2}{1}{0}\label{eqn:relcl23_1}
+(m_{11}-m_{22})\cp{2}{3}{M}{2}{1}{1}\\
&+\bkt{C_1(M)+\chi_-(M)}\cp{2}{3}{M}{1}{1}{0}
+(m_{33}-m_{22}-1)\cp{2}{3}{M}{1}{1}{1}\nonumber \\
&=(-2m_{11}+m_{23}+m_{33}-2)\cp{2}{3}{M}{2}{2}{1}\nonumber ,\\[2mm]
&(m_{12}-m_{23})\chi_-(M)\cp{2}{3}{M}{2}{1}{1}\label{eqn:relcl23_2}
+\bkt{C_1(M)-1+\chi_-(M)}\cp{2}{3}{M}{1}{1}{1}=0.\\ \nonumber
\end{align}

We have
\begin{align*}
&(m_{11}-m_{22}+1)\cp{2}{3}{M}{2}{1}{0}+(m_{33}-m_{22})\cp{2}{3}{M}{1}{1}{0}\\
&=-2(m_{11}-m_{22}+1)(m_{22}-m_{33})+(m_{33}-m_{22})(2m_{22}-m_{23}-m_{33})\\
&=(-2m_{11}+m_{23}+m_{33}-2)(m_{22}-m_{33})\\
&=(-2m_{11}+m_{23}+m_{33}-2)\cp{2}{3}{M}{2}{2}{0}.
\end{align*}
Hence we obtain the equation (\ref{eqn:relcl23_0}).\\

We have
\begin{align*}
&(m_{12}-m_{23}+1)\chi_-(M)\cp{2}{3}{M}{2}{1}{0}
+(m_{11}-m_{22})\cp{2}{3}{M}{2}{1}{1}\\
&+\bkt{C_1(M)+\chi_-(M)}\cp{2}{3}{M}{1}{1}{0}
+(m_{33}-m_{22}-1)\cp{2}{3}{M}{1}{1}{1}\nonumber \\
&=-2(m_{12}-m_{23}+1)\chi_-(M)(m_{22}-m_{33})\\
&\hphantom{=}+(m_{11}-m_{22})
\{\bar{C_1}(M) -(m_{22}-m_{33})+\delta (M)\chi_-(M)\}\\
&\hphantom{=}+\bkt{C_1(M)+\chi_-(M)}(2m_{22}-m_{23}-m_{33})
-(m_{33}-m_{22}-1)(\bar{C_1}(M)+\bar{D}(M))\chi_-(M)\nonumber \\
&=-2(m_{12}-m_{23}+1)(m_{22}-m_{33})\chi_-(M)\\
&\hphantom{=}+(m_{11}-m_{22})(-2m_{22}+m_{23}+m_{33}+2\delta (M)\chi_-(M))\\
&\hphantom{=}
+\bkt{m_{11}-m_{22}+(\delta (M)+1)\chi_-(M)}(2m_{22}-m_{23}-m_{33})\\
&\hphantom{=}-(m_{33}-m_{22}-1)(m_{12}-m_{11}+\bar{D}(M))\chi_-(M)\nonumber \\
&=(2m_{11}-m_{22}-m_{23}+1)(-m_{22}+m_{33}+\delta (M))\chi_-(M)
+(m_{22}-m_{33}+1)\bar{D}(M)\chi_-(M)\nonumber \\
&=(2m_{11}-m_{23}-m_{33}+2)\bar{D}(M)\chi_-(M)\nonumber \\
&=(-2m_{11}+m_{23}+m_{33}-2)\cp{2}{3}{M}{2}{2}{1}\nonumber .
\end{align*}
Hence we obtain the equation (\ref{eqn:relcl23_1}). 
Here we use the relations (\ref{eqn:G-fct001}), (\ref{eqn:G-fct002}) and 
(\ref{eqn:pf_clebsh(1,0,0)_002}).\\

We have
\begin{align*}
&(m_{12}-m_{23})\chi_-(M)\cp{2}{3}{M}{2}{1}{1}
+\bkt{C_1(M)-1+\chi_-(M)}\cp{2}{3}{M}{1}{1}{1}\\
&=(m_{12}-m_{23})\chi_-(M)
\{\bar{C_1}(M) -(m_{22}-m_{33})+\delta (M)\chi_-(M)\}\\
&\hphantom{=}-\bkt{C_1(M)-1+\chi_-(M)}(\bar{C_1}(M)+\bar{D}(M))\chi_-(M)\\
&=(m_{12}-m_{23})\chi_-(M)(\bar{C_1}(M) -m_{22}+m_{33}+\delta (M))
-(m_{12}-m_{23})(\bar{C_1}(M)+\bar{D}(M))\chi_-(M)\\
&=0.
\end{align*}
Hence we obtain the equation (\ref{eqn:relcl23_2}). 
Here we use the relations (\ref{eqn:pf_clebsh(1,0,0)_002}) and 
(\ref{eqn:G-fct011}).

\end{proof}

\begin{thm}
\label{th:main}
For $1\leq i\leq j\leq 3$, we have an following equation 
with the matrix representation 
$R(\Gamma_{\pm ij}^\lambda )\in 
M(d_{\lambda [\pm ij]}^\sigma ,d_\lambda^\sigma ,\mC )$ 
of $\Gamma_{\pm ij}^\lambda $ with respect to the induced basis 
$\{ \evec{\lambda }{M}\}_{M\in G_\sigma (\lambda )}$: 
\begin{equation}
\pmat{\lambda }{\pm }{ij}\eblock{\lambda }
=\eblock{\lambda [\pm ij]}\cdot R(\Gamma_{\pm ij}^\lambda ).
\label{eqn:statement_thm}
\end{equation}
Here $\pmat{\lambda }{\pm }{ij}\eblock{\lambda }$ is 
a matrix of the size $d_{\lambda [\pm ij]}\times d_\lambda^\sigma $, 
whose $l^\sigma (M)$-th column is 
$\pmat{\lambda }{\pm }{ij}\evec{\lambda }{M}$ 
for $M\in G_\sigma (\lambda )$. 

The explicit expression of the matrix $R(\Gamma_{\pm ij}^\lambda )$ is 
given as follows.\\
(i) $\gp_+$-side:\\
$R(\Gamma_{+ij}^\lambda )\quad (1\leq i\leq j\leq 3)$ 
is a matrix of size $d_{\lambda [+ij]}^\sigma \times d_\lambda^\sigma $, 
whose $l^\sigma (M)$-th column is given by 
\begin{align*}
&(\nu_1+\rho_1+\wt{M}_1+k_{ij}(M))\sum_{m=0}^{\cpr{i}{j}{2}{2}}
\cp{i}{j}{M\gpt{}{\me_{i}+\me_{j}}{}{0}{2}{2} [m]}{2}{2}{m} 
\mmu^\sigma_{\lambda [+ij]}\bkt{M\gpt{}{\me_{i}+\me_{j}}{}{0}{2}{2} [m]} \\
&+\sum_{m=0}^{\cpr{i}{j}{0}{2}}
h_{[ij;m]}(\nu_2,M)
\mmu^\sigma_{\lambda [+ij]}\bkt{M\gpt{}{\me_{i}+\me_{j}}{}{0}{2}{0} [m]} \\
&+(\nu_3+\rho_3+\wt{M}_3)\sum_{m=0}^{\cpr{i}{j}{0}{0}}
\cp{i}{j}{M\gpt{}{\me_{i}+\me_{j}}{}{0}{0}{0} [m]}{0}{0}{m} 
\mmu^\sigma_{\lambda [+ij]}\bkt{M\gpt{}{\me_{i}+\me_{j}}{}{0}{0}{0} [m]} 
\end{align*}
for $M\in G_\sigma (\lambda )$. 
Here 
\begin{align*}
h_{[ij;m]}(\nu_2,M)=&(\nu_2+\rho_2+\wt{M}_2)
\cp{i}{j}{M\gpt{}{\me_i+\me_j}{}{0}{2}{0} [m]}{2}{0}{m}\\
&+(m_{22}-m_{33}+1+\delta (M))\chid{+}{-1}(M)
\cp{i}{j}{M\gpt{}{\me_i+\me_j}{}{0}{2}{0} [m]}{1}{0}{m-1}\\
&+(m_{22}-m_{33}+1)
\cp{i}{j}{M\gpt{}{\me_i+\me_j}{}{0}{2}{0} [m]}{1}{0}{m}
\end{align*}
\begin{align*}
\biggl(\ \cp{i}{j}{M\gpt{}{\me_i+\me_j}{}{0}{2}{2}[m]}{l}{k}{m}=&0
\ \textit{ if }\ \cpr{i}{j}{k}{l}<m \textit{ or } \quad m<0.\ \biggl), 
\end{align*}
and 
$\mmu^\sigma_\lambda (M)$ is a column vector of degree 
$d^\sigma_\lambda $ which is defined by 
\[
\mmu^\sigma_\lambda (M)=\left\{\begin{array}{ll}
{}^t(\overbrace{0,\cdots ,0}^{l^\sigma (M)-1},1,
\overbrace{0,\cdots ,0}^{d^\sigma_\lambda-l^\sigma (M)})&
\text{ if }M\in G_\sigma (\lambda),\\
\mathbf{0} & \text{ otherwise }.
\end{array}\right.
\]

\noindent (ii) $\gp_-$-side:\\
$R(\Gamma_{-ij}^\lambda )\quad (1\leq i\leq j\leq 3)$ 
is a matrix of size $d_{\lambda [-ij]}^\sigma \times d_\lambda^\sigma $, 
whose $l^\sigma (M)$-th column is given by the form 
\begin{align*}
&(\nu_1+\rho_1-\wt{M}_1+k_{4-j\hs 4-i}(\hat{M}))\hspace{-1mm}
\sum_{m=0}^{\cpr{4-j\hs }{4-i}{2}{2}}\hspace{-1mm}
\cpd{4-j\hs }{4-i}{\bkt{\hat{M}\gpt{}{\me_{i}+\me_{j}}{}{0}{2}{2}[m]}}
{2}{2}{m} 
\mmu^\sigma_{\lambda [-ij]}\bkt{M\gpt{}{-\me_{i}-\me_{j}}{}{-2}{0}{-2} [m]} \\
&+\sum_{m=0}^{\cpr{4-j\hs }{4-i}{0}{2}}
h_{[4-j\hs 4-i;m]}(\nu_2,\hat{M})
\mmu^\sigma_{\lambda [-ij]}\bkt{M\gpt{}{-\me_{i}-\me_{j}}{}{-2}{0}{0} [m]} \\
&+(\nu_3+\rho_3-\wt{M}_3)\sum_{m=0}^{\cpr{4-j\hs }{4-i}{0}{0}}
\cpd{4-j\hs }{4-i}{\bkt{\hat{M}\gpt{}{\me_{i}+\me_{j}}{}{0}{0}{0}[m]}}
{0}{0}{m} 
\mmu^\sigma_{\lambda [-ij]}\bkt{M\gpt{}{-\me_{i}-\me_{j}}{}{0}{0}{0} [m]} 
\end{align*}
for $M\in G_\sigma (\lambda )$. 
\end{thm}
\begin{proof}
For $M,N\in G(\lambda )$, we define 
\[
\Delta (M,N)=
\left\{ \begin{array}{ll}
1 & \text{if } M=N ,\\
0 & \text{otherwise}. 
\end{array} \right. 
\]
Since 
\begin{equation}
\label{eqn:fmne}
\efct{M}{N}(1_6)=\ip{f(M)^*}{f(N)}=
\Delta (M,N) ,
\end{equation}
we see that the value at $1_6\in G$ of the vector $\evec{\lambda }{M}$ is 
the $l(M)$-th unit column vector ${}^t(0,\cdots ,0,1,0,\cdots ,0)$
of degree $d_\lambda $. Therefore, we note that it suffices to evaluate 
the both of the equation (\ref{eqn:statement_thm}) at $1_6\in G$. \\
First, we compute $\{X_{+pq}\efct{M}{N}\}(1_6)$ for $1\leq p\leq q\leq 3$. 
Since $\{\efct{M}{N}\}_{N\in G(\lambda )}$ is the monomial basis of 
$\langle \evec{\lambda }{M}\rangle $, we obtain
\begin{align*}
&\{\kappa (E_{pp}) \efct{M}{N}\}(1_6)=\wt{N}_p\efct{M}{N}(1_6)=
	\wt{M}_p\Delta (M,N)\ ,\quad   1\leq p\leq 3, \\[2mm]
&\{\kappa (E_{21}) \efct{M}{N}\}(1_6)\\
&\hphantom{=}=(n_{11}-n_{22})\efct{M}{N\cgpt{0}{0}{-1}}(1_6)
	+(n_{12}-n_{23})\chi_-(N)\efct{M}{N\cgpt{0}{0}{-1}[-1]}(1_6) \\
&\hphantom{=}=(m_{11}-m_{22}+1)\Delta (M\cgpt{0}{0}{1},N)
	+(m_{12}-m_{23}+1)\chid{-}{-1}(M)\Delta (M\cgpt{0}{0}{1}[1],N) \\[2mm]
&\{\kappa (E_{31}) \efct{M}{N}\}(1_6)\\
&\hphantom{=}=(n_{33}-n_{22})\efct{M}{N\cgpt{0}{-1}{-1}}(1_6)
	+C_1(N)\efct{M}{N\cgpt{0}{-1}{-1}[-1]}(1_6)\\
&\hphantom{=}=(m_{33}-m_{22}-1)\Delta (M\cgpt{0}{1}{1},N)
	+(C_1(M)+1)\Delta (M\cgpt{0}{1}{1}[1],N),\\[2mm]
&\{\kappa (E_{32}) \efct{M}{N}\}(1_6)\\
&\hphantom{=}=(n_{22}-n_{33})\efct{M}{N\cgpt{0}{-1}{0}}(1_6)
	+\{ n_{22}-n_{33}+\delta (N) \} \chi_+(N)
	\efct{M}{N\cgpt{0}{-1}{0}[-1]}(1_6)\\
&\hphantom{=}=(m_{22}-m_{33}+1)\Delta (M\cgpt{0}{1}{0},N)
	+\{ m_{22}-m_{33}+1+\delta (M)\} \chid{+}{-1}(N)
	\Delta (M\cgpt{0}{1}{0}[1],N).
\end{align*}
by Proposition \ref{prop:action_on_GZ-basis} and the equations 
(\ref{eqn:a_act_root_vec}), (\ref{eqn:a_act_root_vec}).
Moreover, we obtain 
\begin{align*}
 \{E_\alpha \efct{M}{N}\}(1_6)&=0, && \alpha \in \Sigma^+, \\
 \{H_p\efct{M}{N}\}(1_6)&=(\nu_p+\rho_p)\efct{M}{N}(1_6), && 1\leq p\leq 3
\end{align*}
by the definition of principal series representation. 
By above computations and Iwasawa decomposition in Lemma \ref{lem:Iwasawa}, 
we obtain
\begin{align*}
\{X_{+pp} \efct{M}{N}\}(1_6)=
	&(\nu_p+\rho_p+\wt{M}_p)\Delta (M,N)\ ,\quad   1\leq p\leq 3, \\[2mm]
\{X_{+12} \efct{M}{N}\}(1_6)=
	&(m_{11}-m_{22}+1)\Delta (M\cgpt{0}{0}{1},N)\\
	&+(m_{12}-m_{23}+1)\chid{-}{-1}(M)\Delta (M\cgpt{0}{0}{1}[1],N),\\[2mm]
\{X_{+13} \efct{M}{N}\}(1_6)=&(m_{33}-m_{22}-1)\Delta (M\cgpt{0}{1}{1},N)
	+(C_1(M)+1)\Delta (M\cgpt{0}{1}{1}[1],N),\\[2mm]
\{X_{+23} \efct{M}{N}\}(1_6)=
	&(m_{22}-m_{33}+1)\Delta (M\cgpt{0}{1}{0},N)\\
	&+\{ m_{22}-m_{33}+1+\delta (M)\} \chid{+}{-1}(M)
	\Delta (M\cgpt{0}{1}{0}[1],N).
\end{align*}

Let $\mmu_\lambda (N)$ be a column vector of degree 
$d_\lambda $ which is defined by 
\[
\mmu_\lambda (N)=\left\{\begin{array}{ll}
{}^t(\overbrace{0,\cdots ,0}^{l(N)-1},1,
\overbrace{0,\cdots ,0}^{d_\lambda-l(N)})&
\text{ if }M\in G(\lambda),\\
\mathbf{0} & \text{ otherwise }.
\end{array}\right.
\]
We denote by $X_{+pq}\evec{\lambda }{M}$ the vector of degree 
$d_\lambda$ 
whose $l(N)$-th component is $X_{+pq}\efct{M}{N}$. 
Then we obtain 
\begin{align*}
\{X_{+pp} \evec{\lambda}{M}\}(1_6)
=&(\nu_p+\rho_p+\wt{M}_p )\mmu_\lambda (M)
\quad \text{ for }\ 1\leq p\leq 3,\\
\{X_{+12} \evec{\lambda}{M}\}(1_6)
=&(m_{12}-m_{23}+1)\chid{-}{-1}(M)\mmu_\lambda \bkt{M\cgpt{0}{0}{1}[1]}\\
&+(m_{11}-m_{22}+1)\mmu_\lambda \bkt{M\cgpt{0}{0}{1}},\\
\{X_{+13} \evec{\lambda}{M}\}(1_6)
=&(C_1(M)+1)\mmu_\lambda \bkt{M\cgpt{1}{0}{1}[1]}
+(m_{33}-m_{22}-1)\mmu_\lambda \bkt{M\cgpt{1}{0}{1}},\\
\{X_{+23} \evec{\lambda}{M}\}(1_6)
=&\bkt{m_{22}-m_{33}+1+\delta (M)}\chid{+}{-1}(M)
\mmu_\lambda \bkt{M\cgpt{1}{0}{0}[1]}\\
&+(m_{22}-m_{33}+1)\mmu_\lambda \bkt{M\cgpt{1}{0}{0}}.
\end{align*}

Let us compute $\{\pmat{\lambda }{+}{ij}\evec{\lambda }{M}\}(1_6)$.

Since 
\[
L^{\lambda }_{+ij}\bkt{\cgpt{0}{-l}{-k} [-m]}\cdot \mmu_\lambda (N)
=\cp{i}{j}{N\gpt{}{\me_i+\me_j}{}{0}{l}{k}[m]}{l}{k}{m}
\mmu_{\lambda \epe{i}{j}}\bkt{N\gpt{}{\me_i+\me_j}{}{0}{l}{k}[m]}
\]
for $N\in G(\lambda )$, 
we obtain
\begin{align}
&\left\{ \bkt{ \sum_{m=0}^{\cpr{i}{j}{2}{2}}
L^{\lambda }_{+ij}\bkt{\cgpt{0}{-2}{-2} [-m]}
\otimes X_{+11} }\evec{\lambda }{M}\right\}  (1_6)
\label{eqn:+11_pf_Mth}\\
&= \bkt{ \sum_{m=0}^{\cpr{i}{j}{2}{2}}
L^{\lambda }_{+ij}\bkt{\cgpt{0}{-2}{-2} [-m]}}
\cdot \left\{X_{+11} \evec{\lambda }{M}\right\}  (1_6)\nonumber \\
&=(\nu_1+\rho_1+\wt{M}_1)\sum_{m=0}^{\cpr{i}{j}{2}{2}}
\cp{i}{j}{M\gpt{}{\me_i+\me_j}{}{0}{2}{2}[m]}{2}{2}{m}
\mmu_{\lambda \epe{i}{j}}\bkt{M\gpt{}{\me_i+\me_j}{}{0}{2}{2}[m]},\nonumber \\
\nonumber \\
&\left\{ \bkt{ \sum_{m=0}^{\cpr{i}{j}{1}{2}}
L^{\lambda }_{+ij}\bkt{\cgpt{0}{-2}{-1} [-m]}\otimes X_{+12} }
\evec{\lambda }{M}\right\}  (1_6)
\label{eqn:+12_pf_Mth}\\
&= \bkt{ \sum_{m=0}^{\cpr{i}{j}{1}{2}}
L^{\lambda }_{+ij}\bkt{\cgpt{0}{-2}{-1} [-m]}}
\cdot \left\{X_{+12} \evec{\lambda }{M}\right\}  (1_6)\nonumber \\
&=(m_{12}-m_{23}+1)\chid{-}{-1}(M)
\sum_{m=0}^{\cpr{i}{j}{1}{2}}
\cp{i}{j}{M\gpt{}{\me_i+\me_j}{}{0}{2}{2}[m+1]}{2}{1}{m}
\mmu_\lambda \bkt{M\gpt{}{\me_i+\me_j}{}{0}{2}{2}[m+1]}\nonumber \\
&\hphantom{==}+(m_{11}-m_{22}+1)\sum_{m=0}^{\cpr{i}{j}{1}{2}}
\cp{i}{j}{M\gpt{}{\me_i+\me_j}{}{0}{2}{2}[m]}{2}{1}{m}
\mmu_\lambda \bkt{M\gpt{}{\me_i+\me_j}{}{0}{2}{2}[m]}, \nonumber \\
&=\sum_{m=0}^{\cpr{i}{j}{1}{2}+1}
\biggl\{ (m_{12}-m_{23}+1)\chid{-}{-1}(M)
\cp{i}{j}{M\gpt{}{\me_i+\me_j}{}{0}{2}{2}[m]}{2}{1}{m-1}
\mmu_\lambda \bkt{M\gpt{}{\me_i+\me_j}{}{0}{2}{2}[m]}\nonumber \\
&\hphantom{==}+(m_{11}-m_{22}+1)
\cp{i}{j}{M\gpt{}{\me_i+\me_j}{}{0}{2}{2}[m]}{2}{1}{m}\biggl\}
\mmu_\lambda \bkt{M\gpt{}{\me_i+\me_j}{}{0}{2}{2}[m]}, \nonumber \\
\nonumber \\
&\left\{ \bkt{ \sum_{m=0}^{\cpr{i}{j}{1}{1}}
L^{\lambda }_{+ij}\bkt{\cgpt{0}{-1}{-1} [-m]}
\otimes X_{+13} }\evec{\lambda }{M}\right\}  (1_6)
\label{eqn:+13_pf_Mth}\\
&= \bkt{ \sum_{m=0}^{\cpr{i}{j}{1}{1}}
L^{\lambda }_{+ij}\bkt{\cgpt{0}{-1}{-1} [-m]}}
\cdot \left\{X_{+13} \evec{\lambda }{M}\right\}  (1_6)\nonumber \\
&=(C_1(M)+1)\sum_{m=0}^{\cpr{i}{j}{1}{1}}
\cp{i}{j}{M\gpt{}{\me_i+\me_j}{}{0}{2}{2}[m+1]}{1}{1}{m}
\mmu_\lambda \bkt{M\gpt{}{\me_i+\me_j}{}{0}{2}{2}[m+1]}\nonumber \\
&\hphantom{==}+(m_{33}-m_{22}-1)\sum_{m=0}^{\cpr{i}{j}{1}{1}}
\cp{i}{j}{M\gpt{}{\me_i+\me_j}{}{0}{2}{2}[m]}{2}{1}{m}
\mmu_\lambda \bkt{M\gpt{}{\me_i+\me_j}{}{0}{2}{2}[m]}\nonumber \\
&=\sum_{m=0}^{\cpr{i}{j}{1}{1}+1}\biggl\{(C_1(M)+1)
\cp{i}{j}{M\gpt{}{\me_i+\me_j}{}{0}{2}{2}[m]}{1}{1}{m-1}\nonumber \\
&\hphantom{==}+(m_{33}-m_{22}-1)
\cp{i}{j}{M\gpt{}{\me_i+\me_j}{}{0}{2}{2}[m]}{2}{1}{m}\biggl\}
\mmu_\lambda \bkt{M\gpt{}{\me_i+\me_j}{}{0}{2}{2}[m]}.\nonumber 
\end{align}
By the equations (\ref{eqn:+11_pf_Mth}), (\ref{eqn:+12_pf_Mth}), 
(\ref{eqn:+13_pf_Mth}) and Lemma \ref{lem:rel_clebsh}, we obtain 
\begin{align}
&\left\{ \left( \sum_{m=0}^{\cpr{i}{j}{2}{2}}
L^{\lambda }_{+ij}\bkt{\cgpt{0}{-2}{-2} [-m]}
\otimes X_{+11} 
+\sum_{m=0}^{\cpr{i}{j}{1}{2}}
L^{\lambda }_{+ij}\bkt{\cgpt{0}{-2}{-1} [-m]}
\otimes X_{+12} \right. \right.
\label{eqn:+11+12+13_pf_Mth}\\
&\hphantom{=} \left. \left. +\sum_{m=0}^{\cpr{i}{j}{1}{1}}
L^{\lambda }_{+ij}\bkt{\cgpt{0}{-1}{-1} [-m]}
\otimes X_{+13} \right)
\evec{\lambda }{M}\right\}  (1_6)\nonumber \\
&=(\nu_1+\rho_1+\wt{M}_1+k_{ij}(M))\sum_{m=0}^{\cpr{i}{j}{2}{2}}
\cp{i}{j}{M\gpt{}{\me_i+\me_j}{}{0}{2}{2}[m]}{2}{2}{m}
\mmu_{\lambda \epe{i}{j}}\bkt{M\gpt{}{\me_i+\me_j}{}{0}{2}{2}[m]}.\nonumber
\end{align}
Similarly,
\begin{align}
&\left\{ \bkt{ \sum_{m=0}^{4}
L^{\lambda }_{+11}\bkt{\cgpt{0}{-2}{0} [-m]}
\otimes X_{+22} }\evec{\lambda }{M}\right\}  (1_6)
\label{eqn:+22_pf_Mth}\\
&= \bkt{ \sum_{m=0}^{\cpr{i}{j}{0}{2}}
L^{\lambda }_{+ij}\bkt{\cgpt{0}{-2}{0} [-m]}}
\cdot \left\{X_{+22} \evec{\lambda }{M}\right\}  (1_6)\nonumber \\
&=(\nu_2+\rho_2+\wt{M}_2)\sum_{m=0}^{\cpr{i}{j}{0}{2}}
\cp{i}{j}{M\gpt{}{\me_i+\me_j}{}{0}{2}{0}[m]}{2}{0}{m}
\mmu_{\lambda \epe{i}{j}}\bkt{M\gpt{}{\me_i+\me_j}{}{0}{2}{0}[m]},\nonumber \\
\nonumber \\
&\left\{ \bkt{ \sum_{m=0}^{\cpr{i}{j}{0}{1}}
L^{\lambda }_{+ij}\bkt{\cgpt{0}{-1}{0} [-m]}
\otimes X_{+23} }\evec{\lambda }{M}\right\}  (1_6)
\label{eqn:+23_pf_Mth}\\
&= \bkt{ \sum_{m=0}^{\cpr{i}{j}{0}{1}}
L^{\lambda }_{+ij}\bkt{\cgpt{0}{-1}{0} [-m]}}
\cdot \left\{X_{+23} \evec{\lambda }{M}\right\}  (1_6)\nonumber \\
&=(m_{22}-m_{33}+1+\delta (M))\chid{+}{-1}(M)
\sum_{m=0}^{\cpr{i}{j}{0}{1}}
\cp{i}{j}{M\gpt{}{\me_i+\me_j}{}{0}{2}{0}[m+1]}{1}{0}{m}\nonumber \\
&\hphantom{==}\times 
\mmu_{\lambda \epe{i}{j}}\bkt{M\gpt{}{\me_i+\me_j}{}{0}{2}{0}[m+1]}\nonumber \\
&\hphantom{==}+(m_{22}-m_{33}+1)
\sum_{m=0}^{\cpr{i}{j}{0}{1}}
\cp{i}{j}{M\gpt{}{\me_i+\me_j}{}{0}{2}{0}[m]}{1}{0}{m}
\mmu_{\lambda \epe{i}{j}}\bkt{M\gpt{}{\me_i+\me_j}{}{0}{2}{0}[m]}\nonumber \\
&=\sum_{m=0}^{\cpr{i}{j}{0}{1}+1}
\biggl\{(m_{22}-m_{33}+1+\delta (M))\chid{+}{-1}(M)
\cp{i}{j}{M\gpt{}{\me_i+\me_j}{}{0}{2}{0}[m]}{1}{0}{m-1}\nonumber \\
&\hphantom{==}+(m_{22}-m_{33}+1)
\cp{i}{j}{M\gpt{}{\me_i+\me_j}{}{0}{2}{0}[m]}{1}{0}{m}\biggl\}
\mmu_{\lambda \epe{i}{j}}\bkt{M\gpt{}{\me_i+\me_j}{}{0}{2}{0}[m]}.\nonumber 
\end{align}
By direct computation from the equations (\ref{eqn:+22_pf_Mth}) 
and (\ref{eqn:+23_pf_Mth}), we obtain 
\begin{align}
&\left\{ \bkt{ \sum_{m=0}^{\cpr{i}{j}{0}{2}}
L^{\lambda }_{+ij}\bkt{\cgpt{0}{-2}{0} [-m]}
\otimes X_{+22} 
+\sum_{m=0}^{\cpr{i}{j}{0}{1}}L^{\lambda }_{+ij}\bkt{\cgpt{0}{-1}{0} [-m]}
\otimes X_{+23} }\evec{\lambda }{M}\right\}  (1_6)
\label{eqn:+22+23_pf_Mth}\\
&=\sum_{m=0}^{\cpr{i}{j}{0}{2}}h_{[ij;m]}(\nu_2,M)
\mmu_{\lambda \epe{i}{j}}\bkt{M\gpt{}{\me_i+\me_j}{}{0}{2}{0}[m]}.\nonumber 
\end{align}
Similarly, 
\begin{align}
&\left\{ \bkt{ \sum_{m=0}^{\cpr{i}{j}{0}{0}}
L^{\lambda }_{+ij}\bkt{\cgpt{0}{0}{0} [-m]}
\otimes X_{+33} }\evec{\lambda }{M}\right\}  (1_6)
\label{eqn:+33_pf_Mth}\\
&= \bkt{ \sum_{m=0}^{\cpr{i}{j}{0}{0}}
L^{\lambda }_{+ij}\bkt{\cgpt{0}{0}{0} [-m]}}
\cdot \left\{X_{+33} \evec{\lambda }{M}\right\}  (1_6)\nonumber \\
&=(\nu_3+\rho_3+\wt{M}_3)
\sum_{m=0}^{\cpr{i}{j}{0}{0}}
\cp{i}{j}{M\gpt{}{\me_i+\me_j}{}{0}{0}{0}[m]}{0}{0}{m}
\mmu_{\lambda \epe{i}{j}}\bkt{M\gpt{}{\me_i+\me_j}{}{0}{0}{0}[m]}.\nonumber
\end{align}
Summing up the equations (\ref{eqn:+11+12+13_pf_Mth}), 
(\ref{eqn:+22+23_pf_Mth}) and (\ref{eqn:+33_pf_Mth}), 
we obtain 
\begin{align}
&\{ \pmat{\lambda }{1}{1} \evec{\lambda}{M} \} (1_6)\\
&=(\nu_1+\rho_1+\wt{M}_1+k_{ij}(M))\sum_{m=0}^{\cpr{i}{j}{2}{2}}
\cp{i}{j}{M\gpt{}{\me_i+\me_j}{}{0}{2}{2}[m]}{2}{2}{m}
\mmu_{\lambda \epe{i}{j}}\bkt{M\gpt{}{\me_i+\me_j}{}{0}{2}{2}[m]}\nonumber \\
&\hphantom{=}+\sum_{m=0}^{\cpr{i}{j}{0}{2}}h_{[ij;m]}(\nu_2,M)
\mmu_{\lambda \epe{i}{j}}\bkt{M\gpt{}{\me_i+\me_j}{}{0}{2}{0}[m]}\nonumber \\
&\hphantom{=}+(\nu_3+\rho_3+\wt{M}_3)
\sum_{m=0}^{\cpr{i}{j}{0}{0}}
\cp{i}{j}{M\gpt{}{\me_i+\me_j}{}{0}{0}{0}[m]}{0}{0}{m}
\mmu_{\lambda \epe{i}{j}}\bkt{M\gpt{}{\me_i+\me_j}{}{0}{0}{0}[m]}.\nonumber
\end{align}
By the remark at the beginning of this proof, we obtain the 
assertion of the case of (i). The case of (ii) is treated 
similarly.
\end{proof}

\section{Examples of contiguous relations and their applications}
\label{sec:examples}

Here are some examples of the contiguous relations 
at the peripheral $K$-types. 
Moreover, as their applications, 
we determine the holonomic systems, 
whose solutions are Whittaker functions. 

\subsection{Whittaker functions}

For a unitary character $\xi $ of $N_{\min}$, we denote 
the derivative of $\xi$ by the same letter. 
Since
\[
\gn /[\gn ,\gn ]\simeq \g_{e_1-e_2}\oplus \g_{e_2-e_3}\oplus \g_{2e_3},
\]
$\xi $ is specified by three real numbers $c_{12},c_{23}$ and $c_3$ such that 
\[
\xi (E_{e_1-e_2})=2\pi \sqrt{-1}c_{12},\ 
\xi (E_{e_2-e_3})=2\pi \sqrt{-1}c_{23}\text{ and }
\xi (E_{2e_3})=2\pi \sqrt{-1}c_{3}.
\]
When $c_{12}c_{23}c_{3}\neq 0$, a unitary character $\xi$ of $N_{\min}$ is 
called \textit{non-degenerate}.

For a finite dimensional representation $(\tau ,V)$ of $K$ and a 
non-degenerate unitary character $\xi $ of $N_{\min }$, we consider the space 
$C^\infty_{\xi ,\tau }(N_{\min }\backslash G/K)$ of smooth functions 
$\varphi \colon G\to V_\tau $ with the property
\[
\varphi (ngk)=\xi (n)\tau (k)^{-1}\varphi (g),\ (n,g,k)\in 
N_{\min}\times G\times K.
\]
Here we remark that any functions 
$\varphi \in C^\infty_{\xi ,\tau }(N_{\min }\backslash G/K)$ is 
determined by its restriction $\varphi |_{A_{\min }}$ to $A_{\min }$ 
from the Iwasawa decomposition 
$G=N_{\min }A_{\min }K$ of $G$. $\varphi |_{A_{\min }}$ is called 
the $A_{\min }$-radial part of $\varphi $. 
Also let $C^\infty \Ind_{N_{\min}}^G(\xi )$ be the $C^\infty $-induced 
representation from $\xi$ with the representation space 
\[
C^\infty_\xi (N_{\min}\backslash G)
=\{ \varphi \in C^\infty(G) \mid 
\varphi (ng)=\xi (n)\varphi (g),\ (n,g)\in N_{\min}
\times G\},
\]
on which $G$ acts by right translation. Then the space 
$C^\infty_{\xi ,\tau }(N_{\min }\backslash G/K)$ is isomorphic to 
$\Hom_K(V^*,C^\infty_\xi (N_{\min}\backslash G))$ 
via the correspondence between 
$\iota \in \Hom_K(V^*,C^\infty_\xi (N_{\min}\backslash G))$ 
and $F^{[\iota ]}\in C^\infty_{\xi ,\tau }(N_{\min }\backslash G/K)$ 
given by the relation $\iota (v^*)(g)=\ip{v^*}{F^{[\iota ]}(g)}$ for 
$v^*\in V^*$ and $g\in G$ with canonical pairing $\ip{}{} $ on 
$V^*\times V$.

Let $(\pi ,H_\pi )$ be an irreducible admissible representation of $G$, and 
take a $K$-type $(\tau^*,V^*)$ of $\pi $ with an injective 
$K$-homomorphism $\eta \colon V^*\to H_\pi $. Then, for each element $T$ in the 
intertwining space $\Isp{\xi ,\pi }=
\Hom_{(\g_\mC ,K)}(H_{\pi ,K},C^\infty_\xi (N_{\min}\backslash G))$, 
the relation 
$T(\eta (v^*))(g)=\ip{v^*}{\Phi (T,\eta )(g)}\ (v^*\in V^*,\ g\in G)$ 
determines an element 
$\Phi (T,\eta )\in C^\infty_{\xi ,\tau }(N_{\min }\backslash G/K)$. Here 
$H_{\pi ,K}$ is a subspace of $H_\pi $, consisting of all $K$-finite vectors.
Now we put 
\[
\Wh (\pi ,\xi ,\tau )=\bigcup_{\eta \in \Hom_K(V^*, H_{\pi ,K})} 
\{ \Phi (T,\eta ) \in C^\infty_{\xi ,\tau }(N_{\min }\backslash G/K)\mid 
T\in \Isp{\xi ,\pi }\}
\]
and call $\Wh (\pi ,\xi ,\tau )$ 
\textit{the space of Whittaker functions} for $(\pi ,\xi ,\tau )$. 
We consider the Whittaker functions for the irreducible principal 
series representation $\pi =\pi_{(\sigma ,\nu )}$. 

\subsection{The $\pm$-chilarity matrices}

We define the $\pm $-chilarity matrices as follows.

\begin{defn}
\textit 
The $\pm$-chilarity matrices $m_i(C_\pm )$ for $1\leq i\leq 3$ are defined by 
\begin{align*}
m_1(C_\pm )=&\left( \begin{array}{ccc}
X_{\pm 11} & X_{\pm 12} & X_{\pm 13} \\
X_{\pm 12} & X_{\pm 22} & X_{\pm 23} \\
X_{\pm 13} & X_{\pm 23} & X_{\pm 33} 
\end{array} \right) ,&
m_2(C_\pm )=&\left( \begin{array}{ccc}
M_{\pm 11} & -M_{\pm 12} & M_{\pm 13} \\
-M_{\pm 12} & M_{\pm 22} & -M_{\pm 23} \\
M_{\pm 13} & -M_{\pm 23} & M_{\pm 33} 
\end{array} \right) ,
\end{align*}
and $m_3(C_\pm )=\det (m_1(C_\pm ))$. Here $M_{\pm ij}$ is 
$(i,j)$-minor of the matrix $m_1(C_\pm )$ for each 
$1\leq i\leq j\leq 3$, that is,
\begin{align*}
M_{\pm 11}=&\left| \begin{array}{cc}
 X_{\pm 22} & X_{\pm 23} \\
 X_{\pm 23} & X_{\pm 33} 
\end{array} \right| ,&
M_{\pm 22}=&\left| \begin{array}{cc}
 X_{\pm 11} & X_{\pm 13} \\
 X_{\pm 13} & X_{\pm 33} 
\end{array} \right| ,&
M_{\pm 33}=&\left| \begin{array}{cc}
 X_{\pm 11} & X_{\pm 12} \\
 X_{\pm 12} & X_{\pm 22} 
\end{array} \right| ,\\
M_{\pm 12}=&\left| \begin{array}{cc}
 X_{\pm 12} & X_{\pm 23} \\
 X_{\pm 13} & X_{\pm 33} 
\end{array} \right| ,&
M_{\pm 13}=&\left| \begin{array}{cc}
 X_{\pm 12} & X_{\pm 22} \\
 X_{\pm 13} & X_{\pm 23} 
\end{array} \right| ,&
M_{\pm 23}=&\left| \begin{array}{cc}
 X_{\pm 11} & X_{\pm 12} \\
 X_{\pm 13} & X_{\pm 23} 
\end{array} \right| .
\end{align*}
\end{defn}
Then we can find the following lemma immediately from 
the definition of the chilarity matrices. 
\begin{lem}
\label{lem:chirality_def}
\textit 
For each $1\leq i\leq 3$, the element 
$C_{2i}=\Tr (m_i(C_+)m_i(C_-))$ in $U(\g_\mC )$ is 
invariant under the adjoint action of $K$, that is, 
\[
C_{2i}\in U(\g_\mC )^K=\{ X\in U(\g_\mC )\mid \Ad (k)X=X,\ k\in K \} .
\]
\end{lem}
\begin{rem}
\textit{
In the case of $Sp(n,\mR )$, we can define $C_{2i}$ for each $1\leq i\leq n$ 
belonging to $U(\g_\mC )^K$ similarly. The operator $C_{2i}$ is essentially 
same as the so-called Maass shift operator in the classical literature 
\cite{MR0344198}. Also, the chilarity matrices are 
used to construct the Capelli elements for a symmetric pair in 
\cite{Capelli_identities}, recently.}
\end{rem}
Now we consider a system of differential equations which are satisfied by the 
$A_{\min }$-radial part of each element in $\Wh (\pi ,\xi ,\tau )$ 
when $\tau^*$ is a multiplicity one $K$-type of $\pi_{(\sigma ,\lambda )}$. 
Let $\lambda $ be the highest weight of $\tau^*$. 
The elements $C_2,\ C_4$ and $C_6$ in $U(\g_\mC )^K$ defined in 
Lemma \ref{lem:chirality_def} are acting on the space 
$C^\infty_\xi (N_{\min}\backslash G)$ as differential 
operators and acting on the space $H_{(\sigma ,\nu )}$ 
as scalar operators. 
Therefore, each element $\Phi $ in $\Wh (\pi ,\xi ,\tau )$ satisfies 
the following system of differential equations 
\begin{align*}
C_{2i}\Phi =&\chi_{2i,\sigma ,\nu ,\lambda }\Phi \quad (1\leq i\leq 3),
\end{align*}
where $\chi_{2i,\sigma ,\nu ,\lambda }$ is the scalar value 
for the action of the operator $C_{2i}$. 
In the later subsections, we compute those values for peripheral $K$-types.
%The minimal $K$-types of $\pi_{(\sigma ,\nu )}$ given as follows.
%\begin{gather*}
%\begin{array}{|c|c|}
%\hline
%(\sigma_1,\sigma_2,\sigma_3)&\text{highest weights of 
%minimal $K$-types of $\pi_{(\sigma ,\nu )}$}\\ \hline
%(0,0,0)&(0,0,0)\\ \hline
%(1,0,0),\ (0,1,0),\ (0,0,1)
%&(1,0,0)\text{ and }(0,0,-1)\\ \hline
%(1,1,0),\ (1,0,1),\ (0,1,1)
%&(1,1,0)\text{ and }(0,-1,-1)\\ \hline
%(1,1,1)&(1,1,1)\text{ and }(-1,-1,-1)
%\\ \hline
%\end{array}
%\end{gather*}

\subsection{The chilarity operators in the cases of 
$(\sigma_1,\sigma_2,\sigma_3)=(0,0,0),\ (1,1,1)$}
In this subsection, we consider the cases of 
$(\sigma_1,\sigma_2,\sigma_3)=(0,0,0),\ (1,1,1)$. 
We set $M_l=\gpt{l}{l}{l}{l}{l}{l}\in G((l,l,l))$ for $l\in \mZ$. 
Let $\varepsilon_\sigma $ be the integer defined by $\varepsilon $
\[
\varepsilon_\sigma =
\left\{\begin{array}{ll}
0 & \text{ if } (\sigma_1,\sigma_2,\sigma_3)=
(0,0,0),\\
1 & \text{ if } (\sigma_1,\sigma_2,\sigma_3)=
(1,1,1).
\end{array}\right.
\]

Since 
$(\sigma_1,\sigma_2,\sigma_3)=(0,0,0),\ (1,1,1)$, 
there are multiplicity one $K$-types 
$\tau_{(l,l,l)}\ (l\equiv \varepsilon_\sigma \bmod 2)$ 
of $\pi_{(\sigma ,\nu )}$. 
Let us compute the values of 
$\chi_{2i,\sigma ,\nu ,(l,l,l)}\ (1\leq i\leq 3)$. 

\begin{lem}\label{lem:ex1_1}
\textit 
(i) Let $\sigma $ be a character of $M_{\min}$ such that 
$(\sigma_1,\sigma_2,\sigma_3)=(0,0,0),\ (1,1,1)$. 
For $l\equiv \varepsilon_\sigma \bmod 2$, the following equations hold:
\begin{align*}
\pmat{(l-2,l-2,l-2)}{+}{11}\eblock{(l-2,l-2,l-2)}
=&\eblock{(l,l-2,l-2)}\cdot R(\Gamma_{+11}^{(l-2,l-2,l-2)} ),\\
\pmat{(l,l-2,l-2)}{+}{22}\eblock{(l,l-2,l-2)}
=&\eblock{(l,l,l-2)}\cdot R(\Gamma_{+22}^{(l,l-2,l-2)} ),\\
\pmat{(l,l,l-2)}{+}{33}\eblock{(l,l,l-2)}
=&\eblock{(l,l,l)}\cdot R(\Gamma_{+33}^{(l,l,l-2)} ),\\
\pmat{(l,l,l)}{-}{33}\eblock{(l,l,l)}
=&\eblock{(l,l,l-2)}\cdot R(\Gamma_{-33}^{(l,l,l)} ),\\
\pmat{(l,l,l-2)}{-}{22}\eblock{(l,l,l-2)}
=&\eblock{(l,l-2,l-2)}\cdot R(\Gamma_{-22}^{(l,l,l-2)} ),\\
\pmat{(l,l-2,l-2)}{-}{11}\eblock{(l,l-2,l-2)}
=&\eblock{(l-2,l-2,l-2)}\cdot R(\Gamma_{-11}^{(l,l-2,l-2)} ).
\end{align*}
Here
\begin{align*}
&\pmat{(l-2,l-2,l-2)}{+}{11}=12\left(\begin{array}{c}
X_{+11}\\
X_{+12}\\
X_{+13}\\
X_{+22}\\
X_{+23}\\
X_{+33}
\end{array}\right),\quad 
R(\Gamma_{+11}^{(l-2,l-2,l-2)})
=12\left(\begin{array}{c}
\nu_1+l+1\\ \nu_2+l\\ \nu_3+l-1
\end{array}\right),\\[2mm]
&\pmat{(l,l-2,l-2)}{+}{22}=2\left(\begin{array}{cccccc}
X_{+22}&-2X_{+12}&0&X_{+11}&0&0\\
X_{+23}&-X_{+13}&-X_{+12}&0&X_{+11}&0\\
X_{+33}&0&-2X_{+13}&0&0&X_{+11}\\
0&X_{+23}&-X_{+22}&-X_{+13}&X_{+12}&0\\
0&X_{+33}&-X_{+23}&0&-X_{+13}&X_{+12}\\
0&0&0&X_{+33}&-2X_{+23}&X_{+22}
\end{array}\right),\\
&R(\Gamma_{+22}^{(l,l-2,l-2)})
=2\left(\begin{array}{ccc}
\nu_2+l&\nu_1+l-1&0\\
\nu_3+l-1&0&\nu_1+l-1\\
0&\nu_3+l-1&\nu_2+l-2
\end{array}\right),\\[2mm]
&\pmat{(l,l,l-2)}{+}{33}
=\left(\begin{array}{cccccc}
X_{+33}&-2X_{+23}&X_{+22}&2X_{+13}&-2X_{+12}&X_{+11}
\end{array}\right),\\
&R(\Gamma_{+33}^{(l,l,l-2)})
=\left(\begin{array}{ccc}
\nu_3+l-1&\nu_2+l-2&\nu_1+l-3
\end{array}\right),\\
&\pmat{(l,l,l)}{-}{33}=12\left(\begin{array}{c}
X_{-33}\\
-X_{-23}\\
X_{-22}\\
X_{-13}\\
-X_{-12}\\
X_{-11}
\end{array}\right),\quad 
R(\Gamma_{-33}^{(l,l,l)})
=12\left(\begin{array}{c}
\nu_3-l+1\\
\nu_2-l+2\\
\nu_1-l+3
\end{array}\right),\\
&\pmat{(l,l,l-2)}{-}{22}
=2\left(\begin{array}{cccccc}
X_{-22}&2X_{-23}&0&X_{-33}&0&0\\
-X_{-12}&-X_{-13}&0&X_{-23}&X_{-33}&0\\
0&-X_{-12}&-X_{-13}&-X_{-22}&X_{-33}&0\\
X_{-11}&0&0&-2X_{-13}&0&X_{-33}\\
0&X_{-11}&0&X_{-12}&-X_{-13}&X_{-23}\\
0&0&X_{-11}&0&2X_{-12}&X_{-22}
\end{array}\right),\\
&R(\Gamma_{-22}^{(l,l,l-2)})
=2\left(\begin{array}{ccc}
\nu_2-l&\nu_3-l+1&0\\
\nu_1-l+1&0&\nu_3-l+1\\
0&\nu_1-l+1&\nu_2-l+2
\end{array}\right),\\
&\pmat{(l,l-2,l-2)}{-}{11}=\left(\begin{array}{cccccc}
X_{-11}&2X_{-12}&2X_{-13}&X_{-22}&2X_{-23}&X_{-33}
\end{array}\right),\\
&R(\Gamma_{-11}^{(l,l-2,l-2)})
=\left(\begin{array}{ccc}
\nu_1-l-1&\nu_2-l&\nu_3-l+1
\end{array}\right).
\end{align*}
(ii) The elements $C_{2i}\ (i=1,2,3)$ are represented 
by the $\gp_{\pm }$-matrices 
$\pmat{\lambda }{\pm }{kl}$ as follows: 
\begin{align*}
&C_2=\frac{1}{12}\pmat{(l,l,l-2)}{+}{33}\cdot \pmat{(l,l,l)}{-}{33},\quad
C_4=\frac{1}{192}
\pmat{(l,l,l-2)}{+}{33}\cdot \pmat{(l,l-2,l-2)}{+}{22}\cdot 
\pmat{(l,l,l-2)}{-}{22}\cdot \pmat{(l,l,l)}{-}{33},\\
&C_6=\frac{1}{20736}
\pmat{(l,l,l-2)}{+}{33}\cdot \pmat{(l,l-2,l-2)}{+}{22}
\cdot \pmat{(l-2,l-2,l-2)}{+}{11}\cdot 
\pmat{(l,l-2,l-2)}{-}{11}\cdot \pmat{(l,l,l-2)}{-}{22}
\cdot \pmat{(l,l,l)}{-}{33}.
\end{align*}
\end{lem}
\begin{proof}
From Theorem \ref{th:main} in the case of 
$(\sigma_1,\sigma_2,\sigma_3)=(0,0,0),\ (1,1,1),$ 
we obtain the assertion by direct computation.
\end{proof}

From above lemma, we obtain differential equations which 
Whittaker functions satisfy. 
%check
\begin{prop}\label{prop:ex1}
\textit 
Let $\sigma $ be a character of $M_{\min }$ such that 
$(\sigma_1,\sigma_2,\sigma_3)=(0,0,0)$ or $(1,1,1)$, and 
$T$ be an element of the space $\Isp{\xi ,\pi_{(\sigma ,\nu )} }$. 
For $l\equiv \varepsilon_\sigma \mod 2$, we define 
a function $\phi_{T,l}\in C^\infty_\xi (N_{\min}\backslash G)$ by 
the equation $\Phi (T,\evec{(l,l,l)}{M_l})=\phi_{T,l}\otimes f(M_l)^*$. 
Then $\phi_{T,l}$ satisfies following differential equations:
\[
C_{2i}\phi_{T,l}=\chi_{2i,\sigma ,\nu ,(l,l,l)}\phi_{T,l}\quad 
(i=1,2,3).
\]
Here 
\begin{align}
\chi_{2,\sigma ,\nu ,(l,l,l)}
=&\{\nu_1^2-(l-3)^2\}+\{\nu_2^2-(l-2)^2\}+\{\nu_3^2-(l-1)^2\},
\label{eqn:ex1_1}\\
\chi_{4,\sigma ,\nu ,(l,l,l)}
=&\{\nu_1^2-(l-2)^2\}\{\nu_2^2-(l-2)^2\}
+\{\nu_1^2-(l-2)^2\}\{\nu_3^2-(l-1)^2\}\label{eqn:ex1_2} \\
&+\{\nu_2^2-(l-1)^2\}\{\nu_3^2-(l-1)^2\},\nonumber \\
\chi_{6,\sigma ,\nu ,(l,l,l)}
=&\{\nu_1^2-(l-1)^2\}\{\nu_2^2-(l-1)^2\}\{\nu_3^2-(l-1)^2\}.\label{eqn:ex1_3}
\end{align}
\end{prop}
\begin{proof}
Let $\chi_{2i,\sigma ,\nu ,(l,l,l)} (i=1,2,3)$ be the complex numbers 
defined by the equations 
(\ref{eqn:ex1_1}), (\ref{eqn:ex1_2}) and (\ref{eqn:ex1_3}). 
From Lemma \ref{lem:ex1_1} and 
$\evec{(l,l,l) }{M_l}=\efct{M_l}{M_l}$, we see that 
\begin{align*}
C_2\efct{M_l}{M_l}=&\frac{1}{12}\efct{M_l}{M_l}
R(\Gamma_{+33}^{(l,l,l-2)})\cdot R(\Gamma_{-33}^{(l,l,l)})
=\chi_{2,\sigma ,\nu ,(l,l,l)}\efct{M_l}{M_l}.
\end{align*}
Similarly, we obtain 
$C_{2i}\efct{M_l}{M_l}=\chi_{2i,\sigma ,\nu ,(l,l,l)}\efct{M_l}{M_l}\ 
(i=2,3)$. 
Since $\phi_{T,l}=T(\efct{M_l}{M_l})$,
we obtain the assertion 
from these equations.
\end{proof}

\subsection{The chilarity operators in the case of 
$(\sigma_1,\sigma_2,\sigma_3)\neq (0,0,0),\ (1,1,1)$}
In this subsection, we consider the cases of 
$(\sigma_1,\sigma_2,\sigma_3)\neq (0,0,0),\ (1,1,1)$. 
Let $\varepsilon_\sigma ,\ \delta_{\sigma ;i}\in \{ 0,1 \} \ (i=1,2,3)$ 
be the integers defined by 
\[
\varepsilon_\sigma =
\left\{\begin{array}{ll}
0 & \text{ if } (\sigma_1,\sigma_2,\sigma_3)=
(1,0,0),\ (0,1,0),\ (0,0,1),\\
1 & \text{ if } (\sigma_1,\sigma_2,\sigma_3)=
(1,1,0),\ (1,0,1),\ (0,1,1),
\end{array}\right.
\text{ and }
\delta_{\sigma ;i} =
\left\{\begin{array}{ll}
0 & \text{ if } \varepsilon_\sigma -\sigma_i=0,\\
1 & \text{ otherwise .} 
\end{array}\right.
\]
In the cases of 
$(\sigma_1,\sigma_2,\sigma_3)\neq (0,0,0),\ (1,1,1)$, 
there are multiplicity one $K$-types 
$\tau_{(l+1,l,l)}$ and $\tau_{(l,l,l-1)}\ 
(l\equiv \varepsilon_\sigma \bmod 2)$
of $\pi_{(\sigma ,\nu )}$. 
Let us compute the values 
$\chi_{2i,\sigma ,\nu ,(l+,l,l)},\ $ 
$\chi_{2i,\sigma ,\nu ,(l,l,l-1)}\ (1\leq i\leq 3)$.

\begin{lem}\label{lem:ex2_1}
\textit 
(i) Let $\sigma $ be a character of $M_{\min}$ such that 
$(\sigma_1,\sigma_2,\sigma_3)\neq (0,0,0),\ (1,1,1)$. 
For $l\equiv \varepsilon_\sigma \bmod 2$, 
the following equations hold:
\begin{align*}
\pmat{(l+1,l-2,-2)}{+}{22}\eblock{(l+1,l-2,-2)}
=&\eblock{(l+1,l,l-2)}\cdot R(\Gamma_{+22}^{(l+1,l-2,-2)} ),\\
\pmat{(l+1,l,l-2)}{+}{33}\eblock{(l+1,l,l-2)}
=&\eblock{(l+1,l,l)}\cdot R(\Gamma_{+33}^{(l+1,l,l-2)} ),\\
\pmat{(l,l,l-3)}{+}{33}\eblock{(l,l,l-3)}
=&\eblock{(l,l,l-1)}\cdot R(\Gamma_{+33}^{(l,l,l-3)} ),\\
\pmat{(l-1,l-2,l-2)}{+}{12}\eblock{(l-1,l-2,l-2)}
=&\eblock{(l,l-1,l-2)}\cdot R(\Gamma_{+12}^{(l-1,l-2,l-2)} ),\\
\pmat{(l,l,l-1)}{+}{13}\eblock{(l,l,l-1)}
=&\eblock{(l+1,l,l)}\cdot R(\Gamma_{+13}^{(l,l,l-1)} ),\\
\pmat{(l-2,l-2,l-3)}{+}{13}\eblock{(l-2,l-2,l-3)}
=&\eblock{(l-1,l-2,l-2)}\cdot R(\Gamma_{+13}^{(l-2,l-2,l-3)} ),\\
\pmat{(l,l-1,l-2)}{+}{23}\eblock{(l,l-1,l-2)}
=&\eblock{(l,l,l-1)}\cdot R(\Gamma_{+23}^{(l,l-1,l-2)} ),\\
\pmat{(l,l-2,l-3)}{+}{23}\eblock{(l,l-2,l-3)}
=&\eblock{(l,l-1,l-2)}\cdot R(\Gamma_{+23}^{(l,l-2,l-3)} ),\\
\pmat{(l+1,l,l-2)}{-}{22}\eblock{(l+1,l,l-2)}
=&\eblock{(l+1,l-2,-2)}\cdot R(\Gamma_{-22}^{(l+1,l,l-2)} ),\\
\pmat{(l+1,l,l)}{-}{33}\eblock{(l+1,l,l)}
=&\eblock{(l+1,l,l-2)}\cdot R(\Gamma_{-33}^{(l+1,l,l)} ),\\
\pmat{(l,l,l-1)}{-}{33}\eblock{(l,l,l-1)}
=&\eblock{(l,l,l-3)}\cdot R(\Gamma_{-33}^{(l,l,l-1)} ),\\
\pmat{(l,l-1,l-2)}{-}{12}\eblock{(l,l-1,l-2)}
=&\eblock{(l-1,l-2,l-2)}\cdot R(\Gamma_{-12}^{(l,l-1,l-2)} ),\\
\pmat{(l+1,l,l)}{-}{13}\eblock{(l+1,l,l)}
=&\eblock{(l,l,l-1)}\cdot R(\Gamma_{-13}^{(l+1,l,l)} ),\\
\pmat{(l-1,l-2,l-2)}{-}{13}\eblock{(l-1,l-2,l-2)}
=&\eblock{(l-2,l-2,l-3)}\cdot R(\Gamma_{-13}^{(l-1,l-2,l-2)} ),\\
\pmat{(l,l,l-1)}{-}{23}\eblock{(l,l,l-1)}
=&\eblock{(l,l-1,l-2)}\cdot R(\Gamma_{-23}^{(l,l,l-1)} ),\\
\pmat{(l,l-1,l-2)}{-}{23}\eblock{(l,l-1,l-2)}
=&\eblock{(l,l-2,l-3)}\cdot R(\Gamma_{-23}^{(l,l-1,l-2)} ).
\end{align*}
The explicit expressions of $\pmat{\lambda }{\pm }{ij}$ and 
$R(\Gamma_{\pm ij}^{\lambda})$ in the above equations are 
given as follows:
\begin{align*}
&\pmat{(l+1,l-2,-2)}{+}{22}=\\
&2{\small \left(\begin{array}{cccccccccc}
X_{+22}\mhs &-2X_{+12}\mhs &0\mhs &X_{+11}\mhs &0\mhs 
&0\mhs &0\mhs &0\mhs &0\mhs &0\\
X_{+23}\mhs &-X_{+13}\mhs &-X_{+12}\mhs &0\mhs &X_{+11}\mhs 
&0\mhs &0\mhs &0\mhs &0\mhs &0\\
X_{+33}\mhs &0\mhs &-2X_{+13}\mhs &0\mhs &0\mhs 
&X_{+11}\mhs &0\mhs &0\mhs &0\mhs &0\\
0\mhs &X_{+22}\mhs &0\mhs &-2X_{+12}\mhs &0\mhs 
&0\mhs &X_{+11}\mhs &0\mhs &0\mhs &0\\
0\mhs &X_{+23}\mhs &-X_{+22}\mhs &-X_{+13}\mhs &X_{+12}\mhs 
&0\mhs &0\mhs &0\mhs &0\mhs &0\\
0\mhs &0\mhs &X_{+22}\mhs &0\mhs &-2X_{+12}\mhs 
&0\mhs &0\mhs &X_{+11}\mhs &0\mhs &0\\
0\mhs &X_{+33}\mhs &-X_{+23}\mhs &0\mhs &-X_{+13}\mhs 
&X_{+12}\mhs &0\mhs &0\mhs &0\mhs &0\\
0\mhs &0\mhs &X_{+23}\mhs &0\mhs &-X_{+13}\mhs 
&-X_{+12}\mhs &0\mhs &0\mhs &X_{+11}\mhs &0\\
0\mhs &0\mhs &X_{+33}\mhs &0\mhs &0\mhs 
&-2X_{+13}\mhs &0\mhs &0\mhs &0\mhs &X_{+11}\\
0\mhs &0\mhs &0\mhs &X_{+23}\mhs &-X_{+22}\mhs 
&0\mhs &-X_{+13}\mhs &X_{+12}\mhs &0\mhs &0\\
0\mhs &0\mhs &0\mhs &X_{+33}\mhs &-2X_{+23}\mhs 
&X_{+22}\mhs &0\mhs &0\mhs &0\mhs &0\\
0\mhs &0\mhs &0\mhs &0\mhs &X_{+23}\mhs &-X_{+22}\mhs 
&0\mhs &-X_{+13}\mhs &X_{+12}\mhs &0\\
0\mhs &0\mhs &0\mhs &0\mhs &X_{+33}\mhs &-X_{+23}\mhs 
&0\mhs &0\mhs &-X_{+13}\mhs &X_{+12}\\
0\mhs &0\mhs &0\mhs &0\mhs &0\mhs &0\mhs 
&X_{+33}\mhs &-2X_{+23}\mhs &X_{+22}&0\\
0\mhs &0\mhs &0\mhs &0\mhs &0\mhs &0\mhs 
&0\mhs &X_{+33}\mhs &-2X_{+23}\mhs &X_{+22}\\
\end{array}\right),}\\
&\pmat{(l+1,l,l-2)}{+}{33}=
{\small \begin{array}{c}
\\ \\ \\ \\ \\ \\ \\ \\ \\ \\ \\ \\ \\ \\
\end{array}^t
\left(\begin{array}{ccc}
X_{+33}&\mhs 0&\mhs 0\\
-2X_{+23}&\mhs 0&\mhs 0\\
X_{+22}&\mhs 0&\mhs 0\\
0&\mhs X_{+33}&\mhs 0\\
2X_{+13}&\mhs -2X_{+23}&\mhs 0\\
0&\mhs -2X_{+23}&\mhs X_{+33}\\
-2X_{+12}&\mhs X_{+22}&\mhs 0\\
0&\mhs X_{+22}&\mhs -2X_{+23}\\
0&\mhs 0&\mhs X_{+22}\\
0&\mhs 2X_{+13}&\mhs 0\\
X_{+11}&\mhs -2X_{+12}&\mhs 0\\
0&\mhs -2X_{+12}&\mhs 2X_{+13}\\
0&\mhs 0&\mhs -2X_{+12}\\
0&\mhs X_{+11}&\mhs 0\\
0&\mhs 0&\mhs X_{+11}
\end{array}\right) },\\
&\pmat{(l,l,l-3)}{+}{33}=
{\small \left(\begin{array}{cccccccccc}
X_{+33}&\mhs -2X_{+23}&\mhs X_{+22}&\mhs 0&\mhs 2X_{+13}
&\mhs -2X_{+12}&\mhs 0&\mhs X_{+11}&\mhs 0&\mhs 0\\
0&\mhs X_{+33}&\mhs -2X_{+23}&\mhs X_{+22}&\mhs 0
&\mhs 2X_{+13}&\mhs -2X_{+12}&\mhs 0&\mhs X_{+11}&\mhs 0\\
0&\mhs 0&\mhs 0&\mhs 0&\mhs X_{+33}
&\mhs -2X_{+23}&\mhs X_{+22}&\mhs 2X_{+13}&\mhs -2X_{+12}&\mhs X_{+11}\\
\end{array}\right)},\\
&\pmat{(l-1,l-2,l-2)}{+}{12}=3
{\small \left(\begin{array}{ccc}
X_{+12}& -X_{+11}& 0\\
X_{+13}& 0& -X_{+11}\\
X_{+22}& -X_{+12}& 0\\
0& X_{+13}& -X_{+12}\\
X_{+23}& -X_{+13}& 0\\
X_{+33}& 0& -X_{+13}\\
0& X_{+23}& -X_{+22}\\
0& X_{+33}& -X_{+23}
\end{array}\right) },\\
&\pmat{(l,l,l-1)}{+}{13}=\pmat{(l-2,l-2,l-3)}{+}{13}
=2
{\small \left(\begin{array}{ccc}
-X_{+13} & X_{+12} & -X_{+11} \\
-X_{+23} & X_{+22} & -X_{+12} \\
-X_{+33} & X_{+23} & -X_{+13} 
\end{array} \right),}\\
&\pmat{(l,l-1,l-2)}{+}{23}=
{\small \left(\begin{array}{cccccccc}
-X_{+23}& X_{+22}& X_{+13}& -2X_{+12}& -X_{+12}& 0& X_{+11}& 0\\
-X_{+33}& X_{+23}& 0& -X_{+13}& X_{+13}& -X_{+12}& 0& X_{+11}\\
0& 0& -X_{+33}& X_{+23}& 2X_{+23}& -X_{+22}& -X_{+13}& X_{+12}
\end{array} \right),}\\
&\pmat{(l,l-2,l-3)}{+}{23}=
{\small \hspace{-5mm}
{\begin{array}{c}
\\ \\ \\ \\ \\ 
\\ \\ \\ \\ \\
\\ \\ \\ \\
\end{array}}^t
\left(\begin{array}{cccccccc}
-X_{+23}& -X_{+33}& 0& 0& 0& 0& 0& 0\\
X_{+22}& X_{+23}& 0& 0& 0& 0& 0& 0\\
X_{+13}& 0& -X_{+23}& -X_{+33}& 0& 0& 0& 0\\
-2X_{+12}& -X_{+13}& X_{+22}& X_{+23}& 0& 0& 0& 0\\
-X_{+12}& X_{+13}& X_{+22}& 2X_{+23}& -X_{+23}& -X_{+33}& 0& 0\\
0& -X_{+12}& 0& -X_{+22}& X_{+22}& X_{+23}& 0& 0\\
0& 0& X_{+13}& 0& 0& 0& -X_{+33}& 0\\
X_{+11}& 0& -2X_{+12}& -X_{+13}& 0& 0& X_{+23}& 0\\
0& 0& -X_{+12}& 0& X_{+13}& 0& 2X_{+23}& -X_{+33}\\
0& X_{+11}& 0& X_{+12}& -2X_{+12}& -X_{+13}& -X_{+22}& X_{+23}\\
0& 0& 0& 0& -X_{+12}& X_{+13}& -X_{+22}& 2X_{+23}\\
0& 0& 0& 0& 0& -X_{+12}& 0& -X_{+22}\\
0& 0& X_{+11}& 0& 0& 0& -X_{+13}& 0\\
0& 0& 0& 0& X_{+11}& 0& X_{+12}& -X_{+13}\\
0& 0& 0& 0& 0& X_{+11}& 0& X_{+12}
\end{array}\right),}
\end{align*}
\pagebreak
\begin{align*}
&\pmat{(l+1,l,l-2)}{-}{22}=\\
&2\hspace{-3mm}{\small 
{\begin{array}{r}
\\ \\ \\ \\ \\ 
\\ \\ \\ \\ \\
\\ \\ \\ \\ 
\end{array}}^t
\left( \begin{array}{cccccccccc}
3X_{-22}&\mhs -2X_{-12}&\mhs 0&\mhs X_{-11}&\mhs 0
&\mhs 0&\mhs 0&\mhs 0&\mhs 0&\mhs 0\\
6X_{-23}&\mhs -2X_{-13}&\mhs -2X_{-12}&\mhs 0&\mhs X_{-11}
&\mhs 0&\mhs 0&\mhs 0&\mhs 0&\mhs 0\\
3X_{-33}&\mhs 0&\mhs -2X_{-13}&\mhs 0&\mhs 0
&\mhs X_{-11}&\mhs 0&\mhs 0&\mhs 0&\mhs 0\\
0&\mhs X_{-22}&\mhs 0&\mhs -2X_{-12}&\mhs 0
&\mhs 0&\mhs 3X_{-11}&\mhs 0&\mhs 0&\mhs 0\\
0&\mhs 4X_{-23}&\mhs -2X_{-22}&\mhs -4X_{-13}&\mhs 0
&\mhs 0&\mhs 0&\mhs 2X_{-11}&\mhs 0&\mhs 0\\
0&\mhs 2X_{-23}&\mhs X_{-22}&\mhs -2X_{-13}&\mhs -2X_{-12}
&\mhs 0&\mhs 0&\mhs 3X_{-11}&\mhs 0&\mhs 0\\
0&\mhs 3X_{-33}&\mhs -2X_{-23}&\mhs 0&\mhs -2X_{-13}
&\mhs 2X_{-12}&\mhs 0&\mhs 0&\mhs X_{-11}&\mhs 0\\
0&\mhs X_{-33}&\mhs 2X_{-23}&\mhs 0&\mhs -2X_{-13}
&\mhs -2X_{-12}&\mhs 0&\mhs 0&\mhs 3X_{-11}&\mhs 0\\
0&\mhs 0&\mhs X_{-33}&\mhs 0&\mhs 0
&\mhs -2X_{-13}&\mhs 0&\mhs 0&\mhs 0&\mhs 3X_{-11}\\
0&\mhs 0&\mhs 0&\mhs 2X_{-23}&\mhs -1X_{-22}
&\mhs 0&\mhs -6X_{-13}&\mhs 2X_{-12}&\mhs 0&\mhs 0\\
0&\mhs 0&\mhs 0&\mhs 3X_{-33}&\mhs -2X_{-23}
&\mhs X_{-22}&\mhs 0&\mhs -2X_{-13}&\mhs 2X_{-12}&\mhs 0\\
0&\mhs 0&\mhs 0&\mhs 2X_{-33}&\mhs 0
&\mhs -2X_{-22}&\mhs 0&\mhs -4X_{-13}&\mhs 4X_{-12}&\mhs 0\\
0&\mhs 0&\mhs 0&\mhs 0&\mhs X_{-33}
&\mhs -2X_{-23}&\mhs 0&\mhs 0&\mhs -2X_{-13}&\mhs 6X_{-12}\\
0&\mhs 0&\mhs 0&\mhs 0&\mhs 0
&\mhs 0&\mhs 3X_{-33}&\mhs -2X_{-23}&\mhs X_{-22}&\mhs 0\\
0&\mhs 0&\mhs 0&\mhs 0&\mhs 0
&\mhs 0&\mhs 0&\mhs X_{-33}&\mhs -2X_{-23}&\mhs 3X_{-22}
\end{array} \right) ,} \\
&\pmat{(l+1,l,l)}{-}{33}=6
\left( \begin{array}{ccc}
4X_{-33} & 0 & 0 \\
-4X_{-23} & 0 & 0 \\
4X_{-22} & 0 & 0 \\
0 & 4X_{-33} & 0 \\
3X_{-13} & -X_{-23} & -X_{-33} \\
-2X_{-13} & -2X_{-23} & 2X_{-33} \\
-3X_{-12} & X_{-22} & X_{-23} \\
X_{-12} & X_{-22} & -3X_{-23} \\
0 & 0 & 4X_{-22} \\
0 & 4X_{-13} & 0 \\
2X_{-11} & -2X_{-12} & -2X_{-13} \\
-X_{-11} & -X_{-12} & 3X_{-13} \\
0 & 0 & -4X_{-12} \\
0 & 4X_{-11} & 0 \\
0 & 0 & 4X_{-11} \\
\end{array}\right),\\
&\pmat{(l,l,l-1)}{-}{33}=24
\left(\begin{array}{ccc}
3X_{-33}& 0& 0\\
-2X_{-23}& X_{-33}& 0\\
X_{-22}& -2X_{-23}& 0\\
0& 3X_{-22}& 0\\
2X_{-13}& 0& X_{-33}\\
-X_{-12}& X_{-13}& -X_{-23}\\
0& -2X_{-12}& X_{-22}\\
X_{-11}& 0& 2X_{-13}\\
0& X_{-11}& -2X_{-12}\\
0& 0& 3X_{-11}
\end{array}\right),\\
&\pmat{(l,l-1,l-2)}{-}{12}=
\left(\begin{array}{cccccccc}
-X_{-12}& -X_{-13}& -X_{-22}& -X_{-23}& -2X_{-23}& -X_{-33}& 0& 0\\
X_{-11}& 0& X_{-12}& -X_{-13}& X_{-13}& 0& -X_{-23}& -X_{-33}\\
0& X_{-11}& 0& 2X_{-12}& X_{-12}& X_{-13}& X_{-22}& X_{-23}\\
\end{array}\right),\\
&\pmat{(l+1,l,l)}{-}{13}=\pmat{(l-1,l-2,l-2)}{-}{13}
=2
\left(\begin{array}{ccc}
-X_{-13}\hspace{-1mm} & -X_{-23}\hspace{-1mm} & -X_{-33} \\
 X_{-12}\hspace{-1mm} &  X_{-22}\hspace{-1mm} &  X_{-23} \\
-X_{-11}\hspace{-1mm} & -X_{-12}\hspace{-1mm} & -X_{-13} 
\end{array} \right),\\
&\pmat{(l,l,l-1)}{-}{23}=
3\left(\begin{array}{ccc}
X_{-23}& X_{-33}& 0\\
-X_{-22}& -X_{-23}& 0\\
-X_{-13}& 0& X_{-33}\\
X_{-12}& X_{-13}& 0\\
0& -X_{-13}& -X_{-23}\\
0& X_{-12}& X_{-22}\\
-X_{-11}& 0& X_{-13}\\
0& -X_{-11}& -X_{-12}
\end{array}\right),\\
&\pmat{(l,l-1,l-2)}{-}{23}=\\
&\left(\begin{array}{cccccccc}
8X_{-23}& 8X_{-33}& 0& 0& 0& 0& 0& 0\\
-8X_{-22}& -8X_{-23}& 0& 0& 0& 0& 0& 0\\
-4X_{-13}& 0& 4X_{-23}& 8X_{-33}& 4X_{-33}& 0& 0& 0\\
6X_{-12}& 6X_{-13}& -2X_{-22}& -2X_{-23}& -4X_{-23}& -2X_{-33}& 0& 0\\
-X_{-12}& -5X_{-13}& -X_{-22}& -5X_{-23}& 2X_{-23}& 3X_{-33}& 0& 0\\
0& 4X_{-12}& 0& 4X_{-22}& -4X_{-22}& -4X_{-23}& 0& 0\\
0& 0& -8X_{-13}& 0& 0& 0& 8X_{-33}& 0\\
-3X_{-11}& 0& 5X_{-12}& 7X_{-13}& 5X_{-13}& 0& -X_{-23}& -X_{-33}\\
X_{-11}& 0& X_{-12}& -5X_{-13}& -7X_{-13}& 0& -5X_{-23}& 3X_{-33}\\
0& -3X_{-11}& 0& -2X_{-12}& 5X_{-12}& 5X_{-13}& X_{-22}& X_{-23}\\
0& 2X_{-11}& 0& 4X_{-12}& 2X_{-12}& -6X_{-13}& 2X_{-22}& -6X_{-23}\\
0& 0& 0& 0& 0& 8X_{-12}& 0& 8X_{-22}\\
0& 0& -8X_{-11}& 0& 0& 0& 8X_{-13}& 0\\
0& 0& 0& -4X_{-11}& -8X_{-11}& 0& -4X_{-12}& 4X_{-13}\\
0& 0& 0& 0& 0& -8X_{-11}& 0& -8X_{-12}
\end{array}\right).
\end{align*}

\noindent $\bullet $ the case of 
$(\sigma_1, \sigma_2, \sigma_3)=(1,0,0),\ (0,1,1)$.
\begin{align*}
&R(\Gamma_{+22}^{(l+1,l-2,-2)} )=
2\left(\begin{array}{ccc}
\nu_2+l&  \nu_1+l-2&  0\\
\nu_3+l-1&  0&  \nu_1+l-2\\
0&  \nu_3+l-1&  \nu_2+l-2\\
0&  0&  -(\nu_2+l-1)\\
\end{array}\right),\\
&R(\Gamma_{+33}^{(l+1,l,l-2)} )=
\left(\begin{array}{cccc}
\nu_3+l-1& \nu_2+l-2& \nu_1+l-4& 0\\
\end{array}\right),\\
&R(\Gamma_{+33}^{(l,l,l-3)} )=
\left(\begin{array}{ccc}
\nu_3+l-1& \nu_2+l-2& \nu_1+l-4
\end{array}\right),\\
&R(\Gamma_{+12}^{(l-1,l-2,l-2)} )=
3\left(\begin{array}{c}
\nu_2+l\\
\nu_3+l-1\\
\end{array}\right),\quad 
R(\Gamma_{+13}^{(l,l,l-1)} )=-2(\nu_1+l),\\
&R(\Gamma_{+13}^{(l-2,l-2,l-3)} )=-2(\nu_1+l-2),\quad 
R(\Gamma_{+23}^{(l,l-1,l-2)} )=
-\left(\begin{array}{ccc}
\nu_3+l-1& \nu_2+l-2
\end{array}\right),\\
&R(\Gamma_{+23}^{(l,l-2,l-3)} )=
\left(\begin{array}{cccc}
\nu_2+l-1& \nu_2+l-2& \nu_1+l-3& 0\\
0& -(\nu_3+l-1)& 0& \nu_1+l-3
\end{array}\right),\\
&R(\Gamma_{-22}^{(l+1,l,l-2)} )=
2\left(\begin{array}{cccc}
3(\nu_2-l)& 3(\nu_3-l+1)& 0& 0\\
\nu_1-l+2& 0& 3(\nu_3-l+1)& 2(\nu_3-l+1)\\
0& \nu_1-l+2& \nu_2-l+4& -2(\nu_2-l)
\end{array}\right),\\
&R(\Gamma_{-33}^{(l+1,l,l)} )=
6\left(\begin{array}{c}
4(\nu_3-l+1)\\
4(\nu_2-l+2)\\
2(\nu_1-l+4)\\
-(\nu_1-l+4)
\end{array}\right),\quad 
R(\Gamma_{-33}^{(l,l,l-1)} )=
24\left(\begin{array}{c}
\nu_3-l+1\\
\nu_2-l+2\\
3(\nu_r-l+4)
\end{array}\right),\\
&R(\Gamma_{-12}^{(l,l-1,l-2)} )=
-\left(\begin{array}{ccc}
\nu_2-l& \nu_3-l+1
\end{array}\right),\quad 
R(\Gamma_{-13}^{(l+1,l,l)} )=-2(\nu_1-l),\\
&R(\Gamma_{-13}^{(l-1,l-2,l-2)} )=-2(\nu_1-l+2),\quad 
R(\Gamma_{-23}^{(l,l,l-1)} )=
3\left(\begin{array}{cc}
\nu_3-l+1\\
\nu_2-l+2
\end{array}\right),\\
&R(\Gamma_{-23}^{(l,l-1,l-2)} )=
\left(\begin{array}{cc}
-2(\nu_2-l)& -2(\nu_3-l+1)\\
-(\nu_2-l+4)& 3(\nu_3-l+1)\\
-8(\nu_1-l+3)& 0\\
0& -8(\nu_1-l+3)
\end{array}\right).
\end{align*}
\noindent $\bullet $ the case of 
$(\sigma_1, \sigma_2, \sigma_3)=(0,1,0),\ (1,0,1)$.
\begin{align*}
&R(\Gamma_{+22}^{(l+1,l-2,-2)} )=
2\left(\begin{array}{ccc}
\nu_2+l+1& \nu_1+l-1& 0\\
\nu_3+l-1& 0& 0\\
0& 0& \nu_1+l-1\\
0& \nu_3+l-1& \nu_2+l-3
\end{array}\right),\\
&R(\Gamma_{+33}^{(l+1,l,l-2)} )=
\left(\begin{array}{cccc}
\nu_3+l-1& \nu_2+l-1& \nu_2+l-3& \nu_1+l-3
\end{array}\right),\\
&R(\Gamma_{+33}^{(l,l,l-3)} )=
\left(\begin{array}{ccc}
\nu_3+l-1& \nu_2+l-3& \nu_1+l-5
\end{array}\right),\\
&R(\Gamma_{+12}^{(l-1,l-2,l-2)} )=
3\left(\begin{array}{c}
-(\nu_1+l)\\
\nu_3+l-1
\end{array}\right),\quad 
R(\Gamma_{+13}^{(l,l,l-1)} )=2(\nu_2+l),\\
&R(\Gamma_{+13}^{(l-2,l-2,l-3)} )=2(\nu_2+l-2),\quad 
R(\Gamma_{+23}^{(l,l-1,l-2)} )=
\left(\begin{array}{cc}
-(\nu_3+l-1)& \nu_1+l-2
\end{array}\right),\\
&R(\Gamma_{+23}^{(l,l-2,l-3)} )=
\left(\begin{array}{cccc}
\nu_2+l-2& \nu_1+l-4& 0& 0\\
0& 0& -(\nu_3+l-1)& -(\nu_2+l-3)
\end{array}\right),\\
&R(\Gamma_{-22}^{(l+1,l,l-2)} )=
2\left(\begin{array}{cccc}
\nu_2-l-1& 3(\nu_3-l+1)& \nu_3-l+1& 0\\
3(\nu_1-l+1)& 0& 0& 3(\nu_3-l+1)\\
0& \nu_1-l+1& 3(\nu_1-l+1)& (\nu_2-l+3)
\end{array}\right),\\
&R(\Gamma_{-33}^{(l+1,l,l)} )=
6\left(\begin{array}{c}
4(\nu_3-l+1)\\
(\nu_2-l)\\
(\nu_2-l+4)\\
4(\nu_1-l+3)
\end{array}\right),\quad 
R(\Gamma_{-33}^{(l,l,l-1)} )=
24\left(\begin{array}{c}
\nu_3-l+1\\
3(\nu_2-l+3)\\
\nu_1-l+5
\end{array}\right),\\
&R(\Gamma_{-12}^{(l,l-1,l-2)} )=
\left(\begin{array}{cc}
\nu_1-l& -(\nu_3-l+1)
\end{array}\right),\quad 
R(\Gamma_{-13}^{(l+1,l,l)} )=2(\nu_2-l),\\
&R(\Gamma_{-13}^{(l-1,l-2,l-2)} )=2(\nu_2-l+2),\quad 
R(\Gamma_{-23}^{(l,l,l-1)} )=
3\left(\begin{array}{c}
\nu_3-l+1\\
-(\nu_1-l+2)
\end{array}\right),\\
&R(\Gamma_{-23}^{(l,l-1,l-2)} )=
\left(\begin{array}{cc}
-8(\nu_2-l+2)& 0\\
-3(\nu_1-l+4)& -(\nu_3-l+1)\\
\nu_1-l+4& 3(\nu_3-l+1)\\
0& 8(\nu_2-l+3)
\end{array}\right).
\end{align*}
\noindent $\bullet $ the case of 
$(\sigma_1, \sigma_2, \sigma_3)=(0,0,1),\ (1,1,0)$.
\begin{align*}
&R(\Gamma_{+22}^{(l+1,l-2,-2)} )=
2\left(\begin{array}{ccc}
-(\nu_2+l-1)& 0& 0\\
\nu_2+l& \nu_1+l-1& 0\\
\nu_3+l& 0& \nu_1+l-1\\
0& \nu_3+l& \nu_2+l-2
\end{array}\right),\\
&R(\Gamma_{+33}^{(l+1,l,l-2)} )=
\left(\begin{array}{cccc}
0& \nu_3+l& \nu_2+l-2& \nu_1+l-3
\end{array}\right),\\
&R(\Gamma_{+33}^{(l,l,l-3)} )=
\left(\begin{array}{ccc}
\nu_3+l-2& \nu_2+l-4& \nu_1+l-5
\end{array}\right),\\
&R(\Gamma_{+12}^{(l-1,l-2,l-2)} )=
-3\left(\begin{array}{c}
\nu_1+l\\
\nu_2+l-1
\end{array}\right),\quad 
R(\Gamma_{+13}^{(l,l,l-1)} )=-2(\nu_3+l),\\
&R(\Gamma_{+13}^{(l-2,l-2,l-3)} )=-2(\nu_3+l-2),\quad
R(\Gamma_{+23}^{(l,l-1,l-2)} )=
\left(\begin{array}{cc}
\nu_2+l-1& \nu_1+l-2
\end{array}\right),\\
&R(\Gamma_{+23}^{(l,l-2,l-3)} )=
\left(\begin{array}{cccc}
-(\nu_3+l-2)& 0& \nu_1+l-4& 0\\
0& -(\nu_3+l-2)& -(\nu_2+l-3)& -(\nu_2+l-4)
\end{array}\right),\\
&R(\Gamma_{-22}^{(l+1,l,l-2)} )=
2\left(\begin{array}{cccc}
-2(\nu_2-l+2)& \nu_2-l-2& \nu_3-l& 0\\
2(\nu_1-l+1)& 3(\nu_1-l+1)& 0& \nu_3-l\\
0& 0& 3(\nu_1-l+1)& 3(\nu_2-l+2)
\end{array}\right),\\
&R(\Gamma_{-33}^{(l+1,l,l)} )=
6\left(\begin{array}{c}
-(\nu_3-l)\\
2(\nu_3-l)\\
4(\nu_2-l+2)\\
4(\nu_1-l+3)
\end{array}\right),\quad 
R(\Gamma_{-33}^{(l,l,l-1)} )=
24\left(\begin{array}{c}
3(\nu_3-l+2)\\
\nu_2-l+4\\
\nu_1-l+5
\end{array}\right),\\
&R(\Gamma_{-12}^{(l,l-1,l-2)} )=
\left(\begin{array}{cc}
\nu_1-l& \nu_2-l+1
\end{array}\right),\quad 
R(\Gamma_{-13}^{(l+1,l,l)} )=-2(\nu_3-l),\\
&R(\Gamma_{-13}^{(l-1,l-2,l-2)} )=-2(\nu_3-l+2),\quad
R(\Gamma_{-23}^{(l,l,l-1)} )=
-3\left(\begin{array}{c}
\nu_2-l+1\\
\nu_1-l+2
\end{array}\right),\\
&R(\Gamma_{-23}^{(l,l-1,l-2)} )=
\left(\begin{array}{cc}
8(\nu_3-l+2)& 0\\
0& 8(\nu_3-l+2)\\
-3(\nu_1-l+4)& \nu_2-l+1\\
2(\nu_1-l+4)& 2(\nu_2-l+5)
\end{array}\right).
\end{align*}
(ii) The elements $C_{2i}\ (i=1,2,3)$ are 
represented by the $\gp_{\pm }$-matrices 
$\pmat{\lambda }{\pm }{kl}$ as follows: 
\begin{align*}
\left(\begin{array}{ccc}
C_2&0&0\\
0&C_2&0\\
0&0&C_2
\end{array}\right)=&
\frac{1}{24}\Big\{\pmat{(l+1,l,l-2)}{+}{33}\cdot \pmat{(l+1,l,l)}{-}{33}
+3\hs \pmat{(l,l,l-1)}{+}{13}\cdot \pmat{(l+1,l,l)}{-}{13}\Big\} \\
=&\frac{1}{72}\Big\{
\pmat{(l,l,l-3)}{+}{33}\cdot \pmat{(l,l,l-1)}{-}{33}
-16\hs \pmat{(l,l-1,l-2)}{+}{23}\cdot \pmat{(l,l,l-1)}{-}{23}\Big\},\\
\left(\begin{array}{ccc}
C_4&0&0\\
0&C_4&0\\
0&0&C_4
\end{array}\right)=
&\frac{1}{1152}\Big\{
\pmat{(l+1,l,l-2)}{+}{33}\cdot \pmat{(l+1,l-2,-2)}{+}{22}\cdot 
\pmat{(l+1,l,l-2)}{-}{22}\cdot \pmat{(l+1,l,l)}{-}{33}\\
&-64\hs \pmat{(l,l,l-1)}{+}{13}\cdot \pmat{(l,l-1,l-2)}{+}{23}\cdot 
\pmat{(l,l,l-1)}{-}{23}\cdot \pmat{(l+1,l,l)}{-}{13}\Big\}\\
=&\frac{1}{72}\Big\{
\pmat{(l,l-1,l-2)}{+}{23}\cdot \pmat{(l-1,l-2,l-2)}{+}{12}\cdot 
\pmat{(l,l-1,l-2)}{-}{12}\cdot \pmat{(l,l,l-1)}{-}{23}\\
&+3\hs \pmat{(l,l-1,l-2)}{+}{23}\cdot 
\pmat{(l,l-2,l-3)}{+}{23}\cdot 
\pmat{(l,l-1,l-2)}{-}{23}\cdot \pmat{(l,l,l-1)}{-}{23}\Big\},\\
\left(\begin{array}{ccc}
C_6&0&0\\
0&C_6&0\\
0&0&C_6
\end{array}\right)=&\frac{1}{144}\Big\{
\pmat{(l,l,l-1)}{+}{13}\cdot \pmat{(l,l-1,l-2)}{+}{23}\cdot 
\pmat{(l-1,l-2,l-2)}{+}{12}\cdot \pmat{(l,l-1,l-2)}{-}{12}\cdot 
\pmat{(l,l,l-1)}{-}{23}\cdot \pmat{(l+1,l,l)}{-}{13}\Big\}\\
=&\frac{1}{144}\Big\{
\pmat{(l,l-1,l-2)}{+}{23}\cdot \pmat{(l-1,l-2,l-2)}{+}{12}
\cdot \pmat{(l-2,l-2,l-3)}{+}{13}\\
&\times \pmat{(l-1,l-2,l-2)}{-}{13}
\cdot \pmat{(l,l-1,l-2)}{-}{12}\cdot \pmat{(l,l,l-1)}{-}{23}\Big\}.
\end{align*}
\end{lem}
\begin{proof}
From Theorem \ref{th:main} in the case of 
$(\sigma_1,\sigma_2,\sigma_3)\neq (0,0,0),\ (1,1,1)$, 
we obtain the assertion by direct computation. 
\end{proof}

From above lemma, we obtain differential equations which 
Whittaker functions satisfy. 
For $l\in \mZ $, we set 
\begin{align*}
M_{l;1}^{\small (1)}&=M_l\gpt{1}{0}{0}{1}{0}{1},&
M_{l;2}^{\small (1)}&=M_l\gpt{1}{0}{0}{1}{0}{0},&
M_{l;3}^{\small (1)}&=M_l\gpt{1}{0}{0}{0}{0}{0},\\
M_{l;1}^{\small (2)}&=M_l\gpt{0}{0}{-1}{0}{0}{0},&
M_{l;2}^{\small (2)}&=M_l\gpt{0}{0}{-1}{0}{-1}{0},&
M_{l;3}^{\small (2)}&=M_l\gpt{0}{0}{-1}{0}{-1}{-1}.
\end{align*}
Then $G((l+1,l,l))=\{ M_{l;j}^{\small (1)}\mid j=1,2,3\}$ and 
$G((l,l,l-1))=\{ M_{l;j}^{\small (2)}\mid j=1,2,3\}$.
For $l\equiv \varepsilon_\sigma \mod 2 $,
let $M_l^{(\sigma ;1)}$ and $M_l^{(\sigma ;2)}$ be 
the unique element of $G_\sigma ((l+1,l,l))$ 
and $G_\sigma ((l,l,l-1))$, respectively. 
%check
\begin{prop}\label{prop:ex2}
\textit 
Let $\sigma $ be a character of $M_{\min }$ such that 
$(\sigma_1,\sigma_2,\sigma_3)$ is neither $(0,0,0)$ nor $(1,1,1)$, and 
$T$ be an element of the space $\Isp{\xi ,\pi_{(\sigma ,\nu )} }$. \\
(i) For $l\equiv \varepsilon_\sigma \mod 2$, we define functions
$\phi_{T,l;j}^{(1)}\in C^\infty_\xi (N_{\min}\backslash G)\ (j=1,2,3)$ by 
the equation 
$\Phi (T,S_{(l+1,l,l)}(M_l^{(\sigma ;1)}))=
\sum_{j=1}^3\phi_{T,l;j}^{(1)}\otimes f(M_{l;j}^{(1)})^*.$ 
Then $\phi_{T,l;1}^{(1)},\ \phi_{T,l;2}^{(1)}$ and 
$\phi_{T,l;3}^{(1)}$ satisfy following differential equations:
\begin{align*}
&C_{2i}\phi_{T,l;j}^{(1)}=
\chi_{2i,\sigma ,\nu ,(l+1,l,l)}\phi_{T,l;j}^{(1)},\\
&D^{(+,-)}_{i1}\phi_{T,l;1}^{(1)}+D^{(+,-)}_{i2}\phi_{T,l;2}^{(1)}+
D^{(+,-)}_{i3}\phi_{T,l;3}^{(1)}
=\tilde{\chi}_{2,\sigma ,\nu ,(l+1,l,l)}\phi_{T,l;i}^{(1)}
\end{align*}
for $i,j=1,2,3$. Here
\[
\left( \begin{array}{ccc}
D^{(+,-)}_{11} & D^{(+,-)}_{12} & D^{(+,-)}_{13} \\
D^{(+,-)}_{21} & D^{(+,-)}_{22} & D^{(+,-)}_{23} \\
D^{(+,-)}_{31} & D^{(+,-)}_{32} & D^{(+,-)}_{33}
\end{array} \right)
=m_1(C_{+})m_1(C_{-}),
\]
\begin{align*}
\chi_{2,\sigma ,\nu ,(l+1,l,l)}
=&\{\nu_1^2-(l-3)^2\}
+\{\nu_2^2-(l-2)^2\}
+\{\nu_3^2-(l-1)^2\}-2l+1,\\
\chi_{4,\sigma ,\nu ,(l+1,l,l)}
=&\{\nu_1^2-(l-2)^2-(2l-1)\delta_{\sigma ;1}\}
\{\nu_2^2-(l-2)^2-(2l-1)\delta_{\sigma ;2}\}\\
&+\{\nu_1^2-(l-2)^2-(2l-1)\delta_{\sigma ;1}\}
\{\nu_3^2-(l-1)^2-(2l-1)\delta_{\sigma ;3}\}\\
&+\{\nu_2^2-(l-1)^2-(2l-1)\delta_{\sigma ;2}\}
\{\nu_3^2-(l-1)^2-(2l-1)\delta_{\sigma ;3}\},\\
\chi_{6,\sigma ,\nu ,(l+1,l,l)}
=&\{\nu_1^2-(l-1)^2-(2l-1)\delta_{\sigma ;1}\}
\{\nu_2^2-(l-1)^2-(2l-1)\delta_{\sigma ;2}\}\\
&\times \{\nu_3^2-(l-1)^2-(2l-1)\delta_{\sigma ;3}\},\\
\tilde{\chi}_{2,\sigma ,\nu ,(l+1,l,l)}=&\nu_1^2\delta_{\sigma ;1}
+\nu_2^2\delta_{\sigma ;2}+\nu_3^2\delta_{\sigma ;3}-l^2.
\end{align*}
(ii) For $l\equiv \varepsilon_\sigma \mod 2$, we define functions
$\phi_{T,l;j}^{(2)}\in C^\infty_\xi (N_{\min}\backslash G)\ (j=1,2,3)$ by 
the equation 
$\Phi (T,S_{(l,l,l-1)}(M_l^{(\sigma ;2)}))
=\sum_{j=1}^3\phi_{T,l;j}^{(2)}\otimes f(M_{l;j}^{(2)})^*.$ 
Then $\phi_{T,l;1}^{(2)},\ \phi_{T,l;2}^{(2)}$ and 
$\phi_{T,l;3}^{(2)}$ satisfy following differential equations:
\begin{align*}
&C_{2i}\phi_{T,l;j}^{(2)}=
\chi_{2i,\sigma ,\nu ,(l,l,l-1)}\phi_{T,l;j}^{(2)},\\
&-D^{(-,+)}_{i3}\phi_{T,l;1}^{(2)}+D^{(-,+)}_{i2}\phi_{T,l;2}^{(2)}-
D^{(-,+)}_{i1}\phi_{T,l;3}^{(2)}
=(-1)^i\tilde{\chi}_{2,\sigma ,\nu ,(l,l,l-1)}\phi_{T,l;4-i}^{(2)}
\end{align*}
for $i,j=1,2,3$. 
Here
\[
\left( \begin{array}{ccc}
D^{(-,+)}_{11} & D^{(-,+)}_{12} & D^{(-,+)}_{13} \\
D^{(-,+)}_{21} & D^{(-,+)}_{22} & D^{(-,+)}_{23} \\
D^{(-,+)}_{31} & D^{(-,+)}_{32} & D^{(-,+)}_{33}
\end{array} \right)
=m_1(C_{-})m_1(C_{+}),
\]
\begin{align*}
\chi_{2,\sigma ,\nu ,(l,l,l-1)}
=&\{\nu_1^2-(l-3)^2\}
+\{\nu_2^2-(l-2)^2\}
+\{\nu_3^2-(l-1)^2\}+2l-7,\\
\chi_{4,\sigma ,\nu ,(l,l,l-1)}
=&\{\nu_1^2-(l-2)^2+(2l-5)\delta_{\sigma ;1}\}
\{\nu_2^2-(l-2)^2+(2l-5)\delta_{\sigma ;2}\}\\
&+\{\nu_1^2-(l-2)^2+(2l-5)\delta_{\sigma ;1}\}
\{\nu_3^2-(l-1)^2+(2l-5)\delta_{\sigma ;3}\}\\
&+\{\nu_2^2-(l-1)^2+(2l-5)\delta_{\sigma ;2}\}
\{\nu_3^2-(l-1)^2+(2l-5)\delta_{\sigma ;3}\},\\
\chi_{6,\sigma ,\nu ,(l,l,l-1)}
=&\{\nu_1^2-(l-1)^2+(2l-3)\delta_{\sigma ;1}\}
\{\nu_2^2-(l-1)^2+(2l-3)\delta_{\sigma ;2}\}\\
&\times \{\nu_3^2-(l-1)^2+(2l-3)\delta_{\sigma ;3}\},\\
\tilde{\chi}_{2,\sigma ,\nu ,(l,l,l-1)}=&\nu_1^2\delta_{\sigma ;1}
+\nu_2^2\delta_{\sigma ;2}+\nu_3^2\delta_{\sigma ;3}-l^2.
\end{align*}
\end{prop}
\begin{proof}
From the Lemma \ref{lem:ex2_1} and the equations 
\begin{align*}
&m_1(C_+)m_1(C_-)=\frac{1}{4}\pmat{(l,l,l-1)}{+}{13}\pmat{(l+1,l,l)}{-}{13},\\
&m_1(C_-)m_1(C_+)\left(\begin{array}{ccc}
0&0&-1\\
0& 1&0\\
-1&0&0
\end{array}\right)
=\frac{1}{4}\left(\begin{array}{ccc}
0&0&-1\\
0& 1&0\\
-1&0&0
\end{array}\right)
\pmat{(l+1,l,l)}{-}{13}\pmat{(l,l,l-1)}{+}{13},\\
&S_{(l+1,l,l)}(M_l^{(\sigma ;1)})
=\left(\begin{array}{c}
\efct{M_l^{(\sigma ;1)}}{M_{l;1}^{(1)}}\\
\efct{M_l^{(\sigma ;1)}}{M_{l;2}^{(1)}}\\
\efct{M_l^{(\sigma ;1)}}{M_{l;3}^{(1)}}
\end{array}\right),\quad 
S_{(l,l,l-1)}(M_l^{(\sigma ;2)})
=\left(\begin{array}{c}
\efct{M_l^{(\sigma ;2)}}{M_{l;1}^{(2)}}\\
\efct{M_l^{(\sigma ;2)}}{M_{l;2}^{(2)}}\\
\efct{M_l^{(\sigma ;2)}}{M_{l;3}^{(2)}}
\end{array}\right),
\end{align*}
we obtain
\begin{align*}
&C_{2i}\phi_{T,l;j}^{(1)}=
\chi_{2i,\sigma ,\nu ,(l+1,l,l)}\phi_{T,l;j}^{(1)},\quad 
C_{2i}\phi_{T,l;j}^{(2)}=
\chi_{2i,\sigma ,\nu ,(l,l,l-1)}\phi_{T,l;j}^{(2)}\quad (i,j=1,2,3)\\
&\left( \begin{array}{ccc}
D^{(+,-)}_{11} & D^{(+,-)}_{12} & D^{(+,-)}_{13} \\
D^{(+,-)}_{21} & D^{(+,-)}_{22} & D^{(+,-)}_{23} \\
D^{(+,-)}_{31} & D^{(+,-)}_{32} & D^{(+,-)}_{33}
\end{array} \right) \left(\begin{array}{c}
\efct{M_l^{(\sigma ;1)}}{M_{l;1}^{(1)}}\\
\efct{M_l^{(\sigma ;1)}}{M_{l;2}^{(1)}}\\
\efct{M_l^{(\sigma ;1)}}{M_{l;3}^{(1)}}
\end{array}\right)=\tilde{\chi}_{2,\sigma ,\nu ,(l+1,l,l)}
\left(\begin{array}{c}
\efct{M_l^{(\sigma ;1)}}{M_{l;1}^{(1)}}\\
\efct{M_l^{(\sigma ;1)}}{M_{l;2}^{(1)}}\\
\efct{M_l^{(\sigma ;1)}}{M_{l;3}^{(1)}}
\end{array}\right),\\
&\left( \begin{array}{ccc}
D^{(-,+)}_{11} & D^{(-,+)}_{12} & D^{(-,+)}_{13} \\
D^{(-,+)}_{21} & D^{(-,+)}_{22} & D^{(-,+)}_{23} \\
D^{(-,+)}_{31} & D^{(-,+)}_{32} & D^{(-,+)}_{33}
\end{array} \right) \left(\begin{array}{c}
-\efct{M_l^{(\sigma ;2)}}{M_{l;3}^{(2)}}\\
 \efct{M_l^{(\sigma ;2)}}{M_{l;2}^{(2)}}\\
-\efct{M_l^{(\sigma ;2)}}{M_{l;1}^{(2)}}
\end{array}\right)=\tilde{\chi}_{2,\sigma ,\nu ,(l,l,l-1)}
\left(\begin{array}{c}
-\efct{M_l^{(\sigma ;2)}}{M_{l;3}^{(2)}}\\
 \efct{M_l^{(\sigma ;2)}}{M_{l;2}^{(2)}}\\
-\efct{M_l^{(\sigma ;2)}}{M_{l;1}^{(2)}}
\end{array}\right).
\end{align*}
From these equations, we obtain the assertion.
\end{proof}
\begin{rem}
\textit{ 
Since $V_{(l+1,l,l)}$ and $V_{(l,l,l-1)}$ are three dimensional, 
the differential equations obtained from $C_{2i}\ (i=1,2,3)$ do not suffice 
to characterize the Whittaker functions. }
\end{rem}

\subsection{Differential equations}
To obtain the explicit actions of the operators 
$C_{2i},\ D^{(\pm ,\mp )}_{jk}\ (1\leq i,j,k\leq 3)$, 
we may express these operators in the normal order modulo 
$[\gn ,\gn ]$ with respect to the Iwasawa decomposition of $\g$, 
according to the following lemma. 

\begin{lem}\label{lem:Kact_func}
\textit 
(i) Let $\phi_{T,l},\ \phi_{T,l;i}^{(1)},\ \phi_{T,l;i}^{(2)}\ (i=1,2,3)$ be 
the elements of $C^\infty_\xi (N_{\min}\backslash G)$ defined in 
Proposition \ref{prop:ex1} and \ref{prop:ex2}. 
The explicit expressions of the action of $\gk_\mC$ on these functions 
given as follows.\\
$\bullet $ the case of $(\sigma_1,\sigma_2,\sigma_3)=(0,0,0),(1,1,1)$.
\begin{align*}
&\kappa (E_{ii})\phi_{T,l}=l\phi_{T,l}&&(1\leq i\leq 3),\\
&\kappa (E_{jk})\phi_{T,l}=0&&(1\leq j\neq k\leq 3).
\end{align*}
$\bullet $ the case of $(\sigma_1,\sigma_2,\sigma_3)\neq (0,0,0),(1,1,1)$.
\begin{align*}
&\kappa (E_{jj})\phi_{T,l;i}^{(1)}=
(l+\delta_{ij})\phi_{T,l;i}^{(1)}&&(1\leq i,j\leq 3),\\
&\kappa (E_{mn})\phi_{T,l;k}^{(1)}=\delta_{nk} \phi_{T,l;m}^{(1)}
&&(1\leq k\leq 3,\ 1\leq m\neq n\leq 3),
\end{align*}
and
\begin{align*}
&\kappa (E_{jj})\phi_{T,l;i}^{(2)}=
(l-\delta_{4-i\hs j})\phi_{T,l;i}^{(2)}&&(1\leq i,j\leq 3),\\
&\kappa (E_{mn})\phi_{T,l;k}^{(2)}
=(-1)^{m+n+1}\delta_{4-m\hs k} \phi_{T,l;4-n}^{(2)}
&&(1\leq k\leq 3,\ 1\leq m\neq n\leq 3).
\end{align*}
(ii) Let $\phi \in C^\infty_\xi (N_{\min}\backslash G)$. 
For $X\in U(\gn_\mC ),\ Y\in U(\ga_\mC )$ and 
$a\in A_{\min}$, we have the equation
$(\Ad (a^{-1})X)Y\phi (a)=\xi (X)(Y\phi )(a)$. 
In particular, for 
$a=\diag (a_1,a_2,a_3,a_1^{-1},a_2^{-1},a_3^{-1})\in A_{\min}$, 
we have 
\begin{align*}
&H_i\phi (a)=a_i\frac{\partial}{\partial a_i}\phi (a)\quad (1\leq i\leq 3),
&&E_{e_1-e_2}\phi (a)=2\pi \sqrt{-1}c_{12}\frac{a_1}{a_2}\phi (a),\\
&E_{e_2-e_3}\phi (a)=2\pi \sqrt{-1}c_{23}\frac{a_2}{a_3}\phi (a),
&&E_{2e_3}\phi (a)=2\pi \sqrt{-1}c_{3}a_3^2\phi (a),
\end{align*}
and $E_\alpha \phi (a)=0$ for 
$\alpha \in \Sigma^+\backslash \{e_1-e_2,\ e_2-e_3,\ 2e_3\}$.
\end{lem}
\begin{proof}
From the definition of 
$\phi_{T,l},\ \phi_{T,l;i}^{(1)},\ \phi_{T,l;i}^{(2)}\ (i=1,2,3)$ and 
Lemma \ref{prop:action_on_GZ-basis}, we obtain the statement (i). 
The statement (ii) is obvious from the definition of 
$C^\infty_\xi (N_{\min}\backslash G)$.
\end{proof}

Moreover, we have the following lemma which is required to get 
the expressions of the elements in $U(\g_\mC )$ in normal order. 
In the following, we denote $X\equiv Y$ for two elements 
$X$ and $Y$ in $U(\g_\mC )$ when $X-Y\in [\gn ,\gn ]U(\g_\mC )$.
\begin{lem}
(i) The root vectors $X_{\pm ij}$ in $\gp_{\pm}$ 
have the following expressions:
\begin{align*}
X_{+ij}&\equiv 
\left\{ \begin{array}{ll}
H_i+\kappa (E_{ii})&(i=j=1,2),\\
E_{e_i-e_j}+\kappa (E_{ji})&((i,j)=(1,2),(2,3)),
\end{array}\right.\\
X_{-ij}&\equiv 
\left\{ \begin{array}{ll}
H_i-\kappa (E_{ii})&(i=j=1,2),\\
E_{e_i-e_j}-\kappa (E_{ij})&((i,j)=(1,2),(2,3)),
\end{array}\right.
\end{align*}
and
$X_{\pm 33}\equiv \pm 2\sqrt{-1}E_{2e_3}+H_3\pm \kappa (E_{33}),\ 
X_{+13}\equiv \kappa (E_{31}),\ X_{-13}\equiv -\kappa (E_{13})$.\\
(ii) Each $(i,j)$-minor $M_{+ij}$ in the matrix $m_2(C_+)$ has 
the following expression.
\begin{align*}
M_{+11}\equiv &(H_2-1)X_{+33}+X_{+33}\kappa (E_{22})
	-E_{e_2-e_3}X_{+23}-X_{+23}\kappa (E_{32}),\\
M_{+22}\equiv &(H_1-1)X_{+33}+X_{+33}\kappa (E_{11})-X_{+13}\kappa (E_{31}),\\
M_{+33}\equiv &(H_1-1)X_{+22}+X_{+22}\kappa (E_{11})
	-E_{e_1-e_2}X_{+12}-X_{+12}\kappa (E_{21}),\\
M_{+12}\equiv &E_{e_1-e_2}X_{+33}+X_{+33}\kappa (E_{21})
	-X_{+23}\kappa (E_{31}),\\
M_{+23}\equiv &(H_1-1)X_{+23}+X_{+23}\kappa (E_{11})-X_{+12}\kappa (E_{31}),\\
M_{+13}\equiv &E_{e_1-e_2}X_{+23}+X_{+23}\kappa (E_{21})
	-X_{+22}\kappa (E_{31}).
\end{align*}
(iii) Each $(i,j)$-minor $M_{-ij}$ in the matrix $m_2(C_-)$ has 
the following expression.
\begin{align*}
M_{-11}\equiv &(H_2-1)X_{-33}-X_{-33}\kappa (E_{22})
	-E_{e_2-e_3}X_{-23}+X_{-23}\kappa (E_{23}),\\
M_{-22}\equiv &(H_1-1)X_{-33}-X_{-33}\kappa (E_{11})+X_{-13}\kappa (E_{13}),\\
M_{-33}\equiv &(H_1-1)X_{-22}-X_{-22}\kappa (E_{11})
	-E_{e_1-e_2}X_{-12}+X_{-12}\kappa (E_{12}),\\
M_{-12}\equiv &E_{e_1-e_2}X_{-33}-X_{-33}\kappa (E_{12})
	+X_{-23}\kappa (E_{13}),\\
M_{-23}\equiv &(H_1-1)X_{-23}-X_{-23}\kappa (E_{11})+X_{-12}\kappa (E_{13}),\\
M_{-13}\equiv &E_{e_1-e_2}X_{-23}-X_{-23}\kappa (E_{12})
	+X_{-22}\kappa (E_{13}).
\end{align*}
\end{lem}
\begin{proof}
The statement (i) is obvious from Lemma \ref{lem:K-action}. 
The statements (ii), (iii) are obtain by direct computation 
using tables in the proof of Lemma \ref{lem:K-action}.
\end{proof}
By using above lemma and tables in the proof of Lemma \ref{lem:K-action}, 
we have the following expressions of 
the elements $D^{(\pm ,\mp )}_{ij}\ (1\leq i,j\leq 3),\ 
C_{2k}\ (1\leq k\leq 2),\ m_3(C_\pm )$ in normal order:

For $1\leq i\leq 3$, we have
\begin{align*}
D_{1i}^{(+,-)}\equiv &
(H_1-4)X_{-1i}+X_{-1i}\kappa (E_{11})
+E_{e_1-e_2}X_{-2i}+X_{-2i}\kappa (E_{21})
+X_{-3i}\kappa (E_{31}),\\
D_{2i}^{(+,-)}\equiv &
E_{e_1-e_2}X_{-1i}+X_{-1i}\kappa (E_{21})
+(H_2-3+\delta_{1i})X_{-2i}+X_{-2i}\kappa (E_{22})\\
&+E_{e_2-e_3}X_{-3i}+X_{-3i}\kappa (E_{32})-\delta_{2i}X_{11},\\
D_{3i}^{(+,-)}\equiv &
X_{-1i}\kappa (E_{31})
+E_{e_2-e_3}X_{-2i}+X_{-2i}\kappa (E_{32})\\
&+(H_3-1-\delta_{3i}+2\sqrt{-1}E_{2e_3})X_{-3i}+X_{-3i}\kappa (E_{33})
-\delta_{3i}(X_{-11}+X_{-22}),
\end{align*}
and
\begin{align*}
D_{1i}^{(-,+)}\equiv &
(H_1-4)X_{+1i}-X_{+1i}\kappa (E_{11})
+E_{e_1-e_2}X_{+2i}-X_{+2i}\kappa (E_{12})-X_{+3i}\kappa (E_{13}),\\
D_{2i}^{(-,+)}\equiv &
E_{e_1-e_2}X_{+1i}-X_{+1i}\kappa (E_{12})
+(H_2-3+\delta_{1i})X_{+2i}-X_{+2i}\kappa (E_{22})\\
&+E_{e_2-e_3}X_{+3i}-X_{+3i}\kappa (E_{23})-\delta_{2i}X_{+11},\\
D_{3i}^{(-,+)}\equiv &
-X_{+1i}\kappa (E_{13})
+E_{e_2-e_3}X_{+2i}-X_{+2i}\kappa (E_{23})\\
&+(H_3-1-\delta_{3i}-2\sqrt{-1}E_{2e_3})X_{+3i}-X_{+3i}\kappa (E_{33})
-\delta_{3i}(X_{+11}+X_{+22}).
\end{align*}
Here we denote $X_{\pm ji}=X_{\pm ij}\ (1\leq i<j\leq 3)$. 
\begin{align*}
C_2\ \equiv \ 
&(H_1-6)X_{-11}+X_{-11}\kappa (E_{11})+(H_2-4)X_{-22}+X_{-22}\kappa (E_{22})\\
&+(H_3+2\sqrt{-1}E_{2e_3}-2)X_{-33}+X_{-33}\kappa (E_{33})
	+2E_{e_1-e_2}X_{-12}\\
&+2X_{-12}\kappa (E_{21})+2E_{e_2-e_3}X_{-23}+2X_{-23}\kappa (E_{32})
	+2X_{-13}\kappa (E_{31}).\\
C_4\ \equiv \ 
&(H_2-1)
\big\{ (2\sqrt{-1}E_{2e_3}+H_3)M_{-11}+M_{-11}(\kappa (E_{33})-2)\big\} \\
&+\big\{ (2\sqrt{-1}E_{2e_3}+H_3)M_{-11}+M_{-11}(\kappa (E_{33})-2)\big\}
(\kappa (E_{22})-2)\\
&-E^2_{e_2-e_3}M_{-11}-2E_{e_2-e_3}M_{-11}\kappa (E_{32})
-M_{-11}\kappa (E_{32})^2\\
&+(H_1-1)
\big\{ (2\sqrt{-1}E_{2e_3}+H_3)M_{-22}+M_{-22}(\kappa (E_{33})-2)\big\} \\
&+\big\{ (2\sqrt{-1}E_{2e_3}+H_3)M_{-22}+M_{-22}(\kappa (E_{33})-2)\big\}
(\kappa (E_{11})-2)-M_{-22}\kappa (E_{31})^2\\
&+(H_1-1)
\big\{ H_2M_{-33}+M_{-33}(\kappa (E_{22})-2)\big\} \\
&+\big\{ H_2M_{-33}+M_{-33}(\kappa (E_{22})-2)\big\} 
(\kappa (E_{11})-2)\\
&-E^2_{e_1-e_2}M_{-33}-2E_{e_1-e_2}M_{-33}\kappa (E_{21})
-M_{-33}\kappa (E_{21})^2\\
&+2E_{e_1-e_2}
\big\{ (2\sqrt{-1}E_{2e_3}+H_3)M_{-12}+M_{-12}(\kappa (E_{33})-2)\big\} \\
&+2\big\{ (2\sqrt{-1}E_{2e_3}+H_3)M_{-12}+M_{-12}(\kappa (E_{33})-2)\big\} 
\kappa (E_{21})\\
&-2\big\{ (2\sqrt{-1}E_{2e_3}+H_3)M_{-22}+M_{-22}(\kappa (E_{33})-2)\big\} 
-2E_{e_2-e_3}M_{-23}\\
&-2(E_{e_2-e_3}M_{-12}+M_{-12}\kappa (E_{32})-M_{-13})\kappa (E_{31})
+2M_{-33}-2M_{-23}\kappa (E_{32})\\
&+2(H_1-1)(E_{e_2-e_3}M_{-23}+M_{-23}\kappa (E_{32})-M_{-33})\\
&+2(E_{e_2-e_3}M_{-23}+M_{-23}\kappa (E_{32})-M_{-33})(\kappa (E_{11})-2)\\
&-2(E_{e_1-e_2}M_{-23}+M_{-23}\kappa (E_{21}))\kappa (E_{31})\\
&+2E_{e_1-e_2}(E_{e_2-e_3}M_{-13}+M_{-13}\kappa (E_{32}))-2E_{e_2-e_3}M_{-23}\\
&-2(H_2-3)M_{-33}-2M_{-33}\kappa (E_{22})-2M_{-23}\kappa (E_{32})\\
&+2(E_{e_2-e_3}M_{-13}+M_{-13}\kappa (E_{32}))\kappa (E_{21})\\
&-2\big\{ (H_2-2)M_{-13}+M_{-13}\kappa (E_{22})\big\} \kappa (E_{31}),\\
m_3(C_+)\equiv 
&(H_1-2)M_{+11}+M_{+11}\kappa (E_{11})-E_{e_1-e_2}M_{+12}
-M_{+12}\kappa (E_{21})+M_{+13}\kappa (E_{31}),\\
m_3(C_-)\equiv 
&(H_1-2)M_{-11}-M_{-11}\kappa (E_{11})-E_{e_1-e_2}M_{-12}
+M_{-12}\kappa (E_{12})-M_{-13}\kappa (E_{13}).\\
\end{align*}
From above expressions and Lemma \ref{lem:Kact_func}(i), 
We can summarize the explicit actions of the 
operators $C_{2i}\ (1\leq i\leq 3)$ and 
$D_{jk}^{(\pm ,\mp )}\ (1\leq j,k\leq 3)$ on the functions in 
Proposition \ref{prop:ex1}, \ref{prop:ex2}. 
\begin{prop}\label{prop:ex3}
The operators $C_{2i}\ (1\leq i\leq 3)$ and 
$D_{jk}^{(\pm ,\mp )}\ (1\leq j,k\leq 3)$ acting on 
the functions 
$\phi_{T,l},\ \phi_{T,l;i}^{(1)},\ \phi_{T,l;i}^{(2)}\ (i=1,2,3)$ 
in Proposition \ref{prop:ex1}, \ref{prop:ex2} are given 
as follows. \\
$\bullet $ the case of $(\sigma_1,\sigma_2,\sigma_3)=(0,0,0),(1,1,1)$.
\begin{align*}
C_2\phi_{T,l}
=&\{(H_1+l-6)(H_1-l)
+(H_2+l-4)(H_2-l)\\
&\hphantom{=}+(H_3+l-2+2\sqrt{-1}E_{2e_3})
(H_3-l-2\sqrt{-1}E_{2e_3})+2E_{e_1-e_2}^2
+2E_{e_2-e_3}^2\} \phi_{T,l},\\
C_4\phi_{T,l}
=&\Big\{
\{(H_2+l-3)
(H_3+l-2+2\sqrt{-1}E_{2e_3})
-E_{e_2-e_3}^2
\}\\
&\hphantom{==} \times \{(H_2-l-1)
(H_3-l-2\sqrt{-1}E_{2e_3})
-E_{e_2-e_3}^2\} \\
&\hphantom{=}+(H_1+l-5)
(H_3+l-2+2\sqrt{-1}E_{2e_3})
(H_1-l-1)(H_3-l-2\sqrt{-1}E_{2e_3})\\
&\hphantom{=}+\{(H_1+l-5)(H_2+l-4)
-E_{e_1-e_2}^2\} 
\{(H_1-l-1)(H_2-l)
-E_{e_1-e_2}^2\} \\
&\hphantom{=}+2E_{e_1-e_2}^2
(H_3+l-2+2\sqrt{-1}E_{2e_3})
(H_3-l-2\sqrt{-1}E_{2e_3})\\
&\hphantom{=}+2E_{e_1-e_2}^2E_{e_2-e_3}^2
+2E_{e_2-e_3}^2
(H_1+l-5)(H_1-l-1)
\Big\}\phi_{T,l},\\
C_6\phi_{T,l}
=&\{
(H_1+l-4)(H_2+l-3)
(H_3+l-2+2\sqrt{-1}E_{2e_3})\\
&\hphantom{=}-E_{e_2-e_3}^2(H_1+l-4)
-E_{e_1-e_2}^2(H_3+l-2+2\sqrt{-1}E_{2e_3})\}\\
&\times \{
(H_1-l-2)(H_2-l-1)
(H_3-l-2\sqrt{-1}E_{2e_3})\\
&\hphantom{=}-E_{e_2-e_3}^2(H_1-l-2)
-E_{e_1-e_2}^2(H_3-l-2\sqrt{-1}E_{2e_3})\}\phi_{T,l}.
\end{align*}
$\bullet $ the case of $(\sigma_1,\sigma_2,\sigma_3)\neq (0,0,0),(1,1,1)$.
\begin{align*}
&\hspace{-4mm}D_{11}^{(+,-)}\phi_{T,l;1}^{(1)}
+D_{12}^{(+,-)}\phi_{T,l;2}^{(1)}
+D_{13}^{(+,-)}\phi_{T,l;3}^{(1)}\\
=&\{ 
(H_1+l-3)(H_1-l-3)+E_{e_1-e_2}^2
\} \phi_{T,l;1}^{(1)}
+E_{e_1-e_2}
(H_1+H_2-4)\phi_{T,l;2}^{(1)}\\
&+E_{e_1-e_2}E_{e_2-e_3}\phi_{T,l;3}^{(1)},\\
&\hspace{-4mm}D_{21}^{(+,-)}\phi_{T,l;1}^{(1)}+
D_{22}^{(+,-)}\phi_{T,l;2}^{(1)}+
D_{23}^{(+,-)}\phi_{T,l;3}^{(1)}\\
=&E_{e_1-e_2}(H_1+H_2-6)
\phi_{T,l;1}^{(1)}
+\{
(H_2+l-2)(H_2-l-2)+E_{e_1-e_2}^2
+E_{e_2-e_3}^2
\} \phi_{T,l;2}^{(1)}\\
&+E_{e_2-e_3}
(H_2+H_3-2-2\sqrt{-1}E_{2e_3})\phi_{T,l;3}^{(1)},\\
&\hspace{-4mm}D_{31}^{(+,-)}\phi_{T,l;1}^{(1)}+
D_{32}^{(+,-)}\phi_{T,l;2}^{(1)}+
D_{33}^{(+,-)}\phi_{T,l;3}^{(1)}\\
=&+E_{e_1-e_2}E_{e_2-e_3}\phi_{T,l;1}^{(1)}
+E_{e_2-e_3}
(H_2+H_3-4+2\sqrt{-1}E_{2e_3})\phi_{T,l;2}^{(1)}\\
&+\{
(H_3+l-1+2\sqrt{-1}E_{2e_3})(H_3-l-1-2\sqrt{-1}E_{2e_3})
+E_{e_2-e_3}^2\} \phi_{T,l;3}^{(1)},
\end{align*}
\begin{align*}
C_2\phi_{T,l;i}^{(1)}
=&\{(H_1+l-6+\delta_{1i})(H_1-l-\delta_{1i})
+(H_2+l-4+\delta_{2i})(H_2-l-\delta_{2i})\\
&\hphantom{=}+(H_3+l-2+\delta_{3i}+2\sqrt{-1}E_{2e_3})
(H_3-l-\delta_{3i}-2\sqrt{-1}E_{2e_3})\\
&\hphantom{=}+2E_{e_1-e_2}^2
+2E_{e_2-e_3}^2
\} \phi_{T,l;i}^{(1)}\\
&-2\delta_{1i}\{
2\phi_{T,l;1}^{(1)}
-E_{e_1-e_2}\phi_{T,l;2}^{(1)}\} 
-2\delta_{2i}\{
E_{e_1-e_2}\phi_{T,l;1}^{(1)}
+\phi_{T,l;2}^{(1)}
-E_{e_2-e_3}\phi_{T,l;3}^{(1)}
\}\\
&-2\delta_{3i}E_{e_2-e_3}\phi_{T,l;2}^{(1)}, \\
C_4\phi_{T,l;i}^{(1)}
=&\Big\{
\{(H_2+l-3+\delta_{2i})
(H_3+l-2+\delta_{3i}+2\sqrt{-1}E_{2e_3})
-E_{e_2-e_3}^2
\}\\
&\hphantom{==} \times \{(H_2-l-1-\delta_{2i})
(H_3-l-\delta_{3i}-2\sqrt{-1}E_{2e_3})
-E_{e_2-e_3}^2\} \\
&\hphantom{=}+(H_1+l-5+\delta_{1i})
(H_3+l-2+\delta_{3i}+2\sqrt{-1}E_{2e_3})\\
&\hphantom{==} \times 
(H_1-l-1-\delta_{1i})(H_3-l-\delta_{3i}-2\sqrt{-1}E_{2e_3})\\
&\hphantom{=}+\{(H_1+l-5+\delta_{1i})(H_2+l-4+\delta_{2i})
-E_{e_1-e_2}^2\} \\
&\hphantom{==} 
\times \{(H_1-l-1-\delta_{1i})(H_2-l-\delta_{2i})
-E_{e_1-e_2}^2\} \\
&\hphantom{=}+2E_{e_1-e_2}^2
(H_3+l-2+\delta_{3i}+2\sqrt{-1}E_{2e_3})
(H_3-l-\delta_{3i}-2\sqrt{-1}E_{2e_3})\\
&\hphantom{=}+2E_{e_1-e_2}^2E_{e_2-e_3}^2
+2E_{e_2-e_3}^2
(H_1+l-5+\delta_{1i})(H_1-l-1-\delta_{1i})
\Big\}\phi_{T,l;i}^{(1)}\\
&+2\delta_{1i}\Big\{
-\{(H_2+l-4)(H_2-l)\\
&\hphantom{==}+(H_3+l-2+2\sqrt{-1}E_{2e_3})
(H_3-l-2\sqrt{-1}E_{2e_3})+3E_{e_1-e_2}^2
+2E_{e_2-e_3}^2
\}\phi_{T,l;1}^{(1)}\\
&\hphantom{=}+E_{e_1-e_2}
\{(H_3+l-2+2\sqrt{-1}E_{2e_3})(H_3-l-2\sqrt{-1}E_{2e_3})\\
&\hphantom{==}-(H_1-l-1)(H_2-l-1)
+H_1+H_2-6
+E_{e_1-e_2}^2
+E_{e_2-e_3}^2
\}\phi_{T,l;2}^{(1)}\\
&\hphantom{=}-E_{e_1-e_2}E_{e_2-e_3}
(H_1+H_2+H_3-l-6-2\sqrt{-1}E_{2e_3})
\phi_{T,l;3}^{(1)}
\Big\}\\
&+2\delta_{2i}\Big\{
-E_{e_1-e_2}
\{
(H_3+l-2+2\sqrt{-1}E_{2e_3})(H_3-l-2\sqrt{-1}E_{2e_3})\\
&\hphantom{==}-(H_1+l-4)(H_2+l-4)
-(H_1+H_2-4)+E_{e_1-e_2}^2+E_{e_2-e_3}^2
\}\phi_{T,l;1}^{(1)}\\
&\hphantom{=}-\{
(H_1+l-5)(H_1-l-1)
+E_{e_1-e_2}^2
+2E_{e_2-e_3}^2
\}
\phi_{T,l;2}^{(1)}\\
&\hphantom{=}+E_{e_2-e_3}
\{
(H_1+l-5)(H_1-l-1)\\
&\hphantom{==}-(H_2-l-1)(H_3-l-1-2\sqrt{-1}E_{2e_3})
+E_{e_1-e_2}^2+E_{e_2-e_3}^2
\}\phi_{T,l;3}^{(1)}
\Big\}\\
&+2\delta_{3i}\Big\{
E_{e_1-e_2}E_{e_2-e_3}
(H_1+H_2+H_3+l-7+2\sqrt{-1}E_{2e_3})\phi_{T,l;1}^{(1)}\\
&\hphantom{=}+E_{e_2-e_3}
\{(H_2+l-2)(H_3+l-2+2\sqrt{-1}E_{2e_3})\\
&\hphantom{==}-(H_1+l-5)(H_1-l-1)-E_{e_1-e_2}^2-E_{e_2-e_3}^2
\}\phi_{T,l;2}^{(1)}
\Big\} ,\\
C_6\phi_{T,l;i}^{(1)}
=&\{
(H_1+l-4+\delta_{1i})(H_2+l-3+\delta_{2i})
(H_3+l-2+\delta_{3i}+2\sqrt{-1}E_{2e_3})\\
&\hphantom{=}-E_{e_2-e_3}^2(H_1+l-4+\delta_{1i})
-E_{e_1-e_2}^2(H_3+l-2+\delta_{3i}+2\sqrt{-1}E_{2e_3})
\}\\
&\hphantom{=}
\times \{(H_1-l-2-\delta_{1i})(H_2-l-1-\delta_{2i})
(H_3-l-\delta_{3i}-2\sqrt{-1}E_{2e_3})\\
&-E_{e_2-e_3}^2(H_1-l-2-\delta_{1i})
-E_{e_1-e_2}^2(H_3-l-\delta_{3i}-2\sqrt{-1}E_{2e_3})
\}\phi_{T,l;i}^{(1)}\\
&+2\delta_{1i}\Big\{
-2E_{e_1-e_2}^2\{(H_3+l-2+2\sqrt{-1}E_{2e_3})
(H_3-l-2\sqrt{-1}E_{2e_3})+E_{e_2-e_3}^2\}
\phi_{T,l;1}^{(1)}\\
&\hphantom{=}-E_{e_1-e_2}
\{(H_3+l-2+2\sqrt{-1}E_{2e_3})
(H_1-l-2)(H_2-l-2)
(H_3-l-2\sqrt{-1}E_{2e_3})\\
&\hphantom{==}-E_{e_2-e_3}^2(H_3+l-2+2\sqrt{-1}E_{2e_3})(H_1-l-2)\\
&\hphantom{==}
-E_{e_1-e_2}^2(H_3+l-2+2\sqrt{-1}E_{2e_3})(H_3-l-2\sqrt{-1}E_{2e_3})
\}\phi_{T,l;2}^{(1)}\\
&\hphantom{=}+E_{e_1-e_2}E_{e_2-e_3}
\{(H_1-l-2)(H_2-l-1)
(H_3-l-1-2\sqrt{-1}E_{2e_3})\\
&\hphantom{==}-E_{e_2-e_3}^2(H_1-l-2)
-E_{e_1-e_2}^2(H_3-l-1-2\sqrt{-1}E_{2e_3})
\}\phi_{T,l;3}^{(1)}
\Big\}\\
&+2\delta_{2i}\Big\{
E_{e_1-e_2}\{
(H_1+l-3)(H_2+l-3)\\
&\hphantom{===}\times (H_3+l-2+2\sqrt{-1}E_{2e_3})(H_3-l-2\sqrt{-1}E_{2e_3})\\
&\hphantom{==}-E_{e_2-e_3}^2(H_1+l-3)(H_3-l-2-2\sqrt{-1}E_{2e_3})\\
&\hphantom{==}-E_{e_1-e_2}^2(H_3+l-2+2\sqrt{-1}E_{2e_3})
(H_3-l-2\sqrt{-1}E_{2e_3})\}\phi_{T,l;1}^{(1)}\\
&\hphantom{=}-2E_{e_2-e_3}^2(H_1+l-4)
(H_1-l-2)\phi_{T,l;2}^{(1)}\\
&\hphantom{=}-E_{e_2-e_3}(H_1+l-4)
\{(H_1-l-2)(H_2-l-1)
(H_3-l-1-2\sqrt{-1}E_{2e_3})\\
&\hphantom{==}-E_{e_2-e_3}^2(H_1-l-2)
-E_{e_1-e_2}^2(H_3-l-1-2\sqrt{-1}E_{2e_3})
\}\phi_{T,l;3}^{(1)}
\Big\}\\
&+2\delta_{3i}\Big\{
-E_{e_1-e_2}E_{e_2-e_3}\{
(H_1+l-3)(H_2+l-3)
(H_3+l-2+2\sqrt{-1}E_{2e_3})\\
&\hphantom{==}-E_{e_2-e_3}^2(H_1+l-3)
-E_{e_1-e_2}^2(H_3+l-2+2\sqrt{-1}E_{2e_3})
\}\phi_{T,l;1}^{(1)}\\
&\hphantom{=}+E_{e_2-e_3}\{
(H_1+l-4)(H_2+l-2)
(H_3+l-2+2\sqrt{-1}E_{2e_3})\\
&\hphantom{==}-E_{e_2-e_3}^2(H_1+l-4)
-E_{e_1-e_2}^2(H_3+l-2+2\sqrt{-1}E_{2e_3})
\}(H_1-l-2)\phi_{T,l;2}^{(1)}
\Big\} ,
\end{align*}
and
\begin{align*}
&\hspace{-4mm}-D_{11}^{(-,+)}\phi_{T,l;3}^{(2)}
+D_{12}^{(-,+)}\phi_{T,l;2}^{(2)}
-D_{13}^{(-,+)}\phi_{T,l;1}^{(2)}\\
=&-E_{e_1-e_2}E_{e_2-e_3}\phi_{T,l;1}^{(2)}
+E_{e_1-e_2}(H_1+H_2-4)\phi_{T,l;2}^{(2)}\\
&-\{(H_1-l-3)(H_1+l-3)+E_{e_1-e_2}^2\}\phi_{T,l;3}^{(2)},\\
&\hspace{-4mm}-D_{21}^{(-,+)}\phi_{T,l;3}^{(2)}
+D_{22}^{(-,+)}\phi_{T,l;2}^{(2)}
-D_{23}^{(-,+)}\phi_{T,l;1}^{(2)}\\
=&-E_{e_2-e_3}(H_2+H_3-2+2\sqrt{-1}E_{2e_3})\phi_{T,l;1}^{(2)}\\
&+\{(H_2-l-2)(H_2+l-2)+E_{e_1-e_2}^2+E_{e_2-e_3}^2\}\phi_{T,l;2}^{(2)}
-E_{e_1-e_2}(H_1+H_2-6)\phi_{T,l;3}^{(2)},\\
&\hspace{-4mm}-D_{31}^{(-,+)}\phi_{T,l;3}^{(2)}
+D_{32}^{(-,+)}\phi_{T,l;2}^{(2)}
-D_{33}^{(-,+)}\phi_{T,l;1}^{(2)}\\
=&-\{(H_3-l-1-2\sqrt{-1}E_{2e_3})(H_3+l-1+2\sqrt{-1}E_{2e_3})
+E_{e_2-e_3}^2\}\phi_{T,l;1}^{(2)}\\
&+E_{e_2-e_3}(H_2+H_3-4-2\sqrt{-1}E_{2e_3})\phi_{T,l;2}^{(2)}
-E_{e_1-e_2}E_{e_2-e_3}\phi_{T,l;3}^{(2)},
\end{align*}
\begin{align*}
C_2\phi_{T,l;i}^{(2)}=
&\{(H_1+l-6-\delta_{3i})(H_1-l+\delta_{3i})
+(H_2+l-4-\delta_{2i})(H_2-l+\delta_{2i})\\
&\hphantom{=}+(H_3+l-2-\delta_{1i}+2\sqrt{-1}E_{2e_3})
(H_3-l+\delta_{1i}-2\sqrt{-1}E_{2e_3})\\
&\hphantom{=}+2E_{e_1-e_2}^2+2E_{e_2-e_3}^2\}\phi_{T,l;i}^{(2)}\\
&-2\delta_{1i}\{2\delta_{1i}\phi_{T,l;1}^{(2)}
-E_{e_2-e_3}\delta_{1i}\phi_{T,l;2}^{(2)}\}
-2\delta_{2i}\{
E_{e_2-e_3}\phi_{T,l;1}^{(2)}+\phi_{T,l;2}^{(2)}
-E_{e_1-e_2}\phi_{T,l;3}^{(2)}\}\\
&-2\delta_{3i}E_{e_1-e_2}\phi_{T,l;2}^{(2)},\\
C_4\phi_{T,l;i}^{(2)}=
&\Big\{
\{(H_2+l-3-\delta_{2i})(H_3+l-2-\delta_{1i}+2\sqrt{-1}E_{2e_3})
-E_{e_2-e_3}^2\}\\
&\hphantom{==}\times 
\{(H_2-l-1+\delta_{2i})(H_3-l+\delta_{1i}-2\sqrt{-1}E_{2e_3})
-E_{e_2-e_3}^2\}\\
&\hphantom{=}+(H_1+l-5-\delta_{3i})(H_3+l-2-\delta_{1i}+2\sqrt{-1}E_{2e_3})\\
&\hphantom{==}\times 
(H_1-l-1+\delta_{3i})(H_3-l+\delta_{1i}-2\sqrt{-1}E_{2e_3})\\
&\hphantom{=}+\{(H_1+l-5-\delta_{3i})(H_2+l-4-\delta_{2i})-E_{e_1-e_2}^2\}\\
&\hphantom{==}\times 
\{(H_1-l-1+\delta_{3i})(H_2-l+\delta_{2i})-E_{e_1-e_2}^2\}\\
&\hphantom{=}+2E_{e_1-e_2}^2
(H_3+l-2-\delta_{1i}+2\sqrt{-1}E_{2e_3})
(H_3-l+\delta_{1i}-2\sqrt{-1}E_{2e_3})\\
&\hphantom{=}+2E_{e_2-e_3}^2
\{(H_1+l-5-\delta_{3i})(H_1-l-1+\delta_{3i})+E_{e_1-e_2}^2\}
\Big\}\phi_{T,l;i}^{(2)}\\
&+2\delta_{1i}\Big\{
-\{(H_1+l-5)(H_1-l-1)+(H_2+l-3)(H_2-l-1)\\
&\hphantom{==}+2E_{e_1-e_2}^2+3E_{e_2-e_3}^2\}\phi_{T,l;1}^{(2)}\\
&\hphantom{=}+E_{e_2-e_3}
\{(H_1+l-5)(H_1-l-1)\\
&\hphantom{==}-(H_2-l)(H_3-l-2\sqrt{-1}E_{2e_3})\\
&\hphantom{==}-(H_2+H_3-2-2\sqrt{-1}E_{2e_3})+E_{e_1-e_2}^2+E_{e_2-e_3}^2
\}\phi_{T,l;2}^{(2)}\\
&\hphantom{=}+E_{e_1-e_2}E_{e_2-e_3}(H_1+H_2+H_3-l-3-2\sqrt{-1}E_{2e_3})
\phi_{T,l;3}^{(2)}
\Big\}\\
&+2\delta_{2i}\Big\{
-E_{e_2-e_3}\{(H_1+l-5)(H_1-l-1)\\
&\hphantom{==}-(H_2+l-3)(H_3+l-3+2\sqrt{-1}E_{2e_3})\\
&\hphantom{==}+H_2+H_3-4+2\sqrt{-1}E_{2e_3}
+E_{e_1-e_2}^2+E_{e_2-e_3}^2\}\phi_{T,l;1}^{(2)}\\
&\hphantom{=}-\{(H_3+l-2+2\sqrt{-1}E_{2e_3})(H_3-l-2\sqrt{-1}E_{2e_3})\\
&\hphantom{==}+2E_{e_1-e_2}^2+E_{e_2-e_3}^2\}
\phi_{T,l;2}^{(2)}\\
&\hphantom{=}+E_{e_1-e_2}\{-(H_1-l)(H_2-l)\\
&\hphantom{==}+(H_3+l-2+2\sqrt{-1}E_{2e_3})(H_3-l-2\sqrt{-1}E_{2e_3})\\
&\hphantom{==}+E_{e_1-e_2}^2+E_{e_2-e_3}^2\}
\phi_{T,l;3}^{(2)}
\Big\}\\
&+2\delta_{3i}\Big\{
-E_{e_1-e_2}E_{e_2-e_3}(H_1+H_2+H_3+l-9+2\sqrt{-1}E_{2e_3})
\phi_{T,l;1}^{(2)}\\
&\hphantom{=}+E_{e_1-e_2}\{(H_1+l-5)(H_2+l-5)\\
&\hphantom{==}-(H_3+l-2+2\sqrt{-1}E_{2e_3})(H_3-l-2\sqrt{-1}E_{2e_3})\\
&\hphantom{==}-E_{e_1-e_2}^2-E_{e_2-e_3}^2\}
\phi_{T,l;2}^{(2)}\Big\}, \\
C_6\phi_{T,l;i}^{(2)}=
&\{(H_1+l-4-\delta_{3i})(H_2+l-3-\delta_{2i})
(H_3+l-2-\delta_{1i}+2\sqrt{-1}E_{2e_3})\\
&\hphantom{=}-E_{e_2-e_3}^2(H_1+l-4-\delta_{3i})
-E_{e_1-e_2}^2(H_3+l-2-\delta_{1i}+2\sqrt{-1}E_{2e_3})\}\\
&\hphantom{==}\times \{(H_1-l-2+\delta_{3i})(H_2-l-1+\delta_{2i})
(H_3-l+\delta_{1i}-2\sqrt{-1}E_{2e_3})\\
&\hphantom{===}-E_{e_1-e_2}^2(H_3-l+\delta_{1i}-2\sqrt{-1}E_{2e_3})
-E_{e_2-e_3}^2(H_1-l-2+\delta_{3i})\}\phi_{T,l;i}^{(2)}\\
&+2\delta_{1i}\Big\{
-2E_{e_2-e_3}^2\{(H_1+l-4)(H_1-l-2)+E_{e_1-e_2}^2\}\phi_{T,l;1}^{(2)}\\
&\hphantom{=}-E_{e_2-e_3}\{(H_1+l-4)(H_1-l-2)\\
&\hphantom{===}\times (H_2-l)(H_3-l-2\sqrt{-1}E_{2e_3})\\
&\hphantom{==}-E_{e_1-e_2}^2(H_1+l-4)(H_3-l-2\sqrt{-1}E_{2e_3})\\
&\hphantom{==}-E_{e_2-e_3}^2(H_1+l-4)(H_1-l-2)\}\phi_{T,l;2}^{(2)}\\
&\hphantom{=}-E_{e_1-e_2}E_{e_2-e_3}\{(H_1-l-1)(H_2-l-1)
(H_3-l-2\sqrt{-1}E_{2e_3})\\
&\hphantom{==}-E_{e_1-e_2}^2(H_3-l-2\sqrt{-1}E_{2e_3})
-E_{e_2-e_3}^2(H_1-l-1)\}\phi_{T,l;3}^{(2)}
\Big\}\\
&+2\delta_{2i}\Big\{
E_{e_2-e_3}\{(H_1+l-4)(H_2+l-3)\\
&\hphantom{===}\times (H_3+l-3+2\sqrt{-1}E_{2e_3})(H_1-l-2)\\
&\hphantom{==}-E_{e_2-e_3}^2(H_1+l-4)(H_1-l-2)\\
&\hphantom{==}-E_{e_1-e_2}^2(H_3+l-3+2\sqrt{-1}E_{2e_3})(H_1-l)\}
\phi_{T,l;1}^{(2)}\\
&\hphantom{=}-2E_{e_1-e_2}^2(H_3+l-2+2\sqrt{-1}E_{2e_3})
(H_3-l-2\sqrt{-1}E_{2e_3})\phi_{T,l;2}^{(2)}\\
&\hphantom{==}-E_{e_1-e_2}(H_3+l-2+2\sqrt{-1}E_{2e_3})\\
&\hphantom{===}\times \{(H_1-l-1)(H_2-l-1)
(H_3-l-2\sqrt{-1}E_{2e_3})\\
&\hphantom{====}-E_{e_1-e_2}^2(H_3-l-2\sqrt{-1}E_{2e_3})
-E_{e_2-e_3}^2(H_1-l-1)\}\phi_{T,l;3}^{(2)}
\Big\} \\
&+2\delta_{3i}\Big\{
E_{e_1-e_2}E_{e_2-e_3}\{(H_1+l-4)(H_2+l-3)
(H_3+l-3+2\sqrt{-1}E_{2e_3})\\
&\hphantom{==}-E_{e_2-e_3}^2(H_1+l-4)
-E_{e_1-e_2}^2(H_3+l-3+2\sqrt{-1}E_{2e_3})\}
\phi_{T,l;1}^{(2)}\\
&\hphantom{=}+E_{e_1-e_2}\{(H_1+l-4)(H_2+l-4)
(H_3+l-2+2\sqrt{-1}E_{2e_3})\\
&\hphantom{==}-E_{e_2-e_3}^2(H_1+l-4)
-E_{e_1-e_2}^2(H_3+l-2+2\sqrt{-1}E_{2e_3})\}\\
&\hphantom{==}\times (H_3-l-2\sqrt{-1}E_{2e_3})\phi_{T,l;2}^{(2)}
\Big\}.
\end{align*}
\end{prop} 
\begin{proof}
The actions of the operators $C_{2i}\ (1\leq i\leq 2),\ 
D_{jk}^{(\pm ,\mp )}\ (1\leq j,k\leq 3)$ and $m_3(C_\pm )$
are obtained by direct computation from the expressions 
in normal order and Lemma \ref{lem:Kact_func}(i). 

We can summarize the explicit actions of the 
operators $C_{2i}\ (1\leq i\leq 3)$ and 
$D_{jk}^{(\pm ,\mp )}\ (1\leq j,k\leq 3)$ on the functions in 
Proposition \ref{prop:ex1}, \ref{prop:ex2}. 
The action of $C_6=m_3(C_+)m_3(C_-)$ is obtained by the composite 
of following actions of $m_3(C_+)$ and $m_3(C_-)$:\\
$\bullet $ the case of $(\sigma_1,\sigma_2,\sigma_3)=(0,0,0),(1,1,1)$.
\begin{align*}
m_3(C_+)\phi_{T,l-2}
=&\{(H_1+l-4)(H_2+l-3)(H_3+l-2+2\sqrt{-1}E_{2e_3})\\
&-E_{e_2-e_3}^2(H_1+l-4)
-E_{e_1-e_2}^2(H_3+l-2+2\sqrt{-1}E_{2e_3})
\}\phi_{T,l-2},\\
m_3(C_-)\phi_{T,l}
=&\{
(H_1-l-2)(H_2-l-1)
(H_3-l-2\sqrt{-1}E_{2e_3})\\
&-E_{e_2-e_3}^2(H_1-l-2)
-E_{e_1-e_2}^2(H_3-l-2\sqrt{-1}E_{2e_3})
\}\phi_{T,l}.
\end{align*}
$\bullet $ the case of $(\sigma_1,\sigma_2,\sigma_3)\neq (0,0,0),(1,1,1)$.
\begin{align*}
m_3(C_+)\phi_{T,l-2;i}^{(1)}
=&\{
(H_1+l-4+\delta_{1i})(H_2+l-3+\delta_{2i})
(H_3+l-2+\delta_{3i}+2\sqrt{-1}E_{2e_3})\\
&\hphantom{=}-E_{e_2-e_3}^2(H_1+l-4+\delta_{1i})
-E_{e_1-e_2}^2(H_3+l-2+\delta_{3i}\\
&\hphantom{=}+2\sqrt{-1}E_{2e_3})\}\phi_{T,l-2;i}^{(1)}\\
&-2\delta_{1i}\{
E_{e_1-e_2}(H_3+l-2+2\sqrt{-1}E_{2e_3})\phi_{T,l-2;2}^{(1)}
-E_{e_1-e_2}E_{e_2-e_3}\phi_{T,l-2;3}^{(1)}\}\\
&-2\delta_{2i}E_{e_2-e_3}
(H_1+l-4)\phi_{T,l-2;3}^{(1)},
\end{align*}
\begin{align*}
m_3(C_-)\phi_{T,l;i}^{(1)}
=&\{
(H_1-l-2-\delta_{1i})(H_2-l-1-\delta_{2i})
(H_3-l-\delta_{3i}-2\sqrt{-1}E_{2e_3})\\
&\hphantom{=}-E_{e_2-e_3}^2(H_1-l-2-\delta_{1i})
-E_{e_1-e_2}^2(H_3-l-\delta_{3i}-2\sqrt{-1}E_{2e_3})
\}\phi_{T,l;i}^{(1)}\\
&+2E_{e_1-e_2}(H_3-l-2\sqrt{-1}E_{2e_3})
\delta_{2i}\phi_{T,l;1}^{(1)}\\
&-2\delta_{3i}\{
E_{e_1-e_2}E_{e_2-e_3}\phi_{T,l;1}^{(1)}
-E_{e_2-e_3}(H_1-l-2)\phi_{T,l;2}^{(1)}\} ,
\end{align*}
and
\begin{align*}
m_3(C_+)\phi_{T,l-2;i}^{(2)}=
&\{(H_1+l-4-\delta_{3i})(H_2+l-3-\delta_{2i})
(H_3+l-2-\delta_{1i}+2\sqrt{-1}E_{2e_3})\\
&\hphantom{=}-E_{e_2-e_3}^2(H_1+l-4-\delta_{3i})\\
&\hphantom{=}-E_{e_1-e_2}^2(H_3+l-2-\delta_{1i}
+2\sqrt{-1}E_{2e_3})\}\phi_{T,l-2;i}^{(2)}\\
&-2\delta_{1i}\{E_{e_2-e_3}(H_1+l-4)\phi_{T,l-2;2}^{(2)}
+E_{e_1-e_2}E_{e_2-e_3}\phi_{T,l-2;3}^{(2)}\} \\
&-2\delta_{2i}E_{e_1-e_2}(H_3+l-2+2\sqrt{-1}E_{2e_3})\phi_{T,l-2;3}^{(2)},\\
m_3(C_-)\phi_{T,l;i}^{(2)}=
&\{(H_1-l-2+\delta_{3i})(H_2-l-1+\delta_{2i})
(H_3-l+1-1+\delta_{1i}-2\sqrt{-1}E_{2e_3})\\
&\hphantom{=}-E_{e_1-e_2}^2(H_3-l+\delta_{1i}-2\sqrt{-1}E_{2e_3})
-E_{e_2-e_3}^2(H_1-l-2+\delta_{3i})\}\phi_{T,l;i}^{(2)}\\
&+2\delta_{2i}E_{e_2-e_3}(H_1-l-2)\phi_{T,l;1}^{(2)}\\
&+2\delta_{3i}\{E_{e_1-e_2}E_{e_2-e_3}\phi_{T,l;1}^{(2)}
+E_{e_1-e_2}(H_3-l-2\sqrt{-1}E_{2e_3})\phi_{T,l;2}^{(2)}\}.
\end{align*}
\end{proof}

To state an explicit form of a holonomic system of 
partial differential equations satisfied by the 
$A$-radial part of each element in 
$\Wh (\pi_{(\sigma ,\nu )},\xi ,\tau )$ for peripheral $K$-type $\tau^*$, 
we introduce the coordinates 
$x=(x_1,x_2,x_3)$ and $y=(y_1,y_2,y_3)$ on $A$ defined by 
\begin{align*}
&x_1=\bkt{\pi c_{12}\frac{a_1}{a_2}}^2,\quad 
x_2=\bkt{\pi c_{23}\frac{a_2}{a_3}}^2,\quad 
x_3=4\pi c_3a_3^2,\\
&y_1=2\pi c_{12}\frac{a_1}{a_2},\quad 
y_2=2\pi c_{23}\frac{a_2}{a_3},\quad 
y_3=4\pi c_3a_3^2,
\end{align*}
for $\diag (a_1,a_2,a_3,a_1^{-1},a_2^{-1},a_3^{-1})\in A$. 
We use the coordinate $x$ for the case of 
$(\sigma_1,\sigma_2,\sigma_3)=(0,0,0),(1,1,1)$, 
and use the coordinate $y$ for the case of
$(\sigma_1,\sigma_2,\sigma_3)\neq (0,0,0),(1,1,1)$. 
Then we have following theorem. 
%check
\begin{thm}\label{th:submain}
\textit 
Let $T$ be an element of the space $\Isp{\xi ,\pi_{(\sigma ,\nu )} }$. 

\noindent (i) If $\sigma $ is a character of $M_{\min }$ 
such that $(\sigma_1,\sigma_2,\sigma_3)=(0,0,0)$ or $(1,1,1)$, then 
there exist multiplicity one $K$-types $\tau_{(l,l,l)}\ 
(l\equiv \varepsilon_\sigma \mod 2)$. 
For $l\equiv \varepsilon_\sigma \mod 2$, 
the Whittaker function 
$\Phi (T,\evec{(l,l,l)}{M_l})=\phi_{T,l}\otimes f(M_l)^*$ 
satisfies the following holonomic system of partial differential equations 
of rank 48:
\begin{align*}
&\{(2\partial_{x_1}+l-6)(2\partial_{x_1}-l)
+(-2\partial_{x_1}+2\partial_{x_2}+l-4)(-2\partial_{x_1}+2\partial_{x_2}-l)\\
&+(-2\partial_{x_2}+2\partial_{x_3}+l-2-x_3)
(-2\partial_{x_2}+2\partial_{x_3}-l+x_3)\\
&-8x_1-8x_2-\chi_{2,\sigma ,\nu ,(l,l,l)}
\} \phi_{T,l}=0,\\[5mm]
&\Big\{
\{(-2\partial_{x_1}+2\partial_{x_2}+l-3)
(-2\partial_{x_2}+2\partial_{x_3}+l-2-x_3)
+4x_2\}\\
&\hphantom{==} \times \{(-2\partial_{x_1}+2\partial_{x_2}-l-1)
(-2\partial_{x_2}+2\partial_{x_3}-l+x_3)
+4x_2\} \\
&\hphantom{=}+(2\partial_{x_1}+l-5)
(-2\partial_{x_2}+2\partial_{x_3}+l-2-x_3)\\
&\hphantom{==} \times 
(2\partial_{x_1}-l-1)(-2\partial_{x_2}+2\partial_{x_3}-l+x_3)\\
&\hphantom{=}+\{(2\partial_{x_1}+l-5)(-2\partial_{x_1}+2\partial_{x_2}+l-4)+4x_1\} \\
&\hphantom{==} \times \{(2\partial_{x_1}-l-1)(-2\partial_{x_1}+2\partial_{x_2}-l)+4x_1\} \\
&\hphantom{=}-8x_1(-2\partial_{x_2}+2\partial_{x_3}+l-2-x_3)
(-2\partial_{x_2}+2\partial_{x_3}-l+x_3)\\
&\hphantom{=}+32x_1x_2
-8x_2(2\partial_{x_1}+l-5)(2\partial_{x_1}-l-1)-\chi_{4,\sigma ,\nu ,(l,l,l)}
\Big\}\phi_{T,l}=0,\\[5mm]
&\Big\{ \{(2\partial_{x_1}+l-4)(-2\partial_{x_1}+2\partial_{x_2}+l-3)
(-2\partial_{x_2}+2\partial_{x_3}+l-2-x_3)\\
&\hphantom{=}+4x_2(2\partial_{x_1}+l-4)
+4x_1(-2\partial_{x_2}+2\partial_{x_3}+l-2-x_3)\}\\
&\times \{(2\partial_{x_1}-l-2)(-2\partial_{x_1}+2\partial_{x_2}-l-1)
(-2\partial_{x_2}+2\partial_{x_3}-l+x_3)\\
&\hphantom{=}+4x_2(2\partial_{x_1}-l-2)
+4x_1(-2\partial_{x_2}+2\partial_{x_3}-l+x_3)\}-\chi_{6,\sigma ,\nu ,(l,l,l)}
\Big\}
\phi_{T,l}=0.
\end{align*}

%check
\noindent (ii) If $\sigma $ is a character of $M_{\min }$ 
such that $(\sigma_1,\sigma_2,\sigma_3)\neq (0,0,0)$ or $(1,1,1)$, 
there exists multiplicity one $K$-types 
$\tau_{(l+1,l,l)},\ \tau_{(l,l,l-1)}\ (l\equiv \varepsilon_\sigma \mod 2)$.

For $l\equiv \varepsilon_\sigma \mod 2$, 
Whittaker function
$\Phi (T,S_{(l+1,l,l)}(M_l^{(\sigma ;1)}))=
\sum_{j=1}^3\phi_{T,l;j}^{(1)}\otimes f(M_{l;j}^{(1)})^*$ 
satisfy the following holonomic system of partial differential equations 
of rank 48:
\begin{align*}
&\{ 
(\partial_{y_1}+l-3)(\partial_{y_1}-l-3)-y_1^2
-\tilde{\chi}_{2,\sigma ,\nu ,(l+1,l,l)}
\} \phi_{T,l;1}^{(1)}\\
&+\sqrt{-1}y_1
(\partial_{y_2}-4)\phi_{T,l;2}^{(1)}
-y_1y_2\phi_{T,l;3}^{(1)}=0,\\[5mm]
&\sqrt{-1}y_1(\partial_{y_2}-6)
\phi_{T,l;1}^{(1)}\\
&+\{
(-\partial_{y_1}+\partial_{y_2}+l-2)(-\partial_{y_1}+\partial_{y_2}-l-2)-y_1^2
-y_2^2-\tilde{\chi}_{2,\sigma ,\nu ,(l+1,l,l)}
\} \phi_{T,l;2}^{(1)}\\
&+\sqrt{-1}y_2
(-\partial_{y_1}+2\partial_{y_3}-2+y_3)\phi_{T,l;3}^{(1)}=0,\\[5mm]
&-y_1y_2\phi_{T,l;1}^{(1)}
+\sqrt{-1}y_2
(-\partial_{y_1}+2\partial_{y_3}-4-y_3)\phi_{T,l;2}^{(1)}\\
&+\{
(-\partial_{y_2}+2\partial_{y_3}+l-1-y_3)
(-\partial_{y_2}+2\partial_{y_3}-l-1+y_3)
-y_2^2-\tilde{\chi}_{2,\sigma ,\nu ,(l+1,l,l)}
\} \phi_{T,l;3}^{(1)}=0,\\[5mm]
&\{(\partial_{y_1}+l-6+\delta_{1i})(\partial_{y_1}-l-\delta_{1i})
+(-\partial_{y_1}+\partial_{y_2}+l-4+\delta_{2i})
(-\partial_{y_1}+\partial_{y_2}-l-\delta_{2i})\\
&+(-\partial_{y_2}+2\partial_{y_3}+l-2+\delta_{3i}-y_3)
(-\partial_{y_2}+2\partial_{y_3}-l-\delta_{3i}+y_3)\\
&-2y_1^2-2y_2^2-\chi_{2,\sigma ,\nu ,(l+1,l,l)}
\} \phi_{T,l;i}^{(1)}
-2\delta_{1i}\{
2\phi_{T,l;1}^{(1)}
-\sqrt{-1}y_1\phi_{T,l;2}^{(1)}\} \\
&-2\delta_{2i}\{
\sqrt{-1}y_1\phi_{T,l;1}^{(1)}
+\phi_{T,l;2}^{(1)}
-\sqrt{-1}y_2\phi_{T,l;3}^{(1)}
\}
-2\delta_{3i}\sqrt{-1}y_2\phi_{T,l;2}^{(1)}=0,\\[5mm]
&\Big\{
\{(-\partial_{y_1}+\partial_{y_2}+l-3+\delta_{2i})
(-\partial_{y_2}+2\partial_{y_3}+l-2+\delta_{3i}-y_3)
+y_2^2
\}\\
&\hphantom{==} \times \{(-\partial_{y_1}+\partial_{y_2}-l-1-\delta_{2i})
(-\partial_{y_2}+2\partial_{y_3}-l-\delta_{3i}+y_3)
+y_2^2\} \\
&\hphantom{=}+(\partial_{y_1}+l-5+\delta_{1i})
(-\partial_{y_2}+2\partial_{y_3}+l-2+\delta_{3i}-y_3)\\
&\hphantom{==} \times 
(\partial_{y_1}-l-1-\delta_{1i})
(-\partial_{y_2}+2\partial_{y_3}-l-\delta_{3i}+y_3)\\
&\hphantom{=}+\{(\partial_{y_1}+l-5+\delta_{1i})
(-\partial_{y_1}+\partial_{y_2}+l-4+\delta_{2i})
+y_1^2\} \\
&\hphantom{==} 
\times \{(\partial_{y_1}-l-1-\delta_{1i})
(-\partial_{y_1}+\partial_{y_2}-l-\delta_{2i})
+y_1^2\} \\
&\hphantom{=}-2y_1^2
(-\partial_{y_2}+2\partial_{y_3}+l-2+\delta_{3i}-y_3)
(-\partial_{y_2}+2\partial_{y_3}-l-\delta_{3i}+y_3)\\
&\hphantom{=}+2y_1^2y_2^2
-2y_2^2
(\partial_{y_1}+l-5+\delta_{1i})(\partial_{y_1}-l-1-\delta_{1i})
-\chi_{4,\sigma ,\nu ,(l+1,l,l)}\Big\}\phi_{T,l;i}^{(1)}\\
&+2\delta_{1i}\Big\{
-\{(-\partial_{y_1}+\partial_{y_2}+l-4)(-\partial_{y_1}+\partial_{y_2}-l)\\
&\hphantom{=}+(-\partial_{y_2}+2\partial_{y_3}+l-2-y_3)
(-\partial_{y_2}+2\partial_{y_3}-l+y_3)-3y_1^2
-2y_2^2
\}\phi_{T,l;1}^{(1)}\\
&\hphantom{=}+\sqrt{-1}y_1
\{(-\partial_{y_2}+2\partial_{y_3}+l-2-y_3)
(-\partial_{y_2}+2\partial_{y_3}-l+y_3)\\
&\hphantom{=}-(\partial_{y_1}-l-1)(-\partial_{y_1}+\partial_{y_2}-l-1)
+\partial_{y_2}-6
-y_1^2
-y_2^2
\}\phi_{T,l;2}^{(1)}\\
&\hphantom{=}+y_1y_2
(2\partial_{y_3}-l-6+y_3)
\phi_{T,l;3}^{(1)}
\Big\}\\
&+2\delta_{2i}\Big\{
-\sqrt{-1}y_1
\{
(-\partial_{y_2}+2\partial_{y_3}+l-2-y_3)(-\partial_{y_2}+2\partial_{y_3}-l+y_3)\\
&\hphantom{=}-(\partial_{y_1}+l-4)(-\partial_{y_1}+\partial_{y_2}+l-4)
-(\partial_{y_2}-4)-y_1^2-y_2^2
\}\phi_{T,l;1}^{(1)}\\
&\hphantom{=}-\{
(\partial_{y_1}+l-5)(\partial_{y_1}-l-1)
-y_1^2
-2y_2^2
\}
\phi_{T,l;2}^{(1)}\\
&\hphantom{=}+\sqrt{-1}y_2
\{
(\partial_{y_1}+l-5)(\partial_{y_1}-l-1)\\
&\hphantom{=}-(-\partial_{y_1}+\partial_{y_2}-l-1)(-\partial_{y_2}+2\partial_{y_3}-l-1+y_3)
-y_1^2-y_2^2
\}\phi_{T,l;3}^{(1)}
\Big\}\\
&+2\delta_{3i}\Big\{
-y_1y_2
(2\partial_{y_3}+l-7-y_3)\phi_{T,l;1}^{(1)}\\
&\hphantom{=}+\sqrt{-1}y_2
\{(-\partial_{y_1}+\partial_{y_2}+l-2)(-\partial_{y_2}+2\partial_{y_3}+l-2-y_3)\\
&\hphantom{=}-(\partial_{y_1}+l-5)(\partial_{y_1}-l-1)+y_1^2+y_2^2
\}\phi_{T,l;2}^{(1)}
\Big\} =0,\\[5mm]
&\Big\{ \{(\partial_{y_1}+l-4+\delta_{1i})
(-\partial_{y_1}+\partial_{y_2}+l-3+\delta_{2i})
(-\partial_{y_2}+2\partial_{y_3}+l-2+\delta_{3i}-y_3)\\
&\hphantom{=}+y_2^2(\partial_{y_1}+l-4+\delta_{1i})
+y_1^2(-\partial_{y_2}+2\partial_{y_3}+l-2+\delta_{3i}-y_3)
\}\\
&\hphantom{=}\times \{(\partial_{y_1}-l-2-\delta_{1i})
(-\partial_{y_1}+\partial_{y_2}-l-1-\delta_{2i})
(-\partial_{y_2}+2\partial_{y_3}-l-\delta_{3i}+y_3)\\
&\hphantom{=}+y_2^2(\partial_{y_1}-l-2-\delta_{1i})
+y_1^2(-\partial_{y_2}+2\partial_{y_3}-l-\delta_{3i}+y_3)
\}-\chi_{6,\sigma ,\nu ,(l+1,l,l)}
\Big\} \phi_{T,l;i}^{(1)}\\
&+2\delta_{1i}\Big\{
2y_1^2\{(-\partial_{y_2}+2\partial_{y_3}+l-2-y_3)
(-\partial_{y_2}+2\partial_{y_3}-l+y_3)-y_2^2\}
\phi_{T,l;1}^{(1)}\\
&\hphantom{=}-\sqrt{-1}y_1
\{(-\partial_{y_2}+2\partial_{y_3}+l-2-y_3)
(\partial_{y_1}-l-2)(-\partial_{y_1}+\partial_{y_2}-l-2)
(-\partial_{y_2}+2\partial_{y_3}-l+y_3)\\
&\hphantom{=}+y_2^2(-\partial_{y_2}+2\partial_{y_3}+l-2-y_3)
(\partial_{y_1}-l-2)\\
&\hphantom{=}+y_1^2(-\partial_{y_2}+2\partial_{y_3}+l-2-y_3)
(-\partial_{y_2}+2\partial_{y_3}-l+y_3)\}\phi_{T,l;2}^{(1)}\\
&\hphantom{=}-y_1y_2
\{(\partial_{y_1}-l-2)(-\partial_{y_1}+\partial_{y_2}-l-1)
(-\partial_{y_2}+2\partial_{y_3}-l-1+y_3)\\
&\hphantom{=}+y_2^2(\partial_{y_1}-l-2)
+y_1^2(-\partial_{y_2}+2\partial_{y_3}-l-1+y_3)
\}\phi_{T,l;3}^{(1)}
\Big\}\\
&+2\delta_{2i}\Big\{
\sqrt{-1}y_1\{
(\partial_{y_1}+l-3)(-\partial_{y_1}+\partial_{y_2}+l-3)\\
&\hphantom{==}\times (-\partial_{y_2}+2\partial_{y_3}+l-2-y_3)
(-\partial_{y_2}+2\partial_{y_3}-l+y_3)\\
&\hphantom{=}+y_2^2(\partial_{y_1}+l-3)
(-\partial_{y_2}+2\partial_{y_3}-l-2+y_3)\\
&\hphantom{=}+y_1^2(-\partial_{y_2}+2\partial_{y_3}+l-2-y_3)
(-\partial_{y_2}+2\partial_{y_3}-l+y_3)\}
\phi_{T,l;1}^{(1)}\\
&\hphantom{=}+2y_2^2(\partial_{y_1}+l-4)
(\partial_{y_1}-l-2)\phi_{T,l;2}^{(1)}\\
&\hphantom{=}-\sqrt{-1}y_2(\partial_{y_1}+l-4)
\{(\partial_{y_1}-l-2)(-\partial_{y_1}+\partial_{y_2}-l-1)
(-\partial_{y_2}+2\partial_{y_3}-l-1+y_3)\\
&\hphantom{=}+y_2^2(\partial_{y_1}-l-2)
+y_1^2(-\partial_{y_2}+2\partial_{y_3}-l-1+y_3)
\}\phi_{T,l;3}^{(1)}
\Big\}\\
&+2\delta_{3i}\Big\{
y_1y_2\{
(\partial_{y_1}+l-3)(-\partial_{y_1}+\partial_{y_2}+l-3)
(-\partial_{y_2}+2\partial_{y_3}+l-2-y_3)\\
&\hphantom{=}+y_2^2(\partial_{y_1}+l-3)
+y_1^2(-\partial_{y_2}+2\partial_{y_3}+l-2-y_3)
\}\phi_{T,l;1}^{(1)}\\
&\hphantom{=}+\sqrt{-1}y_2\{
(\partial_{y_1}+l-4)(-\partial_{y_1}+\partial_{y_2}+l-2)
(-\partial_{y_2}+2\partial_{y_3}+l-2-y_3)\\
&\hphantom{=}+y_2^2(\partial_{y_1}+l-4)
+y_1^2(-\partial_{y_2}+2\partial_{y_3}+l-2-y_3)
\}(\partial_{y_1}-l-2)\phi_{T,l;2}^{(1)}
\Big\}=0,
\end{align*}
for $i=1,2,3$.

For $l\equiv \varepsilon_\sigma \mod 2$, 
Whittaker function
$\Phi (T,S_{(l,l,l-1)}(M_l^{(\sigma ;2)}))
=\sum_{j=1}^3\phi_{T,l;j}^{(2)}\otimes f(M_{l;j}^{(2)})^*$
 satisfy the following holonomic system of partial differential equations 
of rank 48:
\begin{align*}
&y_1y_2\phi_{T,l;1}^{(2)}
+\sqrt{-1}y_1(\partial_{y_2}-4)\phi_{T,l;2}^{(2)}\\
&-\{(\partial_{y_1}-l-3)(\partial_{y_1}+l-3)-y_1^2
-\tilde{\chi}_{2,\sigma ,\nu ,(l,l,l-1)}
\}\phi_{T,l;3}^{(2)}=0,\\[5mm]
&-\sqrt{-1}y_2(-\partial_{y_1}+2\partial_{y_3}-2-y_3)\phi_{T,l;1}^{(2)}\\
&+\{(-\partial_{y_1}+\partial_{y_2}-l-2)(-\partial_{y_1}+\partial_{y_2}+l-2)
-y_1^2-y_2^2-\tilde{\chi}_{2,\sigma ,\nu ,(l,l,l-1)}\}\phi_{T,l;2}^{(2)}\\
&-\sqrt{-1}y_1(\partial_{y_2}-6)\phi_{T,l;3}^{(2)}=0,\\[5mm]
&-\{(-\partial_{y_2}+2\partial_{y_3}-l-1+y_3)
(-\partial_{y_2}+2\partial_{y_3}+l-1-y_3)
-y_2^2-\tilde{\chi}_{2,\sigma ,\nu ,(l,l,l-1)}\}\phi_{T,l;1}^{(2)}\\
&+\sqrt{-1}y_2(-\partial_{y_1}+2\partial_{y_3}-4+y_3)\phi_{T,l;2}^{(2)}
+y_1y_2\phi_{T,l;3}^{(2)},\\[5mm]
&\{(\partial_{y_1}+l-6-\delta_{3i})(\partial_{y_1}-l+\delta_{3i})\\
&\hphantom{=}+(-\partial_{y_1}+\partial_{y_2}+l-4-\delta_{2i})
(-\partial_{y_1}+\partial_{y_2}-l+\delta_{2i})\\
&\hphantom{=}+(-\partial_{y_2}+2\partial_{y_3}+l-2-\delta_{1i}-y_3)
(-\partial_{y_2}+2\partial_{y_3}-l+\delta_{1i}+y_3)\\
&\hphantom{=}-2y_1^2-2y_2^2-\chi_{2,\sigma ,\nu ,(l,l,l-1)}\}
\phi_{T,l;i}^{(2)}\\
&-2\delta_{1i}\{2\phi_{T,l;1}^{(2)}
-\sqrt{-1}y_2\phi_{T,l;2}^{(2)}\}\\
&-2\delta_{2i}\{
\sqrt{-1}y_2\phi_{T,l;1}^{(2)}+\phi_{T,l;2}^{(2)}
-\sqrt{-1}y_1\phi_{T,l;3}^{(2)}\}
-2\delta_{3i}\sqrt{-1}y_1\phi_{T,l;2}^{(2)}=0,\\[5mm]
&\Big\{
\{(-\partial_{y_1}+\partial_{y_2}+l-3-\delta_{2i})
(-\partial_{y_2}+2\partial_{y_3}+l-2-\delta_{1i}-y_3)
+y_2^2\}\\
&\hphantom{==}\times 
\{(-\partial_{y_1}+\partial_{y_2}-l-1+\delta_{2i})
(-\partial_{y_2}+2\partial_{y_3}-l+\delta_{1i}+y_3)
+y_2^2\}\\
&\hphantom{=}+(\partial_{y_1}+l-5-\delta_{3i})
(-\partial_{y_2}+2\partial_{y_3}+l-2-\delta_{1i}-y_3)\\
&\hphantom{==}\times 
(\partial_{y_1}-l-1+\delta_{3i})
(-\partial_{y_2}+2\partial_{y_3}-l+\delta_{1i}+y_3)\\
&\hphantom{=}+\{(\partial_{y_1}+l-5-\delta_{3i})
(-\partial_{y_1}+\partial_{y_2}+l-4-\delta_{2i})+y_1^2\}\\
&\hphantom{==}\times 
\{(\partial_{y_1}-l-1+\delta_{3i})
(-\partial_{y_1}+\partial_{y_2}-l+\delta_{2i})+y_1^2\}\\
&\hphantom{=}-2y_1^2
(-\partial_{y_2}+2\partial_{y_3}+l-2-\delta_{1i}-y_3)
(-\partial_{y_2}+2\partial_{y_3}-l+\delta_{1i}+y_3)\\
&\hphantom{=}-2y_2^2
\{(\partial_{y_1}+l-5-\delta_{3i})(\partial_{y_1}-l-1+\delta_{3i})-y_1^2\}
-\chi_{4,\sigma ,\nu ,(l,l,l-1)}\Big\}\phi_{T,l;i}^{(2)}\\
&+2\delta_{1i}\Big\{
-\{(\partial_{y_1}+l-5)(\partial_{y_1}-l-1)+(-\partial_{y_1}+\partial_{y_2}+l-3)(-\partial_{y_1}+\partial_{y_2}-l-1)\\
&\hphantom{==}-2y_1^2-3y_2^2\}\phi_{T,l;1}^{(2)}\\
&\hphantom{=}+\sqrt{-1}y_2
\{(\partial_{y_1}+l-5)(\partial_{y_1}-l-1)-(-\partial_{y_1}+\partial_{y_2}-l)(-\partial_{y_2}+2\partial_{y_3}-l+y_3)\\
&\hphantom{==}-(-\partial_{y_1}+2\partial_{y_3}-2+y_3)-y_1^2-y_2^2
\}\phi_{T,l;2}^{(2)}\\
&\hphantom{=}-y_1y_2(2\partial_{y_3}-l-3+y_3)
\phi_{T,l;3}^{(2)}
\Big\}\\
&+2\delta_{2i}\Big\{
-\sqrt{-1}y_2\{(\partial_{y_1}+l-5)(\partial_{y_1}-l-1)
-(-\partial_{y_1}+\partial_{y_2}+l-3)
(-\partial_{y_2}+2\partial_{y_3}+l-3-y_3)\\
&\hphantom{==}-\partial_{y_1}+2\partial_{y_3}-4-y_3
-y_1^2-y_2^2\}\phi_{T,l;1}^{(2)}\\
&\hphantom{=}-\{(-\partial_{y_2}+2\partial_{y_3}+l-2-y_3)(-\partial_{y_2}+2\partial_{y_3}-l+y_3)
-2y_1^2-y_2^2\}
\phi_{T,l;2}^{(2)}\\
&\hphantom{=}+\sqrt{-1}y_1\{-(\partial_{y_1}-l)
(-\partial_{y_1}+\partial_{y_2}-l)
\\
&\hphantom{==}+(-\partial_{y_2}+2\partial_{y_3}+l-2-y_3)
(-\partial_{y_2}+2\partial_{y_3}-l+y_3)-y_1^2-y_2^2\}
\phi_{T,l;3}^{(2)}
\Big\}\\
&+2\delta_{3i}\Big\{
y_1y_2(2\partial_{y_3}+l-9-y_3)
\phi_{T,l;1}^{(2)}\\
&\hphantom{=}+\sqrt{-1}y_1\{(\partial_{y_1}+l-5)
(-\partial_{y_1}+\partial_{y_2}+l-5)\\
&\hphantom{==}-(-\partial_{y_2}+2\partial_{y_3}+l-2-y_3)
(-\partial_{y_2}+2\partial_{y_3}-l+y_3)+y_1^2+y_2^2\}
\phi_{T,l;2}^{(2)}\Big\}=0,\\[5mm]
&\Big\{ \{(\partial_{y_1}+l-4-\delta_{3i})
(-\partial_{y_1}+\partial_{y_2}+l-3-\delta_{2i})
(-\partial_{y_2}+2\partial_{y_3}+l-2-\delta_{1i}-y_3)\\
&\hphantom{=}+y_2^2(\partial_{y_1}+l-4-\delta_{3i})
+y_1^2(-\partial_{y_2}+2\partial_{y_3}+l-2-\delta_{1i}-y_3)\}\\
&\hphantom{=}\times \{(\partial_{y_1}-l-2+\delta_{3i})
(-\partial_{y_1}+\partial_{y_2}-l-1+\delta_{2i})
(-\partial_{y_2}+2\partial_{y_3}-l+\delta_{1i}+y_3)\\
&\hphantom{==}+y_1^2(-\partial_{y_2}+2\partial_{y_3}-l+\delta_{1i}+y_3)
+y_2^2(\partial_{y_1}-l-2+\delta_{3i})\}
-\chi_{6,\sigma ,\nu ,(l,l,l-1)}\Big\}\phi_{T,l;i}^{(2)}\\
&+2\delta_{1i}\Big\{
2y_2^2\{(\partial_{y_1}+l-4)(\partial_{y_1}-l-2)-y_1^2\}\phi_{T,l;1}^{(2)}\\
&\hphantom{=}-\sqrt{-1}y_2\{(\partial_{y_1}+l-4)(\partial_{y_1}-l-2)\\
&\hphantom{===}\times (-\partial_{y_1}+\partial_{y_2}-l)
(-\partial_{y_2}+2\partial_{y_3}-l+y_3)\\
&\hphantom{==}+y_1^2(\partial_{y_1}+l-4)
(-\partial_{y_2}+2\partial_{y_3}-l+y_3)\\
&\hphantom{==}+y_2^2(\partial_{y_1}+l-4)
(\partial_{y_1}-l-2)\}\phi_{T,l;2}^{(2)}\\
&\hphantom{=}+y_1y_2\{(\partial_{y_1}-l-1)(-\partial_{y_1}+\partial_{y_2}-l-1)
(-\partial_{y_2}+2\partial_{y_3}-l+y_3)\\
&\hphantom{==}+y_1^2(-\partial_{y_2}+2\partial_{y_3}-l+y_3)
+y_2^2(\partial_{y_1}-l-1)\}\phi_{T,l;3}^{(2)}
\Big\}\\
&+2\delta_{2i}\Big\{
\sqrt{-1}y_2\{(\partial_{y_1}+l-4)(-\partial_{y_1}+\partial_{y_2}+l-3)\\
&\hphantom{===}\times (-\partial_{y_2}+2\partial_{y_3}+l-3-y_3)
(\partial_{y_1}-l-2)\\
&\hphantom{==}+y_2^2(\partial_{y_1}+l-4)(\partial_{y_1}-l-2)\\
&\hphantom{==}+y_1^2(-\partial_{y_2}+2\partial_{y_3}+l-3-y_3)
(\partial_{y_1}-l)\}
\phi_{T,l;1}^{(2)}\\
&\hphantom{=}+2y_1^2(-\partial_{y_2}+2\partial_{y_3}+l-2-y_3)
(-\partial_{y_2}+2\partial_{y_3}-l+y_3)\phi_{T,l;2}^{(2)}\\
&\hphantom{==}-\sqrt{-1}y_1(-\partial_{y_2}+2\partial_{y_3}+l-2-y_3)\\
&\hphantom{===}\times \{(\partial_{y_1}-l-1)
(-\partial_{y_1}+\partial_{y_2}-l-1)
(-\partial_{y_2}+2\partial_{y_3}-l+y_3)\\
&\hphantom{==}+y_1^2(-\partial_{y_2}+2\partial_{y_3}-l+y_3)
+y_2^2(\partial_{y_1}-l-1)\}\phi_{T,l;3}^{(2)}
\Big\} \\
&+2\delta_{3i}\Big\{
-y_1y_2\{(\partial_{y_1}+l-4)(-\partial_{y_1}+\partial_{y_2}+l-3)
(-\partial_{y_2}+2\partial_{y_3}+l-3-y_3)\\
&\hphantom{==}+y_2^2(\partial_{y_1}+l-4)
+y_1^2(-\partial_{y_2}+2\partial_{y_3}+l-3-y_3)\}
\phi_{T,l;1}^{(2)}\\
&\hphantom{=}+\sqrt{-1}y_1\{(\partial_{y_1}+l-4)
(-\partial_{y_1}+\partial_{y_2}+l-4)
(-\partial_{y_2}+2\partial_{y_3}+l-2-y_3)\\
&\hphantom{==}+y_2^2(\partial_{y_1}+l-4)
+y_1^2(-\partial_{y_2}+2\partial_{y_3}+l-2-y_3)\}
(-\partial_{y_2}+2\partial_{y_3}-l+y_3)\phi_{T,l;2}^{(2)}
\Big\}=0,
\end{align*}
for $i=1,2,3$.
\end{thm}
\begin{proof}
We obtain the assertion from Proposition \ref{prop:ex1}, 
\ref{prop:ex2} and \ref{prop:ex3}.
\end{proof}
From the results of Kostant (\cite{MR507800} Theorem 6.8.1), it follows that 
the dimension of the intertwining space $\Isp{\xi ,\pi }$ is $48$. 
Hence, for a multiplicity one $K$-type $\tau$, the dimension of the 
space $\Wh (\pi ,\xi ,\tau )$ of Whittaker functions is also $48$. 
Therefore, every solution of the holonomic systems in the Theorem 
\ref{th:submain} gives an element of in $\Wh (\pi ,\xi ,\tau )|_{A_{\min}}$.

\nocite{pre_Switching_engine_2003}
\def\cprime{$'$}

%\bibliography{bib}

\end{document}